\documentclass[a4paper,12pt,twoside,openright]{report}
\pdfoutput=1

\newif\ifprint

\usepackage[
  top=2cm,
  bottom=4cm,
  heightrounded,
\ifprint
  left=2.5cm,
  right=2.5cm,
  bindingoffset=1.0cm
\else
  left=3.0cm,
  right=3.0cm,
  bindingoffset=0.0cm
\fi
]{geometry}

\usepackage{emptypage}

\let\subsubsection\subsection
\let\subsection\section
\let\section\chapter

\usepackage{mystyle}

\ifprint
\hypersetup{hidelinks}
\fi

\makeindex[intoc]
\makeglossaries
\shorthandon{"} %
\newglossarystyle{myglossarystyle}{%
  \setglossarystyle{long-name-desc-loc}%
  
  \settowidth{\glspagelistwidth}{99, 99, 99} %
  \glslongextraSetWidest{\ensuremath{\CComon(\Cvar)}} %
}

\makeatletter
\newcounter{@glossaryGroupCounter}
\def\makeGlossaryGroup#1#2{%
  \stepcounter{@glossaryGroupCounter}%
  \expandafter\edef\csname #1\endcsname{%
    \Alph{@glossaryGroupCounter}%
  }%
  \@namedef{\Alph{@glossaryGroupCounter}groupname}{#2}%
}
\makeatother

\makeGlossaryGroup{CategoryGroup}{Categories}
\makeGlossaryGroup{FunctorGroup}{Functors}
\makeGlossaryGroup{MapGroup}{Maps}
\makeGlossaryGroup{ObjectGroup}{Objects}
\makeGlossaryGroup{RelationGroup}{Relations}
\makeGlossaryGroup{OtherGroup}{Miscellaneous}

\newglossaryentry{Set}{type=symbols,sort={\CategoryGroup Set},name={\ensuremath{\Set}},description={Sets and functions}}
\newglossaryentry{Pos}{type=symbols,sort={\CategoryGroup Pos},name={\ensuremath{\Pos}},description={Posets and monotone functions}}
\newglossaryentry{Rng}{type=symbols,sort={\CategoryGroup CRing},name={\ensuremath{\Rng}},description={(Commutative) rings and ring homomorphisms}}
\newglossaryentry{ModR}{type=symbols,sort={\CategoryGroup Mod R},name={\ensuremath{\Mod{R}}},description={Modules over the ring $R$ and module homomorphisms}} %
\newglossaryentry{JSLat}{type=symbols,sort={\CategoryGroup SLat J},name={\ensuremath{\SLat{\vee}}},description={Join-semilattices and semilattice homomorphisms}} %
\newglossaryentry{DLat}{type=symbols,sort={\CategoryGroup DLat},name={\ensuremath{\DLat}},description={Distributive lattices and lattice homomorphisms}}
\newglossaryentry{Dcpo}{type=symbols,sort={\CategoryGroup DCPO},name={\ensuremath{\Dcpo}},description={Dcpos and dcpo momorphisms}}
\newglossaryentry{Frm}{type=symbols,sort={\CategoryGroup Frm},name={\ensuremath{\Frm}},description={Frames and frame homomorphisms}}
\newglossaryentry{Loc}{type=symbols,sort={\CategoryGroup Loc},name={\ensuremath{\Loc}},description={Locales --- the opposite of the category of frames}}
\newglossaryentry{OLoc}{type=symbols,sort={\CategoryGroup OLoc},name={\ensuremath{\OLoc}},description={Overt locales and locale morphisms}}
\newglossaryentry{Top}{type=symbols,sort={\CategoryGroup Top},name={\ensuremath{\Top}},description={Topological spaces and continuous functions}}
\newglossaryentry{Sup}{type=symbols,sort={\CategoryGroup Sup},name={\ensuremath{\Sup}},description={Suplattices and suplattice homomorphisms}}
\newglossaryentry{Inf}{type=symbols,sort={\CategoryGroup Inf},name={\ensuremath{\Inf}},description={Inf"|lattices and inf"|lattice homomorphisms}}
\newglossaryentry{QuantTwoSided}{type=symbols,sort={\CategoryGroup Quant Top},name={\ensuremath{\Quant}},description={Two-sided quantales and quantale homomorphisms}}
\newglossaryentry{QuantGeneral}{type=symbols,sort={\CategoryGroup Quant},name={\ensuremath{\QuantGeneral}},description={(Commutative) quantales and quantale homomorphisms}}
\newglossaryentry{ModQ}{type=symbols,sort={\CategoryGroup Mod Q},name={\ensuremath{\Mod{Q}}},description={Quantale modules over the quantale $Q$ and quantale module homomorphisms}} %
\newglossaryentry{CComonC}{type=symbols,sort={\CategoryGroup CComon C},name={\ensuremath{\CComon(\Cvar)}},description={Cocommutative comonoids in a monoidal category $\Cvar$}}

\glsxtrnewsymbol[sort={\FunctorGroup 1 18 Sigma},description={Sends a frame to its topological space of points}]{points}{\ensuremath{\points}}
\glsxtrnewsymbol[sort={\FunctorGroup S},description={Sends a locale to its lattice of sublocales. The map $\Sloc f$ takes preimages and $(\Sloc f)_!$ takes images under $f$.}]{Sloc}{\ensuremath{\Sloc}}
\glsxtrnewsymbol[sort={\FunctorGroup SubOW},description={Sends a locale to its suplattice of overt weakly closed sublocales. The map $\SubOW(f)$ takes the weak closure of the image of a sublocale under $f$.}]{SubOW}{\ensuremath{\SubOW}}
\glsxtrnewsymbol[sort={\FunctorGroup Idl},description={Sends a (localic) semiring to its quantale of (overt, weakly closed) ideals}]{Idl}{\ensuremath{\Idl}}
\glsxtrnewsymbol[sort={\FunctorGroup Rad},description={Sends a (localic) semiring to its quantale of (overt, weakly closed) radical ideals}]{Rad}{\ensuremath{\Rad}}
\glsxtrnewsymbol[sort={\FunctorGroup M},description={Sends a (localic) monoid or semiring to its quantale of (overt, weakly closed) monoid ideals}]{MonIdl}{\ensuremath{\MonIdl}}
\glsxtrnewsymbol[sort={\FunctorGroup S 2},description={Sends a localic monoid or semiring to its frame of saturated opens}]{Sats}{\ensuremath{\Sats}} %
\glsxtrnewsymbol[sort={\FunctorGroup OPAIR},description={Sends a locale to the set of open prime anti-ideals of the localic semiring $R$ fibred over it}]{OPAIR}{\ensuremath{\OPAI_R}}
\glsxtrnewsymbol[sort={\FunctorGroup OPAIR 2},description={Sends a quantale to the set of (quantic) open prime anti-ideals of the localic semiring $R$ fibred over it}]{OPAIRoverline}{\ensuremath{\overline{\OPAI}_R}}
\glsxtrnewsymbol[sort={\FunctorGroup OPMAI},description={Sends a quantale to the set of open prime monoid anti-ideals of the localic monoid $M$ fibred over it}]{OPMAIRoverline}{\ensuremath{\overline{\OPMAI}_M}}
\glsxtrnewsymbol[sort={\FunctorGroup Spec},description={A partial functor which gives the (localic) spectrum of a localic semiring when it exists}]{SpecR}{\ensuremath{\Spec}}
\glsxtrnewsymbol[sort={\FunctorGroup I},description={Sends a distributive lattice, join-semilattice or two-sided idempotent semiring to its frame, suplattice or quantale of order ideals}]{I}{\ensuremath{\I}}
\glsxtrnewsymbol[sort={\FunctorGroup D},description={Sends a poset to its suplattice of downsets}]{D}{\ensuremath{\D}}
\glsxtrnewsymbol[sort={\FunctorGroup O},description={Sends a space or locale to its frame of opens}]{O}{\ensuremath{\O}} %
\glsxtrnewsymbol[sort={\FunctorGroup L},description={Sends a quantale to its universal localic quotient}]{locrefl}{\ensuremath{\locrefl}}
\glsxtrnewsymbol[sort={\FunctorGroup hom},description={Internal hom in the category of suplattices}]{hom}{\ensuremath{\hom}}

\glsxtrnewsymbol[sort={\ObjectGroup 1 24 Omega},description={Frame of truth values}]{Omega}{\ensuremath{\Omega}}
\glsxtrnewsymbol[sort={\ObjectGroup S},description={Sierpiński space}]{Srpnsk}{\ensuremath{\Srpnsk}}
\glsxtrnewsymbol[sort={\ObjectGroup N},description={Free two-sided quantale on one generator}]{Nvar}{\ensuremath{\Nvar}}
\glsxtrnewsymbol[sort={\ObjectGroup D},description={Quantalic first-order infinitesimal interval --- the two-sided quantale presented by $\langle \epsilon \mid \epsilon^2 = 0\rangle$}]{Dvar}{\ensuremath{\Dvar}} %

\glsxtrnewsymbol[sort={\MapGroup 0 100 exclamation},description={Unique morphism from an initial object or to a terminal object; especially, the unique frame map from $\Omega$ to a given frame}]{exclamation}{\ensuremath{!}}
\glsxtrnewsymbol[sort={\MapGroup 0 150 exists},description={Left adjoint to the frame map $!\colon \Omega \to L$ for an overt frame $L$}]{exists}{\ensuremath{\exists}}
\glsxtrnewsymbol[sort={\MapGroup 0 150 existsW},description={Suplattice homomorphism into $\Omega$ associated with the overt weakly closed sublocale $W$}]{existsW}{\ensuremath{\exists_W}}
\glsxtrnewsymbol[sort={\MapGroup 0 400 nabla},description={Codiagonal map on the frame $\Omega$ --- sends $p \oplus q$ to $p \wedge q$}]{nabla}{\ensuremath{\nabla}}
\glsxtrnewsymbol[sort={\MapGroup 1 16 pi},description={Sends an overt weakly closed sublocale $V$ to the largest open contained in all saturated opens which meet $V$}]{stabpos}{\ensuremath{\stabpos}} %
\glsxtrnewsymbol[sort={\MapGroup 1 03 gamma},description={Sends an open $a$ to the largest overt weakly closed sublocale contained in all monoid ideals which meet $a$}]{jointcoz}{\ensuremath{\jointcoz}} %
\glsxtrnewsymbol[sort={\MapGroup 1 16 pi12},description={Product projections}]{pi12}{\ensuremath{\pi_1,\, \pi_2}}
\glsxtrnewsymbol[sort={\MapGroup 1 09 iota12},description={Coproduct injections}]{iota_12}{\ensuremath{\iota_1,\, \iota_2}}
\glsxtrnewsymbol[sort={\MapGroup 1 09 iota01},description={Canonical maps from a quantale $Q$ into ${}^\Srpnsk Q$}]{iota_01}{\ensuremath{\iota_0,\, \iota_1}}
\glsxtrnewsymbol[sort={\MapGroup 1 09 iota},description={Canonical quantale map from a quantale $Q$ to ${}^\Dvar Q$}]{iota}{\ensuremath{\iota}}
\glsxtrnewsymbol[sort={\MapGroup d},description={Canonical $\iota$-derivation from a quantale $Q$ into ${}^\Dvar Q$}]{dx}{\ensuremath{\d}}
\glsxtrnewsymbol[sort={\MapGroup 1 05 epsilon0},description={Frame map corresponding to the additive identity of a localic semiring}]{epsilon0}{\ensuremath{\epsilon_0}}
\glsxtrnewsymbol[sort={\MapGroup 1 05 epsilon1},description={Frame map corresponding to the multiplicative identity of a localic semiring or monoid}]{epsilon1}{\ensuremath{\epsilon_1}}
\glsxtrnewsymbol[sort={\MapGroup 1 12 mu 1},description={Frame map corresponding to the multiplication in a localic semiring or monoid}]{mutimes}{\ensuremath{\mu_\times}}
\glsxtrnewsymbol[sort={\MapGroup 1 12 mu 0},description={Frame map corresponding to the addition in a localic semiring}]{muplus}{\ensuremath{\mu_+}}
\glsxtrnewsymbol[sort={\MapGroup 1 17 rho},description={Unit of the localic reflection of a quantale --- the universal localic quotient map}]{rad}{\ensuremath{\rad}}
\glsxtrnewsymbol[sort={\MapGroup 1 07 eta},description={Unit for a dualisable object or coexponential adjunction}]{eta}{\ensuremath{\eta}}
\glsxtrnewsymbol[sort={\MapGroup 1 05 epsilon},description={Counit for a dualisable object. Also used for the generator of $\Dvar$.}]{epsilon}{\ensuremath{\epsilon}}
\glsxtrnewsymbol[sort={\MapGroup s},description={Natural quotient map from $\MonIdl R$ to $\Idl(R)$ for a localic semiring $R$}]{addquot}{\ensuremath{\addquot}}
\glsxtrnewsymbol[sort={\MapGroup 1 10 kappa},description={Inclusion of $\Sats M$ into $\O M$ for a commutative localic monoid $M$}]{inclSat}{\ensuremath{\inclSat}}
\glsxtrnewsymbol[sort={\MapGroup ev},description={Evaluation map for an exponential object}]{ev}{\ensuremath{\ev}} %
\glsxtrnewsymbol[sort={\MapGroup osat},description={Open saturation operator in a topological semiring}]{osat}{\ensuremath{\osat}}
\glsxtrnewsymbol[sort={\MapGroup dQ},description={Canonical comparison map from $\locrefl({}^\Dvar Q)$ to $\locrefl({}^\Srpnsk Q)$ for the two-sided quantale $Q$}]{dfrak}{\ensuremath{\dfrak_Q}}
\glsxtrnewsymbol[sort={\MapGroup j},description={Sends downsets in a suplattice to their joins}]{joinmap}{\ensuremath{\joinmap}}

\glsxtrnewsymbol[sort={\RelationGroup 0 100},description={Indicates that two subsets meet each other or that an overt sublocale meets an open}]{between}{\ensuremath{\between}}
\glsxtrnewsymbol[sort={\RelationGroup 0 201},description={Way-below relation}]{ll}{\ensuremath{\ll}}
\glsxtrnewsymbol[sort={\RelationGroup 0 202},description={Totally-below relation}]{lll}{\ensuremath{\lll}}
\glsxtrnewsymbol[sort={\RelationGroup 0 203},description={A certain transitive relation on the frame of saturated elements of a localic semiring}]{blacktriangleleft}{\ensuremath{\blacktriangleleft}} %
\glsxtrnewsymbol[sort={\RelationGroup 0 204},description={A certain transitive relation on the quantale of monoid ideals of a localic semiring}]{triangleleft}{\ensuremath{\vartriangleleft}} %

\glsxtrnewsymbol[sort={\OtherGroup 0 0600},description={Indicates the open $a$ is positive}]{agt0}{\ensuremath{a > 0}} %

\glsxtrnewsymbol[sort={\OtherGroup 0 1001},description={Principal downset/upset generated by $a$}]{downarrow}{\ensuremath{{\downarrow} a,\, {\uparrow} a}}
\glsxtrnewsymbol[sort={\OtherGroup 0 1003},description={Set of elements way below/above $a$}]{twoheaddownarrow}{\ensuremath{\twoheaddownarrow a,\, \twoheaduparrow a}}
\glsxtrnewsymbol[sort={\OtherGroup 0 1005},description={Set of elements totally below $a$}]{Downarrow}{\ensuremath{{\Downarrow} a}}
\glsxtrnewsymbol[sort={\OtherGroup 0 1010},description={Implication operation in a quantale or Heyting algebra}]{heyting}{\ensuremath{\heyting}}

\glsxtrnewsymbol[sort={\OtherGroup 0 1201},description={Supremum of the directed set $D$}]{dirsupD}{\ensuremath{\dirsup D}}

\glsxtrnewsymbol[sort={\OtherGroup 0 1601},description={Sequent indicating that the formulae $\chi$ and $\psi$ together entail the formula $\phi$ in the context $\vec{x}$}]{sequent}{\ensuremath{\chi,\psi \vdash_{\vec{x}} \phi}} %

\glsxtrnewsymbol[sort={\OtherGroup 0 3101},description={Pairing of the morphisms $f$ and $g$. Also used for the ideal generated by the semiring elements $f$ and $g$.}]{bracketsfg}{\ensuremath{(f,g)}} %
\glsxtrnewsymbol[sort={\OtherGroup 0 3103},description={Set of natural numbers strictly less than $n$}]{sqbracketsN}{\ensuremath{[n]}}
\glsxtrnewsymbol[sort={\OtherGroup 0 3105},description={Truth value of the formula $\phi$}]{llrrbrackets}{\ensuremath{\llbracket \phi \rrbracket}}
\glsxtrnewsymbol[sort={\OtherGroup 0 3111},description={The subalgebra, congruence or ideal generated by the subset $S$}]{langleSrangle}{\ensuremath{\langle S \rangle}}
\glsxtrnewsymbol[sort={\OtherGroup 0 3112},description={Free algebra (of some understood kind) on the generators $g, h$ subject to the relation $t = s$}]{langleSmidRrangle}{\ensuremath{\langle g,h \mid t = s \rangle}} %

\glsxtrnewsymbol[sort={\OtherGroup 0 3501},description={The least radical ideal containing the ideal $I$}]{sqrtI}{\ensuremath{\sqrt{I}}}

\glsxtrnewsymbol[sort={\OtherGroup 0 5111},description={Coproduct of frames $L$ and $M$ or quantales $L$ and $M$}]{LoplusM}{\ensuremath{L \oplus M}} %
\glsxtrnewsymbol[sort={\OtherGroup 0 5112},description={An element in a binary coproduct of quantales equal to $\iota_1(\ell) \cdot \iota_2(m)$
and corresponding to $\ell \otimes m$ under the isomorphism $L \oplus M \cong L \otimes M$}]{elloplusm}{\ensuremath{\ell \oplus m}}

\glsxtrnewsymbol[sort={\OtherGroup 0 5201},description={A generator in a quantale coexponential ${}^A Q$ where $g$ is a generator in $Q$ and $\sigma$ is part of a dual basis for $A$}]{oslash}{\ensuremath{g \oslash \sigma}}

\glsxtrnewsymbol[sort={\OtherGroup 0 5301},description={Denotes the action of a semiring on a semimodule; in particular, the product of quantale elements, repeated addition (for $p \in \N$),
or the action of a quantale on a quantale module (especially for $p \in \Omega$)}]{cdot}{\ensuremath{p \cdot a}} %

\glsxtrnewsymbol[sort={\OtherGroup 0 7101},description={Dual of a dualisable object $A$}]{dual}{\ensuremath{A^*}}
\glsxtrnewsymbol[sort={\OtherGroup 0 7101},description={Frame homomorphism associated to the locale map $f$ or $P(f)$ for a background hyperdoctrine $P$}]{upperStar}{\ensuremath{f^*}}
\glsxtrnewsymbol[sort={\OtherGroup 0 7111},description={Right adjoint of the monotone map $f$}]{rightadjoint}{\ensuremath{f_*}}
\glsxtrnewsymbol[sort={\OtherGroup 0 7112},description={Left adjoint of the (left adjoint) monotone map $f$}]{leftadjoint}{\ensuremath{f_!}} %

\glsxtrnewsymbol[sort={\OtherGroup 0 7201},description={Quantic pseudocomplement of $a$ --- the largest element $b$ in a quantale such that $ab = 0$}]{abullet}{\ensuremath{a^\bullet}}

\glsxtrnewsymbol[sort={\OtherGroup 0 7401},description={Coexponential with base $Q$ and co-exponent $A$ --- the exponential $Q^A$ in the opposite category}]{AQ}{\ensuremath{{}^A Q}} %
\glsxtrnewsymbol[sort={\OtherGroup 0 7421},description={Localisation of the ring $R$ at the prime ideal $P$}]{RP}{\ensuremath{R_P}}

\glsxtrnewsymbol[sort={\OtherGroup 0 7501},description={Localisation of the ring $R$ with respect to the multiplicative submonoid $S$}]{Sneg1R}{\ensuremath{S^{-1}R}} %

\glsxtrnewsymbol[sort={\OtherGroup 0 9101},description={True or the top element of a suplattice or quantale}]{top}{\ensuremath{\top}}
\glsxtrnewsymbol[sort={\OtherGroup 0 9102},description={False}]{bot}{\ensuremath{\bot}} %

\glsxtrnewsymbol[sort={\OtherGroup 0 9901},description={Closed congruence $\langle(0,a)\rangle$ on a frame or two-sided quantale}]{nablaa}{\ensuremath{\nabla_a}}

\glsxtrnewsymbol[sort={\OtherGroup 1 04},description={Open congruence $\langle(a,1)\rangle$ on a frame}]{Deltaa}{\ensuremath{\Delta_a}}
\glsxtrnewsymbol[sort={\OtherGroup 1 20 upsilon},description={Universal (quantic) open prime anti-ideal}]{upsilon}{\ensuremath{\upsilon}} %

\glsxtrnewsymbol[sort={\OtherGroup dim},description={Krull dimension of the ring $A$ or dimension of the vector space $A$}]{dim}{\ensuremath{\dim A}}
\glsxtrnewsymbol[sort={\OtherGroup BeF},description={Sublocale of `points' within $\epsilon$ of the overt sublocale $F$}]{BepsilonF}{\ensuremath{B_\epsilon(F)}} %
\glsxtrnewsymbol[sort={\OtherGroup m},description={Maximal ideal of a local ring}]{m}{\ensuremath{\m}}
\glsxtrnewsymbol[sort={\OtherGroup TPR},description={Tangent space of the ring $R$ at the prime ideal $P$}]{TPR}{\ensuremath{T_P R}}

\shorthandoff{"} %

\title{Quantalic spectra of semirings}
\author{Graham Manuell}
\date{2019}

\begin{document}

\pagenumbering{Alph}

\begin{titlepage}
\begin{center}
\vspace*{-0.5cm}

\makeatletter
\newcommand\HUGE{\@setfontsize\Huge{36}{45}}
\makeatother

\HUGE
\textbf{\MyTitle}

\vspace{2.0cm}

\huge
\MyAuthor

\vfill

\LARGE
Doctor of Philosophy

University of Edinburgh

\MyDate
\end{center}
\end{titlepage}

\cleardoublepage

\subsection*{Declaration}
\thispagestyle{empty}

I declare that I have composed this thesis myself and that it has not been submitted for any previous degree.
The work presented is entirely my own, except where stated otherwise.

\medskip

\begin{flushright}
 Graham Manuell
\end{flushright}
\cleardoublepage %
\pagenumbering{roman}

\subsection*{Acknowledgements}

I would like to thank my supervisor Antony Maciocia for all his support, his generosity of time and for allowing me the freedom to work on a topic outside his area of research.
I thank my second supervisor Arend Bayer for meeting with me on a number of occasions.
I also appreciate the helpful comments of my thesis examiners, Chris Heunen and Steve Vickers. 

I am grateful to the Skye Foundation for providing me with funding.

I would like to thank Peter Faul for discussions about how to present these results, for proofreading this thesis and above all for being a great friend.

Finally, I thank my parents for their constant support and encouragement.

\subsection*{Lay summary}

A \emph{semiring} is a collection of things which can be added and multiplied, just like the natural numbers $\{0,1,2,\dots\}$ or the so-called real numbers (all the numbers on the number line).
The study of such structures forms part of a branch of mathematics called \emph{algebra}.

On the other hand, a \emph{topological space} is a collection of things, states or positions, together with a collection of yes-no questions we can ask about these things.
The questions should be chosen so that if the answer to a question is yes for a certain thing, then we can demonstrate that it is true by `finite means'. The real numbers are
an example of a topological space and one appropriate question we could ask about a real number is ``Is this number greater than 1?''. Notice that if a number is greater than
one, then this will be clear after looking at some (possibly large, but finite) number of decimal places. For example, suppose we want to know if $\pi/3$ is greater than one.
If we compute $\pi/3$ to the first decimal place, we get $1.0$ and so at this point we cannot tell if it is greater than one or not (it might be rounded up from $0.98$). But
by the second decimal place we have $1.05$, which is enough to convince us. However, it is not possible to demonstrate that a given number is \emph{equal} to one. Even if we
know the number to 5 decimal places, we cannot rule out the possibility that the number might be $1.00000003$. The study of topological spaces is called \emph{topology}.

A \emph{spectrum} is a way to take a semiring and construct an associated topological space. In this way spectra provide a bridge between algebra and topology.
There are a number of well-known spectrum constructions in mathematics, each appropriate for different kinds of semiring.
We introduce a single construction which applies to all semirings and reduces to four different previously known constructions in special cases.

\subsection*{Abstract} %

Spectrum constructions appear throughout mathematics as a way of constructing topological spaces from algebraic data.
Given a localic semiring $R$ (the pointfree analogue of a topological semiring), we define a spectrum of $R$ which
generalises the Stone spectrum of a distributive lattice, the Zariski spectrum of a commutative ring, the Gelfand spectrum
of a commutative unital C*-algebra and the Hofmann--Lawson spectrum of a continuous frame.
We then provide an explicit construction of this spectrum under conditions on $R$ which are satisfied by our main examples.

Our results are constructively valid and hence admit interpretation in any elementary topos with natural number object.
For this reason the spectrum we construct should actually be a locale instead of a topological space.

A simple modification to our construction gives rise to a \emph{quantic} spectrum in the form of a commutative quantale.
Such a quantale contains `differential' information in addition to the purely topological information of the localic spectrum.
In the case of a discrete ring, our construction produces the quantale of ideals.

This prompts us to study the quantale of ideals in more detail. We discuss some results from abstract ideal theory in the setting of quantales
and provide a tentative definition for what it might mean for a quantale to be nonsingular by analogy to commutative ring theory.

\tableofcontents
\cleardoublepage %
\pagenumbering{arabic}
\setcounter{chapter}{-1}

\section{Introduction}\label{section:introduction}

Spectrum\index{spectrum} constructions show up in a number of different contexts in mathematics as a way to associate a topological space to some kind of algebra structure.
These constructions tend to all proceed in a similar manner and they have led to great insight into the algebraic structures involved.
This suggests that we should try to understand the notion of spectrum more generally.

In order to understand how one might associate a space to an algebraic structure, it is helpful to consider the following special case. Suppose we are given a ring $R$ of some class of functions
on a topological space $X$ (or more generally, sections of a bundle over $X$ so that the functions may have variable codomain).
In the best cases we might hope that we could use a spectrum construction to recover this space $X$ from $R$.

Every point $x \in X$ yields a ring homomorphism $e_x$ from $R$ corresponding to evaluation at $x$.
In general we cannot say much about the codomain of this function, other than that it is a ring. We could nonetheless ask whether $e_x(f)$ is equal to $0$ or perhaps whether it is equal to $1$.
Since $e_x(f) = 1 \iff e_x(f-1) = 0$, it is enough to consider only the former case. %
In this way, each $f \in R$ defines a zero set in $X$.

We can use these zero sets to recover information about the topology of $X$. We expect the evaluations $e_x(f)$ to be `continuous in $x$'
and so if the functions on $X$ take values in discrete rings, then these zero sets should be clopen.
However, in many important examples the codomain will not be discrete: for example, if $R$ is the ring of continuous real-valued functions on $X$.
So we will assume that these zero sets are merely closed. These can then be used to generate a topology.

Of course, in general we cannot expect to be able to recover the topology of $X$ from $R$.
For example, the ring of $S$-valued constant functions contains no information about $X$ other than whether it is empty.
In this case, the zero sets are either empty or all of $X$. So the points of $X$ are topologically indistinguishable and it seems natural to identify them in the spectrum.

The ring $R = S_1 \times S_2$ can be obtained as either the ring of $R$-valued functions on the one-point space or the ring of sections of the bundle $S_1 \sqcup S_2 \to \{1,2\}$.
The latter fact is more informative --- it encodes the fact that $R$ can be decomposed into a product of two smaller rings. We would like the spectrum of a ring to be as informative as possible
and so we should distinguish zero sets which are distinguished for any choice of substrate $X$.

\paragraph{The Zariski spectrum}

For a (discrete) commutative ring $R$, the \emph{Zariski spectrum}\index{spectrum!Zariski|(} of $R$ gives an associated topological space $\Spec R$ and
the elements of $R$ can then be viewed as (polynomial) functions on this space.

Unlike the discussion above, we do not start with a space $X$. When $X$ is given, its points give rise to ring homomorphisms and
so it stands to reason that we can associate the points of the spectrum with certain ring homomorphisms out of $R$.
Not every ring homomorphism is appropriate for this purpose. Firstly, these ring homomorphisms should be distinguished by which elements of $R$ they send to zero.
This suggests that they should be associated to quotients of $R$. Secondly, we should not map into a ring such as $S_1 \times S_2$, when more fine grained information
can be obtained by into smaller rings such as $S_1$ and $S_2$.

In the Zariski spectrum, the points of the spectrum $\Spec R$ correspond to prime ideals of $R$. The value of a `function' $f \in R$ at the point $P$ is given by the image of $f$ in the field of fractions of
the integral domain $R/P$.
We can then equip $\Spec R$ with a topology by demanding that the zero sets of these functions are closed. %

Note that the zero set of $fg$ is the union of the zero sets of $f$ and $g$. Thus, the subbasic closed sets are closed under finite unions and hence
general closed sets will be given by simultaneous zero sets of functions.
Now if $f$ and $g$ are zero at some point, then $f + g$ and $fh$ are zero at that same point for any $h \in R$. Furthermore, $f$ is zero wherever $f^2$ is.
Therefore closed sets of $\Spec R$ correspond to radical ideals\index{ideal (discrete algebraic)!semiring!radical} of the ring $R$. %

In this way the open sets of the Zariski topology correspond to unions of cozero\index{cozero} sets. That is, regions on which some function $f \in I$ is nonzero (or equivalently, invertible) for radical ideals $I$.

\paragraph{Constructive considerations}

This approach works well in classical mathematics with the axiom of choice. However, without choice axioms rings might fail to have sufficiently many prime ideals for the spectrum to behave appropriately.
For example, we cannot show that a nontrivial ring always has a nonempty spectrum.
It is desirable for mathematical theories to be \emph{constructive}\index{constructive mathematics}. A constructive result has much greater applicability than a classical one, since it holds in any topos.
For instance, interpreting a result in the effective topos can give a computer program which computes a witness for the result.
The extra constraints involved in working constructively can also be useful in narrowing the possibilities when defining concepts and constructions.

The Zariski spectrum can be defined constructively if we replace the classical notion of a topological space with the pointfree notion of a frame.
A frame is an algebraic formulation of a lattice of open sets. From this points can be reconstituted, but they are not fundamental; nontrivial frames might even have no points.

As a frame, the Zariski spectrum of a ring $R$ is given by its frame $\Rad(R)$ of radical ideals.
This frame can also be defined by generators and relations. One generator $\overline{f}$ is given for each element $f \in R$.
This generator is thought of as the cozero\index{cozero|(} set of $f$. The relations correspond to intuitive properties about cozero sets and are as given as follows:
\begin{align*}
 \overline{0}    &\mathmakebox[4.5ex][c]{=} 0 \\
 \overline{f+g}  &\mathmakebox[4.5ex][c]{\le} \overline{f} \vee \overline{g} \\
 \overline{1} &\mathmakebox[4.5ex][c]{=} 1 \\
 \overline{fg}   &\mathmakebox[4.5ex][c]{=} \overline{f} \wedge \overline{g}.
\end{align*}
Without excluded middle, the points of the spectrum correspond to prime \emph{anti-ideals}.
These should be thought of as the functions which are invertible or `cozero' at that point.
In this way, the Zariski spectrum classifies the `places where the elements of $R$ (viewed as functions) can be cozero'.\index{cozero|)}

\paragraph{The quantale of ideals\index{quantale of ideals}}

We have discussed how the frame of radical ideals $\Rad(R)$ is important for the study of a ring $R$. It is natural to ask what would happen if we instead considered the lattice $\Idl(R)$ of \emph{all} ideals.
This is no longer a frame, but it does come with a natural operation of \emph{multiplication of ideals} which makes it a commutative \emph{quantale}. 

Commutative quantales are generalisations of frames where the operation corresponding to `meet' need not be idempotent.
While a frame contains purely `topological' information, this failure of idempotence allows quantales to additionally encode some kind of `differential'\index{spectrum!quantic|(} information.
This will be seen explicitly when we discuss the tangent bundle of a quantale.

The quantale of ideals $\Idl(R)$ can be presented in a similar way to the frame of radical ideals. We have generators $\overline{f}$ for each $f \in R$ subject to the relations
\begin{align*}
 \overline{0}    &\mathmakebox[4.5ex][c]{=} 0 \\
 \overline{f+g}  &\mathmakebox[4.5ex][c]{\le} \overline{f} \vee \overline{g} \\
 \overline{1} &\mathmakebox[4.5ex][c]{=} 1 \ge \overline{f} \\
 \overline{fg}   &\mathmakebox[4.5ex][c]{=} \overline{f} \cdot \overline{g}.
\end{align*}

Here the generators corresponding to $f$ and $f^2$ are different. This should be thought of as arising from the fact that while $f$ and $f^2$ vanish in the same locations, if $f$ vanishes to first order,
then $f^2$ will vanish to second order --- that is, both $f^2$ and its first derivatives will be zero.
The difference with $\Rad(R)$ is particularly stark when $R$ has nontrivial nilpotent elements. When viewed as a function, such an element $\epsilon$ is zero everywhere on $\Spec R$, but it still
gives a nontrivial `cozero' element in $\Idl(R)$. This behaviour of nilpotents is understood in algebraic geometry to be related to `infinitesimal tangent vectors' and `repeated roots'.
We might consider $\Idl(R)$ to be a kind of \emph{quantic spectrum}\index{spectrum!quantic|)} of $R$.

The idea to study rings through their quantales of ideals is not new. Ideal theory can be traced back to work by Dedekind, Krull and Noether, while the abstract study of ideals was initiated by
Dilworth and Ward \cite{DilworthWard1939}. This research has shown that for a number of purposes $\Idl(R)$ contains a rich amount of information about $R$.
There are also a number of streams of research in algebraic geometry, such as the study of Hilbert--Samuel functions, which can be interpreted in purely ideal theoretic terms.\footnote{
Besides intrinsically ideal-theoretic concepts, it seems to me that quantales easily describe algebro-geometric notions in low dimension (for instance, the intersection theory of curves or
first-order differential information), but fail in higher dimensions when higher homology groups become relevant (though perhaps more progress could be made with a more sophisticated approach).}
Of course, a large amount of information is also lost. For example, every field has the same quantale of ideals.

Still, while this prior work is able to extract information from the quantale of ideals, it does not always do so in a natural way. In particular, while quantale-like structures have been used
to study individual rings, the morphisms have been somewhat neglected and the functoriality of $\Idl\colon \Rng \to \Quant$ remains unexploited.
Perhaps much of the information so extracted from $\Idl(R)$ appears only `incidentally' and more categorically minded approaches will not be as successful.
With this in mind, we do not attempt here to conduct a thorough examination of $\Idl(R)$ in service of commutative ring theory.
Instead, we make a start at exploring a few natural constructions involving quantales and then consider what if anything they might tell us.
If this were to be followed through in future, the most promising direction might be to see what the tangent bundle of a quantale has to say about singularity theory.

\paragraph{Other kinds of spectra}

The Zariski spectrum is not the only, nor the first kind of spectrum to be defined.
The idea of associating a topological space to an algebraic structure, in this case Boolean algebras\index{spectrum!Stone|(}, was first proposed in \cite{StoneBoolAlgs} by Stone\footnote{
In the same paper Stone also constructed the spectrum of the (topological) ring of real-valued functions on a compact Hausdorff space.}. %
This was quickly extended to distributive lattices in \cite{StoneDistLattices}.
Shortly thereafter, Gelfand developed a spectrum\index{spectrum!Gelfand|(} for commutative Banach algebras \cite{GelfandNormedSpectrum} and Jacobson found a spectrum for (discrete, possibly noncommutative) rings \cite{JacobsonPrimitive}. %
These spectra were very similar to Stone's --- the main difference being that while Stone thought of his spectrum as a space of prime ideals, Gelfand and Jacobson
both used maximal ideals (or noncommutative variants). Grothendieck observed that maximal spectra are not functorial for general commutative rings, so at least in the discrete case,
it is better to consider prime ideals, as in Stone's original definition. This is the so-called Zariski spectrum described above. %

On the other hand, Gelfand's duality between commutative C*-algebras and compact Hausdorff spaces\footnote{The so-called `real Gelfand duality' is actually due to Stone \cite{StoneBoolAlgs,Stone1940real}
and was proved prior to the complex version described in \cite{GelfandNormedSpectrum,GelfandNaimark}.} is still usually presented as being in terms of maximal ideals. Notwithstanding the satisfactory theory of the
representation of Gelfand rings (see \cite{MulveyGelfandRings} and \cite{banaschewski2002maximal}), it seems to me that the importance here of maximal ideals per se is sometimes overstated.
As normed rings, C*-algebras have a natural topology, which we dare not ignore. It is not maximal ideals that are of importance so much as \emph{closed prime ideals\index{anti-ideal!open prime|(}},
which happen to coincide in this example.

Moreover, in \cite{hofmann1978spectral} Hofmann and Lawson describe a spectrum\index{spectrum!Hofmann-Lawson@Hofmann--Lawson|(} of continuous distributive lattices, whose points are described as certain elements,
which we may interpret as closed prime ideals with respect to the Scott topology.

The striking similarities between these different notions of spectrum suggest we look for a common generalisation.
Prompted by the case of C*-algebras and encouraged by Stone's maxim \emph{``Always topologise''}, we develop a notion of spectrum of a commutative topological semiring.
More accurately, we work with \emph{localic semirings}\index{semiring!localic}, which are the pointfree analogue of topological semirings.
(Using classical logic and topological semirings, a similar spectrum construction is mentioned in \cite{kerkhoff2012tropical}, but this was left largely unexplored.) %

When we work constructively, the picture becomes even clearer. We described the constructive theory of the Zariski spectrum above; it has been understood for a long time with the essential ideas
appearing in \cite{Joyal1975supportOfRing}. A constructive account of Gelfand duality was developed more recently \cite{Henry2016}.
We obtain the following table.

\begin{table}[H]
\begin{tabular}{>{\raggedright\arraybackslash}m{38mm}>{\raggedright\arraybackslash}m{23mm}>{\raggedright\arraybackslash}m{35mm}>{\raggedright\arraybackslash}m{35mm}}
  \toprule
  \textbf{Class of semiring} & \textbf{Spectrum}\index{spectrum} & \textbf{Opens} & \textbf{Points} \\
  \midrule
  Commutative rings & Zariski\index{spectrum!Zariski|)} & Radical ideals & Prime anti-ideals \\
  \rowcolor{gray!20}
  Distributive lattices & Stone\index{spectrum!Stone|)} & Ideals & Prime filters \\
  Comm.\ C*-algebras & Gelfand\index{spectrum!Gelfand|)} & Overt weakly closed ideals \cite{ConstructiveGelfandNonunital} & Closed prime ideals \\ %
  \rowcolor{gray!20}
  Continuous frames & Hofmann--Lawson\index{spectrum!Hofmann-Lawson@Hofmann--Lawson|)} & Scott-closed ideals & Scott-open prime filters \\ %
  \bottomrule
\end{tabular}
\end{table}

This suggests that a general notion of spectrum for a localic semiring might be given by a frame of `overt weakly closed radical ideals'\index{ideal (localic algebraic)!overt weakly closed} and the corresponding points will be closed prime ideals\index{ideal (localic algebraic)!closed prime|see {anti-ideal, open prime}}
(or equivalently open prime anti-ideals\index{anti-ideal!open prime|)}).

We have seen how the points of the Zariski spectrum corresponded to regions where `functions can be cozero'.
Phrasing this as a geometric theory leads to the previously described presentation.
In this thesis, we attempt to generalise this approach to give a kind of presentation for the spectrum of a general localic semiring.
In a wide class of localic semirings, including all the examples mentioned above, this successfully leads to a spectrum which coincides with an appropriate definition of
the frame of overt weakly closed radical ideals discussed above.

By analogy to the Zariski spectrum, we might also hope that each point corresponds to a semiring homomorphism into some notion of `localic semi-field' (with an open sublocale of invertible elements).
Regrettably, this does not work without additional assumptions.
The reason is that quotients and localisations of localic semirings can be badly behaved due to the failure of products of locales to commute with coequalisers. %
The Zariski spectrum can be equipped with a `structure sheaf' from whose global sections we can recover the original ring. Any attempt to equip our generalised spectrum with
a `structure bundle'\index{structure bundle} runs into the same problems. In the absence of a developed theory of localic algebra, we have not explored this further in this thesis. %
Perhaps some progress could be made by placing restrictions on the class of localic semirings or by only taking the localisation `formally'. This is an opportunity for future work.

\paragraph{Quantic spectra}

We can also extend the notion of quantic spectrum to this more general situation.
This extension is more audacious, since for most of our examples this quantale is already a frame and coincides with the localic spectrum.
Nonetheless it appears that there might be some interesting examples where it could lead to more interesting structure: for instance, it could perhaps be
applied to the ring of smooth functions on a differentiable manifold, equipped with the Whitney topology. %

We can recover the localic spectrum as the localic reflection of the quantic spectrum.
But this is not the only frame which we can derive from the quantic spectrum. We can also use the quantic spectrum to construct what we call the \emph{tangent bundle}.
This contains first-order differential information about the quantale, and hence the original localic semiring.
For a discrete ring, this can be used to determine if the ring is \emph{nonsingular} in the sense of algebraic geometry. %

\subsection*{Outline}

We now give a brief summary of the contents of the thesis.

\newcommand*{\fullref}[1]{{\hypersetup{linkcolor=.}\hyperref[{#1}]{\Cref*{#1}: \nameref*{#1}}}} %

\subsubsection*{\fullref{section:background}}

In this chapter we describe the necessary background for understanding the rest of the thesis.
We pay particular attention to constructive frame theory, often including complete proofs, since it is rare to find this information in a single place.
The results are not new, but many of the proofs differ from those found in the literature (if explicit published proofs exist at all).

\subsubsection*{\fullref{section:quantales_and_AIT}}

This chapter presents some links between classical abstract ideal theory and the theories of quantales and suplattices.
We start by characterising the `principal elements' of Dilworth \cite{Dilworth1962} in categorical terms in the category of suplattices (\cref{prop:characterisation_of_principal_elts}).
After finding this characterisation I discovered that similar work had been done in a slightly different context in \cite{blyth1972residuation,nai2014principalMappingsPosets},
but hopefully our approach is still of interest. %

Next with \cref{prop:principal_elements_of_Idl_R} we provide a constructive proof of the characterisation of principal elements in $\Idl(R)$.
This is a well-known result in abstract ideal theory, but the standard proof \cite{McCarthy1971} uses the axiom of choice.
Some partial results can be found in \cite{LombardiQuitte}, but a full constructive proof does not seem to have appeared before.
Along the way, we develop some of the theory of localisations of coherent quantales.
These results are folklore and are stated in \cite{Anderson1976} without proof. We provide constructive proofs.

Finally, under the assumption of excluded middle, we observe that the notion of a discrete valuation can be recovered from natural quantale-theoretic considerations.
We also show how a certain class of primary ideals arises naturally from maps from $\Idl(R)$ into the two-sided quantale $\langle \epsilon \mid \epsilon^2 = 0\rangle$
and relate this to subspaces of the Zariski tangent space at a prime in $R$.

\subsubsection*{\fullref{section:spectrum}}

In this chapter we define the spectrum of a localic semiring (\cref{def:localic_spectrum,def:quantic_spectrum_as_OPAI}) and the quantale of ideals (\cref{section:quantale_of_overt_weakly_closed_ideals})
and prove our main result (\cref{prop:quantic_spectrum_of_approximable_semiring}), which provides a sufficient condition for the quantic spectrum to exist and coincide with the quantale of ideals.
These sufficient conditions are satisfied by all our examples and the spectrum construction gives the expected results in each case.

We start by describing a generalised notion of presentation for frames,
which will allow us to mimic the presentation of the Zariski spectrum of a discrete ring in our more general situation.
In this way we define the localic spectrum of a localic semiring to be the representing object of a certain functor.
This definition is easily extended to give a definition for the quantic spectrum.

Next we construct the quantale of ideals $\Idl(R)$ of a localic semiring $R$ as a quotient of the suplattice $\SubOW(R)$ of overt weakly closed sublocales of $R$.
We also define the quantale of monoid ideals and the frame of radical ideals.

In \cref{section:sufficient_conditions_for_spectrum} we define \emph{approximable} localic semirings. We provide a candidate for the universal element of the spectrum functor
and prove various properties about it under the assumption that the ring is overt and approximable. Finally, we prove that under these assumptions it is indeed the universal element
and hence $\Idl(R)$ is the quantic spectrum. The localic result is obtained as a corollary.

We conclude the chapter by proving that some large classes of localic semirings are approximable and thus have a spectrum. In particular, we prove that the localic ring of real-valued functions
on a compact regular locale is approximable (\cref{prop:gelfand_approximable}) and hence the Gelfand spectrum of a C*-algebra is encompassed by our construction. %

\subsubsection*{\fullref{section:coexponential}}

This chapter can be divided into two parts. In the first part, we discuss the notion of approximable semirings and the spectrum construction from the perspective of dualisable suplattices.
We note that a localic semiring $R$ is approximable if and only if the underlying suplattice of an associated localic monoid $\Sats R$ is a dualisable object in $\Sup$.
This localic monoid can be viewed as a localic incarnation of the `holoid quotient' of a monoid. We prove a number of results about this construction and show that, when it exists,
the dual of $\Sats R$ is the quantale $\MonIdl R$ of monoid ideals (\cref{prop:dual_of_SatsR}). A simple proof then shows that under these assumptions, $\MonIdl R$ gives a kind of
`monoid spectrum' of the multiplicative monoid of $R$ (\cref{prop:spectrum_for_localic_monoids}). This helps to clarify a number of aspects of the spectrum construction for semirings,
though I have not yet extended this approach to give the semiring spectrum in a similar way.

The second part involves coexponentials of quantales. A result of Niefield \cite{Niefield2016} implies that a quantale is coexponentiable if and only if its underlying suplattice is dualisable.
While it is possible to extract an explicit description of the coexponential from her characterisation, the high level of abstraction makes this difficult.
\Cref{prop:presentation_for_coexponential} provides an explicit description of the coexponential in terms of generators and relations.
In \cref{section:tangent_bundle} we then apply this result to construct the tangent bundle of a quantale. We discuss links to algebraic geometry and propose a definition for what it might mean
for a quantale to be nonsingular.

\section{Background}\label{section:background}

In this chapter we summarise the necessary background for understanding the thesis.
Let us start by fixing basic terminology and notation.

Given a set $A$, a subset $S \subseteq A$ and some notion of distinguished subsets of $A$ closed under arbitrary meets (for instance, subalgebras, ideals or congruences),
we will use the notation $\langle S\rangle$\glsadd{langleSrangle} to denote the distinguished subset generated by $S$.

We denote the category of sets by $\Set$\glsadd{Set} and the category of posets by $\Pos$\glsadd{Pos}.

\begin{definition}
 A \emph{semiring}\index{semiring|(textbf} (also called a \emph{rig}\index{rig|see {semiring}}) is a set $R$ equipped with an `additive' commutative monoid $(R, +, 0)$ structure and a `multiplicative' monoid structure $(R, \times, 1)$
 such that $\times$ is bilinear with respect to the additive structure --- that is, $\times$ distributes over $+$ (on both sides) and $0$ is an absorbing element for $\times$.
\end{definition}

A semiring is \emph{commutative} if its multiplicative monoid is commutative. In this thesis, we will implicitly assume that every semiring is commutative.\index{semiring|)}

We regard $0$ as a natural number and so the natural numbers $\N$ with their usual operations form the initial semiring.

A \emph{ring\index{ring|textbf}} is a semiring where the additive monoid is a group.
In particular, note that for us rings are commutative and unital.

A semiring $R$ is said to be \emph{idempotent}\index{idempotent semiring|textbf}\index{semiring!idempotent|see {idempotent semiring}} if its additive monoid is idempotent --- that is, if $x + x = x$ for all $x \in R$.
This makes the additive monoid a semilattice and by convention we order the semiring so that addition corresponds to join.

\begin{definition}
 A subset $S$ of a commutative monoid $M$ is said to be \emph{saturated}\index{saturated!set|(textbf} if $xy \in S \implies (x \in S) \land (y \in S)$.
 By commutativity, it is enough to require that $xy \in S \implies x \in S$.
\end{definition}

\begin{definition}
 We say a subset of a semiring is \emph{saturated}\index{saturated!set|)} if it is saturated with respect to its multiplicative structure.
\end{definition}

\begin{definition}
 A subset $I$ of a commutative monoid $M$ is called a \emph{monoid ideal}\index{ideal (discrete algebraic)!monoid|textbf} if $(x \in I) \land (y \in M) \implies xy \in I$.
 Note in particular that the empty set is a monoid ideal.
\end{definition}

\begin{definition}
 An \emph{ideal}\index{ideal (discrete algebraic)!semiring|textbf} in a semiring $R$ is a subset which is both a submonoid for the additive structure and a monoid ideal for the multiplicative structure.
 We write $\Idl(R)$ for the poset of ideals in $R$.
 
 An ideal $I$ is \emph{radical}\index{ideal (discrete algebraic)!semiring!radical|textbf} if $x^n \in I \implies x \in I$ for all $n \in \N$. Every ideal $J$ is contained in a least radical ideal $\sqrt{J}$\glsadd{sqrtI}.
 We write $\Rad(R)$\glsadd{Rad} for the poset of radical ideals in $R$.
 
 An ideal $I$ is \emph{prime}\index{ideal (discrete algebraic)!semiring!prime|textbf} if whenever a finite product lies in $I$, so does one of the factors. Explicitly, this means $1 \notin I$ and $xy \in I \iff (x \in I) \lor (y \in I)$.
\end{definition}
Of course, these definitions reduce to the more familiar ones in the case that $R$ is a ring\index{ideal (discrete algebraic)!ring}.

We will make use of a number of basic concepts and results in category theory. Much, but not all, of the necessary background can be found in \cite{MacLane1998categories}.
Most of the remainder is covered informally in \cite{lack20102cats} where further references can be found. Unfortunately there does not seem to be single comprehensive source.

Given objects $X$ and $Y$ in a category $\Cvar$, if the categorical product $X \times Y$ exists, then we will often write $\pi_1\colon X \times Y \to X$\glsadd{pi12} and $\pi_2\colon X \times Y \to Y$ for the
corresponding product projections without further comment. Similarly, if the coproduct $X \oplus Y$ exists, we will write $\iota_1\colon X \to X \oplus Y$\glsadd{iota_12} and $\iota_2\colon Y \to X \oplus Y$
for the coproduct injections.
We write $!$\glsadd{exclamation} for the unique morphism from an initial object or to a terminal object.

If $f\colon Z \to X$ and $g\colon Z \to Y$, we use $(f,g)$\glsadd{bracketsfg} to denote the map from $Z$ to $X \times Y$ induced by the universal property of the product.

\subsection{Constructive mathematics}

For the most part, in this thesis we will work constructively\index{constructive mathematics} in the sense that the arguments will hold internal to any topos with natural numbers object.
Any use of nonconstructive principles will be clearly indicated.

Constructive logic is characterised by the failure to assert the law of the excluded middle\index{excluded middle}.\footnote{While the law of the excluded middle is not asserted, neither is its negation.
Thus, classical mathematics just a special case of constructive mathematics. Of course, some logics are more general still. Quantales have strong links to substructural logics such as linear logic.}
See \cite{VanDalenIntuitionisticLogic} for the precise details.
We do not strive to avoid classical axioms for philosophical reasons, but because of the great utility in being able to interpret the results in a number of different settings. For example, constructive proofs are
intrinsically effective (in the sense of computability) and also immediately give proofs of analogous theorems in the setting of fibrewise topology.
It can be useful to keep one of these models in mind for intuition.

Constructive mathematics is not difficult, but understandably classical mathematicians can be unsure exactly which results are constructively valid.
We attempt to briefly describe some of the important ideas which are invisible in the classical theory.

We denote \emph{lattice of truth values} by $\Omega$\glsadd{Omega}. If $\phi$ is a closed logical formula, we will use the notation $\llbracket \phi \rrbracket$\glsadd{llrrbrackets} to denote the corresponding truth value in $\Omega$.
In constructive logic, we cannot prove $\Omega$ is isomorphic to the two-element lattice $\{\bot, \top\}$ (as this would imply excluded middle).
Nonetheless, maps into $\Omega$ are completely determined by what they send to $\top$\glsadd{top} (i.e.\ true), since every $p \in \Omega$ satisfies $p = \llbracket p = \top \rrbracket$.
The lattice $\Omega$ may be identified with the powerset of the singleton by associating each truth value $p$ with the subset $\{ x \in \{*\} \mid p \}$.

Notice in particular that subsets of sets can fail to have complements.
\begin{definition}
 A subset $A \subseteq X$ is said to be \emph{decidable}\index{decidable|(textbf} if $A$ has a complement in the powerset lattice $\powerset{X}$.
\end{definition}

When we restrict our attention to decidable subsets, constructive mathematics begins to look much more like the familiar classical situation.
For example, the \emph{decidable} subsets of $\{*\}$ are just $\emptyset$ and $\{*\}$.\footnote{
While the subobject classifier in a topos might have many complemented subobjects externally, there are nonetheless always exactly two according to the internal logic.
}
Nonetheless, this is not always possible (or even desirable).

A general subset of $\{*\}$ is called a \emph{subsingleton\index{subsingleton|textbf}}. Since a subsingleton cannot be shown to be equal to either $\emptyset$ or $\{*\}$, nonemptiness is a little more subtle constructively.
A set $X$ is \emph{nonempty} if it is not equal to the empty set.
But this is a weaker condition than the set being \emph{inhabited\index{inhabited|textbf}}, which means that there is an element $x \in X$.
Inhabitedness is the more useful property. If a set is inhabited, we are provided with an element which we can then use to do other things; if we only know a set is nonempty,
there is very little more we can say. This is a common theme in constructive mathematics: conditions are better phrased in a \emph{positive} manner. By this we mean that the
condition should avoid using implication (and negation in particular).

Sometimes in classical mathematics, it is important to consider if two subsets $U$ and $V$ are disjoint --- that is, when their intersection is empty. In the constructive setting, it is often better to
consider when the intersection is inhabited. In this case we say the subsets \emph{meet} each other and write $U \between V$\glsadd{between}.

Another case in which constructive mathematics behaves more classically is when equality is decidable.
\begin{definition}
 We say a set $X$ has \emph{decidable equality}\index{decidable|)} if the diagonal is a decidable subset of $X \times X$.
\end{definition}
An alternative way of describing decidable equality is that excluded middle holds for equality statements: $\forall x, y \in X.\ (x = y) \lor (x \ne y)$.
Functions can be defined piecewise when their domain has decidable equality, while this is usually not constructively valid.
Decidable equality holds for many familiar sets, such as the natural numbers and the rationals, but if $\Omega$ has decidable equality, then excluded middle holds.
It is also traditional to mention here that decidable equality can fail for the set of real numbers, but we believe that the real numbers are best viewed as a locale (see \cref{section:frame_propositional_theory})
and so thinking of them as a set discards too much information.\footnote{In topological terms, the diagonal of a discrete space is always open and decidable equality says that it is also closed.
Since the reals are Hausdorff, their diagonal is indeed closed. However, the diagonal fails to be open. So the reals also fail to have a clopen diagonal, but for a different reason to a set such as $\Omega$.}

In constructive mathematics, finiteness splinters into a number of different concepts. We discuss two of the most important ones.

\begin{definition}
 A set is called \emph{finite\index{finite|(textbf}} if it is isomorphic to $[n] \coloneqq \{x \in \N \mid x < n\}$\glsadd{sqbracketsN} for some $n \in \N$.\footnote{
 For us finiteness and other conditions are properties, not structures. When we say there is an isomorphism, we are existentially quantifying over morphisms.
 See \cite{shulman2010stack} for how this is interpreted in a topos.
 In particular, this does \emph{not} necessarily mean that there is an isomorphism externally.
 }
 A set is \emph{finitely indexed} or \emph{Kuratowski-finite} if it is the quotient of a finite set.
\end{definition}
\begin{remark}
One should be aware that the terminology for constructive finiteness in the literature is not completely consistent with some authors using `finite' to mean what we call `Kuratowski-finite'.
\end{remark}

A subset of a finite set need not be finite! In fact, the following lemma shows that if every subsingleton is Kuratowski-finite, then excluded middle holds.
\begin{lemma}
 If $X = \{x \in \{*\} \mid p\}$ is Kuratowski-finite, then $p \vee \neg p$.
\end{lemma}
\begin{proof}
 There is an $n \in \N$ and a surjection $e\colon [n] \twoheadrightarrow X$.
 We have $(n = 0) \lor (n \ge 1)$. In the case that $n = 0$, we have a surjection from $\emptyset$ to $X$. By the definition of surjectivity, we find $X = \emptyset$ and so $\neg p$.
 In the case that $n \ge 1$, we have $e(0) \in X$ and hence $p$.
\end{proof}

On the other hand, decidable subsets are much better behaved.

\begin{lemma}
 A decidable subset of a finite set is finite and a decidable subset of a Kuratowski-finite set is Kuratowski-finite.
\end{lemma}
\begin{proof} %
 We show that decidable subsets of $[n]$ are finite by induction on $n$.
 This is trivial for $n = 0$. Suppose $n = k+1$ and that $S$ is a decidable subset of $[n]$. Then either $k \in S$ or $k \notin S$. In the latter case, $S$ is a decidable subset of $[k]$
 and is therefore finite by the inductive hypothesis. In the former case, $S\setminus \{k\}$ is a decidable subset of $[k]$. Thus, we have an isomorphism $\phi\colon [m] \to S \setminus \{k\}$.
 We can then construct a map $\phi'\colon [m+1] \to S$ by
 \[\phi'(x) = \begin{cases}
               \phi(x) & \text{if $x < m$} \\
               k       & \text{if $x = m$}
              \end{cases},
 \]
 which is easily seen to be an isomorphism.
 
 Now given an arbitrary set $X$, a surjection $e\colon [n] \twoheadrightarrow X$ and a decidable subset $S \subseteq X$, we have that $e^{-1}(S)$ is decidable and hence finite by the above.
 Thus, $S = e(e^{-1}(S))$ is Kuratowski-finite (via the map $e\vert_{e^{-1}(S)}$) and even finite if $e$ is a bijection.
\end{proof}

We can now show the relationship between the two notions of finiteness we defined above.

\begin{lemma}
 A set $X$ is finite\index{finite|)} if and only if it is Kuratowski-finite and has decidable equality. %
\end{lemma}
\begin{proof}
 If $X$ is finite, it is clearly Kuratowski-finite and inherits decidable equality from $\N$.
 
 For the converse, call a set $n$-listable if it admits a surjection from $[n]$. We show every decidable $n$-listable set $X$ is finite by strong induction on $n$. Assume this holds for all $k < n$.
 The result is clear for $n = 0$, since the supplied surjection is automatically a bijection. So we may assume $n \ge 1$.
 Let $X$ be a set with decidable equality and $e\colon [n] \twoheadrightarrow X$ a surjection.
 Since $X$ has decidable equality, the subset $X' = \{x \in X \mid x \ne e(n-1)\}$ is decidable. Thus, $e^{-1}(X')$ is decidable and hence finite. Furthermore, $n-1 \notin e^{-1}(X')$
 and so $X'$ is $k$-listable for some $k < n$. Therefore, $X'$ is finite by the inductive hypothesis. Now let $\phi\colon [m] \to X'$ be a bijection. As before, we may use this to define an
 isomorphism from $[m+1]$ to $X$ in a piecewise manner, sending $m$ to $e(n-1)$.
\end{proof}

Finally, to give an impression of the situation with constructive algebra, we briefly discuss ideals of rings.
Let us consider the ring of integers $\Z$ in particular. Even though the integers have decidable equality, general ideals still behave in an unfamiliar way.

First note that $\Z$ need not be a principal ideal domain in the usual sense. Let $p$ be a truth value and suppose the ideal $\{a \in \Z \mid a = 0 \text{ or } p = \top\}$ is principal with generator $x$.
Then we have $p \iff \abs{x} = 1$ and hence $p \lor \neg p$, since $\Z$ has decidable equality. On the other hand, every finitely generated ideal of $\Z$ is principal
and so $\Z$ is what is known as a \emph{Bézout domain}.

Sometimes where ideals are used classically, constructive mathematicians need to use \emph{anti-ideals} instead.
\begin{definition}
 Let $R$ be a semiring. An \emph{anti-ideal}\index{anti-ideal|textbf} of $R$ is a subset $A \subseteq R$ such that $0 \notin A$, $x + y \in A \implies (x \in A) \lor (y \in A)$
 and $xy \in A \implies (x \in A) \land (y \in A)$. We say $A$ is \emph{prime}\index{anti-ideal!prime|textbf} if $1 \in A$ and $xy \in A \iff (x \in A) \land (y \in A)$.
\end{definition}
Classically, an anti-ideal is just the complement of an ideal. It is the prime anti-ideals, not the prime ideals, which occur as the points of the Zariski spectrum.\footnote{
Of course, as we discuss below, the Zariski spectrum should be thought of as a locale, and it might not have `enough points'.
Nonetheless, having an idea of what the points are is important when viewing the spectrum as a classifying locale.}

\subsection{Basic notions in order theory}

We will need a number of basic notions from order theory which we now recall.

For us lattices\index{lattice} and semilattices\index{semilattice|(} are always bounded (only in one direction in the latter case), so that semilattices are simply idempotent commutative monoids and distributive lattices are a kind of
idempotent semirings. We will move between the algebraic and order-theoretic descriptions of these structures without comment.

The definition of a join-semilattice need only require the existence of nullary and binary joins and we may then inductively define finite joins in the usual way.
But we can even prove the existence of joins of general Kuratowski-finite subsets.
\begin{lemma}
 Let $X$ be a join-semilattice\index{semilattice|)} and let $S \subseteq X$ be a Kuratowski-finite subset. Then $\bigvee S$ exists and is equal to $\bigvee_{i = 0}^{n-1} e(i)$ for any surjection $e\colon [n] \twoheadrightarrow S$.
\end{lemma}
\begin{proof}
 For each $s \in S$ there is an $i \in [n]$ such that $s = e(i) \le \bigvee_{i = 0}^{n-1} e(i)$ and hence $\bigvee_{i = 0}^{n-1} e(i)$ is an upper bound of $S$.
 Now suppose $u \in X$ is an upper bound of $S$. Then for each $i \in [n]$, we have $e(i) \le u$ and hence $\bigvee_{i = 0}^{n-1} e(i) \le u$. Thus, this is the least upper bound of $S$ as required.
\end{proof}

Most, but not all, of our lattices will be distributive. On occasion we will also use some weaker conditions.
Suppose $a$, $b$ and $x$ are elements of a lattice $L$ and $a \le b$. There are two natural ways to restrict $x$ to the interval $[a,b]$:
$(x \wedge b) \vee a$ and $(x \vee a) \wedge b$.
If these always coincide, the lattice is \emph{modular}\index{modular|(textbf}\index{lattice!modular|textbf}. %
An important example of a modular lattice is given by the lattice of submodules of a module over a ring.
We will also have use for some more fine-grained notions of modularity.
\begin{definition}
 A pair of elements $(x,b)$ in a lattice $L$ is said to be a \emph{modular pair} if $(x \wedge b) \vee a = (x \vee a) \wedge b$ for all $a \le b$. %
 We say $x$ is a \emph{left modular element}\index{modular|)} if $(x,b)$ is a modular pair for all $b \in L$ and we say $b$ is a \emph{(right) modular element} if $(x,b)$ is a modular pair for all $x \in L$.
 A pair $(x,a)$ is a called \emph{dual modular pair} if it is a modular pair in $L\op$ and $a$ is called a \emph{dual modular element} if $a$ is a modular element in $L\op$.
 The notion of left modular element is self-dual.
\end{definition}
A lattice is modular if and only if every element is modular if and only if every element is dual modular if and only if every element is left modular.

\begin{definition}
 Let $X$ be a poset. A subset of $S$ of $X$ is a \emph{downset\index{downset|textbf}} if it satisfies $(x \in S) \land (y \le x) \implies y \in S$.
 It is an \emph{upset\index{upset|textbf}} if $(x \in S) \land (y \ge x) \implies y \in S$.
\end{definition}

The principal downset generated by $a$ is ${\downarrow} a = \{x \in X \mid x \le a\}$\glsadd{downarrow} and the corresponding principal upset is ${\uparrow} a = \{x \in X \mid x \ge a\}$.

\begin{definition}
 A poset $D$ is \emph{directed}\index{directed set|textbf} if every Kuratowski-finite subset $K \subseteq D$ has an upper bound $u \in D$. (Note that, in particular, directed sets are inhabited.)
\end{definition}

\begin{definition}
 Let $X$ be a poset. An \emph{order ideal}\index{ideal (order)|textbf} in $X$ is a directed downset. A \emph{filter}\index{filter|textbf} is an ideal in the opposite order.
\end{definition}
If $X$ admits finite joins, then a downset $I$ is an ideal if and only if it is closed under finite joins.
In particular, for a distributive lattice this coincides with the notion of a semiring ideal.

\begin{definition}
 Let $X$ be a join-semilattice. We say that a filter $F$ is \emph{prime}\index{filter!prime|textbf} if $0 \notin F$ and $x \vee y \in F \implies (x \in F) \lor (y \in F)$.
\end{definition}
Observe that a filter $F$ is prime if and only if the associated predicate $x \mapsto \llbracket x \in F\rrbracket$ preserves finite joins.
Thus, if $X$ is a lattice, the prime filters on $X$ are in bijection with lattice homomorphisms from $X$ to $\Omega$.
Also note that in a distributive lattice, prime filters coincide with prime anti-ideals.

\begin{definition}\label{def:dcpo}
 We say a poset $X$ is a \emph{dcpo\index{dcpo|textbf}} (directed-complete partially ordered set) if it admits joins of all directed subsets.
 We will often write $\dirsup D$\glsadd{dirsupD} for the join of a directed subset $D$ when we want to emphasise that $D$ is directed.
 A morphism in the category $\Dcpo$\glsadd{Dcpo} of dcpos is a monotone map which preserves directed joins.
\end{definition}

Directed joins and finite joins are complementary in the sense that every poset admits joins of the subsets which are simultaneously finite and directed, but if a poset admits both kinds of joins,
then it admits all joins.
\begin{lemma}
 If $X$ is both a dcpo and join-semilattice, then $X$ admits arbitrary joins. Explicitly, an arbitrary join $\bigvee S$ can be decomposed as directed join of the Kuratowski-finite subsets of $S$.
\end{lemma}

It is notable that a function on a product of two dcpos is \emph{bilinear} with respect to directed joins if and only if it is \emph{linear} with respect to directed joins.
\begin{lemma}
 If $X$, $Y$ and $Z$ are dcpos, then $X \times Y$ is a dcpo and if $f\colon X \times Y \to Z$ is a monotone map, then $f$ preserves directed joins if and only if
 it preserves directed joins in each variable separately (with the other variable held fixed).
\end{lemma}

We now define two important classes of subsets of dcpos.

\begin{definition}
 Let $X$ be a dcpo.
 A subset $U \subseteq X$ is \emph{Scott-open}\index{Scott topology|(textbf} if it is an upset and for all directed sets $D \subseteq X$, $\dirsup D \in U \implies D \between U$.
 A subset $C \subseteq X$ is \emph{Scott-closed}\index{Scott closed set@Scott-closed set|(textbf} if it is a downset which is closed under directed suprema.
\end{definition}

The Scott-open sets on a dcpo are indeed the open sets of a topology --- called the \emph{Scott topology\index{Scott topology|)}}.
\begin{lemma}
 The Scott-open sets on a dcpo $X$ are the open sets of a topology on $X$.
\end{lemma}
\begin{proof}
 The only nontrivial axiom to check is closure under binary intersections.
 Suppose $U$ and $V$ are Scott-open,
 $D$ is directed and $\dirsup D \in U \cap V$. Then there are elements $u \in D \cap U$ and $v \in D \cap V$.
 Since $D$ is directed, we have a $w \in D$ with $w \ge u, v$ and then $w \in D \cap U \cap V$, since $U$ and $V$ are upsets.
\end{proof}

However, the nature of Scott-closed sets\index{Scott closed set@Scott-closed set|)} is more subtle. Classically, these are the closed sets of the Scott topology. But since subsets need not have complements constructively,
there is not a good constructive notion of closed subset (in the sense of being the complement of an open set). Later we will introduce the notion of \emph{weak closedness}
and show that Scott-closedness fits nicely into this picture, at least under certain assumptions.

Finally, recall that Galois connections\index{Galois connection|(textbf} are the order-theoretic manifestation of categorical adjunctions. We define them here to fix terminology and notation.
\begin{definition}
 A (monotone) \emph{Galois connection} is a pair of order-preserving maps between posets $f\colon X \to Y$ and  $g\colon Y \to X$ such that $f(x) \le y \iff x \le g(y)$.
 We call $f$ the left adjoint and $g$ the right adjoint and we write $g = f_*$\glsadd{rightadjoint}. If $f$ has a further left adjoint, we will write it as $f_!$\glsadd{leftadjoint}.
\end{definition}
An order-preserving map between complete lattices is a left adjoint if and only if it preserves arbitrary joins and a right adjoint if and only if it preserves arbitrary meets.
Furthermore, $fgf = f$ and $gfg = g$. Hence $f$ is injective if and only if $g$ is surjective and visa versa. In this case, the one adjoint is a one-sided inverse for the other.

\begin{definition}
 An \emph{order-reversing Galois connection} between $X$ and $Y$ is monotone Galois connection between $X$ and $Y\op$ (where the \emph{left} adjoint maps from $X$ to $Y\op$).
\end{definition}

Note that $f\colon X \to Y$ and $g\colon Y \to X$ form an order-reversing Galois connection if and only if $gf(x) \ge x$ and $fg(y) \ge y$.
Both members of an order-reversing Galois connection\index{Galois connection|)} send joins to meets.

\begin{definition}
 A \emph{Heyting algebra\index{Heyting algebra|textbf}} is a lattice $L$ equipped with a \emph{Heyting implication} operation $\heyting$\glsadd{heyting} which satisfies $a \wedge b \le c \iff a \le b \heyting c$.
 In order words, $a \wedge (-)$ is left adjoint to $a \heyting (-)$.
 A Heyting algebra homomorphism preserves finite meets, finite joins and implication.
\end{definition}

The lattice of truth values $\Omega$ is a Heyting algebra, where $\heyting$ is the usual logical implication.

\subsection{Frames}\label{section:frames}

The classical approach to topology equips a topological structure to a prespecified set of points. In pointfree topology it is the lattice of open sets which is of fundamental importance
and the points are derived from this. This shift of focus from the points to the topology has a number of pleasant consequences. Of particular importance for us is that pointfree topology has a good
constructive formulation, whereas the constructive theory of topological spaces is badly behaved.

\begin{definition}
 A \emph{frame\index{frame|textbf}} is a complete lattice satisfying the frame distributivity condition \[x \wedge \bigvee_{\alpha \in I} x_\alpha = \bigvee_{\alpha \in I} (x \wedge x_\alpha)\] for arbitrary families
 $(x_\alpha)_{\alpha \in I}$.
 We denote the smallest element of a frame by $0$ and the largest element by $1$. A frame homomorphism is a function which preserves finite meets and arbitrary joins (and thus in particular $0$ and $1$).
\end{definition}

The distributivity condition gives that $x \wedge (-)$ preserves arbitrary joins and thus has a right adjoint. So a frame can be given the structure of a complete Heyting algebra\index{Heyting algebra!complete}, though the morphisms are different.

The prototypical example of a frame is the lattice of open sets of a topological space. In fact, a topological space can be defined as a set $X$ equipped with a subframe of the powerset of $X$.
Note that a continuous map of topological spaces induces a frame homomorphism \emph{in the opposite direction}. So it is not the category $\Frm$\glsadd{Frm} of frames, but its opposite category which is analogous
to the category of topological spaces. We call this opposite category $\Loc$\glsadd{Loc}, the category of \emph{locales}\index{locale|(textbf}.

The lattice of truth values $\Omega$\glsadd{Omega} is a frame. In fact, it is the initial frame. The unique frame homomorphism from $\Omega$ to a frame $L$ is given by $t \mapsto \bigvee\{1 \mid t \}$.

We discuss some relevant aspects of frame theory in the sections below.
For further background on frames, see \cite{PicadoPultr} for a comprehensive account of the classical theory and \cite[Part C]{Elephant2} for the constructive theory.

\subsubsection{Frames and topological spaces}

Let us describe the relationship between frames and spaces in more detail.
There is an obvious functor $\O\colon \Top \to \Frm\op$\glsadd[format=(]{O} sending spaces to their frames of open sets. %
This has a right adjoint $\points\colon \Frm\op \to \Top$\glsadd{points} which constructs a topological space out of a frame as follows.
The underlying set of $\points L$ is given by $\Hom_\Loc(1, L) = \Hom_\Frm(L, \Omega)$, where $1$ is the terminal locale, and the open sets are of the form $\{p\colon L \to \Omega \mid p(a) = \top\}$ for $a \in L$.
The action of $\Sigma$ on morphisms is given by composition in the obvious way.

Since maps from $L$ to $\Omega$ are determined by the preimage of $\top$, we can identify points of $L$ with certain subsets of $L$: the \emph{completely prime filters}\index{filter!completely prime|textbf}.
A filter $F$ is completely prime if $\bigvee A \in F \implies A \between F$ for all $A \subseteq L$, or equivalently, if $F$ is prime and Scott-open.
Assuming excluded middle, a completely prime filter is the complement of certain principal downsets --- those generated by what are called the \emph{prime elements} of $L$ --- but this often fails constructively.

The adjunction is idempotent and restricts to an equivalence between the so-called \emph{sober} spaces\index{sober space|textbf} and the \emph{spatial} locales\index{locale!spatial|textbf}\index{spatial locale|textbf}\index{enough points|see {spatial locale}}. %
The restriction to sober spaces is not severe, since almost all spaces encountered in practice are sober.
Of particular interest are the locales corresponding to discrete topological spaces. These are called \emph{discrete locales}\index{locale!discrete|textbf}\index{discrete locale|textbf} and their frames are given by powerset lattices.

\begin{remark}
 In this thesis we will treat sober spaces as locales without comment. In order to avoid confusion between the points of a space and elements of its frame of open sets,
 we will extend the notation $\O X$\glsadd[format=)]{O} to refer to the underlying frame of a locale $X$, even when it is not necessarily spatial. If $f\colon X \to Y$ is a morphism
 of locales, we will write $f^*$\glsadd{upperStar} for the corresponding frame homomorphism.\index{locale|)}
\end{remark}

Just as `inhabited' is a positive formulation of a set being nonempty, there is a positive version of a locale being nontrivial. We do not simply ask for the locale to possess a point, since
even classically there are nontrivial locales with no points.
\begin{definition}
 We say an element $a$ of a frame $L$ is \emph{positive} and write $a > 0$\glsadd{agt0} if every open cover of $a$ is inhabited --- that is, $a \le \bigvee S$ only if $S$ is inhabited. %
 We say $L$ is positive\index{positive locale|textbf} if $1 \in L$ is a positive element.
\end{definition}

\begin{lemma}\label{prop:positive_if_epic}
 A locale $X$ is positive if and only if the unique locale map $!$ from $X$ to the terminal locale is epic. %
\end{lemma}

\begin{proof}
 Suppose $X$ is positive. We must show that the corresponding frame map $!\colon \Omega \to \O X$ sending $p$ to $\bigvee\{1 \mid p\}$ is injective.
 Suppose $\bigvee\{1 \mid p\} = \bigvee\{1 \mid q\}$. Assume $p$ holds. Then $\bigvee\{1 \mid q\}$ is a cover which is inhabited by positivity
 and so $q = \top$. Thus, we have shown $p \implies q$. Symmetrically, we also have $q \implies p$ and hence $p = q$, as required.
 
 Conversely, suppose the map $!: \Omega \to \O X$ is injective and let $\bigvee S = 1 \in \O X$. Now notice that $1 \le \bigvee\{s \mid s \in S\} \le \bigvee \{1 \mid s \in S\}$.
 But the set $\{1 \mid s \in S\}$ is a subsingleton and thus equal to $\{1 \mid p\}$ for some $p \in \Omega$. We now have $\bigvee\{1 \mid p\} = 1 = \bigvee\{1 \mid \top\}$
 and hence $p = \top$ by injectivity. We may conclude that $S$ is inhabited and consequently $X$ is positive.
\end{proof}

Note that a set $X$ is inhabited if and only if the unique map $X \to 1$ is epic.
Furthermore, we have that a set is inhabited if and only if the corresponding discrete locale is positive.

\begin{lemma}\label{prop:positive_spaces}
 If $X$ is a topological space, then $\O X$ is positive if and only if $X$ is inhabited.
\end{lemma}

\begin{proof}
 Suppose $X$ is inhabited. Then $X \to 1$ is surjective and hence epic in $\Top$. But since $\O$ is a left adjoint, it preserves epimorphisms and hence $\O X$ is positive. %
 
 Now suppose $\O X$ is positive. Consider the open cover $X = \bigcup \{X \mid x \in X\}$. Since $\O X$ is positive, this join is inhabited and hence so is $X$.
\end{proof}

\subsubsection{Frames as algebraic structures}

Frames\index{frame|(} are best viewed as algebraic structures, with a nullary operation for $1$, a binary operation for binary meet and one operation for joins of each cardinality. %
Despite them having a proper class of operations, they are monadic over $\Set$ and most of the usual notions from universal algebra are available.
Thus, we can talk about \emph{subframes} and \emph{quotient frames}, which by duality correspond loosely to quotient spaces (as well as subtopologies) and subspaces respectively.
Notably frames may be presented\index{presentation|(} by generators and relations.

The free frame on a set $X$ can be constructed in two steps. First form the free meet-semilattice: the opposite of the poset of Kuratowski-finite subsets of $X$. %
Then the free frame is the lattice of downsets on this semilattice.
A presentation yields a frame by quotienting the free frame on the generators by the congruence generated by the relations.

Limits and colimits of frames can now be constructed as for any algebraic structure.
In particular, the coproduct\index{coproduct!of frames|(} $L \oplus M$ is given by the free frame on the set of symbols $\{\iota_1(\ell) \mid \ell \in L\} \sqcup \{\iota_2(m) \mid m \in M\}$ subject to relations ensuring that the maps
$\iota_1$ and $\iota_2$ are frame homomorphisms. The elements $\iota_1(\ell) \wedge \iota_2(m)$ are the pointfree analogues of the basic open rectangles in the Tychonoff topology
and are denoted by $\ell \oplus m$\glsadd{elloplusm}.

It should be mentioned that the coproduct of the frame of open sets of two topological spaces sometimes differs from the frame of open sets of the product topology.\index{coproduct!of frames|)}
Consequently, a monoid structure on a topological space might fail to lift to a localic monoid\index{monoid!localic|textbf}\footnote{By \emph{localic monoid} we mean a monoid object in
the (cartesian) category of locales or equivalently a comonoid in the monoidal category of frames with coproduct as the monoidal product.} structure on its corresponding locale.
(However, in the reverse direction, we do have that every spatial localic monoid is a topological monoid.)
It is true that finite products of discrete locales are discrete and hence discrete algebraic structures are always localic algebras.

It can also be useful to consider the forgetful functor from $\Frm$ to the category $\DLat$ of distributive lattices.
This has a left adjoint $\I$\glsadd{I} sending a distributive lattice to its frame of ideals.
The image of $\I$ is the category of so-called \emph{coherent} frames\index{coherent!frame}\index{frame!coherent} (see \cref{section:local_compactness}).
This is the pointfree version of Stone duality\index{spectrum!Stone} for distributive lattices (see \cite{StoneDistLattices} and \cite{Johnstone1982stone}).

The free frame on a single generator\glsadd[format=(]{Srpnsk} can be constructed as the frame of ideals of the free distributive lattice, which is the three element chain $\{0 \le m \le 1\}$.
It is also isomorphic to the frame of downsets on the two-element meet-semilattice $\{0 \le 1\}$. Finally, this frame can be viewed as the frame of opens of \emph{Sierpiński space}, which
has $\Omega$ as its underlying set and a topology generated by the subbasic open set $\{\top\}$. We denote the corresponding locale by $\Srpnsk$\glsadd[format=)]{Srpnsk}.

\subsubsection{Frames as propositional theories}\label{section:frame_propositional_theory}

Frames can also be viewed as Lindenbaum--Tarski algebras\index{Lindenbaum-Tarski algebra@Lindenbaum--Tarski algebra} (i.e.\ algebras of logical formulae considered up to logical equivalence) of a \emph{propositional geometric theory}\index{logic!geometric|(textbf}.
Points (i.e.\ completely prime filters) of the frame correspond to models of the propositional theory, with the completely prime filter\index{filter!completely prime} giving the true propositions in the model.
A propositional \emph{coherent} theory\index{logic!coherent}\index{coherent!logic} is described by sequents\index{sequent|(} involving formulae that are built up from $\top$, $\bot$, $\wedge$ and $\vee$.
These are described in detail in \cref{section:hyperdoctrine}. A propositional \emph{geometric} theory is similar, but the disjunctions are allowed to be infinitary. 

Geometric logic is the logic of \emph{verifiability}. When we say a proposition $p$ is verifiable\index{verifiable property|(}, we mean that if $p$ is true, there is some `proof' or `certificate' of this fact which we could confirm the
validity of in finite time.\footnote{Compare the concept of a semidecidable subset in computability theory or the class $\mathsf{NP}$ in computational complexity theory.}
Note that if $p$ is verifiable, it might not be the case that $\neg p$ is too. If $p$ and $q$ are verifiable, we can also verify $p \wedge q$ first verifying $p$ and then verifying $q$.
However, there is no obvious way to verify an infinitary conjunction. On the other hand, suppose $p_\alpha$ is verifiable for all $\alpha$. If $\bigvee_\alpha p_\alpha$ holds,
then $p_\alpha$ holds for some $\alpha$ and so the combination of $\alpha$ and a certificate for $p_\alpha$ gives a certificate for the disjunction. Thus, verifiable properties\index{verifiable property|)} are closed under infinite disjunction.
Further discussion can be found in \cite{Vickers1989topology,vickers2014geometric}.

The basic propositions and axioms of a geometric theory\index{logic!geometric|)} lead directly to a presentation of a frame.
As an example we consider the theory of Dedekind cuts on the rationals. Recall that a Dedekind cut
is a pair $(L,U)$ of subsets of $\Q$ satisfying certain conditions. For each $q \in \Q$, we have a
basic proposition $\ell_q$, which we interpret as meaning $q \in L$, and a basic proposition $u_q$,
which we interpret as meaning $q \in U$. These satisfy the following axioms. \pagebreak[0] %
\begin{displaymath}
\begin{array}{l@{\qquad\quad}r@{\hspace{1.5ex}}c@{\hspace{1.5ex}}l@{\quad}@{}l@{\qquad\quad}r@{}}
 (1) & \ell_q &\vdash& \ell_p & \text{ for $p \le q$} & \text{($L$ downward closed)} \\
 (2) & \ell_q &\vdash& \bigvee_{p > q} \ell_p & \text{ for $q \in \Q$} & \text{($L$ rounded)} \\
 (3) & \top &\vdash& \bigvee_{q \in \Q} \ell_q && \text{($L$ inhabited)} \\
 (4) & u_p &\vdash& u_q & \text{ for $p \le q$} & \text{($U$ upward closed)} \\
 (5) & u_q &\vdash& \bigvee_{p < q} u_p & \text{ for $q \in \Q$} & \text{($U$ rounded)} \\
 (6) & \top &\vdash& \bigvee_{q \in \Q} u_q && \text{($U$ inhabited)} \\
 (7) & \ell_p \land u_q &\vdash& \bot & \text{ for $p \ge q$} & \text{($L$ and $U$ disjoint)} \\
 (8) & \top &\vdash& \ell_p \lor u_q & \text{ for $p < q$} & \text{(locatedness)}
\end{array}
\end{displaymath}
In order to turn this axiomatisation into a presentation, we simply take one generator for each elementary proposition and one relation for each axiom, with a sequent\index{sequent|)} $\phi \vdash \psi$ corresponding to setting
$\phi \le \psi$ in the presented frame. Explicitly, we could write this as the relation $\phi \vee \psi = \psi$. However, we will often write $\le$ in a presentation for clarity.\index{presentation|)}

The frame obtained from the theory of Dedekind cuts is the frame of reals $\O \R$. The points of the locale $\R$ correspond to the usual constructive notion of Dedekind reals.
However, the locale $\R$ need not be spatial constructively.\index{frame|)}

\begin{remark}
 Taking only the basic propositions $\ell_p$ and axioms 1--3 we obtain the theory of lower one-sided Dedekind cuts. The corresponding locale is the locale of lower reals.
 Classically, this has the same points as $\R$ together with $\infty$, while its topology is the topology of lower semicontinuity. Similarly, taking the propositions $u_p$ and axioms 4--6
 we obtain the theory of upper one-sided Dedekind cuts and the locale of upper reals. The norm of a localic C*-algebra takes values in the (nonnegative, extended) upper reals (see \cite{Henry2016} for details).
\end{remark}

Another example, which is central to our theme of spectra, is given by the theory of prime anti-ideals\index{anti-ideal!prime} of a commutative ring $R$. Each element $f \in R$ gives a proposition $\overline{f}$
which we interpret as meaning that $f$ lies in the prime anti-ideal under consideration.
Expressing the axioms of a prime anti-ideal as a geometric theory, we have the following.
\begin{align*}
 \overline{0}    &\mathmakebox[4.5ex][c]{\vdash} \bot \\
 \overline{f+g}  &\mathmakebox[4.5ex][c]{\vdash} \overline{f} \lor \overline{g} \\
 \top &\mathmakebox[4.5ex][c]{\vdash} \overline{1} \\
 \overline{fg}   &\mathmakebox[4.5ex][c]{\dashv\vdash} \overline{f} \land \overline{g}
\end{align*}
Not coincidentally, this is precisely the same geometric theory we would obtain by attempting to axiomatise the theory of `places at which elements of $R$ can be cozero'.
This is a presentation of the \emph{Zariski spectrum\index{spectrum!Zariski|textbf}} of $R$.

Explicitly, the frame obtained from this propositional theory has the following presentation.
\[\langle \overline{f} : f \in R \mid \overline{0} = 0,\, \overline{f+g} \le \overline{f} \vee \overline{g},\,
\overline{1} = 1,\, \overline{fg} = \overline{f} \wedge \overline{g} \rangle\]
We will see later that this frame corresponds to the lattice $\Rad(R)$ of radical ideals of $R$.

\subsubsection{Sublocales} \label{section:sublocales}

Let us discuss frame quotients in more detail. Taking a quotient of a frame corresponds to adding axioms to any associated geometric theory.
We will often call quotient frames \emph{sublocales}\index{sublocale|(textbf} when we wish to emphasise the analogy with subspaces.
The lattice of sublocales $\Sloc L$\glsadd{Sloc} of a frame $L$ is distributive (in fact, it is a coframe) and pullbacks of sublocales in $\Loc$ respect these lattice (coframe) operations.
This gives rise to a functor $\Sloc\colon \Loc\op \to \DLat$, where $\DLat$ is the category of distributive lattices.

If $f\colon X \to Y$ is a morphism of locales, then $\Sloc f$ can be thought of mapping sublocales of $Y$ to their preimages under $f$. The map $\Sloc f$ has a left adjoint $(\Sloc f)_!$
which we interpret as sending sublocales of $X$ to their images under $f$.

Sublocales can be specified by congruences\index{congruence on a frame}. Each equivalence class of a frame congruence $C$ is inhabited and closed under inhabited joins; therefore, it has a largest element.
This gives a canonical representative of each equivalence class. The resulting map $a \mapsto \bigvee [a]_C$ is called the \emph{nucleus}\index{nucleus|(textbf} associated with the congruence.
\begin{definition}
 A \emph{nucleus}\index{nucleus|)} $\nu$ on a frame is a closure operator which preserves finite meets.
\end{definition}
Nuclei then provide an alternative approach to constructing frame quotients. A nucleus $\nu$ induces a congruence $\{(a,b) \mid \nu(a) = \nu(b)\}$ and the resulting quotient frame is order isomorphic to
the poset of fixed points of the nucleus (or equivalently, the image of the nucleus). Given a surjective frame homomorphism $q\colon L \twoheadrightarrow M$, we can recover the nucleus directly as $q_*q$.

We now describe some important classes of sublocales.
If $a$ is an element of a frame $L$, the principal downset ${\downarrow} a$ is also a frame and there is a natural surjective frame homomorphism onto ${\downarrow} a$ sending $x$ to $x \wedge a$.
We call this an \emph{open quotient}\index{sublocale!open|textbf}, since if $L$ is the frame of a topological space and $U$ is an open set, then ${\downarrow} U$ is the frame of $U$ viewed as an open subspace.
The corresponding congruence is $\Delta_a = \langle (a,1) \rangle$\glsadd{Deltaa} and can be given explicitly as $\Delta_a = \{(x,y) \mid x \wedge a = y \wedge a\}$.
As we might expect, open sublocales are closed under finite meets and arbitrary joins in the lattice of sublocales.

While open subsets of a topological space might not have complements constructively, open sublocales always have complements in the lattice of sublocales. We call these \emph{closed sublocales}\index{sublocale!closed|textbf}\index{closed quotient!of a frame|textbf}.
They can be given by the surjection from $L$ to ${\uparrow} a$ sending $x$ to $x \vee a$ and the corresponding congruence is $\nabla_a = \langle (0,a) \rangle = \{(x,y) \mid x \vee a = y \vee a\}$\glsadd{nablaa}.
It is not hard to see that every closed sublocale of $\Omega$ is open only if excluded middle holds.

The principal congruence\index{congruence on a frame!principal|textbf} generated by a pair $(a,b)$ can be expressed in terms of open and closed congruences as
$\langle (a,b)\rangle = \langle (0, a \vee b) \rangle \cap \langle (a \wedge b,1) \rangle = \nabla_{a \vee b} \cap \Delta_{a \wedge b}$.
In this way every congruence can be expressed as a join of finite meets of open and closed congruences.\index{sublocale|)}

Open sublocales of discrete locales correspond to subsets. This provides an avenue to generalise certain relations on sets to general locales.
For example, ideals of a semiring $R$ are subsets $I$ satisfying $0 \in I$, $(x \in I) \land (y \in I) \implies x + y \in I$ and $x \in I \implies xy \in I$.
We can then interpret this in the internal logic of the indexed distributive lattice of open sublocales\index{hyperdoctrine!of open sublocales}
(see \cref{def:indexed_dist_lattice}), as in \cref{section:hyperdoctrine}, to describe open ideals of localic semirings.
We can also interpret the same axioms in the indexed distributive lattice $\Sloc$ of all sublocales to define a general notion of ideal, though this will be strictly more general than usual notion in the discrete case.
Finally, it can be useful to interpret ideals in the indexed distributive lattice of \emph{closed} sublocales, yielding \emph{closed ideals} in our example. Closed ideals of a discrete ring are precisely the
closed complements of the \emph{anti-ideals}\index{anti-ideal} of constructive algebra. The following lemma describes how to interpret more general closed substructures in discrete locales.
This lemma uses concepts we will only introduce in \cref{section:hyperdoctrine}, but we state it here since it gives some understanding of the behaviour of closed sublocales in a constructive setting.
\begin{proposition}
Consider the theory $T$ of $n$-ary relations on a set $X$ satisfying axioms given as sequents $\phi \vdash \psi$ in the fragment of coherent logic without quantification.
We define the \emph{dual theory} $T\op$ to be the theory of relations satisfying the dual axioms, $\psi\op \vdash \phi\op$, where $\phi\op$ is obtained from $\phi$ by exchanging
$\wedge$ and $\vee$ (and $\top$ and $\bot$). Then there is a correspondence between the models of $T$ in the indexed distributive lattice of closed sublocales and models of $T\op$
in the hyperdoctrine of subsets.
\end{proposition} %
\begin{proof}
 The hyperdoctrine of subsets is equivalent to the hyperdoctrine of open sublocales of discrete locales, which we can then view as an indexed distributive sublattice of $\Sloc$.
 Taking complements provides a bijection between open sublocales and closed sublocales which reverses order and sends joins to meets and visa versa. By the nature of how the theory is interpreted
 in these structures, this order reversal corresponds to moving to the dual theory as described above.
\end{proof}

With the notion of closed sublocale comes the notion of \emph{density}.
\begin{definition}
 A frame homomorphism $f$ is \emph{dense}\index{dense|textbf} if $f(x) = 0 \implies x = 0$.
\end{definition}
Every frame homomorphism $f\colon L \to M$ factors as $L \twoheadrightarrow {\uparrow} f_*(0) \to M$ to give a closed quotient followed by a dense frame map.

\subsubsection{Compactness}\label{section:local_compactness} %

It is easy to define compactness for frames.
\begin{definition}
 An element $a$ of a frame is \emph{compact}\index{compact|textbf} if whenever $a \le \dirsup D$ for a directed set $D$, there exists an element $d \in D$ such that $a \le d$.
 We say a frame $L$ is compact if $1 \in L$ is a compact element.
\end{definition}

It is straightforward to see that compact elements are closed under finite joins.

The definition of local compactness might be slightly less familiar.
\begin{definition}
 We say $b$ is \emph{way below} $a$ and write $b \ll a$\glsadd{ll} if whenever $\dirsup D \ge a$ there exists an element $d \in D$ such that $d \ge b$.
 In particular, $a \ll a$ if and only if $a$ is compact.
 We define $\twoheaddownarrow a = \{b \mid b \ll a \}$\glsadd{twoheaddownarrow} and $\twoheaduparrow b = \{a \mid b \ll a \}$. %
\end{definition}
\begin{definition}
 A frame $L$ is \emph{locally compact}\index{locale!locally compact|textbf}\index{locally compact locale|textbf} or \emph{continuous}\index{continuous frame|textbf}\index{frame!continuous|textbf} if every $a \in L$ can be expressed as a join
 \[a = \bigvee_{b \ll a} b.\]
\end{definition}

Compactness, the way-below relation and continuity make sense for any dcpo. Though for a dcpo to be continuous, we must additionally require that $\twoheaddownarrow a$ is directed for all $a$.\index{continuous dcpo|textbf}\index{dcpo!continuous|textbf}
This condition is automatic for a complete lattice, since $b \ll a$ and $b' \ll a$ easily gives $b \vee b' \ll a$.
If a complete lattice is distributive and continuous, it is automatically a frame.

\begin{lemma}\label{prop:way_below_interpolates}
 If $X$ is a continuous dcpo, the way below relation on $X$ interpolates --- that is, if $c \ll a$ there is an element $b \in X$ such that $c \ll b \ll a$.
\end{lemma}
\begin{proof}
 Suppose $c \ll a$ and write $a = \dirsup[b \ll a] b$. Now for each $b \ll a$, we have $b = \dirsup[d \ll b] d$
 and hence $a = \dirsup[d \ll b \ll a] d$. Then since $c \ll a$, there are some $d$ and $b$ with $d \ll b \ll a$ such that $c \le d$
 and thus $c \ll b \ll a$ as required.
\end{proof}

\begin{lemma}\label{prop:base_for_Scott_topology}
 If $X$ is a continuous dcpo, then subsets of the form $\twoheaduparrow x$ are Scott-open and such open sets form a base for the Scott topology\index{Scott topology|(} on $X$.
\end{lemma}
\begin{proof}
 We first show $\twoheaduparrow x$ is Scott-open. Clearly it is an upset. Now suppose $\dirsup D \in \twoheaduparrow x$.
 This means $x \ll \dirsup D$ and interpolating we have $x \ll y \ll \dirsup D$ for some $y \in X$. Now $x \ll y \le d$ for some $d \in D$ by the definition of the way below relation
 and hence $d \in \twoheaduparrow x$ as required.
 
 Now suppose $U$ is a general Scott-open set and consider $x \in U$. Then $\dirsup \twoheaddownarrow x \in U$
 and thus there is a $y \in U$ with $y \ll x$. Therefore, $x \in \twoheaduparrow y \subseteq U$ as required.
\end{proof}

\begin{remark}\label{rem:continuous_dpco_map_into_omega}
 We can also interpret the fact that $\twoheaduparrow x$ is Scott-open as saying that $\llbracket x \ll (-)\rrbracket\colon X \to \Omega$ is a dcpo morphism. %
\end{remark}

It is now not difficult to show that continuous dcpos equipped with the Scott topology are sober topological spaces
and that a function between continuous dcpos is continuous with respect to the Scott topology if and only if it is monotone and preserves directed joins.
However, we will not prove this here as these results are somewhat tangential to our core interests.

We now return our focus to continuous \emph{frames}. Continuous frames are precisely the \emph{coexponentiable} frames ---
equivalently, locally compact locales are the exponentiable locales\index{exponential of locales|(}.
For more details see \cite{hyland1981exponentials} where an explicit description of the coexponential\index{coexponential!of frames} is given in terms of generators and relations.

Of particular interest are exponentials of the form $\Srpnsk^X$. If $X$ is a locally compact locale, then $\Srpnsk^X$ is isomorphic to a sober space with underlying set $\O X$ equipped with the Scott topology\index{Scott topology|)}.
See \cite[Section C4.1]{Elephant2} for details.\index{exponential of locales|)}

Note that $\Srpnsk^X$ is a localic distributive lattice with lattice operations inherited from $\Srpnsk$. These are simply the usual meet and join operations on the underlying set $\O X$.

\begin{definition}
 A frame is \emph{coherent}\index{coherent!frame|textbf}\index{frame!coherent|textbf} if it has a base of compact elements which is closed under finite meets.
\end{definition}
It is easy to see that a coherent frame is continuous.
The compact elements of a coherent frame form a distributive lattice and the poset of ideals of a distributive lattice is a frame.
These are inverse operations and yield an equivalence between the category of distributive lattices and the category of coherent frames and frame homomorphisms which preserve compact elements.
Coherent frames are precisely the frames which can be obtained as the Lindenbaum--Tarski algebras of coherent theories.

An example of a coherent frame is the frame of radical ideals of a semiring. The compact elements are given by the finitely generated radical ideals.

We are also interested in a somewhat less well-known variant of local compactness.
\begin{definition}
 We say $b$ is \emph{totally below} $a$ and write $b \lll a$\glsadd{lll} if whenever $\bigvee S \ge a$ there exists an element $d \in S$ such that $d \ge b$.
 We define ${\Downarrow} a = \{b \mid b \lll a\}$\glsadd{Downarrow}.
\end{definition}
\begin{definition}\label{def:supercontinuous}
 A frame $L$ is \emph{supercontinuous}\index{supercontinuous|(textbf}\index{frame!supercontinuous|(textbf}\index{suplattice!supercontinuous|(textbf} if every $a \in L$ can be expressed as a join
 \[a = \bigvee {\Downarrow} a = \bigvee_{b \lll a} b.\]
\end{definition}
Supercontinuity can be defined for general complete lattices, but any supercontinuous complete lattice is automatically a frame.
In fact, supercontinuous lattices are called `constructively completely distributive' in \cite{Fawcett1990completeDistributivity}, since under the assumption of the axiom of choice it
is equivalent to complete distributivity of the lattice.\index{supercontinuous|)}\index{frame!supercontinuous|)}\index{suplattice!supercontinuous|)}

We will find the following results about the totally below relation to be useful.
The proofs proceed as in \cref{prop:way_below_interpolates} and the first half of \cref{prop:base_for_Scott_topology} respectively.

\begin{lemma}
 In a supercontinuous frame the totally below relation interpolates.
\end{lemma}

\begin{corollary}
 Let $L$ be a supercontinuous frame and take $b \in L$. Then the map $\llbracket b \lll (-) \rrbracket \colon L \to \Omega$ preserves arbitrary joins.
\end{corollary}

\subsection{Suplattices and quantales}

\subsubsection{Suplattices}

A \emph{suplattice}\index{suplattice|textbf} is a poset admitting arbitrary joins. As objects suplattices are the same as complete lattices, but suplattice homomorphisms need only preserve joins. We call the category of suplattices $\Sup$\glsadd{Sup}.

We will write $0$ for the least element of a suplattice and $\top$\glsadd{top} for the greatest element.

Like $\Frm$, the category $\Sup$ is monadic over $\Set$ and hence complete and cocomplete. In fact, the underlying functor of the monad in question is the covariant powerset functor,
which sends a set $X$ to $\powerset{X}$ and a function $f$ to its direct image function $(\Omega^f)_!$.

An \emph{inf"|lattice\index{inflattice@inf{\kern.03em}lattice|textbf}} is a poset admitting arbitrary meets. The category $\Sup$ is equivalent to the category of inf"|lattices $\Inf$\glsadd{Inf} by simply reversing the order of each poset.
But $\Sup$ is also equivalent to $\Inf\op$ by leaving the objects alone and sending each morphism to its right adjoint. Composing these gives a self-duality $\Sup \cong \Sup\op$.

This duality makes colimits particularly easy to compute in $\Sup$. It is easy to see that the category has arbitrary (not just finite) biproducts. Coequalisers can be computed as quotients
in the usual way, but we can use the duality to instead compute them as a sub-inf"|lattice. In particular, if $q\colon L \to M$ is a suplattice quotient, it is often convenient to view $M$ as
a subset of $L$ given by the image of $q_*$. Alternatively, we can describe this subset as the subset fixed by the closure operator $q_*q$.

Since $\Sup$ is monadic over $\Set$, the monomorphisms in $\Sup$ are precisely the injections and the regular epimorphisms are precisely the surjections.
But then by duality all epimorphisms are surjections and all injections are regular monomorphisms. Hence every mono/epi is regular and $\Sup$ has (epi,mono)-factorisations.

The forgetful functor from $\Sup$ to the category $\SLat{\vee}$ of join-semilattices is also occasionally of interest.
It has a left adjoint $\I$\glsadd{I} which sends a join-semilattice to its suplattice of order ideals.
Thus, both the squares in the diagram below commute, where the unmarked arrows are the obvious forgetful functors.
\begin{center}
\begin{tikzpicture}[node distance=2.5cm, auto]
  \node (A) {$\Frm$};
  \node (B) [below of=A] {$\DLat$};
  \node (C) [right of=A] {$\Sup$};
  \node (D) [right of=B] {$\SLat{\vee}$}; %
  \draw[->] (A) to node {} (C);
  \draw[->] (B) to node [swap] {} (D);
  \draw[transform canvas={xshift=1.0ex},->] (A) to node {} (B);
  \draw[transform canvas={xshift=-1.0ex},<-] (A) to node [swap] {$\I$} (B);
  \path (A) to node [auto=false] {\scriptsize$\dashv$} (B);
  \draw[transform canvas={xshift=1.0ex},->] (C) to node {} (D);
  \draw[transform canvas={xshift=-1.0ex},<-] (C) to node [swap] {$\I$} (D);
  \path (C) to node [auto=false] {\scriptsize$\dashv$} (D);
\end{tikzpicture}
\end{center}

Another relevant forgetful functor is the one from $\Sup$ into the category $\Pos$ of posets. This has a left adjoint $\D$\glsadd{D} which sends a poset to its suplattice of downsets.
The counit $\joinmap_L\colon \D L \twoheadrightarrow L$\glsadd{joinmap} of this adjunction sends a downset $D$ to its join $\bigvee D$ in $L$.

The category of suplattices has the structure of a symmetric monoidal category.
The tensor product is defined in a very similar way to that of abelian groups or commutative monoids.
\begin{definition}
 Suppose $L, M$ and $N$ are suplattices. A function $f\colon L \times M \to N$ is said to be \emph{bilinear} if $f(-,y)\colon L \to N$ and $f(x,-)\colon M \to N$ are suplattice homomorphisms
 for all $x \in L$ and $y \in M$.
 
 The \emph{tensor product} $L \otimes M$ of $L$ and $M$ is given by `initial bilinear map' in the sense that it is the representing object %
 of the obvious functor $\mathrm{Bilin_{L,M}}\colon \Sup \to \Set$ which sends an object to the set of bilinear maps from $L \times M$ into it.
\end{definition}

The tensor product $L \otimes M$ can be given explicitly as the free suplattice on the generators $\ell \otimes m$ for $\ell \in L$ and $m \in M$ subject to %
the obvious relations making the map $(\ell,m) \mapsto \ell \otimes m$ bilinear.

This tensor product has a right adjoint, making $\Sup$ symmetric monoidal closed. The internal hom object $\hom(L,M)$\glsadd{hom} is given by the set of suplattice homomorphisms from $L$ to $M$ equipped with the pointwise
suplattice structure. This also means $\Sup$ is enriched over itself.

\subsubsection{Quantales}

We are now in a position to define quantales\index{quantale|textbf}. See \cite{galoisTheoryGrothendieck} and \cite{RosenthalQuantales} for further details.

\begin{definition}
 A \emph{quantale} is a monoid in $\Sup$ --- that is a suplattice equipped with a bilinear associative multiplication operation and a unit for this multiplication.
 A quantale homomorphism is then a suplattice homomorphism which preserves the monoid structure. %
\end{definition}
Every frame is an example of a quantale with meet as the multiplication. Quantales are sometimes described as ``frames with noncommutative meet operations'',
but it is equally important that the multiplication is \emph{non-idempotent}. Indeed, in this thesis all quantales will be \emph{commutative}. %
The category of (commutative) quantales with quantale homomorphisms will be denoted by $\QuantGeneral$\glsadd{QuantGeneral}.

As with frames, we may define an implication operation $\heyting$\glsadd{heyting} on a quantale $Q$, where $a \heyting (-)$ is right adjoint to multiplication by $a$.
We will write $x^\bullet$\glsadd{abullet} for $x \heyting 0$. When $Q$ is a frame, this is known as the \emph{pseudocomplement} of $x$.

Quantales provide models for a fragment of intuitionistic linear logic. Linear logic is a resource-sensitive logic where we keep track of the number of times we must
assume a variable in order to derive a particular conclusion. We will see that this extra structure afforded to us by using quantales will be able to capture some measure
of `differential' information, in addition to the topological information that would be given by a frame.

An important class of quantales is given by the free quantale on a commutative monoid. An analogue of this construction will be important in \cref{section:spectrum}.
\begin{example}\label{ex:quantale_on_a_monoid}
 Let $M$ be a commutative monoid. The free quantale over $M$ is given by the quantale of subsets of $M$ where multiplication is defined by setting $ST = \{st \mid s \in S, t \in T\}$ for $S, T \subseteq M$.
\end{example}

We write $1$ for the unit of a quantale and $\top$ for its largest element.
\begin{definition}
 We say a quantale is \emph{two-sided}\index{quantale!two-sided|(textbf}\index{two-sided|(textbf} if $1 = \top$.
\end{definition}
Most quantales we deal with will be two-sided. The category of \emph{two-sided} quantales will be denoted by $\Quant$\glsadd{QuantTwoSided}.
In logical terms two-sided quantales correspond to \emph{affine} logic, which permits the disposal, but not the duplication, of resources.

\begin{example}\label{ex:quantale_of_monoid_ideals}
 Let $M$ be a commutative monoid. The complete lattice of monoid ideals of $M$ forms a two-sided quantale\index{quantale of monoid ideals} where multiplication is given elementwise. Explicitly, $IJ = \{ ij \mid i \in I, j \in J\}$.
 In fact, this is the free two-sided quantale on the commutative monoid $M$.
\end{example}

\begin{example}\label{ex:quantale_of_ideals}
 Let $R$ be a semiring. The complete lattice of semiring ideals\index{quantale of ideals|(textbf} of $R$ forms a two-sided quantale $\Idl(R)$\glsadd{Idl} where the multiplication is given by $IJ = \langle ij \mid i \in I, j \in J\rangle$.
\end{example}

Like frames and suplattices, quantales can be viewed as algebraic structures and so we can describe quotients of quantales by congruences.
But as with frames, we will also sometimes want to use nuclei.
\begin{definition}
 A \emph{nucleus}\index{nucleus|textbf} $\nu$ on a quantale is a closure operator which satisfies $\nu(a)\nu(b) \le \nu(ab)$.
\end{definition}
The quotient given by a nucleus $\nu$ has the fixed points of $\nu$ as its underlying sublattice and a multiplication defined by $ab = \nu(a \circ b)$ where $\circ$ is the multiplication in the parent quantale.
On the other hand, given a surjective quantale homomorphism $q\colon Q \to S$ we can recover the corresponding nucleus as $q_*q$.

If $Q$ is a two-sided quantale and $c \in Q$, it is easy to show that the congruence $\langle(0,c)\rangle$ corresponds to the nucleus $a \mapsto a \vee c$
and the underlying suplattice of the quotient is isomorphic to ${\uparrow} c$.
By analogy to frames we write $\nabla_c = \langle(0,c)\rangle$\glsadd{nablaa} and call the quotients induced by such congruences \emph{closed quotients}\index{closed quotient!of a quantale|textbf}.

Two-sided quantales\index{quantale!two-sided|)} form a subvariety of the class of all quantales defined by the equation $1 \vee x = 1$.
Consequently, the category $\Quant$ is a reflective subcategory of $\QuantGeneral$. The unit of the reflection is given by quotient determined by the nucleus $a \mapsto a \top$.
(In fact, $\Quant$ is also \emph{coreflective} in the category of quantales, where the counit is given by the inclusion of the subquantale ${\downarrow} 1$.)

Similarly, the category of frames forms a full subcategory of $\Quant$ (and hence of the category of all quantales) which is both reflective and coreflective.
We call the reflector $\locrefl$\glsadd{locrefl}. The unit of reflection $\rad_Q\colon Q \twoheadrightarrow \locrefl Q$\glsadd{rad} is given by the quotient
obtained from setting $a \sim a^2$. The fixed points of the corresponding nucleus are the so-called \emph{semiprime} elements: the elements $a$ such that $b^2 \le a \implies b \le a$.
In \cref{ex:quantale_of_ideals} these elements are the radical ideals\index{ideal (discrete algebraic)!semiring!radical} of $R$
and so the localic reflection of $\Idl(R)$ is the frame of radical ideals $\Rad(R)$\glsadd{Rad} (with the map $\rad_{\Idl(R)}$ sending an ideal $I$ to $\sqrt{I}$).

As algebraic structures, we can present quantales by generators and relations.
In \cref{section:frame_propositional_theory} we gave a presentation of the Zariski spectrum of a ring $R$.
Replacing meets with products gives a presentation for a two-sided quantale which is the subject of the following lemma.
\begin{lemma}\label{prop:presentation_for_Idl_R}
 If $R$ is a semiring, then the two-sided quantale
\[\langle \overline{f} : f \in R \mid \overline{0} = 0,\, \overline{f+g} \le \overline{f} \vee \overline{g},\,
\overline{1} = 1,\, \overline{fg} = \overline{f} \cdot \overline{g} \rangle\]
 is isomorphic to the quantale of ideals $\Idl(R)$.\index{quantale of ideals|)}
\end{lemma}
\begin{proof}
 The not-necessarily-two-sided quantale $\langle \overline{f} : f \in R \mid \overline{1} = 1,\, \overline{fg} = \overline{f} \cdot \overline{g} \rangle$
 is simply the free quantale on the multiplicative monoid of $R$ and is hence isomorphic to $\Omega^R$ with the obvious multiplication.
 Applying the reflection into two-sided quantales gives the quotient quantale corresponding to the nucleus $a \mapsto a \top$. The fixed points of this nucleus are the
 subsets $S \subseteq R$ satisfying $S = S \cdot \top = \{sr \mid s \in S, r \in R\}$, which are precisely the monoid ideals.
 Now add back the additive conditions. The condition $\overline{0} = 0$ corresponds to the nucleus $a \mapsto a \vee \overline{0}$, the fixed points of which are monoid ideals
 which contain $0$. Finally, the conditions $\overline{f+g} \le \overline{f} \vee \overline{g}$ then pick out the monoid ideals which are closed under finite sums.
 A product in the quotient is given by the ideal generated by the product in $\Omega^R$, as expected.
\end{proof}
\begin{corollary}
 The Zariski spectrum\index{spectrum!Zariski} of $R$ is isomorphic to $\Rad(R)$.\index{frame of radical ideals}
\end{corollary}

Binary coproducts of quantales are given by tensor products\index{coproduct!of quantales|textbf} of their underlying suplattices with multiplication induced by $(a \otimes b)(a' \otimes b') = aa' \otimes bb'$,
similarly to how coproducts of rings are given by tensor products of their underlying abelian groups. (Note that these results make essential use of commutativity.)
The initial quantale is the frame of truth values $\Omega$.

The free two-sided quantale on a single generator --- denoted by $\Nvar$\glsadd{Nvar} --- is more interesting. It is the quantale of monoid ideals on $\N$ under addition.
Classically, this is isomorphic to $\N \cup \{\infty\}$ with the reverse order and with addition as the monoid operation.
Constructively, this still gives the free two-sided\index{idempotent semiring!two-sided|textbf}\index{two-sided|)}\footnote{As with quantales, we will define an idempotent semiring to be two-sided if $1$ is the largest element in the natural order structure.} idempotent semiring
on one generator, but it cannot be proved to be complete. We obtain a quantale by applying the functor $\I$ as described below.
(Compare this to how the corresponding free frame $\O \Srpnsk$ is given by applying $\I$ to the 3 element chain.) %

Just as there is a forgetful functor from $\Frm$ to $\DLat$,  there is a forgetful functor from $\Quant$ to the category of two-sided idempotent semirings.
This has a left adjoint, which is obtained from applying the functor $\I$\glsadd{I} to the underlying suplattices. Multiplication is defined in the obvious way by generating an order ideal from the elementwise product.
\begin{remark}
 Note that this functor is given by taking \emph{order ideals}\index{ideal (order)}. This is not the same as the functor $\Idl$, which takes \emph{semiring ideals}.
 Every order ideal of a two-sided idempotent semiring is a semiring ideal, but they do not coincide in general!
 For example, in the free two-sided idempotent semiring on $\{x,y\}$ the principal semiring ideal $(x \vee y)$ fails to contain $x$.
\end{remark}

\begin{definition}
 A two-sided quantale is \emph{coherent}\index{coherent!quantale|textbf}\index{quantale!coherent|textbf} if it has a base of compact elements which is closed under finite products.
\end{definition}
Our main example of a coherent quantale is the quantale of (semiring) ideals $\Idl(R)$ of a semiring $R$. Since directed joins of ideals are given by set-theoretic union,
it is not hard to show that the compact elements of $\Idl(R)$ are precisely the finitely generated ideals.

Coherent quantales behave in a very similar way to coherent frames. They have a two-sided idempotent semiring $K$ of compact elements and can be recovered as $\I K$.
Furthermore, a two-sided idempotent semiring $D$ can be recovered from the compact elements of $\I D$, which are the principal order ideals.

\subsubsection{Quantale modules}

There is a clear analogy in the above sections between suplattices and abelian groups and between quantales and rings. As with rings, it is possible to define `modules' over quantales.
See \cite{galoisTheoryGrothendieck} for further details.

\begin{definition}
 Let $Q$ be a quantale. A \emph{$Q$-module}\index{quantale module|(textbf}\index{Q module@$Q$-module|see {quantale module}} is a suplattice $M$ equipped with a quantale homomorphism $\alpha$ from $Q$ to the noncommutative quantale of endomorphisms $\End(M)$.\footnote{
 We take multiplication in $\End(M)$ to be function composition in the usual order. Our definition is then of a \emph{left} module. However, since $Q$ is assumed to be commutative, we will not be
 careful to distinguish between left and right modules.}
 We will write $q \cdot m$\glsadd[format=(]{cdot} (or $m \cdot q$) for $\alpha(q)(m)$ when the choice of structure map is obvious from context.
 
 A $Q$-module homomorphism between $Q$-modules $M$ and $N$ is a suplattice homomorphism $f\colon M \to N$ satisfying $f(q \cdot m) = q \cdot f(m)$.
\end{definition}

The category $\Mod{Q}$\glsadd{ModQ} of $Q$-modules behaves similarly to that of suplattices. It too is symmetric monoidal closed and its tensor product may be constructed in a similar way.
A monoid in $\Mod{Q}$ is called a \emph{$Q$-algebra} and can be thought of as quantale with a compatible $Q$-module structure. %

Notice that any quantale homomorphism $f\colon Q \to P$ induces a $Q$-module structure on $P$ by $q \cdot p = f(q)p$. This gives $P$ the structure of a $Q$-algebra. In fact, this gives an equivalence between
the category of (commutative) $Q$-algebras and the coslice category $Q/\QuantGeneral$. Note in particular that any quantale $Q$ has a unique $\Omega$-algebra\footnote{This should not be confused with $\Omega$-algebras
in the sense of universal algebra, despite the unfortunate clash of terminology.} structure (and all quantale homomorphisms preserve this structure). %
We will occasionally use this to write $t \cdot a$\glsadd[format=)]{cdot} for $\bigvee \{ a \mid t \}$ where $a \in Q$ and $t \in \Omega$.

Coproducts in the category of $Q$-algebras correspond to pushouts in $\QuantGeneral$ and their underlying suplattices are given by these tensor products in $\Mod{Q}$.

Quantale modules were first introduced to help with the study \emph{open} morphisms of frames.

\begin{definition}
 A frame homomorphism $f\colon L \to M$ is \emph{open}\index{open map|(textbf} if it is a complete Heyting algebra\index{Heyting algebra!complete} homomorphism --- that is, if it preserves implication and arbitrary meets.
\end{definition}

Recall that if $f\colon M \to L$ is a morphism of locales, then $(\Sloc f)_!$ sends sublocales of $M$ to their images under $f$. The above definition is equivalent
to $(\Sloc f)_!$ sending open sublocales to open sublocales. In particular, it is a conservative extension of the notion of open quotient from \cref{section:sublocales}.

\begin{lemma}\label{prop:open_map_module_adjoint}
 A frame homomorphism $f\colon L \to M$ is open if and only if it admits a left adjoint in $\Mod{L}$\index{quantale module|)}.
\end{lemma}
\begin{proof}
 Certainly, the map $f$ has an order-theoretic left adjoint $f_!$ if and only if $f$ preserves arbitrary meets.
 Suppose this is the case. We now show that $f$ preserves implication if and only if $f_!$ is an $L$-module homomorphism. This latter condition is $f_!(q \cdot p) = q \cdot f_!(p)$,
 or more explicitly, $f_!(p \wedge f(q)) = f_!(p) \wedge q$.
 
 Observe the following sequence of equivalences:
 \begin{align*}
       & p \le f(q) \Rightarrow f(r) \\
  \iff & p \wedge f(q) \le f(r) \\
  \iff & f_!(p \wedge f(q)) \le r.
 \end{align*}
 Additionally, we have
 \begin{align*}
       & p \le f(q \Rightarrow r) \\
  \iff & f_!(p) \le q \Rightarrow r \\
  \iff & f_!(p) \wedge q \le r.
 \end{align*}
 
 The desired result now follows easily.
\end{proof}

We may now show open locale maps are stable under pullback.

\begin{lemma}\label{prop:open_pullback_stable}
 Suppose $f\colon L \to M$ and $g\colon L \to N$ are frame homomorphisms. If $f$ is open, then so is the map $f'$ in the pushout diagram below.\index{open map|)}
 Furthermore, we have the so-called Beck--Chevalley condition\index{Beck-Chevalley condition@Beck--Chevalley condition|textbf} $gf_! = f'_! g'$.
 \begin{center}
  \begin{tikzpicture}[node distance=2.5cm, auto]
    \node (A) {$L$};
    \node (B) [below of=A] {$M$};
    \node (C) [right of=A] {$N$};
    \node (D) [right of=B] {$M \otimes_L N$};
    \draw[->] (A) to node [swap] {$f$} (B);
    \draw[->] (A) to node {$g$} (C);
    \draw[->] (B) to node [swap] {$g'$} (D);
    \draw[->] (C) to node {$f'$} (D);
    \begin{scope}[shift=(D)] %
        \draw +(-0.4,0.8) -- +(-0.8,0.8)  -- +(-0.8,0.4);
    \end{scope}
  \end{tikzpicture}
 \end{center}
\end{lemma}

\begin{proof}
 Note that $f'(n) = 1 \otimes n$ and $g'(m) = m \otimes 1$. Then $(f_! \otimes_L N)g'(m) = f_!(m) \cdot 1 = g(f_!(m))$, where the final equality comes from the definition of the $L$-module structure on $N$.
 Thus, we have the following commutative diagram in $\Mod{L}$.
 \begin{center}
  \begin{tikzpicture}[node distance=2.5cm, auto]
    \node (A) {$L$};
    \node (B) [below of=A] {$M$};
    \node (C) [right of=A] {$N$};
    \node (D) [right of=B] {$M \otimes_L N$};
    \draw[<-] (A) to node [swap] {$f_!$} (B);
    \draw[->] (A) to node {$g$} (C);
    \draw[->] (B) to node [swap] {$g'$} (D);
    \draw[<-] (C) to node {$\mathrlap{f_! \otimes_L N}$} (D); %
  \end{tikzpicture}
 \end{center}
 Also note that identifying $N$ with $L \otimes_L N$, we have $f' = f \otimes_L N$. So $f_! \otimes_L N$ is left adjoint to $f' = f \otimes_L N$, since $(-) \otimes_L N$ preserves the order enrichment.
 
 It remains to show that $f_! \otimes_L N$ preserves the action of $N$ on $M \otimes_L N$ induced by $f'$,
 but this is just the equality $f_!(m) \cdot n'n = n'(f_!(m) \cdot  n)$, which comes from the $L$-algebra structure of $N$.
\end{proof}

\subsection{Overtness and weakly closed sublocales}

\subsubsection{Overtness}

Compactness plays an important role in both classical and constructive topology. But there is an equally important `dual' notion, which is classically invisible.
\begin{definition}
 A frame $L$ said to be \emph{overt}\index{overt|(textbf}\index{locale!overt|(textbf}\index{frame!overt|(textbf} (or \emph{locally positive}\index{locally positive|see {overt}}) if the unique frame homomorphism $!\colon \Omega \to L$ is open. Equivalently, $L$ is overt if $!$ has a left adjoint, which we denote by $\exists\colon L \to \Omega$\glsadd{exists}.
\end{definition}

As suggested by the notation, the map $\exists$ measures the degree to which elements of $L$ are positive --- think ``$\exists x \in a$''.
We call it the \emph{positivity predicate}\index{positivity predicate|textbf} on $L$.
\begin{lemma}\label{prop:positivity_predicate}
 If $L$ is overt, then $\exists(a) = \llbracket a > 0 \rrbracket$.
\end{lemma}
\begin{proof}
 Since $\exists$ is left adjoint to $!\colon \Omega \to L$, we have $\exists(a) = \bigwedge \{p \in \Omega \mid a \le p\cdot 1\}$ and so $\exists(a) = \top$ if and only if $a \le \bigvee\{1 \mid p\} \implies p=\top$.
 This in turn is seen to be equivalent to $a$ being positive as in the proof of \cref{prop:positive_if_epic}.
\end{proof}

\begin{lemma}
 A frame $L$ is overt if and only if it admits a base of positive elements.
\end{lemma}
\begin{proof}
 Suppose $L$ is overt and take $a \in L$. Then $a \le {!}\exists(a) = \exists(a)\cdot 1$ and hence $a = \exists(a)\cdot 1 \wedge a = \exists(a)\cdot a = \bigvee\{a \mid \exists(a)\}$.
 This represents $a$ as a join of positive elements, since every element of $\{a \mid \exists(a)\}$ is positive by \cref{prop:positivity_predicate}.
 
 Now suppose $L$ has a base of positive elements. We will show that $\llbracket (-) > 0\rrbracket$ is left adjoint to ${!}\colon \Omega \to L$.
 
 We first note that $\llbracket (-) > 0\rrbracket\circ{!} \le \id_\Omega$. It is enough to compare the preimages of $\top$ and the condition then follows as in the proof of \cref{prop:positivity_predicate}.
 
 Next we show $\id_L \le {!}\circ\llbracket (-) > 0\rrbracket$. Take $a \in L$ and suppose $a = \bigvee_\alpha a_\alpha$ for some family of positive elements $(a_\alpha)_\alpha$.
 Then $a = \bigvee_\alpha \llbracket a_\alpha > 0\rrbracket \cdot a_\alpha \le \bigvee_\alpha \llbracket a_\alpha > 0\rrbracket \cdot 1 \le \llbracket a > 0\rrbracket \cdot 1$ as required.
\end{proof}

Assuming excluded middle, every element is either positive or zero and hence every locale has a base of positive elements and is therefore overt. This explains why overtness is not discussed in
nonconstructive approaches to topology. Even constructively, many locales of interest are overt.

\begin{lemma}\label{prop:discete_is_overt}
 Every discrete locale is overt.
\end{lemma}
\begin{proof}
 If $X$ is a set, then the set of singletons of $X$ form a base of positive elements for the frame $\powerset{X}$.
\end{proof}

Furthermore, the reals are overt. In fact, this is true of every metric locale \cite{Henry2016}.
However, if every locale is overt then we can deduce excluded middle: recall that excluded middle holds if every sublocale of $\Omega$ is open and
simply note that if $S$ is any overt sublocale of $\Omega$, then the frame map $\Omega \twoheadrightarrow S$ is open by definition.

The relationship between overtness and open maps allows for simple proofs of some basic closure properties.

\begin{lemma}\label{prop:open_subloc_overt}
 Open sublocales of overt frames are overt.
\end{lemma}
\begin{proof}
 Simply use the fact that the composition of two open maps is open.
\end{proof}

\begin{lemma}\label{prop:overt_projection}
 A locale $V$ is overt if and only if the product projection $\pi_1\colon X \times V \to X$ is open for all locales $X$.
\end{lemma}
\begin{proof}
 The backward direction is immediate by taking $X = 1$.
 For the forward direction we take the pullback of ${!}\colon X \to 1$ and ${!}\colon V \to 1$.
 \begin{center}
  \begin{tikzpicture}[node distance=2.5cm, auto]
    \node (A) {$X \times V$};
    \node (B) [below of=A] {$X$};
    \node (C) [right of=A] {$V$};
    \node (D) [right of=B] {$1$};
    \draw[->] (A) to node [swap] {$\pi_1$} (B);
    \draw[->] (A) to node {$\pi_2$} (C);
    \draw[->] (B) to node [swap] {${!}$} (D);
    \draw[->] (C) to node {${!}$} (D);
    \begin{scope}[shift=(A)] %
        \draw +(0.4,-0.8) -- +(0.8,-0.8)  -- +(0.8,-0.4);
    \end{scope}
  \end{tikzpicture}
 \end{center}
 The openness of $\pi_1$ follows from that of ${!}\colon V \to 1$ by \cref{prop:open_pullback_stable}.
\end{proof}

\begin{corollary}\label{prop:product_of_overt_locales}
 Finite products of overt locales are overt.
\end{corollary}
\begin{proof}
 It is enough to show this for binary products. Simply express ${!}\colon X \times Y \to 1$ as the composite of the open maps $X \times Y \xrightarrow{\,\pi_1\,} X$ and $X \xrightarrow{\;{!}\;} 1$.
\end{proof}

For discrete locales we can prove a converse to \cref{prop:open_subloc_overt}. This is the `dual' of the result that a sublocale of a compact Hausdorff locale is closed if and only if it is compact.
\begin{lemma}
 A sublocale of a discrete locale is overt if and only if it is open.
\end{lemma}
\begin{proof}
 The backward direction is immediate from \cref{prop:open_subloc_overt}. We prove the forward direction.
 Let $X$ be a discrete locale and let $i\colon V \hookrightarrow X$ be the inclusion of an overt sublocale. Observe that the following square is a pullback.
  \begin{center}
  \begin{tikzpicture}[node distance=2.7cm, auto]
    \node (A) {$V$};
    \node (B) [below of=A] {$X$};
    \node (C) [right of=A] {$X \times V$};
    \node (D) [right of=B] {$X \times X$};
    \draw[right hook->] (A) to node [swap] {$i$} (B);
    \draw[right hook->] (A) to node {$(i,\id)$} (C);
    \draw[right hook->] (B) to node [swap] {$(\id,\id)$} (D);
    \draw[right hook->] (C) to node {$X \times i$} (D);
    \begin{scope}[shift=(A)] %
        \draw +(0.4,-0.8) -- +(0.8,-0.8)  -- +(0.8,-0.4);
    \end{scope}
  \end{tikzpicture}
 \end{center}
 Now the diagonal inclusion $(\id,\id)\colon X \to X \times X$ is given by the usual inclusion of subsets and is therefore open.
 Thus, $(i,\id)$ is open by \cref{prop:open_pullback_stable} and since $V$ is overt, the projection $\pi_1\colon X \times V \to X$ is open by \cref{prop:overt_projection}.
 Composing these gives that $i = \pi_1 (i, \id)$ is open, as required.
\end{proof}

Overtness is best understood as the property that allows us to existentially quantify an open predicate over a locale.
Such an open predicate would be given by an open in a product locale $X \times V$. The image of this open under the projection $\pi_1 \colon X \times V \to X$
intuitively selects those `points' $x \in X$ for which there is some $v \in V$ such that $(x,v)$ satisfies the predicate. 
\Cref{prop:overt_projection} shows that we can be sure that the resulting predicate on $X$ is open when $V$ is overt\index{overt|)}\index{locale!overt|)}\index{frame!overt|)}.
This idea is explored further in \cref{section:hyperdoctrine}.

\subsubsection{Strong density and weak closedness}

Classically, a closed set can either be defined as the complement of an open set or as a set which contains all of its limit points.
Constructively, these two notions are no longer equivalent. Our definition of closed sublocale is analogous to the former definition,
but an analogue of the latter notion --- called \emph{weak closedness} --- is also of interest. (See \cite{JohnstoneClosedSubgroup} and \cite{JibladzeFibrewiseClosed} for further details on weak closedness.)

Let us first take a closer look the notion of density. Recall that a frame homomorphism $f\colon L \to M$ is dense if $f(x) = 0 \implies x = 0$.
But if $L$ and $M$ are overt, it is perhaps even more natural to ask when $x > 0 \implies f(x) > 0$.\index{strongly dense|(textbf}
(This is analogous to the limit-point definition of density, since every point of $B \supseteq A$ is an adherent point of $A$ if and only if $U \between B \implies U \between A$ for all open $U$.)

\begin{lemma}
 Let $f\colon L \to M$ be a frame homomorphism between overt frames. Then $x > 0 \implies f(x) > 0$ if and only if $f(a) \le {!}(p) \implies a \le {!}(p)$.
\end{lemma}
\begin{proof}
 The condition $x > 0 \implies f(x) > 0$ simply means that $\exists(x) \le \exists f(x)$. Rephrasing this condition gives $\exists f(x) \le p \implies \exists(x) \le p$ for all $p \in \Omega$
 and the result follows immediately by using the adjunction between $\exists$ and ${!}$.
\end{proof}

This second condition gives a reasonable notion even if $L$ and $M$ are not overt.

\begin{definition}
 A frame map $f\colon L \to M$ is called \emph{strongly} (or \emph{fibrewise}) \emph{dense}\index{fibrewise dense|see {strongly dense}} if $f(a) \le {!}(p) \implies a \le {!}(p)$.
\end{definition}

Note that a strongly dense frame homomorphism is in particular dense (by taking $p = 0$).
Furthermore, writing the condition as $f(a) = f(p\cdot a) \implies a = p \cdot a$ shows that every injective frame homomorphism is strongly dense.\index{strongly dense|)}
It is also easy to see that strongly dense frame homomorphisms are closed under composition.

Suppose $f\colon L \to M$ is a frame map. We have $f(a) \le {!}(p) \iff f(a) = f(p\cdot a)$ %
and hence $f$ factors through the quotient $q\colon L \twoheadrightarrow L/C$ where $C$ is the congruence $\langle (a, p \cdot a) \mid f(a) \le {!}(p) \rangle$.
The resulting map $f'\colon L/C \to M'$ is strongly dense.

\begin{definition}
 A frame quotient is \emph{weakly closed}\index{sublocale!weakly closed|(textbf} (or \emph{fibrewise closed}\index{sublocale!fibrewise closed|see {sublocale, weakly closed}}) if its associated congruence is generated by pairs of the form $(a,p \cdot a)$. %
\end{definition}
We have shown that every frame homomorphism $f$ factors into a weakly closed quotient followed by a strongly dense map.
Note that if $f = gh$ where $g$ is strongly dense, then the weakly closed quotient obtained from $f$ coincides with that obtained from $h$. %
It follows easily that the (weakly closed, strongly dense)-factorisation is unique. %
If the map $f$ is surjective, we call $L/C$ the \emph{weak closure} of the sublocale $M$.

By replacing closedness with weak closedness we can prove an analogue of the classical result that every subspace of a discrete space is closed,
which can otherwise fail badly without excluded middle.

\begin{lemma}\label{prop:sublocales_of_discrete_weakly_closed}
 Every sublocale of a discrete locale is weakly closed.\index{sublocale!weakly closed|)}
\end{lemma}
\begin{proof}
 Let $X$ be a discrete locale. Since singletons form a base, any congruence on $\O X$ can be generated by relations of the form $\{x\} \le \bigvee_{\alpha \in A} \{y_\alpha\}$
 and taking the meet of both sides of this inequality with $\{x\}$ we see that we may take $y_\alpha = x$ for each $\alpha \in A$.
 We can rewrite this as $x \le p \cdot x$ where $p = \llbracket \exists \alpha \in A \rrbracket$ and hence the quotient by any such congruence is weakly closed.
\end{proof}

We end with a useful lemma for showing proving that certain locales are overt.

\begin{lemma}\label{prop:weak_closure_overt}
 Suppose $f\colon L \to M$ is a strongly dense frame homomorphism. If $M$ is overt, then so is $L$. If $f$ is additionally assumed to be surjective, the converse also holds.
\end{lemma}
\begin{proof}
We have $a \le {!}(p) \iff f(a) \le {!}(p)$, since $f$ is strongly dense, and $f(a) \le {!}(p) \iff \exists f(a) \le p$, since $M$ is overt.
Thus, $\exists f$ is left adjoint to ${!}\colon L \to \Omega$ and hence $L$ is overt.

For the converse, we have $b = ff_*(b) \le {!}(p) \iff f_*(b) \le {!}(p) \iff \exists f_*(b) \le p$. So $\exists f_*$ is left adjoint to ${!}\colon M \to \Omega$ and $M$ is overt.
\end{proof}

\begin{corollary}
 The weak closure of an overt sublocale is overt.
\end{corollary}

\begin{corollary}
 Spatial locales are overt.
\end{corollary}
\begin{proof}
 For a topological space $X$ with underlying set $|X|$, we have an injective frame homomorphism $\O X \hookrightarrow \Omega^{|X|}$.
 The discrete locale given by $\Omega^{|X|}$ is overt by \cref{prop:discete_is_overt} and injective frame homomorphisms are strongly dense.
 Thus, the result follows by \cref{prop:weak_closure_overt}.
\end{proof}

\subsubsection{Overt weakly closed sublocales}

Overt weakly closed sublocales\index{sublocale!overt weakly closed} have a description in terms of suplattice homomorphisms, which will be very important for us going forward.
Further details can be found in \cite{BungeFunkLowerPowerlocale} and \cite{Vickers1997PowerlocalePoints}.

We will use the following lemma, which is of interest in itself.

\begin{lemma}\label{prop:join_frame_congruence}
 Suppose $L$ is a frame and $(C_\alpha)_{\alpha \in X}$ is a family of frame congruences on $L$. Then the join $\bigvee_{\alpha \in X} C_\alpha$ in the lattice of frame congruences coincides with
 the join $\bigvee_{\alpha \in X} C_\alpha$ in the lattice of suplattice congruences.
\end{lemma}
\begin{proof}
 Recall that quotient suplattices correspond to sub-inf"|lattices under the duality between $\Sup$ and $\Inf$. The join of suplattice congruences then corresponds to the intersection of the corresponding
 sub-inf"|lattices. We must show that if all the input congruences are frame congruences, then the inclusion of the resulting inf"|lattice is right adjoint to a frame homomorphism. %
 Equivalently, we can show that the associated closure operator $\nu'\colon x \mapsto \bigwedge \{ s \ge x \mid \forall \alpha \in X.\ \nu_\alpha(s) \le s\}$ is a nucleus, where $\nu_\alpha$ is
 the nucleus associated to $C_\alpha$.
 
 It is enough to prove that $\nu'(x) \wedge y \le \nu'(x \wedge y)$, since then $\nu'(x \wedge y) \le \nu'(x) \wedge \nu'(y) \le \nu'(x \wedge \nu'(y)) \le \nu'(\nu'(x \wedge y)) = \nu'(x \wedge y)$.
 So given $s \ge x \wedge y$ such that $\nu_\alpha(s) \le s$ for all $\alpha \in X$, we require $\nu'(x) \wedge y \le s$, or equivalently $\nu'(x) \le y \heyting s$.
 It suffices to show that $y \heyting s$ occurs in the meet defining $\nu'(x)$. We have $y \heyting s \ge x$ by the assumption that $s \ge x \wedge y$.
 But now $\nu_\alpha(y \heyting s) \wedge y \le \nu_\alpha((y \heyting s) \wedge y) \le \nu_\alpha(s) \le s$ and thus $\nu_\alpha(y \heyting s) \le y \heyting s$, as required.
\end{proof}
\begin{remark}
 This result can also be derived from the `coverage theorem' of \cite{Abramsky1993quantales}, which shows how certain frame presentations can be described in terms of corresponding suplattice presentations. %
\end{remark}

We can now prove the correspondence between overt weakly closed sublocales\index{sublocale!overt weakly closed|(} and suplattice homomorphisms into $\Omega$. The intuition behind this link is that the suplattice homomorphism indicates when a given open
meets the sublocale in question and factors through the quotient to give a positivity predicate.

\begin{proposition}
 There is a bijection between the overt weakly closed sublocales\index{sublocale!overt weakly closed|)} of a frame $L$ and the suplattice homomorphisms from $L$ to $\Omega$. %
\end{proposition}
\begin{proof}
 Let $q\colon L \twoheadrightarrow L/C$ be a quotient where $L/C$ is overt. Then the unique map ${!}\colon \Omega \to L/C$ has a left adjoint $\exists\colon L/C \to \Omega$
 and so the composite $\exists \circ q$ gives a corresponding suplattice homomorphism.
 
 On the other hand, given a suplattice homomorphism $h\colon L \to \Omega$ we can construct a weakly closed congruence $W_h = \langle (a, h(a)\cdot a) \mid a \in L\rangle$ on $L$.
 In fact, the quotient frame $L/W_h$ is overt. Let us prove this.
 
 The positivity predicate is obtained by factoring $h$ through the quotient $q\colon L \twoheadrightarrow L/W_h$.
 To see that this is possible first note that $W = \bigvee_{a \in L} \langle (a, h(a)\cdot a)\rangle$. By \cref{prop:join_frame_congruence} this join may be taken in the lattice of suplattice congruences
 and so we need only show that $h$ factors through $L/\langle (a, h(a)\cdot a)\rangle$ for each $a \in L$.
 We have $\langle (a, h(a)\cdot a)\rangle = \nabla_a \wedge \Delta_{h(a) \cdot a} = \{(x,y) \in L \times L \mid x \vee a = y \vee a \text{ and } h(a)\cdot a \wedge x = h(a)\cdot a \wedge y \}$.
 So if $(x,y) \in \langle (a, h(a)\cdot a)\rangle$, then $x \vee a = y \vee a$ and so $h(x) \vee h(a) = h(y) \vee h(a)$. But also, $h(a) \cdot a \wedge x = h(a) \cdot a \wedge y$ and hence we have
 $h(x) \wedge h(a) = h(h(a) \cdot a \wedge x) = h(h(a) \cdot a \wedge y) = h(y) \wedge h(a)$.
 Now recall that in a distributive lattice, if $u \vee w = v \vee w$ and $u \wedge w = v \wedge w$ for some $w$, then $u = v$.
 Thus, we have shown that $h(x) = h(y)$ and so $h$ factors through $L/\langle (a, h(a)\cdot a)\rangle$. Consequently, $h$ factors through $q$ to give a map $\overline{h}\colon L/W_h \to \Omega$.
 
 We now show $\overline{h}$ is left adjoint to ${!}\colon \Omega \to L/W_h$. We must show $x \le ({!} \circ \overline{h})(x)$ and $(\overline{h}\circ{!})(p) \le p$.
 For the former, take $a$ such that $x = q(a)$. Then we need $q(a) \le ({!} \circ h)(a)$. But $q(a) = q(h(a) \cdot a) = h(a) \cdot q(a) \le h(a) \cdot 1 = ({!} \circ h)(a)$ and so we are done.
 For the other inequality, we have $(\overline{h}\circ {!})(p) = \overline{h}(p \cdot 1) = p \cdot \overline{h}(1) \le p$.
 
 Finally, we show the above assignments are inverse operations when restricted to weakly closed overt sublocales. Starting with $h \colon L \to \Omega$ we construct $q\colon L \twoheadrightarrow L/W_h$, but then we immediately see that $\exists \circ q = \overline{h} q = h$, as required.
 
 In the other direction, given an overt quotient $q\colon L \twoheadrightarrow L/C$ we form the suplattice homomorphism $\exists \circ q$. Then we obtain the congruence
 $W_{\exists\circ q} = \langle (a, (\exists \circ q)(a) \cdot a) \mid a \in L\rangle$. Note that $q((\exists \circ q)(a) \cdot a) = (\exists \circ q)(a) \cdot q(a) = ({!}\circ\exists\circ q)(a) \wedge q(a) = q(a)$,
 since ${!}\circ\exists \ge \mathrm{id}$. Thus, $W_{\exists\circ q} \subseteq C$.
 
 Now suppose $(a, p\cdot a) \in C$. Then $q(a) = p \cdot q(a)$ and so $(\exists\circ q)(a) \cdot a = (p \cdot (\exists \circ q)(a)) \cdot a \le p \cdot a$.
 But $((\exists\circ q)(a) \cdot a, a) \in W_{\exists\circ q}$ and $(\exists\circ q)(a) \cdot a \le p \cdot a \le a$ so that $(a, p\cdot a) \in C$.
 Therefore, the weak closure of $C$ is contained in $W_{\exists\circ q}$ and hence, $C = W_{\exists\circ q}$ in the case that $C$ is weakly closed.
\end{proof}

\begin{definition}
 If $W$ is an overt weakly closed sublocale, we will write $\exists_W$\glsadd{existsW} for the corresponding suplattice homomorphism.
 If $W$ is the weak closure of an overt sublocale $V$, we will also write $V \between a$\glsadd{between} to mean $\exists_W(a) = \top$.
\end{definition}
This definition of the $\between$ relation agrees with our previous definition in the case that the parent locale is discrete.

\begin{definition}
 Let $X$ be a locale. We denote the suplattice of overt weakly closed sublocales of $X$ by $\SubOW(X)$\glsadd{SubOW}.
 This extends to a functor $\SubOW\colon \Loc \to \Sup$ where the action on morphisms is induced by
 the suplattice hom functor $\hom(-,\Omega)$.
\end{definition}
If $f\colon X \to Y$ is a locale map, then $\SubOW(f)(W) \between a \iff W \between f^*(a)$.
Furthermore, we note that $\SubOW(f)(W)$ coincides with the weak closure of $(\Sloc f)_!(W)$.

\begin{remark}
 Overt weakly closed sublocales do not give rise to an indexed distributive lattice like open and closed sublocales do.
 Thus in \cref{section:quantale_of_overt_weakly_closed_ideals} we will take a somewhat different approach when defining overt weakly closed ideals
 compared to that used in \cref{section:sublocales}.
\end{remark}

We can now return to the matter of Scott-closed sets. The following result is surely known, but I have been unable to find a reference for it.
\begin{proposition}\label{prop:scott_closed_weakly_closed}
 Let $X$ be a continuous dcpo. The Scott-closed subsets\index{Scott closed set@Scott-closed set} of $X$ correspond to overt weakly closed sublocales in the Scott topology.
\end{proposition}
\begin{proof}
 Given a Scott-closed set $S \subseteq X$, we can define a suplattice homomorphism $h_S\colon \O X \to \Omega$ as above: $h_S(U) = \llbracket S \between U \rrbracket$.
 
 Conversely, given a suplattice homomorphism $h\colon \O X \to \Omega$, we define a set $S_h \subseteq X$ as the Scott-closure of $\{ x \in X \mid h(\twoheaduparrow x) = 1\}$.
 (This set is already downwards closed, but needs to be closed under directed suprema.)
 
 We now show these are inverse operations. Take $x \in S$. If $y \ll x$, then $S \between \twoheaduparrow y$ and so $y \in S_{h_S}$.
 But $x = \dirsup \twoheaddownarrow x$ by continuity and hence $x \in S_{h_S}$ by Scott-closedness. Thus, $S \subseteq S_{h_S}$.
 Conversely, if $S \between \twoheaduparrow x$ then $x \in S$ by downward closure and hence $S_{h_S} \subseteq S$, since $S$ is Scott-closed.
 
 Suppose $h(U) = 1$. By \cref{prop:base_for_Scott_topology} we may write $U = \bigcup_{x \in U} \twoheaduparrow x$ %
 and so $h(\twoheaduparrow x) = 1$ for some $x \in U$. Thus, $S_h \between U$ and hence $h_{S_h}(U) = 1$. Consequently, $h \le h_{S_h}$.
 On the other hand, if $S_h \between U$, then $\{ x \in X \mid h(\twoheaduparrow x) = 1\} \between U$ by the Scott-openness of $U$.
 Thus, $h(\twoheaduparrow x) = 1$ for some $x \in U$ so that $h(U) = 1$. Thus, $h_{S_h} \le h$ and the assignments are inverse operations.
\end{proof}
\begin{remark}
 Note that we have not quite shown that (the sublocales induced by) Scott-closed subsets are weakly closed sublocales, but instead that Scott-closed subsets are in bijection with their weak closures.
 It is easy to see that a Scott-closed set is the largest subspace contained in its weak closure, but I do not know if this weak closure is necessarily spatial. Nonetheless, we
 will not need this for what follows.
\end{remark}

\begin{remark}
 Using classical logic the above result is true without the assumption that $X$ is continuous. I do not know whether this more general result is true in the constructive setting.
\end{remark}

\subsection{Dualisable suplattices and string diagrams} \label{section:dualisable}

In \cite{Niefield1982} Niefield gives a number of equivalent descriptions of an important class of suplattices:
projective suplattices, completely distributive suplattices, supercontinuous suplattices (as in \cref{def:supercontinuous})
and suplattices $L$ for which $L \otimes (-)$ has a left adjoint.
The proofs are nonconstructive, but constructive versions can be found in~\cite{galoisTheoryGrothendieck,Fawcett1990completeDistributivity,Niefield2016}. %

Another description of this class is as the \emph{dualisable} suplattices\index{dualisable!suplattice}\index{suplattice!dualisable}. %
\begin{definition}
 An object $A$ in a monoidal category $\Cvar$ is \emph{(left) dualisable\index{dualisable|(textbf}} if when viewing $\Cvar$ as one-object bicategory, the $1$-morphism corresponding to $A$ is a (left) adjoint.
 Its (right) adjoint is denoted by $A^*$\glsadd{dual} and is called the \emph{(right) dual} of $A$.
 
 Explicitly, $A$ has a (right) dual $A^*$ if there exist maps $\eta\colon I \to A^* \otimes A$\glsadd[format=(]{eta} and $\epsilon\colon A \otimes A^* \to I$\glsadd[format=(]{epsilon}
 such that the following diagrams commute.
 \begin{center}
 \begin{minipage}{.45\textwidth}
  \centering
  \begin{tikzpicture}[node distance=3.5cm, auto]
    \node (A) {$A \otimes (A^* \otimes A)$};
    \node (B) [right of=A] {$(A \otimes A^*) \otimes A$};
    \node (C) [below of=A] {$A \otimes I$};
    \node (D) [right of=C] {$I \otimes A$};
    \path (C) -- node[anchor=center] (mid) {$A$} (D);
    \draw[->] (A) to node {$\sim$} (B); %
    \draw[->] (C) to node {$A \otimes \eta$} (A);
    \draw[->] (C) to node {$\sim$} (mid); %
    \draw[<-] (mid) to node {$\sim$} (D); %
    \draw[->] (B) to node {$\epsilon \otimes A$} (D);
  \end{tikzpicture}
 \end{minipage}
 \begin{minipage}{.45\textwidth}
  \centering
  \begin{tikzpicture}[node distance=3.5cm, auto]
    \node (A) {$(A^* \otimes A) \otimes A^*$};
    \node (B) [right of=A] {$A^* \otimes (A \otimes A^*)$};
    \node (C) [below of=A] {$I \otimes A^*$};
    \node (D) [right of=C] {$A^* \otimes I$};
    \path (C) -- node[anchor=center] (mid) {$A^*$} (D);
    \draw[->] (A) to node {$\sim$} (B); %
    \draw[->] (C) to node {$\eta \otimes A^*$} (A);
    \draw[->] (C) to node {$\sim$} (mid); %
    \draw[<-] (mid) to node {$\sim$} (D); %
    \draw[->] (B) to node {$A^* \otimes \epsilon$} (D);
  \end{tikzpicture}
 \end{minipage}\glsadd[format=)]{eta}\glsadd[format=)]{epsilon}
 \end{center}
 Here $I$ is the unit of the monoidal category and the unnamed isomorphisms are the associator (or its inverse) and the left and right unitors as appropriate.
\end{definition}

In a symmetric monoidal category left and right dualisability are equivalent and so we simple call such an object \emph{dualisable}. %
This also gives that $A^{**} \cong A$.

Note that if $A$ is left dualisable, then the functor $A \otimes (-)$ has $A^* \otimes (-)$ as a right adjoint. %
The unit of the adjunction is given by $(\eta \otimes B)_B$ and the counit is given by $(\epsilon \otimes B)_B$ up to composition with the appropriate unitors.

If the category $\Cvar$ is monoidal closed, then $A \otimes (-) \dashv \hom(A, -)$. So by uniqueness of adjoints, $A^* \otimes (-) \cong \hom(A,-)$ whenever $A^*$ exists.
In particular, $A^* \cong A^* \otimes I \cong \hom(A,I)$. Moreover, if we take $A^* = \hom(A,I)$, the counit $\epsilon\colon A \otimes A^* \to I$ is given by the (flipped)
evaluation map $A \otimes \hom(A,I) \to I$.\index{dualisable|)}

When doing calculations involving duals in a monoidal category it can be very helpful to use \emph{string diagrams}\index{string diagram|(textbf}.
We will describe string diagrams very briefly; more information can be found in \cite{baez2010rosetta,selinger2010survey,marsden2014category}.

In a string diagram morphisms are represented by vertices and objects are represented by wires. Our diagrams will be read vertically, from bottom to top.
For instance, if $f\colon A \to B$ and $g\colon B \to C$, we can write the composite $gf$ as the following string diagram.
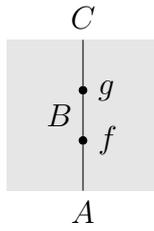
\begin{center}
\vspace{-3pt}
\begin{tikzpicture}[scale=0.5]
\path coordinate[label=below:$A$] (b) ++(0,4) coordinate[label=above:$C$] (t);
\coordinate[dot, label=right:$f$] (alpha) at ($(b)!0.333!(t)$);
\coordinate[dot, label=right:$g$] (alpha') at ($(b)!0.666!(t)$);
\draw (b) -- (alpha) to node[left]{$B$} (alpha') -- (t);
\fillbackground{$(t) + (-2,0)$}{$(b) + (2,0)$};
\end{tikzpicture}
\end{center}

Identity morphisms are suppressed in string diagrams so that $\id_A\colon A \to A$ is simply represented by a vertical line.
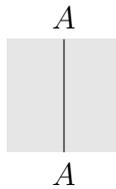
\begin{center}
\vspace{-3pt}
\begin{tikzpicture}[scale=0.75]
\path coordinate[label=below:$A$] (b) ++(0,2) coordinate[label=above:$A$] (t);
\draw (b) -- (t);
\fillbackground{$(t) + (-1,0)$}{$(b) + (1,0)$};
\end{tikzpicture}
\end{center}

The unit of the monoidal category is also suppressed. A map $e\colon I \to A$ will be written:
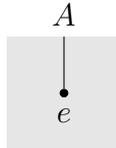
\begin{center}
\vspace{-3pt}
\begin{tikzpicture}[scale=0.75]
\path coordinate[dot, label=below:$e$] (eta) ++(0,1) coordinate[label=above:$A$] (t);
\draw (eta) -- (t);
\fillbackground{$(t) + (-1,0)$}{$(eta) + (1,-1)$};
\end{tikzpicture}
\end{center}

Parallel wires in a string diagram will denote tensor product. So if $f\colon A \to B$ and $f'\colon A' \to B'$, then $f \otimes f'\colon A \otimes A' \to B \otimes B'$ is represented as follows.
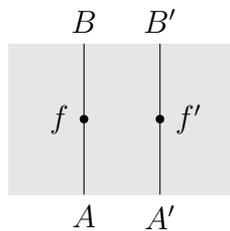
\begin{center}
\vspace{-3pt}
\begin{tikzpicture}[scale=0.5]
\path coordinate[label=below:$A$] (bl) ++(0,2) coordinate[dot, label=left:$f$] (alpha) ++(0,2) coordinate[label=above:$B$] (tl)
 (bl) ++(2,0) coordinate[label=below:$A'$] (br) ++(0,2) coordinate[dot, label=right:$f'$] (beta) ++(0,2) coordinate[label=above:$B'$] (tr);
\draw (bl) -- (alpha) -- (tl)
      (br) -- (beta) -- (tr);
\fillbackground{$(tl) + (-2,0)$}{$(br) + (2,0)$};
\end{tikzpicture}
\end{center}
Note that the associativity and unit laws are implicit in this notation.

String diagrams can also accommodate morphisms with `multiple inputs or outputs'. For example, a map $m\colon A \otimes A \to A$ is shown in the following diagram.
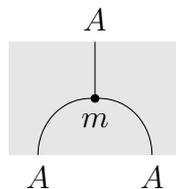
\begin{center}
\vspace{-3pt}
\begin{tikzpicture}[scale=0.75]
\path coordinate[dot, label=below:$m$] (mu)
 +(0,1) coordinate[label=above:$A$] (t)
 +(-1,-1) coordinate[label=below:$A$] (bl)
 +(1,-1) coordinate[label=below:$A$] (br);
\draw (bl) to[out=90, in=180] (mu.west) -- (mu.east) to[out=0, in=90] (br)
 (mu) -- (t);
\fillbackground{$(bl) + (-0.5,0)$}{$(t) + (1.5,0)$};
\end{tikzpicture}
\end{center}

Finally, in a symmetric monoidal category the symmetry map $A \otimes A' \cong A' \otimes A$ is represented as:
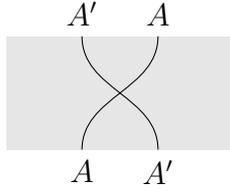
\begin{center}
\vspace{-3pt}
\begin{tikzpicture}[scale=0.5]
\path coordinate[label=below:$A$] (bl) ++(0,3) coordinate[label=above:$A'$] (tl)
 (bl) ++(2,0) coordinate[label=below:$A'$] (br) ++(0,3) coordinate[label=above:$A$] (tr);
\draw (bl) to[out=90, in=270]  (tr)
      (br) to[out=90, in=270] (tl);
\fillbackground{$(tl) + (-2,0)$}{$(br) + (2,0)$};
\end{tikzpicture}
\end{center}

The advantage of string diagrams is that two string diagrams are isomorphic in the obvious sense if and only if the morphisms they represent can be proved equal from the axioms of symmetric monoidal categories. %
A similar result holds for general monoidal categories if we replace diagram isomorphism with planar isotopy, since in that case we cannot allow lines to cross (see \cite{selinger2010survey} for further details).
This allows for intuitive proofs of otherwise complicated identities.

The conditions for $A$ to have a dual $A^*$ can now be expressed in a particularly evocative way.\index{dualisable|(}
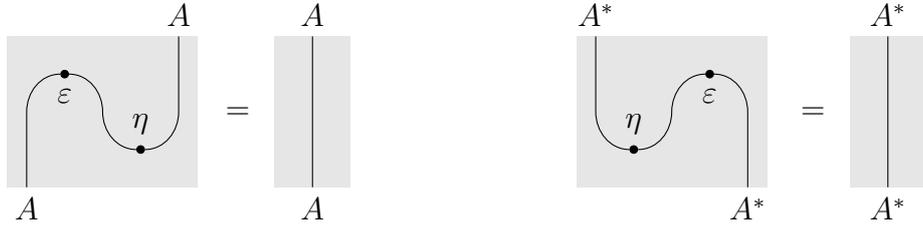
\begin{center} %
\vspace{-3pt}
\begin{minipage}{0.49\textwidth}
\begin{equation*}
\begin{tikzpicture}[scale=0.5,baseline={([yshift=-0.5ex]current bounding box.center)}]
\path coordinate[dot, label=below:$\epsilon$] (epsilon) ++(1,-1) coordinate (a) ++(1,-1) coordinate[dot, label=above:$\eta$] (eta)
 ++(1,1) coordinate (b) ++(0,2) coordinate[label=above:$A$] (tr)
 (epsilon) ++(-1,-1) coordinate (c) ++(0,-2) coordinate[label=below:$A$] (bl);
\draw (bl) -- (c) to[out=90, in=180] (epsilon) to[out=0, in=90] (a) to[out=-90, in=180] (eta) to[out=0, in=-90] (b) -- (tr);
\fillbackground{$(bl) + (-0.5,0)$}{$(tr) + (0.5,0)$};
\end{tikzpicture}
\enspace=\enspace
\begin{tikzpicture}[scale=0.5,baseline={([yshift=-0.5ex]current bounding box.center)}]
\path coordinate[label=below:$A$] (b) ++(0,4) coordinate[label=above:$A$] (t);
\draw (b) -- (t);
\fillbackground{$(t) + (-1,0)$}{$(b) + (1,0)$};
\end{tikzpicture}
\end{equation*}
\end{minipage}
\begin{minipage}{0.49\textwidth}
\begin{equation*}
\begin{tikzpicture}[scale=0.5,baseline={([yshift=-0.5ex]current bounding box.center)}]
\path coordinate[dot, label=above:$\eta$] (eta) ++(1,1) coordinate (a) ++(1,1) coordinate[dot, label=below:$\epsilon$] (epsilon)
 ++(1,-1) coordinate (b) ++(0,-2) coordinate[label=below:$A^*$] (br)
 (eta) ++(-1,1) coordinate (c) ++(0,2) coordinate[label=above:$A^*$] (tl);
\draw (tl) -- (c) to[out=-90, in=180] (eta) to[out=0, in=-90] (a) to[out=90, in=180] (epsilon) to[out=0, in=90] (b) -- (br);
\fillbackground{$(tl) + (-0.5,0)$}{$(br) + (0.5,0)$};
\end{tikzpicture}
\enspace=\enspace
\begin{tikzpicture}[scale=0.5,baseline={([yshift=-0.5ex]current bounding box.center)}]
\path coordinate[label=below:$A^*$] (b) ++(0,4) coordinate[label=above:$A^*$] (t);
\draw (b) -- (t);
\fillbackground{$(t) + (-1,0)$}{$(b) + (1,0)$};
\end{tikzpicture}
\end{equation*}
\end{minipage}
\end{center}

So the maps $\epsilon$ and $\eta$ allow us to `turn corners' in string diagrams and the identities can be thought of as saying we can `pull the wires straight'.\index{dualisable|)}
We will usually suppress the dots and labels for $\epsilon$ and $\eta$ in these situations as they can be readily understood from context.\index{string diagram|)}

Let us now return to the case of dualisable objects\index{dualisable!suplattice|(textbf}\index{suplattice!dualisable|(textbf} in $\Sup$.
\begin{definition}
 Let $L$ be a suplattice. We say that a family $(r_x)_{x \in X}$ of elements of $L$ and a family $(\sigma_x)_{x \in X}$ of elements of $\hom(L, \Omega)$ form a \emph{dual basis\index{dual basis|(textbf}} for $L$
 if $a = \bigvee_{x \in X} \sigma_x(a) \cdot r_x$ for all $a \in L$.
\end{definition}

\begin{lemma}\label{prop:dual_basis}
 A suplattice $L$ is dualisable if and only if it admits a dual basis\index{dual basis|)}.
\end{lemma}
\begin{proof}
 Suppose $L$ is dualisable. We may assume $L^* = \hom(L, \Omega)$ and that the counit $\epsilon\colon L \otimes \hom(L,\Omega) \to \Omega$ is given by evaluation.
 Now express $\eta(\top)$ as a join of basic elements $\eta(\top) = \bigvee_{x \in X} \sigma_x \otimes r_x$. The first `triangle' identity for the adjunction is precisely the
 fact that $(r_x)_{x \in X}$ and $(\sigma_x)_{x \in X}$ give a dual basis.
 
 Now suppose $(r_x)_{x \in X}$ and $(\sigma_x)_{x \in X}$ are a dual basis for $L$. Taking $L^* = \hom(L, \Omega)$ and $\epsilon$ to be evaluation,
 we have the first triangle identity as above. Now note that
 $\alpha(a) = \alpha(\bigvee_{x \in X} \sigma_x(a) \cdot r_x) = \bigvee_{x \in X} \sigma_x(a) \cdot \alpha(r_x)$
 so that $\alpha = \bigvee_{x \in X} \sigma_x \cdot \alpha(r_x)$ for all $\alpha \in \hom(L, \Omega)$, which gives the second identity.
\end{proof}

We will be interested in the following equivalent formulations of dualisability.

\begin{proposition}\label{prop:supercontinuous_vs_dualisable}
 Let $L$ be a suplattice. The following are equivalent:
 \begin{enumerate}
  \item $L$ is dualisable\index{dualisable!suplattice|)}\index{suplattice!dualisable|)}
  \item $L$ is supercontinuous\index{supercontinuous}\index{frame!supercontinuous}\index{suplattice!supercontinuous}
  \item The join map $\joinmap\colon \D L \to L$ has a left adjoint.
 \end{enumerate}
\end{proposition}
\begin{proof}
 $(i) \Rightarrow (ii)$ Suppose $L$ has a dual basis consisting of $(r_x)_{x \in X}$ and $(\sigma_x)_{x \in X}$. Consider $a \in L$ such that $\sigma_x(a) = 1$. We will show that $r_x \lll a$.
 Suppose $a \le \bigvee B$. Then $\bigvee_{b \in B} \sigma_x(b) = 1$ and hence $\sigma_x(b) = 1$ for some $b \in B$. But $b = \bigvee_{y \in X} \sigma_y(b) \cdot r_y$ and
 this join includes $r_x = \sigma_x(b) \cdot r_x$. So $r_x \le b$ and hence $r_x \lll a$ as claimed. Finally, we have $a = \bigvee_x \sigma_x(a) \cdot r_x \le \bigvee_{r \lll a} r$ and so $L$ is supercontinuous.
 
 $(ii) \Rightarrow (iii)$ Now suppose $L$ is supercontinuous. We will show that the map $a \mapsto {\Downarrow} a$ is left adjoint to $\joinmap\colon \D L \to L$.
 Take $a \in L$ and $D \in \D L$.
 If ${\Downarrow} a \subseteq D$, then taking joins gives $a \le \bigvee D$. On the other hand, if $a \le \bigvee D$ and $b \lll a$, then $b \le d$ for some $d \in D$. Hence $b \in D$
 and so ${\Downarrow} a \subseteq D$.
 
 $(iii) \Rightarrow (i)$ Suppose $\joinmap$ has a left adjoint $\joinmap_!\colon L \to \D L$. We construct a dual basis by taking $X = L$ setting $r_x = x$ and
 $\sigma_x(a) = \llbracket x \in \joinmap_!(a)\rrbracket$ for each $x, a \in L$. Note that that maps $\sigma_x$ so defined do indeed preserve joins.
 Then $\bigvee_x \sigma_x(a) \cdot r_x = \bigvee_{x \in \joinmap_!(a)} x = \bigvee \joinmap_!(a) = \joinmap\joinmap_!(a) = a$
 as required, where the final equality holds since $\joinmap$ is surjective.
\end{proof}

The third equivalent condition above is what is called `constructive complete distributivity' in \cite{Fawcett1990completeDistributivity}.
In particular, it easily implies the frame distributivity law.

\begin{corollary}
 A dualisable suplattice $L$ is a frame.
\end{corollary}
\begin{proof}
 Note that $\D L$ is a frame, since it is a complete sublattice of $\Omega^L$. But as a right adjoint, $\joinmap$ preserves arbitrary meets and so $L$ is a frame quotient of $\D L$.
\end{proof}

Finally, we note a link between supercontinuity and overtness.
\begin{lemma}
 A supercontinuous frame is overt.
\end{lemma}
\begin{proof}
 Simply note that $\llbracket x > 0\rrbracket = \llbracket 0 \lll x\rrbracket$. The elements $x$ for which $0 \lll x$ then clearly form a base by supercontinuity.
\end{proof}

\subsection{Categorical logic} \label{section:hyperdoctrine}

A number of our proofs will be more perspicuous when written using internal logic.
In order to make this precise we will need to make use of the theory of hyperdoctrines.
We give a brief introduction to the concepts involved. A more detailed account of this material can be found in \cite{PittsLogic}.

\begin{definition}
 A \emph{signature} is given by a set of \emph{sorts}, a set of \emph{function symbols}, a set of \emph{relation symbols} and an assignment of function and relations symbols to their `types'.
 The type of a function is given by a finite list of input sorts and an output sort. We write $f\colon A_1 \times \dots \times A_n \to B$ to mean
 the type of $f$ has $A_1,\dots,A_n$ as the input sorts and $B$ as the output sort. In the case that $n = 0$ we write $f\colon B$ and call $f$ a \emph{constant}.
 The type of a relation $R$ is given by a finite list of sorts. We write $R \subseteq A_1 \times \dots \times A_n$. In the case that $n = 0$ we write $R \subseteq 1$ %
 and call $R$ an \emph{atomic proposition}.
\end{definition}

\begin{definition}
 Given a signature $\Sigma$, a \emph{context\index{context|textbf}} $\vec{x}$ is a finite list\footnote{In practice we will treat the context more like a set. We should describe rules which allow for the reordering of
 variables in the context wherever they appear, but we omit them for simplicity.} of distinct symbols $x_1,\dots,x_m$, called \emph{variables}, where each variable is equipped with a sort $X_i$.
\end{definition}
\begin{definition}
 A \emph{term} $t$ of type $B$ in the context $\vec{x}$ is described by the grammar\footnote{We express the grammar in a semi-informal version of Backus--Naur form (BNF).}
 \[t \Coloneqq x \mid f(t_1,\dots,t_n)\]
 where $x$ stands for a variable, $f$ is a function symbol of type $A_1 \times \dots \times A_n \to B$ and each $t_i$ stands for a term of type $A_i$ in the context.
 In practice, we simply write the term as $f$ in the case that $n = 0$.
\end{definition}

\begin{definition}
 A \emph{(coherent) formula}\index{logic!coherent|(textbf}\index{coherent!logic|(textbf} $\phi$ in the context $\vec{x}$ is defined by the grammar
 \[\phi \Coloneqq R(t_1,\dots,t_n) \mid \top \mid \bot \mid \phi \wedge \phi \mid \phi \vee \phi \mid \exists y\colon\! A.\; \psi \] %
 where $y$ is a symbol not in $\vec{x}$, $\psi$ is a formula in the context $\vec{x} \cup \{y\}$ (where $y$ has sort $A$),
 $R$ stands for a relation symbol of type $A_1 \times \dots \times A_n$ and each $t_i$ stands for a term of type $A_i$ in the context $\vec{x}$.
 In practice, we simply write the formula as $R$ in the case that $n=0$.
\end{definition}

We use terms and formulae to construct \emph{judgements}\index{judgement|(textbf}. Judgements allow us to make claims about the syntactic notions we have described above.
In our setup there are two kinds judgements: equations and sequents.

\begin{definition}
 An \emph{equation}\index{equation|(textbf} in the context $\vec{x}$ is simply an expression of the form $t =_{\vec{x}} s$ where $t$ and $s$ are terms in $\vec{x}$ of the same sort.
\end{definition}
An equation should, of course, be thought of as a claim that the interpretations of $t$ and $s$ are the same in any model.

\begin{definition}
 A \emph{sequent}\index{sequent|textbf} in the context $\vec{x}$ is an expression of the form $\Gamma \vdash_{\vec{x}} \phi$\glsadd{sequent},
 where $\Gamma$ is a finite set of formulae in the context $\vec{x}$ and $\phi$ is a single formula in the context $\vec{x}$.
 In practice it is common to use a number of notational shorthands. For instance, we might write $\Gamma, \Delta, \psi \vdash_{\vec{x}} \phi$ instead of $\Gamma \cup \Delta \cup \{\psi\} \vdash_{\vec{x}} \phi$
 or $\vdash_{\vec{x}} \phi$ instead of $\emptyset \vdash_{\vec{x}} \phi$.
\end{definition}
A sequent should be thought of as a restricted form of logical implication. (Recall that we do not have an implication connective in our logic.)

\begin{remark}
 Note that equations\index{equation|)} are judgements, not propositions, and so cannot occur in sequents nor be combined with connectives. This is what is meant when this logic is
 described as coherent logic `without equality'\index{logic!without equality}. It is more common to consider logics `with equality' --- where there is an equality \emph{relation} --- but that would
 rule out our main example.
\end{remark}

Coherent logic has inference rules by which we can deduce judgements\index{judgement|)} from others. Such a rule is of the form
\[
\Axiom$J_1 \ J_2\ \fCenter \cdots\ J_n$
\UnaryInf$\fCenter K$
\DisplayProof\;,
\]
where $K$ and each $J_i$ are judgements. This means that if $J_i$ is a theorem for all $1 \le i \le n$, then so is $K$.
In this way an inference rule is like a metalogical sequent.

As a shorthand we sometimes write
\[
\Axiom$J_1 \ J_2\ \fCenter \cdots\ J_n$
\doubleLine
\UnaryInf$\fCenter K$
\DisplayProof\;
\]
to mean that the rule can both be read forwards in the normal way and also backwards to give \AxiomC{$K$} \UnaryInfC{$J_i$}\DisplayProof for each $i$.

We can now give a complete set of inference rules for coherent logic.

\begin{tcolorbox}[title={Inference rules for coherent logic}, %
                  breakable, before skip=\baselineskip] %

\hspace{.05\textwidth}
\begin{minipage}{.45\textwidth}
\AxiomC{$\Gamma \vdash_{\vec{x}} \phi$}
\LeftLabel{(Weak)} %
\UnaryInfC{$\Gamma \vdash_{\vec{x},y} \phi$}
\DisplayProof
\end{minipage}
\begin{minipage}{.45\textwidth}
\AxiomC{\vphantom{$\Gamma$} $s =_{\vec{x}} t$}
\LeftLabel{\phantom{(Trans)}\llap{(Weak)}} %
\UnaryInfC{$s =_{\vec{x},y} t$}
\DisplayProof
\end{minipage}
\bigskip

\hspace{.05\textwidth}
\begin{minipage}{.45\textwidth} %
\AxiomC{$s =_{\vec{x}} s'$}
\AxiomC{$\Gamma \vdash_{\vec{x}, y} \phi$}
\LeftLabel{\phantom{(Weak)}\llap{(Sub)}} %
\BinaryInfC{$\Gamma[s/y] \vdash_{\vec{x}} \phi[s'/y]$}
\DisplayProof
\end{minipage}
\begin{minipage}{.45\textwidth}
\AxiomC{\vphantom{$\Gamma$} $s =_{\vec{x}} s'$}
\AxiomC{$t =_{\vec{x},y} t'$}
\LeftLabel{\phantom{(Trans)}\llap{(Sub)}} %
\BinaryInfC{$t[s/y] =_{\vec{x}} t'[s'/y]$}
\DisplayProof
\end{minipage}
\bigskip

\hspace{.05\textwidth}
\begin{minipage}{.45\textwidth}
\AxiomC{\vphantom{$\Gamma$}}
\doubleLine %
\LeftLabel{\phantom{(Weak)}\llap{($\top$)}} %
\UnaryInfC{$\Gamma \vdash_{\vec{x}} \top$}
\DisplayProof
\end{minipage}
\begin{minipage}{.45\textwidth}
\AxiomC{$\Gamma \vdash_{\vec{x}} \phi$}
\AxiomC{$\Gamma \vdash_{\vec{x}} \psi$}
\doubleLine
\LeftLabel{\phantom{(Trans)}\llap{($\wedge$)}} %
\BinaryInfC{$\Gamma \vdash_{\vec{x}} \phi \wedge \psi$}
\DisplayProof
\end{minipage}
\bigskip

\hspace{.05\textwidth}
\begin{minipage}{.45\textwidth}
\AxiomC{\vphantom{$\Gamma$}}
\doubleLine %
\LeftLabel{\phantom{(Weak)}\llap{($\bot$)}} %
\UnaryInfC{$\Gamma, \bot \vdash_{\vec{x}} \phi$}
\DisplayProof
\end{minipage}
\begin{minipage}{.45\textwidth}
\AxiomC{$\Gamma, \phi \vdash_{\vec{x}} \chi$}
\AxiomC{$\Gamma, \psi \vdash_{\vec{x}} \chi$}
\doubleLine
\LeftLabel{\phantom{(Trans)}\llap{($\vee$)}} %
\BinaryInfC{$\Gamma, \phi \vee \psi \vdash_{\vec{x}} \chi$}
\DisplayProof
\end{minipage}
\bigskip

\hspace{.05\textwidth}
\begin{minipage}{.45\textwidth}
\AxiomC{$\Gamma, \phi \vdash_{\vec{x}, y} \psi$} %
\doubleLine
\LeftLabel{\phantom{(Weak)}\llap{($\exists$)}} %
\UnaryInfC{$\Gamma, \exists y\colon\! A.\; \phi \vdash_{\vec{x}} \psi$}
\DisplayProof
\end{minipage}
\begin{minipage}{.45\textwidth}
\AxiomC{\vphantom{$\Gamma$} $s =_{\vec{x}} t$}
\LeftLabel{\phantom{(Trans)}\llap{(Sym)}} %
\UnaryInfC{$t =_{\vec{x}} s$}
\DisplayProof
\end{minipage}
\bigskip

\hspace{.05\textwidth}
\begin{minipage}{.45\textwidth}
\AxiomC{}
\LeftLabel{\phantom{(Weak)}\llap{(Id)}} %
\UnaryInfC{$\Gamma, \phi \vdash_{\vec{x}} \phi$}
\DisplayProof
\end{minipage}
\begin{minipage}{.45\textwidth}
\AxiomC{}
\LeftLabel{\phantom{(Trans)}\llap{(Refl)}} %
\UnaryInfC{$t =_{\vec{x}} t$}
\DisplayProof
\end{minipage}
\bigskip

\hspace{.05\textwidth}
\begin{minipage}{.45\textwidth}
\AxiomC{$\Gamma \vdash_{\vec{x}} \phi$}
\AxiomC{$\Delta, \phi \vdash_{\vec{x}} \psi$}
\LeftLabel{\phantom{(Weak)}\llap{(Cut)}} %
\BinaryInfC{$\Gamma, \Delta \vdash_{\vec{x}} \psi$}
\DisplayProof
\end{minipage}
\begin{minipage}{.45\textwidth}
\AxiomC{\vphantom{$\Gamma$} $r =_{\vec{x}} s$}
\AxiomC{$s =_{\vec{x}} t$}
\LeftLabel{(Trans)} %
\BinaryInfC{$r =_{\vec{x}} t$}
\DisplayProof
\end{minipage}

\end{tcolorbox}
In the above there are a number of implicit constraints on the types of variables and whether they occur in the context, which are necessary for the rules to be well-formed.
See \cite[pp. B18--B19]{PittsLogic} for details.
The notation $\phi[s/y]$ denotes substitution of the term $s$ for the variable $y$ (of the same type) in the formula $\phi$. Variables that are bound by quantifies should be renamed %
if necessary so as to not shadow free variables in the context lest they be captured during substitution.

We can use the above inference rules to give all the rules that one might expect to hold for such logic.
Let us prove some such rules as examples.
\begin{lemma}\label{prop:right_existential_rule}
 Suppose $y$ is of sort $A$, $s$ is a term of type $A$ in context $\vec{x}$ and $\phi$ is a formula in the context $\vec{x} \cup \{y\}$.
 If we are given $\Gamma \vdash_{\vec{x}} \phi[s/y]$, we can deduce $\Gamma \vdash_{\vec{x}} \exists y\colon A.\ \phi$ by the rules of coherent logic.
\end{lemma}
\begin{proof}
Simply consider the following proof tree built from the inference rules.
\begin{prooftree}
\AxiomC{$\Gamma \vdash_{\vec{x}} \phi[s/y]$}
\AxiomC{}
\UnaryInfC{$s =_{\vec{x}} s$}
\AxiomC{}
\UnaryInfC{$\exists y\colon A.\ \phi \vdash_{\vec{x}} \exists y\colon A.\ \phi$}
\UnaryInfC{$\phi \vdash_{\vec{x},y} \exists y\colon A.\ \phi$}
\BinaryInfC{$\phi[s/y] \vdash_{\vec{x}} \exists y\colon A.\ \phi$}
\BinaryInfC{$\Gamma \vdash_{\vec{x}} \exists y\colon A.\ \phi$}
\end{prooftree}
The result follows.
\end{proof}

\begin{corollary}
 If we have $\Gamma, \phi \vdash_{\vec{x},y} \psi$, then we can deduce $\Gamma, \exists y\colon A.\ \phi \vdash_{\vec{x}} \exists y\colon A.\ \psi$.
\end{corollary} %

\begin{lemma}\label{prop:sub_out_equal_terms}
 Suppose that $y$ is of sort $A$, that $s$, $s'$ and $s''$ are terms of type $A$ in context $\vec{x}$ and that $\phi$ and $\psi$ are formulae in the context $\vec{x} \cup \{y\}$.
 If we are given $s =_{\vec{x}} s'$, $s =_{\vec{x}} s''$ and $\Gamma, \phi[s/y] \vdash_{\vec{x}} \psi[s/y]$, we can deduce $\Gamma, \phi[s'/y] \vdash_{\vec{x}} \psi[s''/y]$.
\end{lemma}
\begin{proof}
 Consider the following proof tree.
\begin{prooftree}
\AxiomC{$s =_{\vec{x}} s'$}
\UnaryInfC{$s' =_{\vec{x}} s$}
\AxiomC{}
\UnaryInfC{$\phi \vdash_{\vec{x},y} \phi$}
\BinaryInfC{$\phi[s'/y] \vdash_{\vec{x}} \phi[s/y]$}
\AxiomC{$\Gamma, \phi[s/y] \vdash_{\vec{x}} \psi[s/y]$}
\AxiomC{$s =_{\vec{x}} s''$}
\AxiomC{}
\UnaryInfC{$\psi \vdash_{\vec{x},y} \psi$}
\BinaryInfC{$\psi[s/y] \vdash_{\vec{x}} \psi[s''/y]$}
\BinaryInfC{$\Gamma, \phi[s/y] \vdash_{\vec{x}} \psi[s''/y]$}
\BinaryInfC{$\Gamma, \phi[s'/y] \vdash_{\vec{x}} \psi[s''/y]$}
\end{prooftree}
The result follows.
\end{proof}

We can now use these rules in order to deduce judgements, just as with the primitive inference rules.

\begin{definition}
 A \emph{coherent theory} $T$ may be defined from a signature $\Sigma$ and a set of judgements (called axioms). %
 A judgement $J$ is a \emph{theorem} of $T$ if it can be derived from the axioms by repeated application of the inference rules of coherent logic.
\end{definition}

We now describe models of coherent logic. A model provides a concrete instantiation of what were hitherto purely syntactic notions.
Propositional coherent logic can be interpreted in a distributive lattice, while sorts and functions can be interpreted in a category.
To interpret coherent logic these must be combined into a \emph{coherent hyperdoctrine}.

\begin{definition}\label{def:indexed_dist_lattice}
 Let $\Cvar$ be a category with finite products. A \emph{$\Cvar$-indexed distributive lattice\index{indexed distributive lattice|textbf}} is a functor $P\colon \Cvar\op \to \DLat$.
\end{definition}

\begin{definition}\label{def:hyperdoctrine}
 Let $\Cvar$ be a category with finite products. A \emph{coherent hyperdoctrine\index{hyperdoctrine|(textbf} on $\Cvar$ (without equality)} is a $\Cvar$-indexed distributive lattice $P$ such that:
 \begin{enumerate}
  \item for every product projection $\pi_1\colon X \times Y \to X$, the lattice homomorphism $\pi_1^* = P(\pi_1)$ has a left adjoint $\exists_Y$ satisfying $\exists_Y(p \wedge \pi^*_1(q)) = \exists_Y(p) \wedge q$,
  \item the $X$-indexed family of maps $(\exists_Y)_X \colon P( X \times Y) \to P(X)$ is natural in $X$.
 \end{enumerate}
\end{definition}

We think of $P(X)$ as giving the set of predicates on $X$ and the lattice structure allows for the interpretation of conjunctions and disjunctions.
This is most obvious for the coherent hyperdoctrine given by powerset on the category of sets.
The action of $P$ on morphisms corresponds to variable substitution and we write $P(f)$ as $f^*$\glsadd{upperStar}.

The two conditions making an indexed distributive lattice into a coherent hyperdoctrine are necessary to handle existential quantification.
The left adjoint of $\pi^*_1$ allows us to define existential quantification. The rest of condition (i) is known as \emph{Frobenius reciprocity}\index{Frobenius reciprocity|textbf}
and is needed to show $(\exists y \in Y.\ p(x,y) \wedge q(x)) \iff (\exists y \in Y.\ p(x,y)) \wedge q(x)$.
Condition (ii) is called the \emph{Beck--Chevalley condition}\index{Beck-Chevalley condition@Beck--Chevalley condition|textbf} and is necessary for existential quantification to respect substitution.

We can now describe how to interpret coherent logic\index{logic!coherent|)}\index{coherent!logic|)} in a hyperdoctrine.

\begin{definition}
 Let $P\colon \Cvar\op \to \DLat$ be a coherent hyperdoctrine\index{hyperdoctrine|)}. A \emph{structure} $M$ in $(\Cvar,P)$ for a signature $\Sigma$ is given by
 \begin{itemize}
  \item an object $M(A)$ in $\Cvar$ for each sort $A$ in $\Sigma$,
  \item a morphism $M(f)\colon M(A_1)\times\dots\times M(A_n) \to M(B)$ in $\Cvar$ for each function symbol $f\colon A_1\times\dots\times A_n \to B$ in $\Sigma$,
  \item an element $M(R) \in P(M(A_1)\times\dots\times M(A_n))$ for each relation symbol $R \subseteq A_1\times\dots\times A_n$ in $\Sigma$.
 \end{itemize}
\end{definition}
Given a structure $M$, we may assign a denotation to a term $t\colon B$ in the context $\vec{x}$ in the form of a morphism $\llbracket t\rrbracket \colon \prod_{i = 1}^m M(X_i) \to M(B)$ %
by recursively applying the following rules:
\begin{itemize}
 \item $\llbracket x_i\rrbracket = \pi_i$ for a variable $x_i$, where $\pi_i$ is the $i^\text{th}$ product projection,
 \item $\llbracket f(t_1,\dots,t_n) \rrbracket = M(f) \circ (\llbracket t_1\rrbracket, \dots, \llbracket t_n\rrbracket)$.
\end{itemize}
We say that $M$ \emph{satisfies} an equation $t =_{\vec{x}} s$ if $\llbracket t\rrbracket = \llbracket s\rrbracket$ as morphisms.

We may assign a denotation $\llbracket \phi \rrbracket \in P(\prod_{i = 1}^m M(X_i))$ to a formula $\phi$ in the context $\vec{x}$ by recursively applying the rules:
\begin{itemize}
 \item $\llbracket R(t_1,\dots,t_n) \rrbracket = (\llbracket t_1\rrbracket, \dots, \llbracket t_n\rrbracket)^*(M(R))$,
 \item $\llbracket \bot \rrbracket = 1$,
 \item $\llbracket \top \rrbracket = 0$,
 \item $\llbracket \phi \wedge \psi \rrbracket = \llbracket \phi \rrbracket \wedge \llbracket \psi \rrbracket$,
 \item $\llbracket \phi \vee \psi \rrbracket = \llbracket \phi \rrbracket \vee \llbracket \psi \rrbracket$,
 \item $\llbracket \exists y\colon A.\ \phi \rrbracket = \exists_A(\llbracket \phi \rrbracket)$, where $\exists_A$ is left adjoint to $P(\pi_{\widehat A})$ and $\pi_{\widehat A}$ is the product projection from 
       $(\prod_{i = 1}^m M(X_i)) \times A$ to $\prod_{i = 1}^m M(X_i)$.
\end{itemize}
We say $M$ \emph{satisfies} a sequent\index{sequent} $\Gamma \vdash_{\vec{x}} \phi$ if $\bigwedge_{\gamma \in \Gamma} \llbracket \gamma \rrbracket \le \llbracket \phi \rrbracket$.

A structure $M$ with signature $\Sigma$ is a \emph{model}\index{model|textbf} of such a theory $T$ if it satisfies every axiom of $T$.

\begin{lemma}[Substitution]
 Suppose $t$ is a term and $\phi$ is a formula in the same context $\vec{x}$ consisting of variables $x_1\colon A_1,\dots,x_n\colon A_n$.
 Furthermore, for each $1 \le i \le n$ suppose $s_i \colon A_i$ is a term in the context $\vec{y}$.
 Then for any structure we have \[\llbracket t[s_1/x_1,\dots,t_n/x_n] \rrbracket = \llbracket t \rrbracket \circ (\llbracket s_1\rrbracket, \dots, \llbracket s_n\rrbracket)\]
 and \[\llbracket \phi[s_1/x_1,\dots,t_n/x_n]\rrbracket = (\llbracket s_1\rrbracket, \dots, \llbracket s_n\rrbracket)^*(\llbracket \phi \rrbracket).\]
\end{lemma}
\begin{proof}
 The proof is by structural induction on $\phi$. We omit the details. %
\end{proof}

\begin{theorem}[Soundness]
 If $M$ is a model of a theory $T$, then $M$ satisfies every theorem of $T$.
\end{theorem}
\begin{proof}
 We must check that the judgements satisfied by $M$ are closed under each of the inference rules. We omit the details.
\end{proof}

In practice we will usually start with a hyperdoctrine and then define an associated theory which we can use to prove results about it.
We define a sort for each object of $\Cvar$ under consideration and a function symbol for each morphism in question. We then define a relation symbol for given elements of
$P(X_1 \times \dots \times X_n)$ for appropriate choices of objects $X_1,\dots,X_n$. Finally, relevant equalities between morphisms and inequalities in between elements
in the distributive lattices are interpreted as logical judgements and become the axioms of the coherent theory. Thus, this theory has a tautologous interpretation in $P$
and results proved in the logic of this theory can then be translated back into results about the original hyperdoctrine.
If $\Cvar$ is small, then the \emph{internal logic}\index{logic!internal|textbf} of the hyperdoctrine $P$ is the theory given by taking \emph{all} objects, morphisms and relations and for which all judgements satisfied by $P$ are axioms,
but we will abuse language and refer to our more restricted theories as internal logics.

In practice, we will often take a more informal approach to using the internal logic than what is demonstrated in \cref{prop:right_existential_rule}.
Just like how we never use completely formal sequent calculus for intuitionistic or classical logic,
we will often leave the context implicit and use natural language instead of proof trees even when dealing with internal logic. It should be straightforward to translate the
informal descriptions into the formal proofs. Occasionally, we will omit a proof in the internal logic when it is identical to a well-known result from standard mathematics.

Our main example of a coherent hyperdoctrine will be given by taking $\Cvar$ to be the category $\OLoc$\glsadd{OLoc} of overt\index{overt}\index{locale!overt}\index{frame!overt} locales (with general locale maps)
and $P$ to be the obvious forgetful functor from its opposite category into the category of distributive lattices\index{hyperdoctrine!of open sublocales|textbf}.
The Frobenius reciprocity condition comes from \cref{prop:overt_projection} and \cref{prop:open_map_module_adjoint},
while the Beck--Chevalley condition follows from \cref{prop:overt_projection} and applying \cref{prop:open_pullback_stable} to $\pi_1\colon X \times Y \to X$ and an arbitrary map $g\colon X' \to X$.

In this example, given an open $a$ we will often write $t \in a$ for the formula $R(t)$ where $R$ is the unary relation defined by $a$.
We will also occasionally write $\{x \colon X \mid \phi(x)\}$ to represent the open in $X$ described by the formula $\phi$ (in the context $X$).

When dealing with exponential objects we will write $f(x)$ to denote $\ev(f,x)$ where $\ev\colon Y^X \times X \to Y$\glsadd{ev} is the evaluation map.
In such a situation, we are permitted to use the following additional inference rule
\[
\AxiomC{$t(y) =_{\vec{x},y} t'(y)$}
\UnaryInfC{$t =_{\vec{x}} t'$}
\DisplayProof,
\]
which follows from the universal property of the exponential.

Finally, in the internal logic we will write the operations of an internal semiring using infix notation in the usual way, with $+$ for addition and juxtaposition for multiplication,
instead of using the formal names of the morphisms involved. Similar notation might also be used for other operations.

\subsection{Results from commutative ring theory}

In \cref{section:quantales_and_AIT,section:tangent_bundle} we will need a few results about commutative rings. %
We do not give full proofs here. See \cite{Matsumura1989rings} for a comprehensive account of commutative ring theory and \cite{VakilRisingSea} for a gentler introduction to algebraic geometry.

\begin{definition}
 Let $R$ be a ring and $S$ a multiplicative submonoid of $R$. The \emph{localisation\index{localisation!of rings|(textbf}} of $R$ with respect to $S$ is the initial ring homomorphism $\ell\colon R \to S^{-1}R$\glsadd{Sneg1R}
 sending every element of $S$ to an invertible element in the codomain. %
 
 Explicitly, for every map $f\colon R \to R'$ such that $f(s)$ is invertible for all $s \in S$, there is a unique map $\overline{f}\colon S^{-1}R \to R'$ making the following diagram commute.
 \begin{center}
  \begin{tikzpicture}[node distance=2.5cm, auto]
    \node (RS) {$S^{-1}R$};
    \node (R) [below of=RS] {$R$};
    \node (T) [right of=R] {$R'$};
    \draw[->] (R) to node [swap] {$f$} (T);
    \draw[->] (R) to node {$\ell$} (RS);
    \draw[dashed,->] (RS) to node {$\overline{f}$} (T);
  \end{tikzpicture}
 \end{center}
\end{definition}

The localisation can be constructed using `fractions' with denominators in $S$. The underlying set of $S^{-1}R$ is given by the quotient of $R \times S$ by the equivalence relation $(r,s) \sim (r',s') \iff \exists t \in S.\ rs't = r'st$.
We write $x/y$ for the equivalence class $[(x,y)]$. The ring operations are given by the usual rules for adding and multiplying fractions and
$\ell$ sends $x \in R$ to $x/1$.\index{localisation!of rings|)}

For the rest of this section we will need to \emph{assume excluded middle and even the occasional choice principle --- namely, dependent choice and the Boolean prime ideal theorem}.
This is so that we can make contact with classical constructions in algebraic geometry. We will be sure to indicate this when we use any of these results.

In classical mathematics with the Boolean prime ideal theorem, the frame $\Rad(R)$ of radical ideals of a ring $R$ is spatial and the points are given by the prime ideals of $R$.
The corresponding sober topological space is the classical description of the (Zariski) spectrum of $R$ and is used extensively in algebraic geometry.

The closure of a point in the spectrum is the simultaneous zero-set for every element of the corresponding prime ideal. These closed sets are called \emph{irreducible closed sets}\index{irreducible closed set|textbf}
as they cannot be expressed as finite unions of closed proper subsets. Closed points in the spectrum correspond to maximal ideals in $R$. 

The spectrum of $R$ is equipped with a sheaf of rings. We will not cover the details, since we use this purely for intuition.
From this we obtain a ring for each point of the spectrum: the \emph{stalk} at that point.
The elements of the stalk should be thought of as functions that are defined in a small neighbourhood of the point.

The stalks can be constructed by localisation\index{localisation!of rings|(textbf}.
Suppose $P$ is a prime ideal. Its complementary anti-ideal $P\comp$ is a submonoid of $R$.
The stalk at $P$ is given by localisation of $R$ with respect to this anti-ideal. We call this the \emph{localisation of $R$ at $P$} and denote the ring by $R_P$\glsadd{RP}.

\begin{definition}
 A \emph{local ring}\index{local ring|(textbf}\index{ring!local|see {local ring}} $(R, \m)$ is a ring $R$ equipped with a unique maximal ideal $\m$\glsadd{m}.
\end{definition}
\begin{remark}
 The elements lying outside of $\m$ in a local ring are precisely the invertible elements.
 Without excluded middle, local rings should instead be defined as rings for which the invertible elements form an anti-ideal\index{local ring|)},
 but we will not need this.
\end{remark}

The localisation of a ring at a prime $P$ is always a local ring where the maximal ideal is given by the image of $P$.
The localisation\index{localisation!of rings|)} map induces a bijection between the prime ideals of $R_P$ and the prime ideals of $R$ which are contained in $P$.

Recall that an ideal $I$ is prime if and only if $R/I$ is an integral domain and maximal if and only if $R/I$ is a field.
The quotient of $R_P$ by its maximal ideal is isomorphic to the field of fractions of the domain $R/P$ and the quotient map
is thought of as taking a `function' $f \in R_P$ defined around $P$ to its value at $P$. %

A function $f \in \m$ is zero at $P$. Conceptually, its Taylor series at $P$ has zero constant term. The linear term --- its first derivative at $P$ ---
is then given by quotienting out by $\m^2$.

\begin{definition}
 The \emph{Zariski cotangent space\index{cotangent space|see {tangent space}}} of a local ring $(R, \m)$ is the module $\m/\m^2$ considered as a vector space over the field $R/\m$.
 For a general ring $R$, the \emph{cotangent space at a prime $P$} %
 refers to Zariski cotangent space of $R_P$.
 
 We can then formally define the \emph{tangent space\index{tangent space|textbf}} to be the linear dual of the cotangent space: $\hom(\m/\m^2, R/\m)$.
 The tangent space of a ring $R$ at a prime $P$ is denoted by $T_P R$\glsadd{TPR}.
 This is the algebraic analogue of the tangent space in differential geometry.
\end{definition}

\begin{definition}
 A ring $R$ is said to be \emph{Noetherian}\index{Noetherian ring|textbf}\index{ring!Noetherian|textbf} if all of its ideals are finitely generated.
\end{definition}
\begin{remark}
 Noetherian rings are only really well-behaved under the assumption of the axiom of dependent choice (in addition to excluded middle),
 since many of the usual definitions become inequivalent without this assumption.
\end{remark}

\begin{proposition}
 Let $(R,\m)$ be a Noetherian local ring. Then $x_1, \dots, x_n$ is a minimal generating set for $\m$ if and only if the images of $x_1,\dots,x_n$ in $\m/\m^2$ form a basis.
\end{proposition}
\begin{proof}
 This follows immediately from \cite[Theorem 2.3]{Matsumura1989rings}. %
\end{proof}

\begin{definition}
 Let $R$ be a ring and $K$ the set of finite chains of prime ideals in $R$.
 Then the \emph{Krull dimension\index{Krull dimension|textbf}\index{dimension of a ring|textbf}} of $R$ is given by $\dim R = \bigvee_{C \in K} (|C|-1) \in \N \cup \{-1,\infty\}$\glsadd{dim}.
\end{definition}

Under this definition a polynomial ring $k[x_1,\dots,x_n]$ over a field $k$ has dimension $n$.
The intuition behind this definition is that a point will be contained in a 1-dimensional curve, which will be contained in a 2-dimensional surface and so on,
giving a chain of length $n+1$ in the spectrum of an $n$-dimensional ring.

\begin{proposition}[Krull's height theorem]
 If $(R, \m)$ is a Noetherian local ring and $\m = \sqrt{(x_1,\dots,x_n)}$, then $\dim R \le n$.
\end{proposition}
\begin{proof}
 See \cite[Theorem 13.5]{Matsumura1989rings}. %
\end{proof}

\begin{corollary}
 The dimension of a Noetherian local ring is finite and bounded by the dimension of its cotangent space.
\end{corollary}

\begin{definition}
 A Noetherian local ring $(R, \m)$ is said to be \emph{regular}\index{local ring!regular|textbf} if $\dim R = \dim \m/\m^2$.
 We say a ring is \emph{nonsingular}\index{nonsingular!ring|textbf}\index{ring!nonsingular|textbf} if all of its stalks are regular local rings.
\end{definition}
A Noetherian local ring is regular if its tangent space is `no larger than expected'.
Regular local rings correspond to the stalks of schemes at nonsingular points.
For instance, below is a graphical representation of the ring $\R[x,y]/(y^2-x^2(x+1))$.
This ring is 1-dimensional, but the cotangent space at the origin $(x,y)$ has dimension 2.

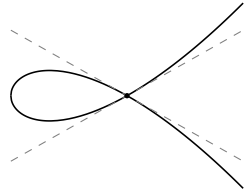
\begin{figure}[H]
\centering
\begin{tikzpicture}
\begin{axis}[hide axis,domain=-1:1,width=0.35\linewidth]
  \addplot[semithick,samples at={-1.0,-0.995,...,-0.65,-0.65,-0.6,...,1.0}] {+x*sqrt(x+1)};
  \addplot[semithick,samples at={-1.0,-0.995,...,-0.65,-0.65,-0.6,...,1.0}] {-x*sqrt(x+1)};
  \addplot[dashed,gray,samples=2] {x};
  \addplot[dashed,gray,samples=2] {-x};
  \fill[black] (0,0) circle(1pt);
\end{axis}
\end{tikzpicture}
\caption*{A nodal cubic curve with two tangent lines at the singularity.}
\end{figure}

If $R$ is a regular local ring of dimension $n$, then $\m$ has a generating set $x_1,\dots,x_n$.
Such a generating set is called a \emph{regular system of parameters}.

\begin{proposition}
 If $x_1,\dots,x_n$ is a regular system of parameters\index{regular system of parameters|textbf} for a regular local ring $R$, then for any $1 \le k \le n$, the quotient $R/(x_1,\dots,x_k)$ is also a regular local ring
 with regular system of parameters $[x_{k+1}], \dots, [x_n]$.
\end{proposition}
\begin{proof}
 See \cite[Theorem 14.2]{Matsumura1989rings}.
\end{proof}

\begin{proposition}
 Regular local rings are integral domains.
\end{proposition}
\begin{proof}
 See \cite[Theorem 14.3]{Matsumura1989rings}.
\end{proof}
\begin{corollary}\label{prop:regular_sequence_generates_prime_ideal}
 If $x_1,\dots,x_n$ is a regular system of parameters for a regular local ring $R$, then $(x_1,\dots,x_k)$ is prime for any $1 \le k \le n$.
\end{corollary}

\section{Quantales and abstract ideal theory} \label{section:quantales_and_AIT}

While ideals have been used to study rings for well over a century, the study of rings purely through their abstract \emph{residuated lattices}\footnote{A residuated lattice is a lattice %
equipped with an additional multiplicative monoid structure and an implication operation which gives a right adjoint to multiplication. As objects, complete residuated lattices are
the same thing as quantales.} of ideals was first initiated by Dilworth and Ward in \cite{DilworthWard1939}. However, progress was limited until Dilworth's publication of
\textit{Abstract Commutative Ideal Theory} \cite{Dilworth1962} in the early 1960s.

In this paper Dilworth defines a number of notions of `principal element'\index{principal element|textbf} of a residuated lattice, which play a similar role to principal ideals and allow for proofs of analogues of the
Lasker--Noether theorem and Krull's principal ideal theorem. %

Unfortunately, most of the literature on Abstract Ideal Theory is rather uncategorical, being focused only on objects with little regard for morphisms.
Indeed, with the most obvious choice of morphisms for residuated lattices, the assignment of a ring to its residuated lattice of ideals is not even functorial.
However, it \emph{is} functorial when the codomain is taken to be the category of (two-sided) quantales.

\begin{definition}
 We define a functor $\Idl\colon \Rng \to \Quant$\glsadd{Idl} which sends a ring to its quantale of ideals and a ring homomorphism $f\colon R \to S$
 to a quantale homomorphism defined by $\Idl(f)(I) = \langle f(I) \rangle$.
\end{definition}

In this chapter we aim to formulate some concepts in Ideal Theory from a more categorical perspective.
We do not attempt here to prove substantial ring-theoretic results in this framework.
Rather, we simply describe some concepts in a way which we hope might better lend themselves to future work
on Abstract Ideal Theory where morphisms play a larger role.

\subsection{Principal elements and strict morphisms}

In 1939 Dilworth and Ward \cite{DilworthWard1939} introduced a weak notion of principal element\index{principal element|(textbf} in a residuated lattice, but
it was another two decades before Dilworth \cite{Dilworth1962} formulated the stronger notion of principal element which was so influential in Abstract Ideal Theory.
As this delay suggests, the definition of principal element is not particularly obvious. The situation is further complicated by there being a number of weaker notions which are also of interest.
A thorough discussion of these different notions can be found in \cite{AndersonJohnson1996} whose terminology we adopt here.

\begin{definition}
 An element $k$ of a two-sided quantale $Q$ is said to be
 \begin{enumerate}
  \item \emph{weak meet principal} if $k(k\heyting x) = x \wedge k$ for all $x \in Q$,
  \item \emph{weak join principal} if $k \heyting kx = x \vee k^\bullet$ for all $x \in Q$ (recall that $k^\bullet = k \heyting 0$),
  \item \emph{meet principal} if $k((k\heyting x) \wedge b) = x \wedge kb$ for all $x,b \in Q$,
  \item \emph{join principal} if $k \heyting (kx \vee a) = x \vee (k\heyting a)$  for all $x,a \in Q$,
  \item \emph{weak principal} if it is weak meet principal and weak join principal,
  \item \emph{principal} if it is meet principal and join principal.
 \end{enumerate}
\end{definition}

It is easy to see that these have the following equivalent characterisations (see \cite{AndersonJohnson1996}). %
\begin{lemma}
 Consider $k \in Q$. We have that
 \begin{enumerate}
  \item $k$ is weak meet principal if and only if $x \le k \implies x = kd$ for some $d \in Q$,
  \item $k$ is weak join principal if and only if $kx \le ky \implies x \le y \vee k^\bullet$,
  \item $k$ is meet principal if and only if $x \le kb \implies x = kd$ for some $d \le b$,
  \item $k$ is join principal if and only if it is weak join principal in every closed quotient of $Q$.
 \end{enumerate}
\end{lemma}

\begin{remark}
 If $R$ is a ring and $I \in \Idl(R)$ contains a non-zero-divisor (i.e.\ an element $r \in R$ such that $rx = 0 \implies x = 0$), then $I^\bullet = 0$ and so $I$ is weak join principal if and only
 if $IJ \le IK \implies J \le K$. This is precisely what is known as a \emph{cancellation ideal} in ring theory.
\end{remark}

So under mild conditions, weak join principal recovers a classical notion.
Furthermore, when trying to prove primary decomposition, it becomes apparent before too long that something like weak meet principal will be helpful.
Indeed, this is what led Dilworth and Ward to their original definition of this notion. The picture simplifies further once one notices that meet principal and join principal are not unrelated notions,
but are dual to each other in a certain sense (see \cite{Anderson1977}).

As one might hope, principal ideals are principal elements\index{principal element|)}.
\begin{lemma}
 If $R$ is a ring and $f \in R$, then $(f)$ is a principal element of $\Idl(R)$.
\end{lemma}
\begin{proof}
 Suppose $I \le f J$ for ideals $I$ and $J$. Then for each $i \in I$, there is a $j \in J$ such that $i = fj$. Set $D_i$ to be the ideal generated by all such $j$ %
 and $D = \bigvee_{i \in I} D_i$. Then $I = f D$ and so $(f)$ is meet principal.
 
 Now suppose $f I \le f J$. Then for each $i \in I$, there is a $j \in J$ such that $fi = fj$. Then $f(i-j) = 0$ and so $i-j \in (f)^\bullet$. Thus, $i = j + (i-j) \in J \vee (f)^\bullet$.
 So $I \le J \vee (f)^\bullet$ and $(f)$ is weak join principal. Then since the same argument holds in every quotient ring and $\Idl(R)/\nabla_I \cong \Idl(R / I)$, we can conclude that $(f)$ is join principal.
\end{proof}

We have tried to convey a sense of the motivation for these definitions, but it is perhaps still not completely clear why these are natural notions.
Our contribution is to show how all these notions of principal element fall out of a simple discussion of normal mono/epimorphisms in $\Sup$.

The category $\Sup$ is pointed with the zero morphisms given by the constant $0$ maps.
In a pointed category there is a notion of kernels and cokernels.
\begin{definition}
 The \emph{kernel}\index{kernel and cokernel|textbf} of a morphism $f\colon L \to M$ in a pointed category is the equaliser of $f$ and $0_{L,M}$.
 The \emph{cokernel} of $f$ is the coequaliser of $f$ and $0_{L,M}$.
 
 A morphism is called a \emph{normal monomorphism}\index{normal monomorphism|(textbf} if it occurs as the kernel of some morphism
 and a \emph{normal epimorphism} if it occurs as the cokernel of some morphism. %
 We will draw such morphisms with triangle tails ($\triangleTailArrow$) and tips ($\triangleHeadArrow$) respectively. %
\end{definition}

Every monomorphism or epimorphism in $\Sup$ is regular, but not every monomorphism or epimorphism is normal.

\begin{lemma}\label{prop:normal_monos_and_epis_in_sup}
 A monomorphism $f\colon L \to M$ in $\Sup$ is normal\index{normal monomorphism|)} if and only if the image of $f$ is downward closed.
 Dually, an epimorphism $g\colon L \to M$ in $\Sup$ is normal if and only if \[g(x) = g(y) \iff x \vee g_*(0) = y \vee g_*(0).\]
\end{lemma}
\begin{proof}
 Let $h\colon M \to N$ be a suplattice homomorphism. The kernel of $h$ is given by $\{x \in M \mid h(x) = 0\} = {\downarrow} h_*(0)$, which is clearly downward closed.
 On the other hand, it is easy to see that for any $a \in M$ the inclusion ${\downarrow} a \hookrightarrow M$ is the kernel of the map into ${\uparrow} a$ sending $x$ to $x \vee a$.
 
 Now in order to apply duality, we note that the image of $f$ is downward closed if and only if $ff_*(x) = x \wedge f(\top)$. %
 The dual condition then gives $g_*g(x) = x \vee g_*(0)$. If $g$ is epic, this is in turn equivalent to the stated condition for $g$ being normal.
\end{proof}
\begin{remark}
 Throughout this section we make extensive use of the autoduality $\Sup \cong \Sup\op$.
 This exchanges $0$ with $\top$, joins with meets, (normal) epis with (normal) monos and maps with their adjoints. %
\end{remark}

Recall that $\Sup$ has (epi,mono)-factorisations. So for any morphism we can ask if its monic/epic factor is normal.
\begin{definition}
 Let $f$ be a morphism in $\Sup$ and suppose $f = me$ where $m$ is monomorphism and $e$ is epimorphism. If $m$ is a normal mono, we say $f$
 is \emph{mono-normal}\index{mono-normal morphism|(textbf}. If $e$ is a normal epi, we call $f$ \emph{epi-normal}\index{epi-normal morphism|see {mono-normal morphism}}.
 If $f$ is both mono-normal and epi-normal, we say it is \emph{strict}\index{strict morphism|(textbf}.
\end{definition}

\begin{corollary}
 A suplattice map $f\colon L \to M$ is epi-normal if and only if its kernel congruence is generated by an element of the form $(0,a)$ and mono-normal\index{mono-normal morphism|)} %
 if and only if its image is downward closed.
\end{corollary}

A reason why we might be interested in strict morphisms of suplattices in the context of abstract ideal theory is that the functor $\Sub\colon \Mod{R} \to \Sup$ %
sending $R$-modules to their suplattices of submodules sends module homomorphisms to strict suplattice homomorphisms\index{strict morphism|)}.
\begin{lemma}
 If $f\colon V \hookrightarrow W$ is a monomorphism of $R$-modules, then $\Sub f$ is a normal monomorphism.
 If $g\colon V \twoheadrightarrow W$ is an epimorphism of $R$-modules, then $\Sub f$ is a normal epimorphism.
\end{lemma}
\begin{proof}
 The first statement is clear if we think of $\Sub(W)$ as the subobjects of $W$. %
 The second statement is clear if we think of $\Sub(V)$ as the quotients of $V$.
\end{proof} %
\begin{remark}
 The fact that principal ideals are join principal, the modularity of $\Sup(V)$ and the above lemma all make essential use of the fact that the additive structure of a module permits inverses.
 (In the final case it is used to prove the bijection between subobjects and quotients.)
\end{remark}

\begin{lemma}\label{prop:normal_closed_under_composition}
 Normal monomorphisms and normal epimorphisms are closed under composition.
\end{lemma}
\begin{proof}
 We show it for monomorphisms; it will then follow for epimorphisms by duality. Suppose $f\colon L \to M$ and $g\colon M \to N$ are normal monos.
 Suppose $x \le gf(\top)$. Then $x \le g(\top)$ and so $x = g(x')$ for some $x' \in M$. So $g(x') \le gf(\top)$. But $g$ is an injective suplattice homomorphism and hence reflects order.
 Thus, $x' \le f(\top)$ and so $x' = f(x'')$ for some $x'' \in L$. Therefore, $x = gf(x'')$ and $gf$ is a normal mono.
\end{proof}

\begin{remark}\label{rem:converse_to_composition_of_normal} %
 Also note that in any pointed category, if $f$ is monic and $fg$ is a normal monomorphism, then $g$ is a normal monomorphism. Dually, if $h$ is epic and $gh$ is a normal epimorphism, then so is $g$.
\end{remark} %

Mono-normal, epi-normal and strict morphisms are \emph{not} closed under composition in general.
If $f$ is epi-normal and $e$ is a normal epi, then $fe$ is epi-normal by \cref{prop:normal_closed_under_composition}, but $ef$ might not be. We
might ask under what conditions on $f$ we have that $ef$ is epi-normal for all normal epis $e$.

\begin{lemma}\label{prop:composing_epinormal_maps}
 Consider $f\colon L \to M$ and let $e\colon M \triangleHeadArrow N$ be a normal epimorphism. Set $a = e_*(0)$.
 Then $ef$ is epi-normal if and only if $x \vee f_*(a) \ge f_*(f(x)\vee a)$ for all $x \in L$.
\end{lemma}
\begin{proof}
 Note that $(ef)_*(0) = f_*(e_*(0)) = f_*(a)$. Now $ef(x) = ef(y) \iff f(x) \vee a = f(y) \vee a$ by normality.
 So $ef$ is epi-normal if and only if $f(x) \vee a = f(y) \vee a \implies x \vee f_*(a) = y \vee f_*(a)$ for all $x,y \in L$.
 
 This is equivalent to $f(y) \le f(x) \vee a \implies y \le x \vee f_*(a)$. %
 But $f(y) \le f(x) \vee a \iff y \le f_*(f(x) \vee a)$ and so the result follows.
\end{proof}

\begin{corollary}\label{prop:composing_mononormal_maps}
 Consider $f\colon M \to L$ and let $m\colon N \triangleTailArrow M$ be a normal monomorphism. Set $b = m(\top)$.
 Then $fm$ is mono-normal if and only if $x \wedge f(b) \le f(f_*(x)\wedge b)$ for all $x \in L$.
\end{corollary}
\begin{proof}
 By duality. %
\end{proof}

\begin{definition}
 We call a morphism $f$ satisfying the conditions of \cref{prop:composing_epinormal_maps} for all normal epimorphisms $e$, a \emph{composably epi-normal morphism}. %
 A morphism satisfying the dual condition is called a \emph{composably mono-normal morphism}\index{mono-normal morphism!composably|textbf}.
 
 Note that each of these is closed under composition.
\end{definition}

Clearly every normal epimorphism is composably epi-normal. To determine if a more general morphism is composably epi-normal, we may consider its monic and epic factors separately.
\begin{lemma}\label{prop:composably_epinormal_depends_on_mono}
 Suppose $f\colon L \to M$ factors as $mn$ where $m$ is monic and $n$ is epic. Then $f$ is composably epi-normal if and only if $n$ is a normal epimorphism and $m$ is composably epi-normal.
\end{lemma}
\begin{proof}
 For $f$ to be composably epi-normal, it must in particular be epi-normal and hence $n$ must certainly be a normal epimorphism.
 Let us assume this is the case.
 
 Now let $e\colon M \triangleHeadArrow N$ be a normal epimorphism. We (epi,mono)-factorise $em$ and $ef$ to obtain the following commutative diagram.
 \begin{center}
 \begin{tikzpicture}[node distance=2.5cm, auto]
  \node (L) {$L$};
  \node (mTarget) [right of=L] {.};
  \node (M) [right of=mTarget] {$M$};
  \node (N) [below of=M] {$N$};
  \path (L) -- node[anchor=center] (efTarget) {.} (N);
  \draw[>=normalTip,->] (L) to node {$n$} (mTarget);
  \draw[right hook->] (mTarget) to node {$m$} (M);
  \draw[>=normalTip,->] (M) to node {$e$} (N);
  \draw[right hook->] (efTarget) to [swap] node {$m'$} (N);
  \draw[->>] (L) to [swap] node {$n'$} (efTarget);
  \draw[->>] (mTarget) to node {$e'$} (efTarget);
 \end{tikzpicture}
 \end{center}
 Then by \cref{prop:normal_closed_under_composition,rem:converse_to_composition_of_normal}, $n'$ is normal if and only if $e'$ is and hence $ef$ is epi-normal if and only if $em$ is.
\end{proof}

At this stage modularity becomes relevant. Many results in abstract ideal theory require the underlying lattice of the quantale in question to be modular.
We provide a possible explanation for this by exhibiting a strong link between modularity properties and the composability properties of epi/mono-normal morphisms.

Recall that a pair $(x,b)$ of elements in a lattice is a \emph{modular pair}\index{modular|(} if $(x \wedge b) \vee a = (x \vee a) \wedge b$ for all $a \le b$ %
and that $b$ is a \emph{modular element} if $(x,b)$ is a modular pair for all $x$.

\begin{lemma}\label{prop:normal_composition_and_modular_pairs}
 Let $m\colon L \triangleTailArrow M$ be a normal monomorphism and let $e\colon M \triangleHeadArrow N$ be a normal epimorphism.
 Then $em$ is epi-normal if and only if $(e_*(0), m(\top))$ is a modular pair and $em$ is mono-normal if and only if $(m(\top), e_*(0))$
 is a dual modular pair.
\end{lemma}
\begin{proof}
 Let $a = e_*(0)$ and $b = m(\top)$. We can view $L$ as the subset ${\downarrow} b \subseteq M$ and take $m$ to be the inclusion map.
 Then $m_*(z) = z \wedge b$. Thus by \cref{prop:composing_epinormal_maps} we find that $em$ is normal if and only if
 $x \vee (a \wedge b) = (x\vee a) \wedge b$ for all $x \le b$. This is precisely the definition of $(a,b)$ being a modular pair. %
 The dual result follows by duality.
\end{proof}

The essence of the following characterisation of complete modular lattices\index{lattice!modular|(} is essentially known, but it is perhaps notable that the characterisation can be stated in purely categorical terms
(and amongst algebraic structures that do not have meets in their signature).

\begin{corollary}\label{cor:strict_maps_compose_for_modular_suplattices}
 Let $M$ be a suplattice. Then $em$ is strict for every normal mono $m\colon L \triangleTailArrow M$ and normal epi $e\colon M \triangleHeadArrow N$
 if and only if it is epi-normal for every such $m$ and $e$ if and only if it is mono-normal for every such $m$ and $e$ if and only if $M$ is modular.\index{lattice!modular|)}
\end{corollary}
Consequently, $M$ is modular if and only if the composition of every strict map into $M$ and every strict map out of $M$ is strict.
Thus, the collections of modular suplattices and strict homomorphisms\index{strict morphism|(} form a category.
\begin{remark}
 The submodule functor $\Sub\colon \Mod{R} \to \Sup$ factors through this subcategory.
 A precise characterisation of the smallest subcategory containing the image of this functor is a very difficult problem,
 even on the level of objects. %
\end{remark}

\begin{corollary}\label{prop:composing_strict_maps_and_modular_elements}
 A strict morphism $k$ is composably epi-normal iff $k(\top)$ is modular and composably mono-normal iff $k_*(0)$ is dual modular\index{modular|)}.
\end{corollary}
\begin{proof}
 Simply use \cref{prop:composably_epinormal_depends_on_mono}, \cref{prop:normal_composition_and_modular_pairs} and the fact that surjections send $\top$ to $\top$.
\end{proof}

It is now apparent that we can rephrase the different notions of principal element\index{principal element|(} in terms of the classes of morphisms discussed above.
\begin{proposition}\label{prop:characterisation_of_principal_elts}
 Let $Q$ be a two-sided quantale and take $k \in Q$. We may view $k$ as suplattice endomorphism on $Q$ (multiplication by $k$). Then
 \begin{enumerate}
  \item $k$ is weak meet principal iff it is mono-normal\index{mono-normal morphism},
  \item $k$ is weak join principal iff it is epi-normal, %
  \item $k$ is meet principal iff it is composably mono-normal\index{mono-normal morphism!composably},
  \item $k$ is join principal iff it is composably epi-normal, %
  \item $k$ is weak principal iff it is strict\index{strict morphism|)},
  \item $k$ is principal\index{principal element|)} iff $k$ is a strict map, $k$ is a modular element and $k^\bullet$ is dual modular.
 \end{enumerate}
\end{proposition}
\begin{proof}
 Claims (i) and (ii) come from \cref{prop:normal_monos_and_epis_in_sup}.
 Claims (iii) and (iv) come from \cref{prop:composing_epinormal_maps,prop:composing_mononormal_maps}.
 Then (v) is immediate and (vi) follows from \cref{prop:composing_strict_maps_and_modular_elements}.
\end{proof}

From this we can immediately conclude the following known results.
\begin{corollary}
 Meet principal, join principal and principal elements are closed under multiplication.
\end{corollary}
\begin{corollary}
 The product of a weak meet principal element and a meet principal element is weak meet principal. The product of a weak join principal element and a join principal element is weak join principal.
\end{corollary}
\begin{corollary}
 Weak principal elements in a modular quantale are principal.
\end{corollary}

\begin{corollary} %
 Weak meet principal elements (and hence principal elements) in a modular quantale are preserved by closed quotients.
\end{corollary}
\begin{proof} %
 Let $k$ be a weak meet principal element of $Q$ and suppose $e\colon Q \twoheadrightarrow Q/\nabla_a$ is a closed quotient. Consider the following commutative diagram.
 
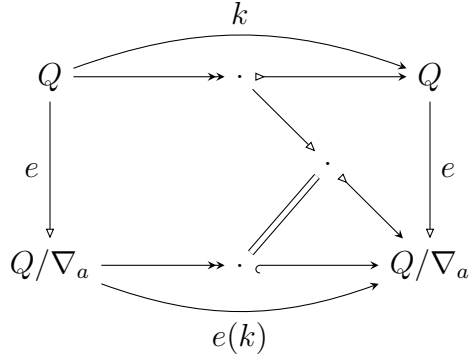
\begin{center}
 \begin{tikzpicture}[node distance=2.5cm, auto]
  \node (Q) {$Q$};
  \node (target) [right of=Q] {.};
  \node (Q2) [right of=target] {$Q$};
  \node (Qquotient) [below of=Q] {$Q/\nabla_a$};
  \node (Qquotient2) [below of=Q2] {$Q/\nabla_a$};
  \node (target2) [below of=target] {.};
  \draw[->, bend left=20] (Q) to node {$k$} (Q2);
  \draw[->>] (Q) to node {} (target);
  \draw[normalTail->] (target) to node {} (Q2);
  \draw[>=normalTip,->] (Q) to [swap] node {$e$} (Qquotient);
  \draw[>=normalTip,->] (Q2) to node {$e$} (Qquotient2);
  \draw[->>] (Qquotient) to node {} (target2);
  \draw[left hook->] (target2) to node {} (Qquotient2);
  \draw[->, bend right=20] (Qquotient) to [swap] node {$e(k)$} (Qquotient2);
  \path (target) -- node[anchor=center] (target3) {.} (Qquotient2);
  \draw[>=normalTip,->] (target) to node {} (target3); %
  \draw[normalTail->] (target3) to node {} (Qquotient2);
  \draw[double equal sign distance] (target2) to node {} (target3);
 \end{tikzpicture}
 \end{center}
 Here we have used \cref{cor:strict_maps_compose_for_modular_suplattices} to commute $e$ past the monic part of $k$
 and the uniqueness of the (epi,mono)-factorisation of $ek$ to obtain the isomorphism.
\end{proof} %

Our approach actually suggests a further notion, which does not seem to have been defined previously. %
\begin{definition}
 We say a map of suplattices is \emph{composably strict}\index{strict morphism!composably|textbf} if the composition of it with every strict map is strict.
 An element $k \in Q$ is \emph{strong principal} if its corresponding suplattice homomorphism is composably strict.
\end{definition}

Note that a quantale element $k$ is strong principal iff $k$ is strict, $k$ is left and right modular and $k^\bullet$ is left and right dual modular (by \cref{prop:normal_composition_and_modular_pairs}).

\begin{remark} %
Let us note some similarities of the above approach to work in two other sources.
The first is \cite{nai2014principalMappingsPosets}, which discusses similar classes of `principal maps' between posets and discusses how they relate to principal elements.
However, the authors do not discuss composition or modularity in much detail.
The second is \cite{blyth1972residuation}, which discusses the notion of `range-closed maps' between posets (corresponding to our mono-normal maps) and proves a number of
similar results involving composition and modularity. However, the link with abstract ideal theory appears to have gone unnoticed.
Our more categorical descriptions place the focus more strongly on composability and hopefully lead to more understandable proofs.
\end{remark}

\subsection{Principal elements in quantales of ideals}

While principal elements were intended to capture the lattice-theoretic behaviour of principal ideals, principal ideals are not characterised by this condition.
For instance, (at least classically) every ideal of a Dedekind domain is a principal element, but not every Dedekind domain is a principal ideal domain.
Rather, the principal elements of $\Idl(R)$ correspond to what are known as the finitely generated \emph{locally principal ideals}\index{ideal (discrete algebraic)!ring!locally principal|(textbf} \cite{McCarthy1971}.
The proofs in the literature are nonconstructive and even require the axiom of choice (though a constructive ring-theoretic proof of the backward direction can be found in \cite{LombardiQuitte}).
We will provide a full constructive proof.

\begin{definition}
 A finitely generated ideal $I$ in a ring $R$ is \emph{locally principal}\index{ideal (discrete algebraic)!ring!locally principal|)} if there are elements $f_1,\dots,f_n \in R$ such that $(f_1,\dots,f_n) = 1$ and the image of $I$ in the localisation $R[f_i^{-1}]$
 is principal for each $1 \le i \le n$.
\end{definition}

To prove the characterisation we will need to a few results on localisation\index{localisation!of idempotent semirings|(textbf} of quantales. The necessary results are stated in \cite{Anderson1976} without proof.
We present them here for completeness.

We first consider a special class of quotients of idempotent semirings.
\begin{definition}
 Let $D$ be a two-sided idempotent semiring and let $S$ be a multiplicative submonoid of $D$.
 We define $D/S$ to be the quotient of $D$ by the congruence $\langle (s,1) \mid s \in S \rangle$
 and call such a quotient a \emph{localisation}.
\end{definition}

\begin{lemma}\label{prop:order_on_localisation}
 If $x, y \in D$, then $[x]_S \le [y]_S$ in $D/S$ if and only if $\exists s \in S.\ xs \le y$.
\end{lemma}
\begin{proof}
 Define a relation ${\lesssim} \subseteq D \times D$ by $x \lesssim y \iff \exists s \in S.\ xs \le y$. This is a preorder.
 It is transitive, since if $s,s' \in S$, $xs \le y$ and $ys' \le z$, then $ss' \in S$ and $xss' \le ys' \le z$.
 It is reflexive, since $1 \in S$. In fact, it is clear that ${\le} \subseteq {\lesssim}$.
 
 Furthermore, $\lesssim$ is a sub-semiring of $D \times D$. If $s,s' \in S$, $xs \le y$ and $x's' \le y'$, then $ss'(x \vee x') = ss'x \vee ss'x' \le sx \vee s'x' \le y \vee y'$
 and $ss'xx' = (sx)(s'x') \le yy'$. %
 
 Thus, ${\sim} \coloneqq {\lesssim} \cap {\gtrsim}$ is a congruence. Notice that for any preorder satisfying the above properties, we have $x \lesssim y \iff (x \vee y) \sim y \iff [x]_\sim \le [y]_\sim$.
 
 Now it is clear that $s \sim 1$ for all $s \in S$. So $\sim$ contains the localisation congruence and hence $[x]_S \le [y]_S \implies x \lesssim y$.
 Finally, if $x \lesssim y$, then $xs \le y$ for some $s \in S$ and so $[x]_S = [x]_S[s]_S \le [y]_S$.
\end{proof}

\begin{lemma}\label{prop:equal_on_localisations_from_a_cover_dioid}
 Let $a_1,\dots,a_k \in D$ be such that $\bigvee_{i=1}^k a_i = 1$ and let $A_i$ be the multiplicative submonoid generated by $a_i$.
 Consider $x, y \in D$. If $[x]_{A_i} \le [y]_{A_i}$ for each $i$, then $x \le y$.
\end{lemma}
\begin{proof}
 By \cref{prop:order_on_localisation} we have $n_1,\dots,n_k \in \N$ such that $a_i^{n_i} x \le y$ for each $i$.
 But setting $N = 1 + \sum_{i=1}^k n_i$, the assumption gives %
 \begin{align*}
  1 = \left(\bigvee_{i = 1}^k a_i \right)^N = \ \smashoperator[l]{\bigvee_{\substack{m_1,\dots, m_k \in \N \\ m_1 + \cdots + m_k = N}}} \prod_{j = 1}^k a_j^{m_j} \le \bigvee_{i=1}^k a_i^{n_i},
 \end{align*}
 where the inequality comes from the fact that there is always some $i$ for which $m_i \ge n_i$ by the pigeonhole principle. %
 Hence $x = \bigvee_{i = 1}^k a_i^{n_i} x \le y$, as required.
\end{proof}

As a left adjoint, the functor $\I$ preserves quotients and so a localisation $D/S$ gives a quotient $\I(D) \twoheadrightarrow \I(D/S)$, which we may call a localisation of coherent two-sided quantales.
Such quantalic localisations\index{localisation!of idempotent semirings|)} have some desirable properties.

\begin{proposition}
 Let $q\colon \I(D) \twoheadrightarrow \I(D/S)$ be a localisation of quantales. Then
 \begin{itemize}
  \item $q$ preserves finite meets,
  \item $q({\downarrow}i \heyting J) = {\downarrow}[i] \heyting q(J)$.
 \end{itemize}
\end{proposition}
\begin{proof}
 It trivially preserves the empty meet, since the quantale is two-sided, and we always have $q(I \cap J) \le q(I) \cap q(J)$.
 For the other direction, suppose $[k] \in q(I) \cap q(J)$. Then $[k] \le [i], [j]$ for some $i \in I$ and some $j \in J$ and
 by \cref{prop:order_on_localisation} there are $s, s' \in S$ such that $ks \le i \in I$ and $ks' \le j \in J$. But then $kss' \in I \cap J$
 and hence $[k] = [kss'] \in \overline{k}(I \cap J)$, as required.
 
 We always have $q(I \heyting J) \le q(I) \heyting q(J)$, since $q(I \heyting J)q(I) = q((I \heyting J)I) \le q(J)$.
 Now take $I = {\downarrow} i$. We must show $[k] \in q(I) \heyting q(J) \implies [k] \in q(I \heyting J)$. Using the adjunction, we suppose $[ik] \in q(J)$
 and find a $k' \gtrsim k$ such that $ik' \in J$. As above, we have $iks \in J$ for some $s \in S$. Then taking $k' = ks$ yields the result.
\end{proof}
\begin{remark}
 Note that $q$ need not preserve the residuation arrow in general, since different elements of $I$ will give different values for $s$ and we need a single $k'$ to work for all $i \in I$.
\end{remark}

\begin{corollary}
 Localisations preserve compact (weak) meet/join principal elements.
\end{corollary}
\begin{proof}
 Simply note that all our notions of principal element $k$ are defined by equations involving multiplication, meet, join and residuation of the form $k \heyting x$.
 These are all preserved by localisation. (For preservation of $k \heyting x$ we use that $k$ is compact.)
\end{proof} %

\begin{lemma}\label{prop:equal_on_localisations_from_a_cover_quantale}
 Let $a_1,\dots,a_k \in D$ be such that $\bigvee_{i=1}^k a_i = 1$ and consider $I, J \in \I(D)$.
 If $q_i(I) \le q_i(J)$ for each localisation $q_i\colon \I(D) \twoheadrightarrow \I(D/\langle a_i\rangle)$, then $I \le J$.
\end{lemma}
\begin{proof}
 Take $x \in I$. We know $[x]_i \in q_i(J)$, so that $[x]_i \le [y_i]_i$ for some $y_i \in J$.
 So setting $y = \bigvee_{i = 1}^k y_i$, we have $[x]_i \le [y]_i$ for all $i$ and thus $x \le y \in J$ by \cref{prop:equal_on_localisations_from_a_cover_dioid}.
\end{proof}

\begin{corollary}\label{prop:compact_locally_principal_is_principal}
 Let $a_1,\dots,a_n \in D$ be such that $\bigvee_{i=1}^n a_i = 1$. Consider compact element $k$ of $\I(D)$ and suppose the image of $k$ in the localisation $D/\langle a_i\rangle$
 is (weak) meet/join principal for each $1 \le i \le n$. Then $k$ is (weak) meet/join principal (as appropriate).
\end{corollary}

Recall that the quantale of ideals of a ring is coherent.
We now observe that localisation of rings induces a localisation of the corresponding quantales of ideals.

\begin{lemma}
 Let $R$ be a ring and $S$ a multiplicative submonoid of $R$. Then the localisation $\ell\colon R \to S^{-1}R$ induces a localisation of quantales $\Idl(\ell)\colon \Idl(R) \twoheadrightarrow \Idl(S^{-1}R)$
 determined by the multiplicative submonoid $\overline{S} = \{(s) \mid s \in S\}$.
\end{lemma}
\begin{proof}
 First note that $\Idl(\ell)$ is surjective. It is enough to show every principal ideal is in the image of $\Idl(\ell)$ (since these form a base).
 Take $f/s \in S^{-1}R$. Since $1/s$ is invertible in $S^{-1}R$, we have $(f/s) = (f/1) = \Idl(\ell)((f))$ and hence $\Idl(\ell)$ is surjective.
 
 Now $\Idl(\ell)$ certainly sends $(s)$ to $1$ and hence factors through $\Idl(R)/\overline{S}$.
 We show the resulting map is injective.
 
 Suppose $\Idl(\ell)(I) \le \Idl(\ell)(J)$ and take $i \in I$. Then $i/1 \in \Idl(\ell)(J)$ and hence $i/1 = j/s$ for some $j \in J, s \in S$.
 This means $iss' = js' \in J$ for some $s' \in S$ and so $(i)(s'') \in J$ for $s'' = ss' \in S$. Therefore, $[(i)]_{\overline{S}} \le [J]_{\overline{S}}$
 and hence $[I]_{\overline{S}} \le [J]_{\overline{S}}$ as required.
\end{proof}

\begin{corollary}
 A finitely generated locally principal ideal of a ring $R$ is a principal element in $\Idl(R)$.
\end{corollary}
\begin{proof}
 Finitely generated ideals are compact elements in $\Idl(R)$ and so we may apply \cref{prop:compact_locally_principal_is_principal} and the fact that principal ideals are principal elements.
\end{proof}

It remains to show that principal elements of $\Idl(R)$ are finitely generated locally principal ideals.
This is consequence of the following proposition.

\begin{proposition}
 Suppose $Q$ is a coherent two-sided quantale with a base $B$ of compact elements which contains $0$. Let $x$ be a weak principal element of $Q$. Then there are elements $b_1,\dots,b_n \in B$
 such that $q_i(x) \in q_i(B)$ for each localisation $q_i\colon Q \twoheadrightarrow Q/\langle b_i\rangle$.
 Furthermore, $x$ is compact.
\end{proposition}
\begin{proof}
 Write $x = \bigvee \{a \in B \mid a \le x\}$.
 Since $x$ is weak meet principal, $a = x(x\heyting a)$ for each $a \le x$. Let $j = \bigvee_{a \in B \cap {\downarrow} x} (x\heyting a)$
 so that $x = xj$. Now since $x$ is weak join principal, $1 = j \vee x^\bullet = \bigvee_{a \in B \cap {\downarrow} x} (x\heyting a)$, where we have the final equality since $0 \in B$.
 
 Now express each $x\heyting a_\alpha$ as a join of elements of $B$. The unit $1$ is compact, since $Q$ is coherent,
 and so we have $1 = b_0 \vee b_1 \vee \dots \vee b_n$ where each $b_i$ is an element of $B$ lying below some $x\heyting a_{\alpha_i}$.
 
 In the quotient $Q/\langle b_i\rangle$, we have $1 = [b_i] \le [x\heyting a_{\alpha_i}] \le [x] \heyting [a_{\alpha_i}]$ and hence $[x] \le [a_{\alpha_i}]$.
 But we know $a_{\alpha_i} \le x$ and so $[x] = [a_{\alpha_i}] \in q_i(B)$ as required.
 
 To see that $x$ is compact, notice that $[x] = [\bigvee_{j=1}^n a_{\alpha_j}]$ in each localisation and thus $x = \bigvee_{j=1}^n a_{\alpha_j}$ by \cref{prop:equal_on_localisations_from_a_cover_quantale}.
 So $x$ is a finite join of compact elements and therefore compact.
\end{proof}

\begin{corollary}\label{prop:principal_elements_of_Idl_R}
 If $R$ is a ring, the principal elements of $\Idl(R)$ are precisely the finitely generated locally principal ideals.
\end{corollary}
\begin{proof}
 We have already shown the backward direction. For the forward direction, simply apply the above result with $B$ taken to be the set of principal ideals.
\end{proof}

\subsection{Valuations}

We will now discuss two concepts from ring theory which arise naturally when considering quantales of ideals.
In order to make links with classical notions, \emph{we will assume excluded middle\index{excluded middle|(} for this section}.

Consider $\Nvar$ --- the free two-sided quantale on one generator $g$.
We will write the elements of $\Nvar$ as $g^n$ for $n \in \N \cup \{\infty\}$ to avoid confusion about what $0$ and $1$ mean.

\begin{definition}
 A \emph{valuation}\index{valuation|(textbf} on a two-sided quantale $Q$ is a quantale homomorphism $w\colon Q \to \Nvar$.
\end{definition}

\begin{definition}
 Let $R$ be a ring. We will call a map $v\colon R \to \N \cup \{\infty\}$ a \emph{positive prevaluation} if
 it satisfies
 \begin{itemize}
 \item $v(0) = \infty$,
 \item $v(a+b) \ge \min(v(a), v(b))$,
 \item $v(1) = 0$,
 \item $v(ab) = v(a) + v(b)$.
\end{itemize}
\end{definition}

\begin{lemma}
 There is a bijective correspondence between positive prevalutations $v$ on a ring $R$ and quantale valuations $\widetilde{v}\colon \Idl(R) \to \Nvar$.
 Here $\widetilde{v}(I) = \bigvee_{a \in I} g^{v(a)}$.
\end{lemma}
\begin{proof}
 This is immediate from the description of $\Idl(R)$ in terms of generators and relations (\cref{prop:presentation_for_Idl_R}).
\end{proof}
\begin{remark}
 This result becomes constructively valid if we restrict to quantale valuations that preserve compact elements. %
\end{remark}

The notion of a positive prevaluation is very reminiscent of that of a \emph{discrete valuation}, which is important in algebraic geometry.
\begin{definition}
 A \emph{discrete valuation}\index{discrete valuation|textbf} on a field $K$ is function $v\colon K \to \Z \cup \{\infty\}$ satisfying
 \begin{itemize}
 \item $v(a) = \infty \iff a = 0$,
 \item $v(a+b) \ge \min(v(a), v(b))$,
 \item $v(1) = 0$,
 \item $v(ab) = v(a) + v(b)$.
 \end{itemize}
\end{definition}

We can make this relationship more precise.
Let $v$ be a positive prevaluation on a ring $R$. First notice that $P = v^{-1}(\infty)$ is a prime ideal of $R$. Indeed, composing $\widetilde{v}$ with the map from $\Nvar$ to $\Omega$ sending $g$ to $1$ gives
a point of $\Idl(R)$, which classically corresponds to a `prime element' in $\Idl(R)$. %
We can now factor $v$ through the quotient $R/P$ to give a positive prevaluation $v'$ on an integral domain satisfying $v'(a) = \infty \iff a = 0$.

Such a positive prevaluation on an integral domain $D$ can then be extended to give a valuation $v''$ on its field of fractions $K(D)$.
We set $v''(a/b) = v'(a) - v'(b)$. Let us show this gives a valuation. Only the additive condition is nontrivial. We have
\begin{align*}
 v''(a/b + c/d) &= v''((ad+cb)/bd) \\
                &= v'(ad+cb) - v'(b) - v'(d) \\
                &\ge \min(v'(a) + v'(d), v'(c) + v'(b)) - v'(b) - v'(d) \\
                &= \min(v'(a) - v'(b), v'(c) - v'(d)) \\
                &= \min(v''(a/b), v''(c/d)).
\end{align*}
In summary, we have the following lemma.

\begin{lemma}
 A positive prevaluation $v$ on a ring $R$ induces a discrete valuation\index{valuation|)} $v''$ on the field $K(R/v^{-1}(\infty))$
 satisfying $v''(a) \ge 0$ for all $a \in R$. \hfill$\Box$
\end{lemma}

However, an element $a \in K(R/P)$ for which $v''(a) \ge 0$ does not necessarily lie in $R$.
This can be understood by considering the other map from $\Nvar$ to $\Omega$, which sends $g$ to $0$.
Composing this with $\widetilde{v}$ gives another point of $\Idl(R)$. This corresponds to a prime ideal $Q \ge P$.
The complement of $Q$ is a multiplicative submonoid consisting of the elements $a$ of $R$ for which $v(a) = 0$.
If we both quotient $R$ by $P$ and localise at $Q$ we obtain a ring $R_Q/P$.
This is precisely the subring of $K(R/P)$ on which $v''$ is nonnegative.

A ring $R_Q/P$ which arises in this manner is known a \emph{discrete valuation ring} (unless it is a field).
It is well known that $\widetilde{v}$ factors through $\Idl(R_Q/P)$ to give an injection. %
However, this might fail for coherent quantales that do not come from rings. %

It can also be illuminating to consider maps from a quantale into a quotient of $\Nvar$.
These might be induced by valuations, but need not be. For instance, a map from a quantale $Q$ into $\O \Srpnsk$ %
corresponds to a pair of prime elements $p \le q$. However, not every such pair can be obtained from a valuation ---
in ring-theoretic terms, $R_Q/P$ might not be a discrete valuation ring. %

More interesting are maps from $Q$ to the two-sided quantale $\Dvar = \langle \epsilon \mid \epsilon^2 = 0 \rangle \cong \Nvar/\nabla_{g^2}$\glsadd{Dvar}. %
As a set $\Dvar = \{0, \epsilon, 1\}$.
This is the quantalic incarnation of an infinitesimal tangent vector at a point.

\begin{lemma}\label{prop:map_to_D_and_primary_elements}
 Let $Q$ be a two-sided quantale. A quantale homomorphism $f\colon Q \to \Dvar$ yields a prime element $p = f_*(\epsilon)$ and
 an element $j = f_*(0)$ with $p^2 \le j \le p$ such that $ab \le j \implies (a \le j) \lor (b \le p)$.
 Furthermore, every pair of elements $(p,j)$ satisfying these conditions arises in this way from a unique quantale map.\index{excluded middle|)}
\end{lemma}
\begin{proof}
 Note that $\epsilon$ is a prime element of $\Dvar$ and hence $p = f_*(\epsilon)$ is a prime element of $Q$.
 Then $f(p^2) = f(p)^2 \le 0$ and so $p^2 \le f_*(0) = j$. And certainly, $j \le p$.
 
 Now suppose $ab \le j$. Then $f(a)f(b) = 0$.
 It is easy to see that this can only happen if $f(a) \le 0$ or $f(b) \le \epsilon$, so that $a \le j$ or $b \le p$.
 
 On the other hand, suppose we are given $p$ and $j$ satisfying the properties above. These uniquely define a meet preserving map from $\Dvar$ to $Q$, which then has a left adjoint $f$.
 We can show $f$ preserves products by simply checking all the cases. We will prove just one case here, but the others are similar.
 
 Suppose $f(x) = \epsilon$ and $f(y) = 1$. Then $f(x) \not\le 0$ and $f(y) \not\le \epsilon$, so that $x \not\le j$ and $y \not\le p$.
 Hence, $xy \not\le j$ by the contraposition of the condition on $j$. This means $f(xy) \not\le 0$ --- that is, $f(xy) \ge \epsilon$.
 But we certainly have $f(xy) \le f(x) = \epsilon$ and hence $f(xy) = \epsilon = \epsilon \cdot 1$, as required.
\end{proof}

Recall the following definition from ring theory.
\begin{definition} %
 An ideal $J$ is \emph{primary}\index{ideal (discrete algebraic)!ring!primary|textbf}, if whenever $ab \in J$, we have $a \in J$ or $b^n \in J$ for some $n \in \N$.
 The radical ideal $\sqrt{J}$ is prime and if $\sqrt{J} = P$ we say $J$ is \emph{$P$-primary}.
\end{definition}
So \cref{prop:map_to_D_and_primary_elements} tells us that quantale homomorphisms from $\Idl(R)$ to $\Dvar$ correspond to a pair of ideals $(P,J)$ where $P$ is prime, $J$ is $P$-primary and $J \ge P^2$. %
Here $P$ represents a point of the spectrum of $R$ and $J$ gives some kind of information about first-order derivatives at that point.
This can be made more explicit by considering the local ring $(R_P, \m)$ obtained by localising $R$ at $P$.
A $P$-primary ideal $J \ge P^2$ corresponds bijectively to an ideal $\m^2 \le \overline{J} \le \m$ in the localisation.
Thus, $J$ corresponds to a subspace of the cotangent space\index{tangent space} $\m/\m^2$ at $P$. Actually, it is better to think of $J$ as defining a \emph{quotient} of the cotangent space %
and hence a subspace of the \emph{tangent space}. We will explore maps into $\Dvar$ in more detail in \cref{section:tangent_bundle}.

\section{The spectrum of localic semirings}\label{section:spectrum}

\index{spectrum|(}
As discussed in the introduction, Stone's spectra of Boolean algebras \cite{StoneBoolAlgs} and distributive lattices \cite{StoneDistLattices},
the Zariski spectrum of rings, Gelfand's spectrum (see \cite{GelfandNormedSpectrum,Henry2016}) of commutative C*-algebras (with their norm topology) and
the Hofmann--Lawson spectrum \cite{hofmann1978spectral} of continuous frames (thought of as topological distributive lattices by equipping them with the Scott topology)
can all be viewed as spectrum constructions of localic semirings where the opens are given by the overt weakly closed radical ideals and the points are given by closed prime ideals.
In this chapter we define the general spectrum which reduces to these as special cases.

The basic idea behind the spectrum of a ring $R$ is to provide a space $X$ such that the elements of $R$ can be thought of as certain continuous functions on $X$.
While the one-point space might technically satisfy this
condition, we would like this space to in some sense be `as general as possible', so that it elucidates the structure of $R$. %
We do not require these functions on $X$ to have a fixed target in some (topological) ring, but instead think of them in the sense of sections of a bundle, so that their codomains differ at different points of $X$.

See \cite{Johnstone1982stone} for a discussion a number of different examples of spectrum constructions, albeit from a classical perspective.

A point of the spectrum is characterised by the values each element is thought to take at that location. While we do not yet know the ring in which the functions take values, we at least know that it contains $0$.
So points will be characterised by the functions which vanish (or do not vanish) at them. We topologise the spectrum so that a function \emph{not} vanishing at a point is a verifiable property.
(This is by analogy to the Zariski topology and because in the Gelfand case $\C$ has open inequality.)

When working constructively we should interpret the notion of being nonzero at a point in a positive sense, not as the negation of being zero.
We will say a function is \emph{cozero}\index{cozero|(} at a point if it is nonzero in this sense. It is helpful to imagine being cozero as akin to being invertible\index{cozero|)}. %

There are a few different ways to describe the spectrum of a ring. Firstly, we can describe it by a geometric theory. The constant function $1$ is cozero everywhere, while zero is cozero nowhere.
If the sum of two functions is cozero at a point, then one of them must be cozero at that point. Finally, the product of two functions is cozero at a point if and only if both of them are.
So we have the following geometric theory (where a function $f$ stands for the predicate `$f$ is cozero').
\begin{align*} %
 0    &\mathmakebox[4.5ex][c]{\vdash} \bot \\
 f+g  &\mathmakebox[4.5ex][c]{\vdash} f \vee g \\
 \top &\mathmakebox[4.5ex][c]{\vdash} 1 \\
 fg   &\mathmakebox[4.5ex][c]{\dashv\vdash} f \wedge g
\end{align*}
The corresponding frame gives the usual Zariski topology.
Another description of the frame is obtained by thinking about the open sets directly. An open set in the spectrum should be a verifiable property and we can verify a point lies in an open set
if we can give a function which is cozero at that point. In this way open sets correspond to sets of functions.
But it isn't hard to see that ``some function in $S$ vanishes at $x$'' holds if and only if ``some function in $I$ vanishes at $x$'', where $I$ is the radical ideal generated by $S$.
So the frame of opens should correspond to the frame of radical ideals and this is indeed the case.

In the above we have restricted our consideration to discrete rings. More generally we will need to deal with localic rings.
This is in accordance with Stone's maxim of ``Always topologise''. In fact, we are taking this maxim to the extreme by topologising Stone duality itself! %
In the case of C*-algebras, ignoring the topology and taking the Zariski spectrum gives the wrong result.
In \cref{section:generalised_presentations,section:spectrum_universal_property} we attempt to mimic this presentation of the Zariski spectrum in the case that $R$ is a localic ring
where we cannot consider its points as generators in a naive manner. In later sections we relate this frame described by this presentation to a construction involving the overt weakly
closed ideals of $R$.

The generalisation from rings to semirings is necessary to handle the cases of discrete distributive lattices and continuous frames (with the Scott-topology).
In this more general setting, there is no reason to expect that the vanishing of functions gives the full story. There might be a finer spectrum which also
encodes information about where a function is equal to $1$ or even $1+1$ and so on. Nonetheless, the information about vanishing still provides the answers to an
important class of questions and it is this spectrum which gives the expected results for our examples.
\index{spectrum|)}

\subsection{A generalised notion of presentation}\label{section:generalised_presentations}

For the classical Zariski spectrum, we construct a presentation of the frame from the elements of the ring. However, this does not take the topology of the ring into account and it isn't even an option
for the spectrum of a localic ring that does not have enough points. We will generalise the notion of presentation to be able to handle this. (An unfortunate aspect of the generalisation is that there might not be
any frame satisfying a given presentation, but it seems that this is unavoidable.)

The classical notion of a presentation\index{presentation|(} of a frame can be viewed as expressing a frame $L$ as the coequaliser of maps between free frames $\langle R\rangle \rightrightarrows \langle G\rangle \twoheadrightarrow L$,
where $G$ is the set of generators and $R$ indexes the relations. From the localic perspective, the free frame on $G$ corresponds to the $G^\text{th}$ power of $\Srpnsk$ and so the presentation corresponds to an
equaliser $L \hookrightarrow \Srpnsk^G \rightrightarrows \Srpnsk^R$ (where we reuse the variable $L$ for the corresponding locale). %

We now reinterpret $G$ and $R$ as discrete locales and the powers of $\Srpnsk$ as exponentials. This allows us to replace $R$ and $G$ with any exponentiable locales. However, since our localic semirings might
not be locally compact, this is not yet sufficiently general for our purposes.

We can circumvent the nonexistence of exponentials by passing to the presheaf category $\Set^{\Loc\op}$. Here we consider the equaliser of a pair of natural transformations
$F \hookrightarrow \Hom_\Loc((-) \times G, \Srpnsk) \rightrightarrows \Hom_\Loc((-) \times R, \Srpnsk)$. The resulting functor is not always representable, but when it is we say the representing object $L$
is presented by the generalised presentation. It is shown in \cite{VickersTownsendDoublePowerlocale} that natural transformations from $\Hom_\Loc((-) \times G, \Srpnsk)$ to $\Hom_\Loc((-) \times R, \Srpnsk)$
are in bijection with dcpo morphisms from $\O G$ to $\O R$ (by taking the component of each natural transformation at $1$).
Thus, a generalised presentation may be described by a locale $G$ of generators and a pair of dcpo morphisms from $G$ into another locale $R$ describing the relations.\index{presentation|)}

\subsection{The universal property}\label{section:spectrum_universal_property}

Let $R$ be a localic semiring\index{semiring!localic}. We have frame homomorphisms $\epsilon_0\colon \O R \to \Omega$, $\epsilon_1\colon \O R \to \Omega$, $\mu_+\colon \O R \to \O R\oplus \O R$ and
$\mu_\times\colon \O R \to \O R\oplus \O R$\glsadd{epsilon0}\glsadd{epsilon1}\glsadd{muplus}\glsadd{mutimes} for the additive identity, the multiplicative identity, addition and multiplication, respectively.

Mimicking the Zariski spectrum, we will describe the spectrum $\Spec R$ by a generalised presentation where $R$ is the locale of generators.
In the Zariski case we have one relation involving $0$, one involving $1$, an $(R\times R)$-indexed family involving addition and an $(R\times R)$-indexed family involving multiplication, so the
relations are indexed by the set $1 \sqcup 1 \sqcup R\times R \sqcup R\times R$. Similarly, in our case the relations are indexed by the locale given by the frame
$\Omega \times \Omega \times (\O R \oplus \O R) \times (\O R \oplus \O R)$.
The relations are described by a parallel pair of dcpo morphisms, which decompose into four pairs of dcpo morphisms $\O R \rightrightarrows \Omega$, $\O R \rightrightarrows \Omega$,
$\O R \rightrightarrows \O R \oplus \O R$ and $\O R \rightrightarrows \O R \oplus \O R$. The first pair is $\epsilon_0$ and the constant function $0$. The second pair is $\epsilon_1$ and the constant function $1$.
The third is $\mu_+ \vee \iota_1 \vee \iota_2$ and $\iota_1 \vee \iota_2$, where $\iota_{1,2}\colon R \to R\oplus R$ are the coproduct injections. The forth pair is $\mu_\times$ and $\iota_1 \wedge \iota_2$.

The resulting equaliser can be explicitly described as a functor $\OPAI_R\colon \Loc\op \to \Set$\glsadd{OPAIR} given by
\begin{align*}
 \OPAI_R\colon X \mapsto \big\{ u \in \O X \oplus \O R \mid {}
 & (\O X \oplus \epsilon_0)(u) = 0,\ (\O X \oplus \epsilon_1)(u) =  1, \\ 
 & (\O X \oplus \mu_+)(u) \le (\O X \oplus \iota_1)(u) \vee (\O X \oplus\iota_2)(u), \\
 &  (\O X \oplus \mu_\times)(u) = (\O X \oplus\iota_1)(u) \wedge (\O X \oplus\iota_2)(u) \big\}
\end{align*}
where we use elements of the frame $\O X \oplus \O R$ in place of locale maps $X \times R \to \Srpnsk$.
If $f\colon Y \to X$ is a map of locales, then $\OPAI_R(f)$ sends $u \in \OPAI_R(X)$ to $(f^* \oplus \O R)(u)$.
We call this $\OPAI$ since it gives the open prime anti-ideals\index{anti-ideal!open prime|textbf} (or equivalently the closed prime ideals) of $R$ `fibred over $X$'. (For example, the first condition specifies that the putative anti-ideal does not
contain $0$ and the third condition is a pointfree formulation of the requirement that either $r \in u$ or $s \in u$ whenever $r + s \in u$.)

\begin{definition}\label{def:localic_spectrum}
 The \emph{(localic) spectrum}\index{spectrum!localic|textbf} of a localic semiring $R$ is the representing object $\Spec R$\glsadd{SpecR} of the presheaf $\OPAI_R$ (equipped with the universal element $\widetilde{\upsilon}_R \in \O \Spec R \oplus \O R$). %
\end{definition}

Note that $\OPAI_R$ is functorial in $R$ and so $\Spec$ is a (partial) functor by parametrised representability. %
Of course, we have no guarantee that $\OPAI_R$ is representable at all. %
We will describe sufficient conditions for the spectrum's existence in \cref{section:sufficient_conditions_for_spectrum}.

The universal element associated with the spectrum can be thought of as the uncurried form of a function that assigns functions to their cozero sets (were it to exist).

We can generalise the spectrum functor from locales to (two-sided) quantales by considering the following functor.
\begin{definition}\label{def:quantic_spectrum_as_OPAI}
 Let $R$ be a localic semiring. The functor $\overline{\OPAI}_R\colon \Quant \to \Set$\glsadd{OPAIRoverline} is given by
\begin{align*}
 \overline{\OPAI}_R\colon Q \mapsto \big\{ u \in Q \oplus \O R \mid {}
 & (Q \oplus \epsilon_0)(u) = 0,\ (Q \oplus \epsilon_1)(u) =  1, \\ 
 & (Q \oplus \mu_+)(u) \le (Q \oplus \iota_1)(u) \vee (Q \oplus\iota_2)(u), \\
 &  (Q \oplus \mu_\times)(u) = (Q \oplus\iota_1)(u) \cdot (Q \oplus\iota_2)(u) \big\}.
\end{align*}

The representing object of $\overline{\OPAI}_R$, if it exists, is called the \emph{quantic spectrum}\index{spectrum!quantic|textbf} of $R$.
In this case, the localic spectrum also exists and is given by the localic reflection of the quantic spectrum. %
\end{definition}

\begin{remark}
 This definition even works when $R$ is a general commutative semiring object in the opposite of the category of quantales, but we will restrict our attention to localic semirings. %
\end{remark}

\subsection{The quantale of ideals} \label{section:quantale_of_overt_weakly_closed_ideals}

At this point we change tack and ask directly what the opens should be in a good notion of a spectrum. In the classical examples, we are often interested in finding the region where every element of a set of ring
elements (viewed as functions on the spectrum) vanishes. This should describe a closed set. If we wish to work with open sets, we can instead ask for the region where \emph{some} element of the set of functions is
\emph{cozero}. In our setting it makes little sense to consider a set of ring elements, but we can replace this with a sublocale of the localic semiring $R$. However, at least constructively, not every
sublocale is appropriate.\index{quantale of ideals|(textbf} We need to be able to ask when `some element of the sublocale is cozero' (interpreted in an appropriately pointfree way).
For this to be possible it seems reasonable to ask that our sublocale admits existential quantification of `open' predicates.
This is precisely what the concept of overtness allows us to do and so we should restrict to overt sublocales.

Of course, not every overt sublocale will give a unique open. For example, adding the zero function to a sublocale will not change the open it describes. Furthermore, if we `close a sublocale under addition', this should
not change the open it describes, since if $f + g$ is cozero somewhere, then one of $f$ or $g$ must be cozero there too. The open is also unaffected if we `multiply the sublocale by the whole locale', since if $fg$
is cozero somewhere, then $f$ is also cozero there. Finally, at least modulo nilpotents, if $f$ is cozero somewhere, then so is $f^2$. So in some sense, the same open is induced by a sublocale and the
`radical ideal generated by that sublocale'. We now make this precise.

The (constructive) closed subgroup theorem (see \cite{JohnstoneClosedSubgroup}) states that every overt subgroup of a localic group is weakly closed.
Consequently, overt ideals in localic rings are always weakly closed.
In terms of existential quantification an overt sublocale and its weak closure (which is still overt) behave identically, so it is reasonable to restrict to weakly closed overt sublocales even for semirings
where the closed subgroup theorem is unavailable.

Recall that the assignment to each locale of its suplattice of overt weakly closed sublocales gives a functor $\SubOW\colon \Loc \to \Sup$,
which factors as (the opposite of) the forgetful functor $U\colon \Frm \to \Sup$ followed by the internal hom functor $\hom(-,\Omega)$ in $\Sup$.
When $\Frm$ is equipped with the cocartesian symmetric monoidal structure, $U$ becomes a strong symmetric monoidal functor.
Furthermore, the internal hom in a symmetric monoidal closed category $\Cvar$ is a lax symmetric monoidal functor $\hom\colon\Cvar\op \times \Cvar \to \Cvar$ and the inclusion from $\Cvar\op$ to $\Cvar\op \times \Cvar$ given by pairing
with the monoidal unit is strong symmetric monoidal. Hence $\hom(-,\Omega)\colon \Sup\op \to \Sup$ is lax symmetric monoidal. Thus, the composite $\SubOW$ is also lax symmetric monoidal.
It follows that $\SubOW$ preserves commutative monoids. %

So if $R$ is a localic semiring, $\SubOW(R)$ is a quantale in two different ways. Explicitly, the multiplicative structure of $R$ gives a unit $1_\times = \epsilon_1$ and multiplication
$fg = \nabla (f \otimes g) \mu_\times$, where $\nabla\colon \Omega\otimes\Omega \cong \Omega$\glsadd{nabla} is the meet in $\Omega$.
Similarly, the additive structure gives an additive identity $0_+ = \epsilon_0$ and an addition $f + g = \nabla (f \otimes g) \mu_+$. Note that $0_+$ is not the same as $0$,
the least element of the suplattice.

We wish to restrict to those overt weakly closed sublocales that are closed under addition --- that is, the elements $a \in \SubOW(R)$ such that $0_+ \le a$ and $a + a \le a$.
It is easy to these elements are closed under arbitrary meets and so they form a sub-inf\/lattice of $\SubOW(R)$. But since we are in the category of suplattices, it is better to take the
adjoint of the inclusion and instead view it as a quotient suplattice. The associated closure operator is given by $\nu\colon a \mapsto \bigvee_{n \in \N} n\cdot a$, where we write $n\cdot a$\glsadd{cdot}
to denote repeated addition. %

\begin{lemma}
 The closure operator $\nu$ is a nucleus for the multiplicative quantale structure on $\SubOW(R)$.
\end{lemma}

\begin{proof}
 We must show that $\nu(a)\nu(b) \le \nu(ab)$. Expanding the definition of $\nu$ we see it is enough to show $(n\cdot a)(m \cdot b) \le nm\cdot ab$ for all $n,m \in \N$
 and this in turn follows from $0_+ x \le 0_+$ and $(x+y)z \le xz + yz$ by induction. We will prove these inequalities.
 
 One of the axioms of a semiring (that zero is absorbing) is given by the following diagram in $\Loc$ (where we write $\epsilon_0$ and $\mu_\times$ for their corresponding locale maps).
 \begin{center}
  \begin{tikzpicture}[node distance=2.5cm, auto]
   \node (RI)  {$1 \times R$};
   \node (RR) [right of=RI, xshift=0.8cm] {$R \times R$};
   \node (I) [below of=RI] {$1$};
   \node (R) [below of=RR] {$R$};
   \draw[->] (RI) to node {$\epsilon_0 \times R$} (RR);
   \draw[->] (RI) to node [swap] {$!$} (I);
   \draw[->] (RR) to node {$\mu_\times$} (R);
   \draw[->] (I) to node [swap] {$\epsilon_0$} (R);
  \end{tikzpicture}
 \end{center}
 So in $\Frm$ we have $(\epsilon_0 \oplus \O R)\mu_\times = {!}\circ\epsilon_0$. Then applying the forgetful functor $U$ and composing with $\Omega \otimes x$ where $x\colon \O R \to \Omega$, we get
 $(\epsilon_0 \otimes x)\mu_\times = \epsilon_0 \otimes x$. Hence, $0_+ x = \nabla (\epsilon_0 \otimes x)\mu_\times = \nabla (\epsilon_0 \otimes x) = \epsilon_0 \wedge x \le \epsilon_0 = 0_+$.
 
 Another axiom of semirings (distributivity) is given by
 \begin{center}
  \begin{tikzpicture}[node distance=2.5cm, auto]
   \node (RRR)  {$R^3$};
   \node (RR) [below of=RRR] {$R^2$};
   \node (RRRR) [right of=RRR, xshift=0.5cm] {$R^4$};
   \node (RRRR') [right of=RRRR, xshift=-1.0cm] {$R^4$};
   \node (RR') [right of=RRRR'] {$R^2$};
   \node (R) [below of=RR'] {$R$};
   \draw[->] (RRR) to node [swap] {$\mu_+ \times R$} (RR);
   \draw[->] (RRR) to node {$R^2 \times \Delta_R$} (RRRR);
   \draw[->] (RRRR) to node {$\tau$} (RRRR');
   \draw[->] (RRRR') to node {$\mu_\times \times \mu_\times$} (RR');
   \draw[->] (RR) to node [swap] {$\mu_\times$} (R);
   \draw[->] (RR') to node {$\mu_+$} (R);
  \end{tikzpicture}
 \end{center}
 where $\Delta_R\colon R \to R\times R$ is the diagonal map and $\tau\colon R^4 \to R^4$ interchanges the middle two factors of the product.
 Thus, if $x,y,z\colon \O R \to \Omega$, proceeding as before we have $((x \otimes y) \mu_+ \otimes z) \mu_\times = (x \otimes y \otimes z \Delta_R) \tau (\mu_\times \otimes \mu_\times) \mu_+$.
 But $z \Delta_R(u \otimes v) = z(u\wedge v) \le z(u) \wedge z(v) = \nabla (z\otimes z) (u \otimes v)$ and so $z \Delta_R \le \nabla (z \otimes z)$.
 Hence, we obtain
 $(x+y)\cdot z = \nabla(\nabla \otimes \Omega)((x \otimes y) \mu_+ \otimes z) \mu_\times \le \nabla(\nabla \otimes \nabla)((x \otimes z)\mu_\times \otimes (y \otimes z)\mu_\times) \mu_+ = x\cdot z + y\cdot z$,
 as required.
\end{proof}

Thus, the quotient is a quantale under multiplication.
\begin{definition}
 This quantale obtained as above from the nucleus $\nu$ on $\SubOW(R)$ will be called the \emph{quantale of additive submonoids of $R$}
 and denoted by $\AddSub(R)$.
\end{definition}

On the other hand, $\nu$ is not quite a nucleus for the additive structure. But we \emph{do} have $\nu(a) + \nu(b) = \nu(a+b)$ when $a,b \ge 0_+$ and so, in some sense, the failure to preserve the
additive structure is entirely due to the fact that $0_+ \ne 0$. We might think of this failure as being rather small. In any case, sums in $\SubOW(R)$ become joins in the quotient and $\AddSub(R)$
is the universal such quantale. Compare the `quantic spectrum' of \cite{QuantalesAndHyperstructures}. %

So finite joins in the $\AddSub(R)$ can be computed as sums in $\SubOW(R)$. The following result shows how to compute directed joins.
\begin{lemma}
 Directed joins in $\AddSub(R)$ are computed as in $\SubOW(R)$.
\end{lemma}
\begin{proof}
 The elements $a$ of $\SubOW(R)$ corresponding to additive submonoids are those satisfying $0_+ \le a$ and $a + a \le a$.
 Suppose $b = \dirsup[\alpha] a_\alpha$ is a directed join of additive submonoids in $\SubOW(R)$.
 Then $0_+ \le b$, since directed joins are inhabited. To see that $b + b \le a$
 we use the fact that whenever a function is bilinear with respect to directed joins, it is also linear with respect to directed joins
 and so $b + b = \dirsup[\alpha] a_\alpha + \dirsup[\alpha] a_\alpha = \dirsup[\alpha] a_\alpha + a_\alpha \le \dirsup[\alpha] a_\alpha = b$.
 Thus $b$ is also an additive submonoid.
\end{proof}

We can take the two-sided and localic reflections of this quantale to obtain the desired quantales of ideals\index{ideal (localic algebraic)!overt weakly closed|textbf} and radical ideals.
\begin{definition}
 The \emph{quantale of ideals of $R$}\index{quantale of ideals|)} is the two-sided reflection of the quantale $\AddSub(R)$ and is denoted by $\Idl(R)$\glsadd[format=(]{Idl}.
 The \emph{frame of radical ideals of $R$}\index{frame of radical ideals|textbf} is then given by the localic quotient of $\Idl(R)$ and denoted by $\Rad(R)$\glsadd[format=(]{Rad}.
\end{definition}

\begin{lemma}
 Joins in $\Idl(R)$ are computed as in $\AddSub(R)$.
\end{lemma}
\begin{proof}
 Simply observe that the nucleus $x \mapsto x \cdot \top$ corresponding to the two-sided quotient preserves arbitrary joins.
\end{proof}

It will also be useful to consider the two-sided reflection of $\SubOW(R)$.
\begin{definition}
 The \emph{quantale of monoid ideals of $R$}\index{quantale of monoid ideals|textbf} is given by the two-sided reflection of $\SubOW(R)$ (with its multiplicative structure)
 and is denoted by $\MonIdl R$\glsadd{MonIdl}.
\end{definition}

By the universal property of the two-sided reflection, the quotient $\SubOW(R) \twoheadrightarrow \Idl(R)$
factors through $\MonIdl R$ so that $\Idl(R)$ is also a quotient of $\MonIdl R$.
It can be shown that the corresponding nucleus on $\MonIdl R$ has the same form as $\nu$ above. %

\begin{definition}
 We write $\addquot$\glsadd{addquot} for this induced surjection from $\MonIdl R$ to $\Idl(R)$.
\end{definition}

The above definitions of $\Idl(R)$ and $\Rad(R)$ are corroborated by the constructive theories of the Zariski and Gelfand spectra.
If $R$ is a discrete ring, then the overt weakly closed sublocales are precisely the open sublocales and so the
overt weakly closed localic radical ideals coincide with the set-theoretic radical ideals in the usual constructive approach. Henry shows in \cite{ConstructiveGelfandNonunital} that if $R$ is a localic C*-algebra,
then the underlying frame of the Gelfand spectrum of $R$ is order isomorphic to $\Idl(R)$.\footnote{In particular, %
note that when $R$ is a commutative localic C*-algebra, the underlying lattice of $\Idl(R)$ is a frame. However, it remains to show that the multiplication in $\Idl(R)$ is given by meet
and hence $\Rad(R) \cong \Idl(R)$. We will prove this in \cref{prop:gelfand_expected_spectrum}.}

We conclude by noting that $\Idl$ and $\Rad$ are functorial with an action on morphisms inherited from $\SubOW$.\glsadd[format=)]{Idl}\glsadd[format=)]{Rad} %
To see this it is enough to note that if $f\colon R \to R'$ is a morphism of localic semirings, then $\SubOW(f)$
preserves the additive and multiplicative quantale structures. From this it is immediate that $\SubOW(f)$ sends
additive submonoids to additive submonoids. On the other hand, ideals and radical ideals are not preserved by $\SubOW(f)$,
but we obtain the desired morphism by composing with the unit of the relevant reflection.

\subsection{When does the spectrum exist?}\label{section:sufficient_conditions_for_spectrum}

Suppose $R$ is a localic semiring. We want to know if the quantale $\Idl(R)$ of overt weakly closed ideals constructed in \cref{section:quantale_of_overt_weakly_closed_ideals} satisfies the
universal property of the quantic spectrum. To prove this we need to equip $\Idl(R)$ with a candidate universal open prime anti-ideal fibred over it, which we call $\upsilon$\glsadd{upsilon}.
Since $\Idl(R)$ is defined in terms of overt sublocales, it seems reasonable to restrict to the case that $R$ is overt.

Before we define $\upsilon$, it will be useful to have a good understanding of \emph{saturated opens}\index{saturated!open|(textbf} in $R$.
Recall that a subset $S$ of a discrete semiring is called \emph{saturated} if $xy \in S \implies (x \in S) \land (y \in S)$.

\begin{definition}
 We say an element $s \in \O R$ is \emph{saturated} if $\mu_\times(s) \le s \oplus s$.
\end{definition}
Note that $s \oplus s = \iota_1(s) \wedge \iota_2(s)$, where $\iota_{1,2}\colon \O R \to \O R \oplus \O R$ are the coproduct injections,
and so this is simply saying that $xy \in s \vdash_{x,y} (x \in s) \land (y \in s)$ holds in the internal logic.

Since $\mu_\times$ is commutative, it is enough to require $\mu_\times(s) \le \iota_1(s)$. From this it is easy to see that the saturated elements form a subframe of $\O R$.
The topological meaning of this frame will be discussed in \cref{section:spectrum_via_duals}.
\begin{definition}
 Let $R$ be a localic semiring. We write $\Sats R$\glsadd{Sats} for the frame of saturated elements\index{frame of saturated opens|textbf} of $\O R$.
\end{definition}

Since $R$ is overt, $\iota_1$ has a left adjoint and
so an open $a$ is saturated if and only if $(\iota_1)_!\mu_\times(a) \le a$. Here $(\iota_1)_!\mu_\times(a)$ is called the \emph{(open) saturation}\index{saturation|(textbf} of $a$.
In the internal logic, it can be expressed as $\{x \colon R \mid \exists y\colon R.\ xy \in a\}$.

\begin{lemma}\label{prop:saturation_operator}
 Let $R$ be an overt localic semiring. The map $(\iota_1)_! \mu_\times\colon \O R \to \O R$ is a closure operator whose fixed points are precisely the saturated elements\index{saturated!open|)}.
\end{lemma}
\begin{proof}
 We have explained above that $(\iota_1)_!\mu_\times(a) \le a \iff a \text{ is saturated}$. It remains to show that $(\iota_1)_!\mu_\times$ is a closure operator.
 
 It is clearly monotone. To show it is inflationary we need $a \le (\iota_1)_! \mu_\times(a)$, which in the internal logic means $x \in a \vdash_x \exists y\colon R.\ xy \in a$.
 But this holds for $y = 1$, since $x \cdot 1 =_{x} x$ by the axioms of a semiring. The desired result follows by \cref{prop:right_existential_rule}.
 
 Finally, we prove idempotence. We require $(\iota_1)_! \mu_\times (\iota_1)_! \mu_\times(a) \le (\iota_1)_! \mu_\times(a)$.
 Let us translate this into the internal logic.
 The open $(\iota_1)_! \mu_\times(a)$ can be described by the formula $\exists y'.\ x'y' \in a$ in the context $x' \colon R$.
 Then $\mu_\times (\iota_1)_! \mu_\times(a)$ corresponds to substituting $xz$ for $x'$ to give the formula $\exists y'.\ (xz)y' \in a$ in the context $x\colon R, z\colon R$.
 Finally, $(\iota_1)_! \mu_\times (\iota_1)_! \mu_\times(a)$ is obtained by existentially quantifying over $z$ to give the formula $\exists z.\ \exists y'.\ (xz)y' \in a$.
 Thus, we must prove the sequent $\exists z.\ \exists y'.\ (xz)y' \in a \vdash_x \exists y.\ xy \in a$. This is proved by taking $y = zy'$ and using associativity.
 An explicit proof tree (using \cref{prop:sub_out_equal_terms}) is given below.
 
 \begin{prooftree}
  \AxiomC{$(xz)y' =_{x,y,z} x(zy')$}
  \AxiomC{}
  \UnaryInfC{$x(zy') \in a \vdash_{x,y',z} x(zy') \in a$}
  \UnaryInfC{$x(zy') \in a \vdash_{x,y',z} \exists y.\ xy \in a$}
  \UnaryInfC{$\exists z.\ \exists y'.\ x(zy') \in a \vdash_x \exists y.\ xy \in a$}
  \BinaryInfC{$\exists z.\ \exists y'.\ (xz)y' \in a \vdash_x \exists y.\ xy \in a$}
 \end{prooftree}
 Thus, $(\iota_1)_!\mu_\times$ is idempotent as required.
\end{proof}

\begin{corollary}
 The saturated opens in an overt localic semiring are closed under arbitrary meets and joins. %
\end{corollary}

Saturated opens are also related to overt weakly closed monoid ideals.

\begin{lemma}\label{prop:overt_meeting_saturated}
 Suppose $R$ is an overt localic semiring. Take $V \in \SubOW(R)$ and $a \in \O R$. Then $V \cdot \top \between a \iff V \between (\iota_1)_!\mu_\times(a)$.
\end{lemma}
\begin{proof}
 By definition we have $\exists_{V \cdot \top} = \nabla (\exists_V \otimes \exists_\top) \mu_\times$. The map $\exists_\top$ is simply $\exists$, the left adjoint to ${!}\colon \Omega \to \O R$.
 Note the coproduct injection $\iota_1$ is equal to $\O R \oplus {!}$ up to composition with the isomorphism $\O R \cong \O R \oplus \Omega$.
 Taking left adjoints in $\Sup$ then shows $(\iota_1)_! = \O R \otimes \exists$, %
 so that $\exists_{V \cdot \top} = \nabla (\exists_V \otimes \Omega) (\O R \otimes \exists) \mu_\times = \exists_V (\iota_1)_! \mu_\times$.
\end{proof}

Now we are in a position to discuss the universal element.
The element $\upsilon \in \Idl(R) \oplus \O R$ should be thought of informally as the region containing the pairs $(x,f)$ such that $f(x)$ is cozero.
(Here we are using spatial language and viewing points of $R$ as functions on the spectrum.)

From the intuition behind our definition of $\Rad(R)$, the (overt weakly closed) ideal $(f)$ generated by $f$ should be thought of as the region on which $f$ is cozero.
(We do not take the radical ideal in order to be sensitive to the additional infinitesimal information captured in $\Idl(R)$.)
We might attempt to construct $\upsilon$ as a join $\bigcup_{f \in \abs{R}} (f) \times \{f\}$. Of course, we should not rely on the points $f \in R$.
Moreover, $\{f\}$ is not open in general and we would like to express $\upsilon$ as a join of `basic open rectangles' in the product.

Now notice that if a function $f(x)$ is cozero and $f = gh$, then $(gh)(x) = g(x)h(x)$ is also cozero. %
So if $(f,x) \in \upsilon$, then we also have $(g,x) \in \upsilon$ for every function $g$ in the set-theoretic saturation $\sat(f)$. %
This suggests building $\upsilon$ as $\bigcup_{f \in \abs{R}} (f) \times \sat(f)$. Our two previous objections still apply, but the larger set $\sat(f)$ gives us some more room to manoeuvre.
While $\sat(f)$ need not be open itself, the saturation map $(\iota_1)_!\mu_\times$ of \cref{prop:saturation_operator} ensures we have a wealth of saturated opens at our disposal.

To remove the dependence on points, we replace the functions $f$ with more general overt sublocales $V$, thought of here as subsets of $R$.
The ideal generated by a subset $V$ intuitively represents the region on which \emph{some} element of $V$ is cozero. Since we have potentially enlarged
the first factor of $(f) \times \sat(f)$ to $(V)$, we should similarly reduce the second factor to the \emph{intersection} $\bigcap_{f \in V} \sat(f)$,
for we cannot know beforehand which of the functions are cozero at a given point in the region in question, so we cut down to what the different sets $\sat(f)$ have in common.
We are left with a description for $\upsilon$ as $\bigcup_{V \subseteq \abs{R}} (V) \times \sat(V)$, writing $\sat(V)$ for the intersection in question.

The above expression for $\upsilon$ could easily be made pointfree were it not for the fact that $\sat(V)$ is now defined in terms of points.
Furthermore, we have not yet addressed the fact that $\sat(V)$ is not necessarily open.
But note that $\sat(V)$ can be equivalently defined as $\bigcap\{ S \text{ saturated} \mid \exists f \in V \cap S\}$.
The two issues can now be addressed at once with the following definition.

\begin{definition}
 We define the function $\stabpos\colon \SubOW(R) \to \O R$\glsadd{stabpos} by \[\stabpos(V) = \bigwedge\left\{ t \in \Sats R \mid V \between t \right\}.\]
 Note that, at least when $R$ is overt, $\stabpos(V)$ is saturated, but there is no guarantee it meets $V$.
\end{definition}

This definition gives a pointfree replacement for $\sat(V)$, where we have replaced the saturated sets with saturated opens\index{saturation|)}. Also note that the intersection is now taken in the lattice of opens,
so it is more like the interior of the corresponding set-theoretic intersection.

This suggests setting $\upsilon$ to be $\bigvee_{V \in \SubOW(R)} \addquot(V) \oplus \stabpos(V)$.
At this point we might worry that the passage from general saturated sets to saturated opens might lead to this join between smaller than necessary.
Indeed, since it appears that spectra might fail to exist (see \cref{ex:no_spectrum}),
it stands to reason that we would need some additional assumption that will ensure that the saturated opens $s$ satisfying $I \oplus s \le \upsilon$ can be well approximated by elements of the form $\stabpos(V)$.

\begin{definition}
 We say a localic semiring $R$ is \emph{approximable}\index{approximable|(textbf}\index{semiring!approximable localic|(textbf} if
 \[s = \bigvee_{V \between s} \stabpos(V)\]
 for every saturated element $s \in \Sats R$.
\end{definition}
This is the pointfree version of the equality $S = \bigcup_{f \in S} \sat(f)$, which holds for discrete semirings.
Note that the `$\ge$' inequality always holds, so to show a localic semiring is approximable\index{approximable|)}\index{semiring!approximable localic|)} we need only show $s \le \bigvee_{V \between s} \stabpos(V)$.
Furthermore, we need only check this on a base of saturated elements.

\begin{lemma}\label{prop:saturation_from_join_over_stabpos}
 If $R$ is overt and approximable, $(\iota_1)_!\mu_\times(a) = \bigvee_{V \between a} \stabpos(V)$ for all $a \in \O R$.
\end{lemma}
\begin{proof}
 Recall that $V \between (\iota_1)_!\mu_\times(a) \iff V\cdot\top \between a$ by \cref{prop:overt_meeting_saturated}. It follows that $\stabpos(V) = \stabpos(V \cdot \top)$ and hence
 the join $\bigvee_{V \between s} \stabpos(V)$ may be restricted to monoid ideals.
 Applying \cref{prop:overt_meeting_saturated} again then yields the desired result.
\end{proof}

It will be useful to understand $\stabpos$ and approximability in more detail.
We first observe that the way in which the approximability condition has an open being approximated from below is vaguely reminiscent of continuity or supercontinuity.
We can make the analogy stronger by defining a relation $\blacktriangleleft$\glsadd{blacktriangleleft} on $\Sats R$ by
\begin{align*}
 t \blacktriangleleft s &\iff \exists K \in \SubOW(R).\ K \between s \text{ and } t \le \stabpos(K) \\
                        &\iff \exists K \between s.\ \forall u \text{ saturated}.\ K \between u \implies t \le u. %
\end{align*}
The following reformulation of approximability is then immediate.
\begin{lemma}
 A localic semiring $R$ is approximable if and only if \[s = \bigvee_{t \blacktriangleleft s} t\] for all $s \in \Sats R$.
\end{lemma}

It will sometimes be useful to extend the $\blacktriangleleft$ relation to general elements of $\O R$.
We set $b \blacktriangleleft a \iff b \le a \text{ and } (\iota_1)_!\mu_\times(b) \blacktriangleleft (\iota_1)_!\mu_\times(a)$.
Equivalently, $b \blacktriangleleft a \iff b \le a \text{ and } \exists K \in \MonIdl R.\ K \between a \text{ and } b \le \stabpos(K)$.
Then if $R$ is approximable, we have $a = \bigvee_{b \blacktriangleleft a} b$ for all $a \in \O R$.

\begin{proposition} \label{prop:approximable_supercontinuous}
 An overt localic semiring $R$ is approximable\index{approximable|textbf}\index{semiring!approximable localic|textbf} if and only if its frame $\Sats R$ of saturated elements\index{frame of saturated opens} is supercontinuous.
 Furthermore, under these conditions $t \blacktriangleleft s \iff t \lll s$ for all $s, t \in \Sats R$.
\end{proposition} %
\begin{proof}
 Let us first show $t \blacktriangleleft s \implies t \lll s$. Suppose $t \blacktriangleleft s$ and $\bigvee A \ge s$ for some $A \subseteq \Sats R$.
 Then there is a $K \in \SubOW(R)$ such that $t \le \stabpos(K)$ and $K \between s \le \bigvee A$. So there is an $a \in A$ such that $K \between a$.
 Thus, $t \le \stabpos(K) \le a$ and we have shown $t \lll s$.
 Consequently, if $R$ is approximable, then $\Sats R$ is supercontinuous.
 
 Conversely, suppose $\Sats R$ is supercontinuous and $t \lll s$. Then $\llbracket t \lll (-) \rrbracket \colon \Sats R \to \Omega$ is a suplattice homomorphism
 and hence the composite $\llbracket t \lll (-) \rrbracket \circ (\iota_1)_! \mu_\times$ corresponds to an overt weakly closed sublocale $K \in \SubOW(R)$.
 Note that if $u$ is saturated, then $K \between u \iff t \lll u$.
 Since, $t \lll s$, we have $K \between s$. If  $u \in \Sats R$ and $K \between u$, then $t \lll u$ and in particular, $t \le u$.
 This means that $t \blacktriangleleft s$. Therefore, $R$ is approximable.
\end{proof}

Notice that $\stabpos$ sends joins to meets and thus forms part of an order-reversing Galois connection.
We show that under a mild assumption on $R$ the adjoint of $\stabpos$ can be described as follows.

\begin{definition}
 Define $\jointcoz\colon \O R \to \SubOW(R)$\glsadd{jointcoz} by \[\jointcoz(a) = \bigwedge\left\{ J \in \MonIdl R \mid J \between a \right\}.\]
 Observe that $\jointcoz(a)$ is a monoid ideal, though it need not meet $a$. %
 Furthermore, note that $\jointcoz = \jointcoz(\iota_1)_!\mu_\times$ by \cref{prop:overt_meeting_saturated}.
\end{definition}

\begin{definition}
 We say \emph{overt sublocales distinguish saturated opens} in a localic semiring $R$ if for all $s, t \in \Sats R$,
 we have $s \le t$ whenever $V \between s \implies V \between t$ for all $V \in \SubOW(R)$. (By \cref{prop:overt_meeting_saturated}, it will be enough to check this for all monoid ideals $V$.)
\end{definition}
This is a very weak condition. It is implied by approximability and by spatiality and under the assumption of excluded middle it holds for all localic semirings.
This condition is very similar to what is called \emph{strong overtness} in \cite{CirauloStronglyOvert}, but where we restrict $s$ and $t$ to be saturated opens, the authors require the claim to hold for all opens.

\begin{lemma}\label{prop:jointcoz_stabpos_Galois_connection}
 If $R$ is an overt localic semiring in which overt sublocales distinguish saturated opens, then
 $\stabpos$ and $\jointcoz$ form an order-reversing Galois connection.
\end{lemma}
\begin{proof}
 We first show $V \le \jointcoz(\stabpos(V))$.
 Suppose $J \between \stabpos(V)$ for a monoid ideal $J$. If $t$ is a saturated element such that $V \between t$, then $\stabpos(V) \le t$ and so $J \between t$.
 Thus, $\exists_V(t) \le \exists_J(t)$ for all saturated $t$ and hence for all $t$, yielding $V \le J$. Since this holds for all such $J$, we can conclude that $V \le \jointcoz(\stabpos(V))$.
 
 Now we show $a \le \stabpos(\jointcoz(a))$. Since $a \le (\iota_1)_! \mu_\times(a)$ and $\jointcoz(a) = \jointcoz (\iota_1)_! \mu_\times(a)$, we may assume $a$ is saturated without loss of generality.
 Suppose $t$ is a saturated element satisfying $\jointcoz(a) \between t$. If $J$ is a monoid ideal such that $J \between a$, then $\jointcoz(a) \le J$
 and so $J \between t$. Thus, $J \between a \implies J \between t$ for all monoidal ideals $J$ and therefore $a \le t$, since overt sublocales distinguish saturated opens. The claim follows.
\end{proof}

\begin{corollary}\label{prop:blacktriangle_vs_jointcoz}
 Under the same assumptions, $t \blacktriangleleft s \iff \jointcoz(t) \between s$.
\end{corollary}
\begin{proof}
 Recall that $t \blacktriangleleft s \iff \exists K \in \SubOW(R).\ K \between s \text{ and } t \le \stabpos(K)$.
 But by the result above $t \le \stabpos(K) \iff K \le \jointcoz(t)$. Thus, if any overt weakly closed sublocale meets $s$, so will $\jointcoz(t)$
 and the result follows.
\end{proof}

The approximability condition implies a dual result for the suplattice of overt weakly closed sublocales.

\begin{lemma}\label{prop:approximable_implies_ideal_approximating}
 Suppose $R$ is an overt approximable localic semiring. Then \[V \cdot \top = \bigvee_{a \between V} \jointcoz(a)\] for all $V \in \SubOW(R)$.
\end{lemma}
\begin{proof}
 The `$\ge$' inequality always holds. For the `$\le$' direction, we may assume $V$ is a monoid ideal by familiar arguments using $\jointcoz(a) = \jointcoz(\iota_1)_!\mu_\times(a)$
 and \cref{prop:overt_meeting_saturated}.
 
 Let $s$ be a saturated element and consider the join $s = \bigvee_{t \blacktriangleleft s} t$.
 Then $V \between s$ implies $V \between t$ for some $t \blacktriangleleft s$. In other words, $V \between t$ and $\jointcoz(t) \between s$.
 But then $\left( \bigvee_{t \between V} \jointcoz(t) \right) \between s$ and since monoid ideals are determined by the saturated elements
 they meet, we may conclude $V \le \bigvee_{a \between V} \jointcoz(a)$ as required.
\end{proof}

As above, we can equivalently express this condition on ideals in terms of a relation $\vartriangleleft$\glsadd{triangleleft} on $\MonIdl R$ defined by
\begin{align*}
 W \vartriangleleft V &\iff \exists a \in \O R.\ V \between a \text{ and } W \le \jointcoz(a) \\
                        &\iff \exists a \between V.\ \forall J \in \MonIdl R.\ J \between a \implies W \le J.
\end{align*}
Moreover, under the conditions of \cref{prop:jointcoz_stabpos_Galois_connection}, we have $W \!\vartriangleleft V$
if and only if $V \between \stabpos(W)$.

The conclusion of \cref{prop:approximable_implies_ideal_approximating} is equivalent to the claim that
\[V = \bigvee_{W \vartriangleleft V} W\] for all monoid ideals $V$.

In fact, we can now prove a partial converse to \cref{prop:approximable_implies_ideal_approximating}.
\begin{lemma}
 Suppose $R$ is an overt localic semiring in which overt sublocales distinguish saturated opens. Furthermore, assume that
 $V \le \bigvee_{a \between V} \jointcoz(a)$ for all $V \in \MonIdl R$.
 Then $R$ is approximable.
\end{lemma}
\begin{proof}
 Suppose $s$ is a saturated element. We must show $s \le \bigvee_{W \between s} \stabpos(W)$.
 Suppose $V$ is a monoid ideal such that $V \between s$ and write $V = \bigvee_{W \vartriangleleft V} W$.
 Then $W \between s$ for some $W$ and from $W \vartriangleleft V$, we know $V \between \stabpos(W)$.
 Thus, $V \between \bigvee_{W \between s} \stabpos(W)$ and the result follows since overt sublocales distinguish saturated opens.
\end{proof}

We can also use the $\vartriangleleft$ relation to give some insight into the structure of $\Idl(R)$.
\begin{lemma}
 If $R$ is an overt approximable localic semiring, then $\Idl(R)$ is a continuous lattice.
\end{lemma}
\begin{proof}
 The proof proceeds as in \cref{prop:approximable_supercontinuous}. We show that if $V,W$ are monoid ideals, then $W \vartriangleleft V \implies W \lll V$ in $\MonIdl R$.
 Suppose $W \vartriangleleft V$ and $\bigvee A \ge V$. Then $V \between \stabpos(W)$ and so $J \between \stabpos(W)$ for some $J \in A$. %
 Thus, $W \vartriangleleft J$ and so $W \le J$. This proves the claim and so $\MonIdl R$ is supercontinuous.
 
 In particular, $\MonIdl R$ is continuous. But directed joins in $\Idl(R)$ are computed as in $\MonIdl R$ and so $\Idl(R)$ is also continuous.
\end{proof}
\begin{remark}
 I do not know if $\Rad(R)$ is always continuous, but the above lemma shows this is the case if the right adjoint to the localic quotient $\rad\colon \Idl(R) \twoheadrightarrow \Rad(R)$ preserves directed joins.
 This is true for our core examples.
\end{remark}

We now know enough to tackle the problem of constructing the universal (quantic) open prime anti-ideal.
We first note that under the assumption of approximability, a number of different definitions for $\upsilon$ coincide.

\begin{lemma}\label{prop:different_definitions_of_upsilon_agree}
 If $R$ is an overt approximable localic semiring, then %
 \[\smashoperator{\bigvee_{V \in \SubOW(R)}}\; \addquot(V) \oplus \stabpos(V) = \;\smashoperator{\bigvee_{a \in \O R}}\; \addquot\jointcoz(a) \oplus a =
 \bigvee\left\{ \addquot\jointcoz(a) \oplus \stabpos(V) \mid V \between a\right\}.\]
 We define $\upsilon \in \Idl(R) \oplus \O R$\glsadd[format=(]{upsilon} to be their common value.
\end{lemma}

\begin{proof}
 Expanding the first $V$ in $\bigvee_V \addquot(V) \oplus \stabpos(V)$ using \cref{prop:approximable_implies_ideal_approximating} and
 the second $a$ in $\bigvee_{a \in \O R} \addquot\jointcoz(a) \oplus a$ by the definition of approximability, we arrive at
 the third join
 $\bigvee\left\{ \addquot\jointcoz(a) \oplus \stabpos(V) \mid V \between a\right\}$ in both cases.
\end{proof}

We must show that $\upsilon$ so defined truly is an open prime anti-ideal.

\begin{proposition}
 If $R$ is an overt approximable localic semiring, then $\upsilon$ is an open prime anti-ideal\index{anti-ideal!open prime} fibred over $\Idl(R)$.
\end{proposition}
\begin{proof}
 We must show that $\upsilon$ satisfies the four axioms of an element of $\overline{\OPAI}_R(\Idl(R))$.
 
 \emph{\underline{Zero axiom}\ }
 First we show $(\Idl(R) \oplus \epsilon_0)(\upsilon) = 0$. Expanding the second definition of $\upsilon$ we have
 $(\Idl(R) \oplus \epsilon_0)(\upsilon) = \bigvee_{a} \addquot\jointcoz(a) \oplus \epsilon_0(a) = \bigvee\{\addquot\jointcoz(a) \mid \epsilon_0(a) = 1\}$.
 So the result will follow if we assume $\epsilon_0(a) = 1$ and show $\addquot\jointcoz(a) = 0$.
 But $0_+ = \epsilon_0$ and so if $\epsilon_0(a) = 1$, then $\jointcoz(a) \le 0_+ \cdot \top = 0_+$ and hence $\addquot\jointcoz(a) \le \addquot(0_+) = 0$. %
 
 \emph{\underline{Unit axiom}\ }
 Next we show $(\Idl(R) \oplus \epsilon_1)(\upsilon) = 1$. By \cref{prop:approximable_implies_ideal_approximating} we have
 $1_\times \cdot \top = \bigvee_{a \between 1_\times} \jointcoz(a) = \bigvee_{a} \jointcoz(a) \cdot \epsilon_1(a)$. The unit axiom follows by applying $\addquot$ to both sides.
 
 \emph{\underline{Additive axiom}\ }
 To see that $(\Idl(R) \oplus \mu_+)(\upsilon) \le (\Idl(R) \oplus \iota_1)(\upsilon) \vee (\Idl(R) \oplus\iota_2)(\upsilon)$,
 take $a \in \O R$ and suppose $r \oplus r' \le \mu_+(a)$. We must show $\addquot\jointcoz(a) \oplus r \oplus r' \le \bigvee_{b \in \O R} \addquot\jointcoz(b) \oplus (b \oplus 1 \vee 1 \oplus b)$.
 
 Write $r = \bigvee_{u \blacktriangleleft r} u$ and $r' = \bigvee_{u' \blacktriangleleft r'} u'$.
 It is enough to show $\jointcoz(a) \le (\jointcoz(u) + \jointcoz(u'))\cdot \top$ so that %
 $\addquot\jointcoz(a) \oplus u \oplus u' \le (\addquot\jointcoz(u) \vee \addquot\jointcoz(u')) \oplus u \oplus u' \le \addquot\jointcoz(u) \oplus u \oplus 1 \vee \addquot\jointcoz(u') \oplus 1 \oplus u'
 \le \bigvee_{b \in \O R} \addquot\jointcoz(b) \oplus (b \oplus 1 \vee 1 \oplus b)$,
 which then implies the desired result by taking the join over $u$ and $u'$.

 Take $u \blacktriangleleft r$ and  $u' \blacktriangleleft r'$, so that $\jointcoz(u) \between r$ and $\jointcoz(u') \between r'$.
 Then $1 = \exists_{\jointcoz(u)}(r) \wedge \exists_{\jointcoz(u')}(r') = \nabla(\exists_{\jointcoz(u)} \otimes \exists_{\jointcoz(u')})(r \otimes r') \le
 \nabla(\exists_{\jointcoz(u)} \otimes \exists_{\jointcoz(u')})\mu_+(a)$.
 But $\nabla(\exists_{\jointcoz(u)} \otimes \exists_{\jointcoz(u')})\mu_+ = \exists_{\jointcoz(u) + \jointcoz(u')}$ and so $(\jointcoz(u) + \jointcoz(u')) \between a$.
 Hence, $(\jointcoz(u) + \jointcoz(u'))\cdot \top \between a$ and $\jointcoz(a) \le (\jointcoz(u) + \jointcoz(u'))\cdot \top$ as required.
 
 \emph{\underline{Multiplicative axiom}\ } %
 Now we show the remaining equality $(\Idl(R) \oplus \mu_\times)(\upsilon) = (\Idl(R) \oplus\iota_1)(\upsilon) \cdot (\Idl(R) \oplus\iota_2)(\upsilon)$.
 Let us start with the `$\le$' direction. We need $\addquot(V) \oplus \mu_\times \stabpos(V) \le \bigvee_{a,b \in \O R} \addquot(\jointcoz(a) \jointcoz(b)) \oplus a \oplus b$.
 Suppose $d \oplus d' \le \mu_\times \stabpos(I)$ and take $e \blacktriangleleft d$ and $e' \blacktriangleleft d'$.
 The result will follow if we can show $V \le \jointcoz(e)\jointcoz(e')$, since then
 $V \oplus e \oplus e' \le \jointcoz(e)\jointcoz(e') \oplus e \oplus e' \le \bigvee_{a, b} \jointcoz(a)\jointcoz(b) \oplus a \oplus b$.
 
 Let $s$ be a saturated element such that $V \between s$. We must show $\jointcoz(e)\jointcoz(e') \between s$.
 Note that $\jointcoz(e)\jointcoz(e') \between s \iff \nabla(\exists_{\jointcoz(e)} \otimes \exists_{\jointcoz(e')})\mu_\times(s) = 1$.
 Now $V \between s$ implies $\stabpos(V) \le s$ and hence $d \oplus d' \le \mu_\times \stabpos(V) \le \mu_\times(s)$.
 Therefore, $\nabla(\exists_{\jointcoz(e)} \otimes \exists_{\jointcoz(e')})\mu_\times(s) \ge \exists_{\jointcoz(e)}(d) \wedge \exists_{\jointcoz(e')}(d')$.
 But since $e \blacktriangleleft d$ and $e' \blacktriangleleft d'$, we know $\jointcoz(e) \between d$ and $\jointcoz(e') \between d'$.
 Thus, the right-hand side is $1$ and the result follows.
 
 To prove the other direction, $(\Idl(R) \oplus\iota_1)(\upsilon) \cdot (\Idl(R) \oplus\iota_2)(\upsilon) \le (\Idl(R) \oplus \mu_\times)(\upsilon)$,
 we will show $IJ \oplus \stabpos(I) \oplus \stabpos(J) \le \bigvee_{K \in \SubOW(R)} K \oplus \mu_\times \stabpos(K)$ for $I, J \in \MonIdl R$.
 Let $K \vartriangleleft IJ$. If we can show $\stabpos(I) \oplus \stabpos(J) \le \mu_\times \stabpos(K)$, then the desired inequality follows by applying $K \oplus (-)$ and taking the join of over all such $K$.

 Since $K \vartriangleleft IJ$, we have $IJ \between \stabpos(K)$. That is, $\nabla (\exists_I \otimes \exists_J)\mu_\times\stabpos(K) = 1$.
 
 As an aside, we note that if $t$ is saturated then $\mu_\times(t)$ can be expressed as a join $\bigvee_\alpha s_\alpha \oplus s'_\alpha$ where each $s_\alpha$ and $s'_\alpha$ is saturated.
 To see this we show that if $s \otimes s' \le \mu_\times(t)$, then
 $(\iota_1)_!\mu_\times(s) \otimes (\iota_1)_!\mu_\times(s') = \iota_1(\iota_1)_!\mu_\times(s) \wedge \iota_2(\iota_1)_!\mu_\times(s') \le \mu_\times(t)$.
 In the internal logic, we have $(x \in s) \land (x' \in s') \vdash_{x,x'} xx' \in t$ and want to show $\exists y.\ xy \in s \land \exists y'.\ x'y' \in s' \vdash_{x,x'} xx' \in t$. %
 Well if $\exists y.\ xy \in s \land \exists y'.\ x'y' \in s'$, then $(xy)(x'y') \in t$. But $(xy)(x'y') = (xx')(yy')$ using associativity and commutativity and so $(xx')(yy') \in t$.
 But this means $xx' \in t$, since $t$ is saturated.
 
 So we can write $\mu_\times\stabpos(K) = \bigvee_\alpha s_\alpha \oplus s'_\alpha$ with $s_\alpha, s'_\alpha$ saturated.
 So $1 = \nabla (\exists_I \otimes \exists_J)\mu_\times\stabpos(K) = \bigvee_\alpha \exists_I(s_\alpha) \wedge \exists_J(s'_\alpha)$
 and thus there is an $\alpha$ such that $I \between s_\alpha$ and $J \between s'_\alpha$. Consequently, $\stabpos(I) \le s_\alpha$ and $\stabpos(J) \le s'_\alpha$
 and therefore $\stabpos(I) \oplus \stabpos(J) \le s_\alpha \oplus s'_\alpha \le \mu_\times(K)$. This completes the proof.
\end{proof}

We can now prove our main result.
\begin{theorem}\label{prop:quantic_spectrum_of_approximable_semiring}
 If $R$ is an overt approximable\index{approximable}\index{semiring!approximable localic} localic semiring,
 then the quantic spectrum\index{spectrum!quantic} of $R$ is given by $\Idl(R)$ equipped with the open prime anti-ideal
 \[\upsilon = \bigvee\nolimits_{V} \addquot(V) \oplus \stabpos(V) = \bigvee\nolimits_{a} \addquot\jointcoz(a) \oplus a =
 \bigvee\left\{ \addquot\jointcoz(a) \oplus \stabpos(V) \mid V \between a\right\}.\glsadd[format=)]{upsilon}\]
\end{theorem}
\begin{proof}
 Suppose $p \in Q \oplus \O R$ is an open prime anti-ideal fibred over $Q$. We must show that there is a unique $\overline{p}\colon \Idl(R) \to Q$ such that $p = (\overline{p} \oplus \O R)(\upsilon)$.
 
 Assume $f\colon \Idl(R) \to Q$ is any quantale homomorphism satisfying $p = (f \oplus \O R)(\upsilon)$. Take $I \in \Idl(R)$. By \cref{prop:approximable_implies_ideal_approximating} we have
 $I = (\Idl(R) \otimes \exists_I)(\upsilon)$. But then $f(I) = (f \otimes \Omega)(\Idl(R) \otimes \exists_I)(\upsilon) = (Q \otimes \exists_I)(f \otimes \O R)(\upsilon) = (Q \otimes \exists_I)(p)$.
 So the only choice for $\overline{p}$ is given by $\overline{p}(I) = (Q \otimes \exists_I)(p)$.
 This definition accords well with the intuition that a point lies in the open defined by an ideal when there exists a function in the ideal which is cozero at that point.
 
 It is not obvious that this definition gives a quantale homomorphism.
 Since $p \in \overline{\OPAI}_R(Q)$, we have $(Q \oplus \epsilon_0)(p) = 0$ and $(Q \oplus \epsilon_1)(p) = 1$.
 If $I$ is the zero ideal, then $\exists_I = \epsilon_0$ so that $\overline{p}(0) = Q \oplus \epsilon_0(p) = 0$ and $\overline{p}$ preserves $0$. %
 If $I$ is the unit ideal, then $\epsilon_1 \le \exists_I$ so that $\overline{p}(1) \ge Q \oplus \epsilon_1(p) = 1$ and $\overline{p}$ preserves $1$.
 
 Now we show that $\overline{p}$ preserves binary joins. It is obviously monotone and so it is enough to show $\overline{p}(I \vee J) \le \overline{p}(I) \vee \overline{p}(J)$.
 Recall that $I \vee J$ in $\Idl(R)$ is equal to $I + J$ in $\SubOW(R)$ %
 and so we need to show $(Q \otimes \exists_{I + J})(p) \le (Q \otimes \exists_I)(p) \vee (Q \otimes \exists_J)(p)$.
 
 We have $(Q \otimes \exists_{I + J})(p) = (Q \otimes \nabla(\exists_I \otimes \exists_J)\mu_+)(p) \le (Q \otimes \nabla(\exists_I \otimes \exists_J))((Q \oplus \iota_1)(p) \vee (Q \oplus \iota_2)(p))$,
 where the inequality follows from the additive condition on elements of $\overline{\OPAI}_R(Q)$. Now we can use the fact that $\nabla(\exists_I \otimes \exists_J)$ preserves joins
 (as well as $\exists_I(1) \ge \epsilon_0(1) = 1$) to conclude that $(Q \otimes \exists_{I + J})(p) \le (Q \otimes \exists_I)(p) \vee (Q \otimes \exists_J)(p)$ as required.
 
 We note that $\overline{p}$ also preserves directed joins, since directed joins in $\Idl(R)$ coincide with those in $\SubOW(R)$ and joins of suplattice homomorphisms are computed pointwise.
 Thus, $\overline{p}$ preserves arbitrary joins.
 
 Before we show $\overline{p}$ preserves multiplication, it will be helpful to prove an ancillary result.
 We show $(Q \otimes \exists_{W})(p) = (Q \otimes \exists_{\addquot(W)})(p)$ for all monoid ideals $W$.
 The `$\le$' direction is clear, since $W \le \addquot_*\addquot(W)$.
 For the other direction, we recall that $\addquot_*\addquot(W) = \bigvee_{n \in \N} n\cdot W$. %
 So we must show $(Q \otimes \exists_{n\cdot W})(p) \le (Q \otimes \exists_{W})(p)$. We prove this by induction.
 For $n = 0$, this holds because $0_+ = \epsilon_0$
 and $(Q \otimes \epsilon_0)(p) = 0$ by the zero axiom of $p$.
 Now suppose $n = k+1$ and $p = \bigvee_\alpha x_\alpha \oplus r_\alpha$. We apply the additive axiom of $p$ to obtain
 \begin{align*}
  (Q \otimes \exists_{n\cdot W})(p) &= (Q \otimes \exists_{k\cdot W + W})(p) \\
                               &= (Q \otimes \nabla (\exists_{k\cdot W} \otimes \exists_{W})\mu_+)(p) \\
                               &\le (Q \otimes \nabla (\exists_{k\cdot W} \otimes \exists_{W}))\!\left(\bigvee_\alpha x_\alpha \oplus r_\alpha \oplus 1 \vee \bigvee_\alpha x_\alpha \oplus 1 \oplus r_\alpha\right) \\
                               &\le \bigvee_\alpha x_\alpha \otimes \exists_{k\cdot W}(r_\alpha) \vee \bigvee_\alpha x_\alpha \otimes \exists_{W}(r_\alpha) \\ %
                               &= (Q \otimes \exists_{k\cdot W})(p) \vee (Q \otimes \exists_{W})(p) \\
                               &= (Q \otimes \exists_{W})(p),
 \end{align*}
 where the final equality follows from the inductive hypothesis.
 
 Finally, we show $\overline{p}$ preserves multiplication. Let $I\circ J$ denote the product of the ideals $I$ and $J$ as elements of $\MonIdl R$ (or equivalently $\SubOW(R)$).
 By the above, we have $(Q \otimes \exists_{IJ})(p) = (Q \otimes \exists_{I\circ J})(p)$ and so we may conclude that $\overline{p}$ preserves binary products once we show that
 $(Q \otimes \exists_{I \circ J})(p) = (Q \otimes \exists_I)(p) \cdot (Q \otimes \exists_J)(p)$.
 
 By definition, $(Q \otimes \exists_{I \circ J})(p) = (Q \otimes \nabla (\exists_I \otimes \exists_J) \mu_\times)(p)$. But
 $p \in \overline{\OPAI}_R(Q)$ and so $(Q \oplus \mu_\times)(p) = (Q \oplus \iota_1)(p) \cdot (Q \oplus \iota_2)(p)$. Thus, if $p = \bigvee_\alpha x_\alpha \oplus r_\alpha$, we have
 \begin{align*}
  (Q \otimes \exists_{I \circ J})(p) &= (Q \otimes \nabla (\exists_I \otimes \exists_J))((Q \oplus \iota_1)(p) \cdot (Q \oplus \iota_2)(p)) \\
                                     &= (Q \otimes \nabla(\exists_I \otimes \exists_J))\!\left(\bigvee_{\alpha,\beta} x_\alpha x_\beta \oplus r_\alpha \oplus r_\beta\right) \\
                                     &= \bigvee_{\alpha,\beta} x_\alpha x_\beta \otimes (\exists_I(r_\alpha) \wedge \exists_J(r_\beta)) \\
                                     &= (Q \otimes \exists_I)(p) \cdot (Q \otimes \exists_J)(p),
 \end{align*}
 as required. Therefore, $\overline{p}$ is a quantale homomorphism.
 
 It remains to demonstrate that $p = (\overline{p} \oplus \O R)(\upsilon)$.  If $p = \bigvee_\alpha x_\alpha \oplus r_\alpha$, we have
 \begin{align*} %
  (\overline{p} \oplus \O R)(\upsilon) &= \bigvee_V (Q \otimes \exists_{\addquot(V)})(p) \oplus \stabpos(V) \\
                                    &= \bigvee_V (Q \otimes \exists_{V})(p) \oplus \stabpos(V) \\
                                    &= \bigvee_V \bigvee_\alpha x_\alpha \oplus \exists_{V}(r_\alpha) \oplus \stabpos(V) \\
                                    &= \bigvee_\alpha x_\alpha \oplus \bigvee_{V} \exists_{V}(r_\alpha) \cdot \stabpos(V) \\
                                    &= \bigvee_\alpha x_\alpha \oplus (\iota_1)_! \mu_\times(r_\alpha) \\
                                    &= (Q \otimes (\iota_1)_! \mu_\times)(p),
 \end{align*}
 where $V$ ranges over all monoid ideals and the penultimate equality comes from \cref{prop:saturation_from_join_over_stabpos}.
 Clearly, $p \le (Q \otimes (\iota_1)_! \mu_\times)(p)$. For the other direction we use that $p$ is an
 open prime anti-ideal so that $(Q \oplus \mu_\times)(p) = (Q \oplus \iota_1)(p) \cdot (Q \oplus \iota_2)(p) \le (Q \oplus \iota_1)(p)$. Applying $Q \otimes (\iota_1)_!$ we then have
 $(Q \otimes (\iota_1)_! \mu_\times)(p) \le (Q \otimes (\iota_1)_! \iota_1)(p) \le p$, since $(\iota_1)_! \iota_1 \le \mathrm{id}_R$.
 Thus, $(\overline{p} \oplus \O R)(\upsilon) = p$ as required.
\end{proof}

\begin{corollary}\label{prop:localic_spectrum_of_approximable_semiring}
 If $R$ is an overt approximable localic semiring, then the (localic) spectrum\index{spectrum!localic} of $R$ is given by $\Rad(R)$ equipped with the open prime anti-ideal
 \[(\rad \oplus \O R)(\upsilon) = \;\;\smashoperator{\bigvee_{V \in \SubOW(R)}}\; \rad\addquot(V) \oplus \stabpos(V) = \;\smashoperator{\bigvee_{a \in \O R}}\; \rad\addquot\jointcoz(a) \oplus a =
 \bigvee\left\{ \rad\addquot\jointcoz(a) \oplus \stabpos(V) \mid V \between a \right\},\]
 where $\rad\colon \Idl(R) \to \Rad(R)$ is the universal localic quotient.
\end{corollary}

The above results give sufficient conditions for the localic and quantic spectra to exist. It is natural to ask if they are necessary.
I do not know if the hypotheses of \cref{prop:quantic_spectrum_of_approximable_semiring} may be weakened, nor if the spectrum might ever exist but not be given by $\Idl(R)$.
However, do note that the definition of the quantale homomorphism $\overline{p}\colon \Idl(R) \to Q$ in the proof of \cref{prop:quantic_spectrum_of_approximable_semiring} does not require any hypotheses at all.
This gives a natural transformation from $\OPAI_R(-)$ to $\Hom(\Idl(R),-)$ and hence a morphism from $\Idl(R)$ to the quantic spectrum whenever it exists. So it is possible that $\Idl(R)$ might satisfy the same
condition as `weak moduli spaces' in algebraic geometry. If this is the case, then the quantic spectrum will always be of this form. This in turn might make us more confident that the theorem has a converse
(at least for overt semirings). %

On the other hand, the localic spectrum of any locally compact semiring can be found directly from the generalised presentation. There does not appear to be a good reason to expect every locally compact semiring
to be approximable, though we have not attempted to construct a specific counterexample. I do not know if the frame of the localic spectrum would still be given by $\Rad(R)$ in such a situation. %

\subsection{Special cases of the spectrum construction}

We can now apply \cref{prop:quantic_spectrum_of_approximable_semiring,prop:localic_spectrum_of_approximable_semiring} to describe the spectra of a number classes of localic semirings.

\begin{definition}
 Suppose $R$ is a \emph{topological} semiring and take $f \in R$. We define the \emph{open saturation} of $f$ to be $\osat(f) = \interior(\bigcap\{T \text{ saturated open} \mid f \in T\})$\glsadd{osat}.
\end{definition}

\begin{lemma}
 Suppose $R$ is a spatial localic semiring. Then $R$ satisfies the conditions of \cref{prop:quantic_spectrum_of_approximable_semiring} whenever $\forall f \in S.\ \exists g \in S.\ f \in \osat(g)$
 for every saturated open set $S$.
\end{lemma}
\begin{proof}
 Note that $R$ is overt since it is spatial. It is approximable since \[S \subseteq \bigcup_{g \in S} \osat(g) = \bigvee_{\widetilde{g} \between S} \stabpos(\widetilde{g}) \le \bigvee_{V \between S} \stabpos(V),\]
 where $\widetilde{g}$ is the weak closure of $\{g\}$.
\end{proof}

\begin{remark}
 Note that even if $R$ is spatial, there does not in general seem to be any reason to expect that its overt weakly closed ideals can be described constructively in terms of
 certain set-theoretic ideals on the underlying set of points. Instead, this is the case when $R$ is \emph{reducible} in the sense of \cite{CirauloStronglyOvert}. %
\end{remark}

\begin{corollary}
 Every discrete semiring $R$ is approximable and its localic spectrum coincides with the Zariski spectrum\index{spectrum!Zariski}. %
\end{corollary}
\begin{proof}
 It is clear that $R$ is approximable, since $f \in \osat(f)$. By \cref{prop:localic_spectrum_of_approximable_semiring}, the localic spectrum is given by $\Rad(R)$,
 which is the usual Zariski spectrum, since our definition of $\Rad(R)$ coincides with the usual one for discrete semirings.
\end{proof}

\begin{corollary}
 A continuous frame $L$ equipped with the Scott topology is an overt approximable localic semiring and its localic spectrum recovers the locale $X$ represented by $L$ (as in Hofmann--Lawson duality\index{spectrum!Hofmann-Lawson@Hofmann--Lawson}).
\end{corollary}
\begin{proof}
 Recall from \cref{section:local_compactness} that a continuous frame is exponentiable in $\Loc$ and $\O(\Srpnsk^X)$ is the frame of the Scott topology on $L$. Furthermore,
 $\Srpnsk^X$ is a spatial localic semiring. %
 In this proof we view $L$ both as the frame of opens of $X$ and as the spatial localic semiring $\Srpnsk^X$.
 
 Suppose $S \subseteq L$ is Scott-open and take $f \in S$. (Note that $S$ is an upset and hence automatically saturated.) Since $L$ is continuous, we have $f = \dirsup\{g \in L \mid g \ll f\}$.
 Now since $S$ is Scott-open, there is one such $g \in S$. Then $\twoheaduparrow g = \{h \in L \mid g \ll h\}$ is a Scott-open set containing $f$ and contained in every upset containing $g$ and hence $f \in \osat(g)$
 as required.
 
 The quantic spectrum is then given by the quantale of overt weakly closed ideals. By \cref{prop:scott_closed_weakly_closed} these are in bijection with the Scott-closed ideals,
 which are simply principal downsets and hence in turn in bijection with elements of $L$. This quantale is already a frame and so coincides with the localic spectrum.
\end{proof}

It remains to show that the Gelfand spectrum\index{spectrum!Gelfand|(} is a special case of our construction. For simplicity, we will show this for the locale of real-valued functions on a compact regular locale.
The complex case is essentially identical. It should also be possible to prove it directly from the axioms of a commutative localic C*-algebra\index{C* algebra@C*-algebra} as given in \cite{Henry2016}, but we content ourselves
with proving it for function spaces. This viewpoint has the advantage of being more amenable to generalisation.

The proof of this result is somewhat more involved than the other cases. In order to better exhibit the idea behind the proof, we will first prove it with classical axioms before providing the full
constructive result.

\begin{proposition}
 Suppose $X$ is a compact regular locale. Under the assumption of excluded middle and the axiom of choice, $A = \R^X$\index{ring of continuous functions|(} is an (overt) approximable localic ring.
\end{proposition}
\begin{proof}
 Under these assumptions $X$ is spatial (see \cite{PicadoPultr}).
 It is well known that the function space $\R^X$ is metrisable and in fact has the structure of a Banach space with the sup norm $\norm{f} = \sup_{x \in X} \abs{f(x)}$.
 We must show that $\forall f\in u.\ \exists g \in u.\ f \in \osat(g)$ for every open set $u$ in $\R^X$.
 We may take $f \in A$ and assume $u$ is a basic open $B_{\epsilon}(f)$ for some $\epsilon > 0$.
 
 The absolute value map and the meet and join operations on $\R$ induce corresponding pointwise operations on $A$. These satisfy
 $h \vee 0 = \tfrac{1}{2}(h + \abs{h})$ and $h \wedge 0 = \tfrac{1}{2}(h - \abs{h})$. Note that $f = (f \vee 0) + (f \wedge 0)$.
 We define a function $g = ([f - \tfrac{1}{2} \epsilon] \vee 0) + ([f + \tfrac{1}{2} \epsilon] \wedge 0)$ where by $\tfrac{1}{2} \epsilon$ we mean $\tfrac{1}{2} \epsilon \cdot 1 \in A$.
 Intuitively, $g$ is obtained by `thickening' the zero set of $f$.
 An easy calculation shows $\norm{f - g} = \tfrac{1}{2}\norm{\abs{f-\tfrac{1}{2}\epsilon} - \abs{f+\tfrac{1}{2}\epsilon}} \le
 \tfrac{1}{2}\norm{(f-\tfrac{1}{2}\epsilon) - (f+\tfrac{1}{2}\epsilon)} = \norm{\tfrac{1}{2}\epsilon} = \tfrac{1}{2}\epsilon$,
 where the inequality is from the reverse triangle inequality. So $g \in B_\epsilon(f)$.
 
 We claim that every saturated open $t$ containing $g$ also contains $B_{\epsilon/4}(f)$.
 (This immediately gives $f \in B_{\epsilon/4}(f) \subseteq \osat(g)$, as required.)
 We can show this by finding a $k \in A$ for each $h \in B_{\epsilon/4}(f)$ such that $kh = g$.
 It would then follow that $kh \in t$ and hence $h \in t$ by saturation. 
 
 Let $F = B_{\epsilon/4}(f)$. With a view towards making this argument constructive, we will define a continuous map $\psi\colon F \to A$ sending each $h$ to an appropriate $k$.
 
 Note that the restriction of the evaluation map $\ev\colon A \times X \to \R$ to $F \times f^{-1}([-\tfrac{1}{4}\epsilon, \tfrac{1}{4}\epsilon]\comp)$ factors through $\R^* = (-\infty, 0) \cup (0, \infty)$. %
 To see this consider a pair $(h,x)$ in the subspace $F \times f^{-1}([-\tfrac{1}{4}\epsilon, \tfrac{1}{4}\epsilon]\comp)$.
 Here $h$ satisfies $\norm{f - h} < \tfrac{1}{4}\epsilon$ and $x$ satisfies $\tfrac{1}{4}\epsilon < \abs{f(x)}$.
 Thus, $\abs{f(x)-h(x)} < \tfrac{1}{4}\epsilon < \abs{f(x)}$, or equivalently, $0 < \abs{f(x)} - \abs{f(x) - h(x)}$.
 But then $0 < \abs{f(x)} - \abs{f(x) - h(x)} \le \abs{f(x) - (f(x) - h(x))} = \abs{h(x)}$ by the reverse triangle inequality. So we may conclude that $h(x)$ lies in $\R^*$ as required.
 
 We can now define the map $\psi\colon F \to \R^X$ as follows.
 We specify that the uncurried form $\psi^\flat\colon F \times X \to \R$
 restricts to give the maps $\psi^\flat\vert_{F \times f^{-1}((-\tfrac{1}{2}\epsilon, \tfrac{1}{2}\epsilon))}\colon (h,x) \mapsto 0$ and
 $\psi^\flat\vert_{F \times f^{-1}([-\tfrac{1}{4}\epsilon, \tfrac{1}{4}\epsilon]\comp)}\colon (h,x) \mapsto (h(x))^{-1} \times g(x)$.
 To see that these agree on the overlap, we show that $g$ is zero on $f^{-1}(-\tfrac{1}{2}\epsilon, \tfrac{1}{2}\epsilon)$.
 Recall that $g = ([f - \tfrac{1}{2} \epsilon] \vee 0) + ([f + \tfrac{1}{2} \epsilon] \wedge 0)$. If $x$ lies in $f^{-1}(-\tfrac{1}{2}\epsilon, \tfrac{1}{2}\epsilon)$, we have
 $f(x) < \tfrac{1}{2}\epsilon$ and $f(x) > -\tfrac{1}{2}\epsilon$. So $f(x) - \tfrac{1}{2}\epsilon < 0$ and $f(x) + \tfrac{1}{2}\epsilon > 0$. We then quickly see that $g(x) = 0$, as required.
 
 Now if $h \in F$, it is easy to see that $h \times \psi(h) = g$, since $g(x) = 0$ on $f^{-1}(-\tfrac{1}{2}\epsilon, \tfrac{1}{2}\epsilon)$.
\end{proof}

In order to make this result constructive, we will replace the points $f$ and $g$ with overt sublocales $F$ and $G$.
The function $\psi$ will also acquire an additional argument, intuitively so it can vary for different choices of $f$ in $F$.
We rely on a number of results from \cite{Henry2016}. %
\begin{proposition}\label{prop:gelfand_approximable}
 Suppose $X$ is a compact regular locale. Then $A = \R^X$ is an overt approximable localic ring.
\end{proposition} %
\begin{proof}
 It is shown in \cite{Henry2016} that $A$ is overt and furthermore a localic Banach algebra. 
 We must show that $u \le \bigvee_{V \between u} \stabpos(V)$ for all $u \in \O A$.
 
 Given an overt sublocale $F$ of $A$ and a rational $\epsilon > 0$, we define $B_\epsilon(F)$\glsadd{BepsilonF} by the expression $\{h \colon A \mid \exists f\colon F.\ \norm{h-f} < \epsilon\}$ in the internal logic
 of the coherent hyperdoctrine of open sublocales. %
 By \cite{Henry2016}, we may express $u$ as $u = \bigvee_i F_i$ where each $F_i$ is a positive open sublocale such that $B_{\epsilon_i}(F_i) \le u$ for some rational $\epsilon_i > 0$.
 Furthermore, each $F_i$ can then in turn be expressed as a join of the open sets contained in it of diameter less than $\tfrac{1}{4}\epsilon_i$.\footnote{A sublocale $F$ has diameter less than a
 positive rational $\delta$ if $\,\vdash_{x,y \colon F} \norm{x-y} < \delta$ in the internal logic.}
 Thus, we may assume $\diam(F_i) < \frac{1}{4}\epsilon_i$ without loss of generality. %
 
 Consider $F = F_i$ and $\epsilon = \epsilon_i$ for some $i$.
 The idea is to construct an overt sublocale $G$ such that $G \between B_\epsilon(F)$ and such that every saturated open $t$ which meets $G$ contains $F$.
 If we can do this, we could deduce $F \le \bigwedge\{t \in \Sats A \mid G \between t\} = \stabpos(G)$. And clearly $G \between u$,
 since $G \between B_\epsilon(F) \le u$.
 Therefore, $F \le \bigvee_{V \between u} \stabpos(V)$. The result would then follow by taking the join over all such $F$.

 We now describe how to construct $G$. The intuition behind our construction is that we want to find functions which are sufficiently close to those in $F$,
 but which have inflated zero sets. The absolute value map and the meet and join operations on $\R$ induce `pointwise' operations on $A$ which satisfy
 $h \vee 0 = \tfrac{1}{2}(h + \abs{h})$ and $h \wedge 0 = \tfrac{1}{2}(h - \abs{h})$ in the internal logic. Note that the map defined by $f \mapsto (f \vee 0) + (f \wedge 0)$ in the internal logic is simply the identity.
 We define a locale map $\zeta\colon F \to A$ by $f \mapsto ([f - \tfrac{1}{2} \epsilon] \vee 0) + ([f + \tfrac{1}{2} \epsilon] \wedge 0)$.
 We set $G = (\Sloc \zeta)_!(F)$. %
 
 Let us show that $G \between B_\epsilon(F)$. Note that $G = (\Sloc \zeta)_!(F) \between B_\epsilon(F) \iff F \between \zeta^*(B_\epsilon(F))$, or in the internal logic,
 $\exists f\colon F.\ \zeta(f) \in B_\epsilon(F)$. So by the definition of $B_\epsilon(F)$, we must show $\exists f\colon F.\ \exists f'\colon F.\ \norm{\zeta(f) - f'} < \epsilon$.
 A straightforward calculation in the internal logic gives $\norm{\zeta(f) - f} = \tfrac{1}{2}\norm{\abs{f+\tfrac{1}{2}\epsilon} - \abs{f-\tfrac{1}{2}\epsilon}} \le
 \tfrac{1}{2}\norm{(f+\tfrac{1}{2}\epsilon) - (f-\tfrac{1}{2}\epsilon)} = \tfrac{1}{2}\norm{\epsilon} = \tfrac{1}{2}\epsilon$,
 where the inequality is from the reverse triangle inequality. (Here the judgemental inequality relation can be defined from equality and $\wedge$ in the usual way.) %
 Thus, $\norm{\zeta(f) - f} < \epsilon$. Now since $F$ is positive, we have $\exists f\colon F.\ \top$. Combining these yields $\exists f\colon F.\ \norm{\zeta(f) - f} < \epsilon$ and the desired result follows. %
 
 Now we show that every saturated open $t$ which meets $G$, contains $F$.
 The plan is to construct a localic map $\psi\colon F\times F \to A$ such that $\zeta(f) = h \times \psi(f,h)$ in the internal logic.
 
 Given such a map $\psi$, we can conclude the result as follows. %
 Since $G \between t$, we know $\exists f \colon F.\ \zeta(f) \in t$. %
 Now since $\zeta(f) = h \times \psi(f,h)$, we have $\exists f \colon F.\ h \times \psi(f,h) \in t$ for $h\colon F$.
 But $t$ being saturated means $\exists y\colon A.\ xy \in t \vdash_{x \colon A} x \in t$.  Putting these together, %
 we may conclude $\vdash_{h \colon F} h \in t$, which gives $F \le t$ as desired.
 
 We define $\psi\colon F \times F \to \R^X$ by specifying its uncurried form $\psi^\flat\colon F \times F \times X \to \R$
 on a covering pair of opens: $W = \{(f, h, x) \colon F \times F \times X \mid \abs{f(x)} > \tfrac{1}{4}\epsilon \}$
 and $Z = \{(f, h, x) \colon F \times F \times X \mid \abs{f(x)} < \tfrac{1}{2}\epsilon \}$,
 where we write $f(x)$ for $\ev(f,x)$.
 
 Note that restricted evaluation map $\ev(\pi_2,\pi_3)\vert_W$ factors through $\R^* = (-\infty, 0) \vee (0, \infty)$.
 This can be seen using the internal logic: consider a triple $(f,h,x)\colon W$.
 Because $\diam(F) < \tfrac{1}{4}\epsilon$, we have $\norm{f - h} < \tfrac{1}{4}\epsilon$.
 By the definition of the norm, this gives $\abs{f(x)-h(x)} < \tfrac{1}{4}\epsilon$. %
 Putting this together with $\tfrac{1}{4}\epsilon < \abs{f(x)}$, we find $\abs{f(x) - h(x)} < \abs{f(x)}$, or equivalently, $0 < \abs{f(x)} - \abs{f(x) - h(x)}$.
 But then $0 < \abs{f(x)} - \abs{f(x) - h(x)} \le \abs{f(x) - (f(x) - h(x))} = \abs{h(x)}$ by the reverse triangle inequality. %
 We may conclude that $\ev(h,x)$ lies in $\R^*$ as required
 
 We define $\psi^\flat\vert_Z\colon (f,h,x) \mapsto 0$ and $\psi^\flat\vert_W\colon (f,h,x) \mapsto (h(x))^{-1} \times \zeta(f)(x)$ in the internal logic, %
 where $(-)^{-1}\colon \R^* \to \R^*$ is the reciprocal operation and we implicitly factor $\ev(\pi_2,\pi_3)\vert_W$ through $\R^*$.
 To see that these agree on the overlap $W \wedge Z$, we show that $\zeta^\flat(\pi_1,\pi_3)$ is zero on $Z$.
 Recall that $\zeta(f) = ([f - \tfrac{1}{2} \epsilon] \vee 0) + ([f + \tfrac{1}{2} \epsilon] \wedge 0)$ in the internal logic. Now if $(f,h,x)$ lies in $Z$, we have
 $(f(x) < \tfrac{1}{2}\epsilon) \land (f(x) > - \tfrac{1}{2}\epsilon)$. Then $f(x) - \tfrac{1}{2}\epsilon < 0$ and $f(x) + \tfrac{1}{2}\epsilon > 0$. We then quickly see that $\zeta(f)(x) = 0$, as required.

 Finally, if $i$ is the inclusion of $F$ into $A$, we show $\mu_\times (i \pi_2, \psi) = \zeta \pi_1$.
 To see this we uncurry each expression and consider the two restrictions in the internal logic. We must show $0 = \zeta(f)(x)$ on $Z$ and
 $h(x) \times ((h(x))^{-1} \times \zeta(f)(x)) = \zeta(f)(x)$ on $W$. But we have already shown the first equality above and the second one follows
 immediately from associativity of multiplication and properties of the reciprocal.
\end{proof}

\begin{remark}
 It should be possible to prove a similar result when $X$ is only locally compact regular. It appears that Henry's result that $\R^X$ is overt can likely be modified to the locally compact case.
 As for the proof of approximability, we can no longer use the simplifications afforded by metrisability and must instead work directly with the presentation of the (co)exponential,
 but this is unlikely to be an insurmountable obstacle. %
\end{remark}

\begin{proposition}\label{prop:gelfand_expected_spectrum}
 If $X$ is a compact regular locale, then the localic spectrum of $\R^X$ is isomorphic to $X$ (as in Gelfand duality).\index{spectrum!Gelfand|)}\index{ring of continuous functions|)}
\end{proposition}
\begin{proof}
 By the localic Gelfand duality of \cite{Henry2016}, we know that the localic Gelfand spectrum of the C*-algebra $\C^X$ is isomorphic to $X$
 and it is shown in \cite{ConstructiveGelfandNonunital} that the opens of the Gelfand spectrum are in turn in bijection with the overt weakly closed ideals of $\C^X$.
 The use of complex numbers here is not essential and similar results hold for the real algebra $\R^X$ giving $\Idl(\R^X) \cong \O X$. (This can also be deduced from
 the complex case by constructing an isomorphism $\Idl(\R^X) \cong \Idl(\C^X)$.) %
 
 Now by \cref{prop:quantic_spectrum_of_approximable_semiring,prop:gelfand_approximable} we have that the quantic spectrum of $\R^X$ is given by $\Idl(\R^X)$.
 It remains to show that $\Idl(\R^X)$ is a frame, from which it follows that the localic spectrum is given by $\Rad(\R^X) \cong \Idl(\R^X) \cong \O X$. %
 
 The quantale $\Idl(\R^X)$ is a quotient of the quantale of monoid ideals $\MonIdl(\R^X)$ and so it is enough to show the latter is a frame. Here it is helpful to use \cref{prop:monoid_ideals_idempotent}
 from \cref{section:coexponential} so that we only need to show $f \in U \vdash_{f\colon \R^X} \exists k\colon \R^X.\ f^2 k \in U$ for all (basic) opens $U$.
 
 To show this we can proceed in a very similar manner to the proof of \cref{prop:gelfand_approximable}.
 Consider $B_\epsilon(F)$ as in that proof and construct $\zeta$ and $G$ in the same way. We then define a map $\psi\colon F \to \R^X$
 such that $\zeta(f) = f^2 \times \psi(f)$ by setting $\psi(f)(x) = 0$ when $\abs{f(x)} < \tfrac{1}{2}\epsilon$ and $\psi(f)(x) = (f(x)^2)^{-1} \times \zeta(f)(x)$ when $\abs{f(x)} > \tfrac{1}{4}\epsilon$.
 Here $\psi(f)$ plays the role of the $k$ required to deduce the result.
\end{proof}

We have now shown that each of our classes of localic semirings is approximable and their usual notions of spectra are special cases of our construction.
It is interesting to consider what other classes of examples there might be. In particular, it would be nice to have another interesting class of examples for the quantic spectrum,
since in all of our examples aside from discrete rings, the quantic spectrum is already a frame.

One such example that would be particularly interesting is the ring of smooth functions on a differentiable manifold. (Working classically) Whitney defines a natural topology on such a ring in \cite{whitney1948ideals}
and characterises the closed ideals. This topological ring is completely metrisable %
and thus by \cite{Isbell1972} it is a localic ring. It would be very interesting to see if this localic ring is approximable.

Let us end with some unusual examples.
\begin{example}
 Consider the locale $\lowerRealsNonneg$ of nonnegative lower reals --- this is given by the theory of (possibly empty) lower Dedekind cuts on the nonnegative rationals. This is a localic semiring with the usual notions of addition and multiplication.
 It is not hard to show this is overt and approximable and the quantic spectrum is $\Omega$ --- a single point. %
 If a discrete semiring $R$ has a quantic spectrum of $\Omega$, then it is a `Heyting semi-field'
 --- that is, a `local semiring' in which an element is not invertible if and only if it is zero. %
 However, even classically $\lowerRealsNonneg$ is not a semi-field, since it contains the point $\infty$, which has no inverse. It is, however, the frame of opens of the non-sober topological space obtained by
 equipping $[0, \infty)$ with the topology of lower semicontinuity. This is a `paratopological semi-field' --- that is, the subspace $(0,\infty)$ of nonzero elements is a topological monoid under multiplication
 for which every element has an inverse.
\end{example}
\begin{example}
 Consider the Sierpiński locale $\Srpnsk$ with the \emph{reverse} of its usual distributive lattice structure. This is overt and approximable,
 but its frame of radical ideals is isomorphic to the trivial frame and so its spectrum is empty, despite the semiring being nontrivial.
 At the start of this chapter we mentioned the possibility of a finer semiring spectrum. Such a spectrum would likely give more information in this case.
\end{example}

\begin{example}\label{ex:no_spectrum}
 The forgetful functor from $\Frm$ to $\Dcpo$ has a left adjoint, which induces a comonad on $\Frm$ and hence, a monad $\doublePow$ on $\Loc$.
 In \cite{VickersTownsendDoublePowerlocale} it is shown that $\doublePow$ sends a locale $X$ to the double exponential $\Srpnsk^{\Srpnsk^X}$,
 even in the case that $X$ is not exponentiable, so long as the intermediate exponential is taken in $\Set^{\Loc\op}$. We have shown that if $X$ is an exponentiable locale, then $X \cong \Spec(\Srpnsk^X)$.
 This suggests that perhaps $\Spec(\doublePow X)$ might fail to exist in the case that $X$ is not locally compact.
 
 An examination of the definition of the functor $\OPAI_R$ in the case that $R = \doublePow X$ suggests that the elements of $\OPAI_{\doublePow X}(Y)$
 correspond to internal lattice homomorphisms in $\Set^{\Loc\op}$ from $\Srpnsk^{\Srpnsk^X}$ to $\Srpnsk^Y$. %
 We believe that these should in turn correspond to maps from $Y$ to $\Srpnsk^X$ so that $\OPAI_{\doublePow X}$ is representable if and only if $X$ is exponentiable.
 However, we will not attempt to provide a full proof here. %
 
 Let us describe the frame of (radical) ideals of $\doublePow X$.
 The localic distributive lattice structure on $\doublePow X$ is given explicitly in \cite{vickers2004double}.
 Writing ${\boxtimes} a$ for the generators of $\O \doublePow X$, we have $\mu_\times\colon {\boxtimes} a \mapsto {\boxtimes} a \oplus {\boxtimes} a$,
 $\mu_+\colon {\boxtimes} a \mapsto {\boxtimes} a \oplus 1 \vee 1 \oplus {\boxtimes} a$, $\epsilon_1\colon {\boxtimes} a \mapsto 1$ and $\epsilon_0\colon {\boxtimes} a \mapsto 0$.
 An explicit calculation shows that the overt weakly closed ideals correspond to \emph{frame} homomorphisms from $\O \doublePow X$ to $\Omega$. %
 These are then in bijection with dcpo morphisms from $\O X$ into $\Omega$, which in turn correspond to the Scott-open subsets of $\O X$. %
\end{example}

\section{Duality and coexponentials}\label{section:coexponential}

\subsection{The spectrum and dualisability}\label{section:spectrum_via_duals}

An alternative formulation of \cref{prop:approximable_supercontinuous} is that an overt localic semiring $R$ is approximable if and only if its suplattice $\Sats R$ of saturated elements is dualisable.
This suggests that we might be able to use duals in $\Sup$ to better understand some aspects of the spectrum construction.
Most of our discussion will hold for general commutative localic monoids $M$.

Before we proceed with this line of inquiry, it will be worthwhile to consider the topological meaning of $\Sats M$ in more detail.
Let us first consider the discrete case. Here a saturated open is a subset $S \subseteq M$ such that $xy \in S \implies x \in S$.
We now recall the monoid-theoretic version of the divisibility preorder of a ring.

\begin{definition}
 The \emph{natural preorder} on a commutative monoid $M$ is defined by \[x \le y \iff \exists z \in M.\ xz = y.\]
 When $\le$ is a partial order, $M$ is said to be \emph{naturally partially ordered\index{monoid!naturally partially ordered|see {holoid}}} or a \emph{holoid\index{holoid|textbf}}.
\end{definition}
Notice that saturated sets\index{saturated!set} are precisely the downsets in this preorder.
Furthermore, the downsets of a preorder form the open sets of a topology --- the \emph{Alexandroff topology} of the opposite preorder.
Since the natural preorder on $M$ is a submonoid of $M \times M$, the posetal reflection of $M$ inherits the structure of a holoid.
In this way, $\Sats M$ can be thought of the topological embodiment of this holoid quotient of $M$ (with the reverse order).

Categorically, we can express this holoid quotient $M/\sim$ as the following coinserter in the 2-category of posets.
\begin{center}
  \begin{tikzpicture}[node distance=2.5cm, auto]
    \node (A) {$M \times M$};
    \node (B) [right of=A] {$M$};
    \node (C) [right of=B] {$M/\sim$};
    \draw[transform canvas={yshift=0.6ex},->] (A) to node {$\mu_\times$} (B);
    \draw[transform canvas={yshift=-0.6ex},->] (A) to [swap] node {$\pi_1$} (B);
    \draw[->>] (B) to node {} (C);
  \end{tikzpicture}
\end{center}
This suggests that in general we can describe $\Sats M$\index{frame of saturated opens}\index{saturated!open} as an inserter in $\Frm$,
\begin{center}
  \begin{tikzpicture}[node distance=2.5cm, auto]
    \node (A) {$\Sats M$};
    \node (B) [right of=A] {$\O M$};
    \node (C) [right of=B] {$\O M \oplus \O M ,$};
    \draw[transform canvas={yshift=0.6ex},->] (B) to node {$\iota_1$} (C);
    \draw[transform canvas={yshift=-0.6ex},->] (B) to [swap] node {$\mu_\times$} (C);
    \draw[right hook->] (A) to node {$\inclSat$\glsadd[format=(]{inclSat}} (B);
  \end{tikzpicture}
\end{center}
and this is easily seen to be the case. %

\begin{lemma}
 We obtain a functor $\Sats$\glsadd[format=(]{Sats} from the category $\CComon(\Frm)$ of cocommutative comonoids in $\Frm$ to $\Frm$ itself %
 and a natural transformation $\inclSat$\glsadd[format=)]{inclSat} from $\Sats$ to the forgetful functor from $\CComon(\Frm)$ to $\Frm$.
\end{lemma}
\begin{proof}
 Suppose $f\colon \O M \to \O M'$ is a morphism of cocommutative comonoids and consider the following diagram.
 \begin{center}
  \begin{tikzpicture}[node distance=2.5cm, auto]
    \node (A) {$\Sats M$};
    \node (B) [right of=A] {$\O M$};
    \node (C) [right of=B] {$\O M \oplus \O M$};
    \node (A') [below of=A] {$\Sats M'$};
    \node (B') [right of=A'] {$\O M'$};
    \node (C') [right of=B'] {$\O M' \oplus \O M'$};
    \draw[transform canvas={yshift=0.6ex},->] (B) to node {$\iota_1$} (C);
    \draw[transform canvas={yshift=-0.6ex},->] (B) to [swap] node {$\mu_\times$} (C);
    \draw[->] (A) to node {$\inclSat$} (B);
    \draw[transform canvas={yshift=0.6ex},->] (B') to node {$\iota'_1$} (C');
    \draw[transform canvas={yshift=-0.6ex},->] (B') to [swap] node {$\mu'_\times$} (C');
    \draw[->] (A') to node {$\inclSat'$} (B');
    \draw[->] (B) to [swap] node {$f$} (B');
    \draw[->] (C) to node {$f \oplus f$} (C');
    \draw[dashed,->] (A) to [swap] node {$\Sats f$} (A');
  \end{tikzpicture}
 \end{center}
 Since $f$ is a morphism of comonoids, both of the squares on the right-hand side commute.
 Thus $\mu'_\times f \inclSat = (f\oplus f) \mu_\times \inclSat \le (f\oplus f) \iota_1 \inclSat = \iota'_1 f \inclSat$ and so we obtain $\Sats f$ from the universal property of $\inclSat'$.
 A standard argument using uniqueness shows that $\Sats$ preserves composition and identity morphisms. The above diagram then also shows that the inserter maps $\inclSat$ form a natural transformation.
\end{proof}

It is not immediate that $\Sats M$ inherits a comonoid structure from $M$, but progress can be made when $M$ is overt. %

When $M$ is overt, both the coproduct injection $\iota_1$ and the inclusion $\inclSat\colon \Sats M \hookrightarrow \O M$ have left adjoints and we obtain the following diagram in $\Sup$.
\begin{center}
\begin{tikzpicture}[node distance=2.5cm, auto]
  \node (A) {$\Sats M$};
  \node (B) [right of=A] {$\O M$};
  \node (C) [right of=B] {$\O M \oplus \O M$};
  \draw[transform canvas={yshift=1.2ex},->] (B) to node {$\iota_1$} (C);
  \draw[transform canvas={yshift=0.0ex},<-] (B) to node {} (C);
  \draw[transform canvas={yshift=-1.2ex},->] (B) to [swap] node {$\mu_\times$} (C);
  \draw[transform canvas={yshift=0.75ex},right hook->] (A) to node {$\inclSat$} (B);
  \draw[transform canvas={yshift=-0.75ex},->>] (B) to node {} (A);
  \path (A) to node (incl) {} (B);
  \path (B) to node[yshift=0.6ex] (iota) {} (C);
  \node[scale=0.5] at ($(incl) - (0,4pt)$) {$\top$};
  \node[scale=0.5] at ($(iota) - (0,4pt)$) {$\top$};
\end{tikzpicture}
\end{center}
Here $\inclSat \inclSat_! = (\iota_1)_! \mu_\times$ and so this is like an order-enriched analogue of a split equaliser.
Indeed, the following lemma implies that this is an absolute $\Pos$-weighted limit in $\Sup$ --- that is, it is preserved by all order-enriched functors.

\begin{lemma} \label{prop:sat_inserter_absolute}
 Consider the following diagram in an order-enriched category $\Cvar$.
\begin{center}
\begin{tikzpicture}[node distance=2.5cm, auto]
  \node (A) {$S$};
  \node (B) [right of=A] {$X$};
  \node (C) [right of=B] {$Y$};
  \draw[transform canvas={yshift=1.2ex},->] (B) to node {$f$} (C);
  \draw[transform canvas={yshift=0.0ex},<-] (B) to node {} (C);
  \draw[transform canvas={yshift=-1.2ex},->] (B) to [swap] node {$g$} (C);
  \draw[transform canvas={yshift=0.75ex},->] (A) to node {$i$} (B);
  \draw[transform canvas={yshift=-0.75ex},->] (B) to node {} (A);
  \path (A) to node (incl) {} (B);
  \path (B) to node[yshift=0.6ex] (iota) {} (C);
  \node[scale=0.5] at ($(incl) - (0,4pt)$) {$\top$};
  \node[scale=0.5] at ($(iota) - (0,4pt)$) {$\top$};
\end{tikzpicture}
\end{center}
Suppose $gi \le fi$, $i_! i = \id_S$ and $ii_! \le f_! g$. Then $i\colon S \hookrightarrow X$ is an inserter of $g$ and $f$.
\end{lemma}
\begin{proof}
 Suppose $j\colon A \to X$ satisfies $gj \le fj$. We must find a unique map $k\colon A \to S$ such that $ik = j$.
 Such a map must be given by $k = i_! i k = i_! j$.
 
 Then we have $ik = ii_! j \ge j$. But also, $ik = ii_!j \le f_! g j \le f_! f j \le j$. Therefore, $ik = j$ as required.
\end{proof}

We can now show that when $M$ is overt, $\Sats M$ inherits the comonoid structure of $\O M$.
First note that since $\Sats M$ is a subframe of $\O M$, it is overt whenever $M$ is (by \cref{prop:weak_closure_overt}) and so
$\Sats$ restricts to a functor (which we will also call $\Sats$) from overt cocommutative comonoids to overt frames.
Since overtness is preserved by finite coproducts in $\Frm$ (by \cref{prop:product_of_overt_locales}), the former category is equivalent to the category $\CComon(\OLoc\op)$ of cocommutative comonoids in overt frames.

\begin{proposition}
 The functor $\Sats\colon \CComon(\OLoc\op) \to \OLoc\op$\glsadd[format=)]{Sats} preserves finite coproducts and consequently lifts to functor from $\CComon(\OLoc\op)$ to itself.
\end{proposition}
\begin{proof} %
 Let $(M, \mu_\times, \epsilon_1)$ and $(M', \mu'_\times, \epsilon'_1)$ be overt cocommutative comonoids in $\Frm$ and consider the following diagram in $\Frm$,
 where the centre row is the coproduct of the top row and the bottom row and the vertical morphisms are the appropriate coproduct injections.
 (Here we write $A^{\oplus 2}$ for $A \oplus A$.)
 \begin{center}
  \resizebox{\linewidth}{!}{%
  \begin{tikzpicture}[node distance=2.5cm, auto]
    \node (A) {$\Sats M$};
    \node (B) [right=2.6cm of A] {$\O M$};
    \node (C) [right=3.2cm of B] {$\O M^{\oplus 2}$};
    \node (A') [below of=A] {$\Sats M \oplus \Sats M'$};
    \node (B') [below of=B] {$\O M \oplus \O M'$};
    \node (C') [below of=C] {$\O M^{\oplus 2} \oplus \O M'^{\,\oplus 2}$};
    \node (D') [right=0.7cm of C'] {$(\O M \oplus \O M')^{\oplus 2}$};
    \node (Z') [left=0.7cm of A'] {$\Sats(\O M \oplus \O M')$};
    \node (A'') [below of=A'] {$\Sats M'$};
    \node (B'') [below of=B'] {$\O M'$};
    \node (C'') [below of=C'] {$\O M'^{\,\oplus 2}$};
    \draw[transform canvas={yshift=0.6ex},->] (B) to node {$\iota_1$} (C);
    \draw[transform canvas={yshift=-0.6ex},->] (B) to [swap] node {$\mu_\times$} (C);
    \draw[->] (A) to node {$\inclSat$} (B);
    \draw[transform canvas={yshift=0.6ex},->] (B') to node {$\iota_1 \oplus \iota'_1$} (C');
    \draw[transform canvas={yshift=-0.6ex},->] (B') to [swap] node {$\mu_\times \oplus \mu'_\times$} (C');
    \draw[->] (A') to node {$\inclSat \oplus \inclSat'$} (B');
    \draw[transform canvas={yshift=0.6ex},->] (B'') to node {$\iota'_1$} (C'');
    \draw[transform canvas={yshift=-0.6ex},->] (B'') to [swap] node {$\mu'_\times$} (C'');
    \draw[out=35,in=155,->] (B') to [pos=0.45] node {$\iota''_1$} ([xshift=17pt] D'.north west);
    \draw[out=-35,in=-155,->] (B') to [swap,pos=0.45] node {$\mu''_\times$} ([xshift=17pt] D'.south west);
    \draw[->] (A'') to node {$\inclSat'$} (B'');
    \draw[->] (B) to [swap] node {} (B');
    \draw[->] (C) to [swap] node {} (C');
    \draw[->] (A) to [swap] node {} (A');
    \draw[->] (B'') to [swap] node {} (B');
    \draw[->] (C'') to node {} (C');
    \draw[->] (A'') to node {} (A');
    \draw[->] (C') to node {$\sim$} (D');
    \draw[->] (C) to node {} (D');
    \draw[->] (C'') to [swap] node {} (D');
    \draw[->] (Z') to node{$\sim$} (A');
    \draw[<-] (Z') to node{} (A);
    \draw[<-] (Z') to [swap] node{} (A'');
  \end{tikzpicture}
 }
 \end{center}
 Note that coproducts of overt cocommutative comonoids are constructed from the coproducts of their underlying frames.
 Thus, by \cref{prop:sat_inserter_absolute} the central row in this diagram is an inserter in $\Sup$. %
 It is therefore also an inserter in $\Frm$, since the forgetful functor from $\Frm$ to $\Sup$ reflects 2-limits.
 Composing with the isomorphism on the right-hand side shows this is the inserter of $\mu''_\times$ and $\iota''_1$, which are the comultiplication on the coproduct comonoid $\O M \oplus \O M'$
 and a coproduct injection, respectively. But this is the defining property of $\Sats(\O M \oplus \O M')$ and we obtain the isomorphism on the left-hand side.
 
 It is clear that all of the squares and the right-hand triangles (all of which involve coproduct injections) commute. Thus the composite $\Sats M \to \Sats M \oplus \Sats M' \to \Sats(\O M \oplus \O M')$
 of the coproduct injection and the inverse of the isomorphism is the \emph{unique} map making the relevant diagram subcommute. This is precisely the definition of $\Sats$ applied to the coproduct injection
 $\O M \to \O M \oplus \O M'$. This together with a similar argument for the lower left-hand triangle then shows that $\Sats$ sends the colimiting cocone for the coproduct to a colimiting cocone.
 Together with an easy argument showing that $\Sats$ preserves the initial comonoid $\Omega$, this proves that $\Sats$ preserves finite coproducts as required.
 
 The result follows once we remember that $\CComon(\OLoc\op)$ has finite biproducts and so every object in $\CComon(\OLoc\op)$ has a unique internal cocommutative comonoid structure,
 which all morphisms in the category respect. This structure will then be preserved by any functor which preserves finite coproducts. %
\end{proof}
Explicitly, the counit and comultiplication on $\Sats M$ are given by the image of the counit and comultiplication on $\O M$.
By naturality, this gives that $\inclSat_M$ is a comonoid morphism. Thus $\inclSat$ becomes a natural transformation from the lifted functor to the identity functor.
This suggests that $\Sats$ and $\inclSat$ might give a coreflection.

\begin{definition}
 We call a commutative localic monoid $(M, \mu_\times, \epsilon_1)$ \emph{deflationary}\index{deflationary localic monoid|textbf}\index{monoid!deflationary localic|textbf} if $\mu_\times \le \iota_1$. %
\end{definition}

\begin{proposition}
 Overt deflationary comonoids in $\Frm$ form a coreflective subcategory of overt commutative comonoids where $\Sats$ is the coreflector and $\inclSat$ is the counit.
\end{proposition}
\begin{proof}
 We first show that $\Sats M$ is deflationary. Consider the following diagram.
 \begin{center}
  \begin{tikzpicture}[node distance=2.5cm, auto]
    \node (A) {$\Sats M$};
    \node (B) [right=3.0cm of A] {$\O M$};
    \node (A') [below of=A] {$\Sats M \oplus \Sats M$};
    \node (B') [below of=B] {$\O M \oplus \O M$};
    \draw[right hook->] (A) to node {$\inclSat$} (B);
    \draw[right hook->] (A') to node {$\inclSat'$} (B');
    \draw[transform canvas={xshift=-0.6ex},->] (A) to [swap]  node {$\Sats \mu_\times$} (A');
    \draw[transform canvas={xshift=0.6ex},->] (A) to node {$\iota_1^{\Sats M}$} (A');
    \draw[transform canvas={xshift=-0.6ex},->] (B) to [swap] node {$\mu_\times$} (B');
    \draw[transform canvas={xshift=0.6ex},->] (B) to node {$\iota_1^M$} (B');
  \end{tikzpicture}
 \end{center}
 Since $\inclSat$ is the inserter $\mu_\times$ and $\iota_1^M$, we have $\inclSat' \Sats \mu_\times = \mu_\times \inclSat \le \iota_1^M \inclSat = \inclSat' \iota_1^{\Sats M}$.
 But $\inclSat'$ reflects order and so $\Sats \mu_\times \le \iota_1^{\Sats M}$ as required. %
 
 Now suppose $N$ is an overt deflationary monoid and $f\colon \O N \to \O M$. Since $N$ is deflationary, $\inclSat_N$ is an isomorphism and we have the following diagram.
 \begin{center}
  \begin{tikzpicture}[node distance=2.5cm, auto]
    \node (A) {$\Sats N$};
    \node (B) [right of=A] {$\O N$};
    \node (A') [below of=A] {$\Sats M$};
    \node (B') [right of=A'] {$\O M$};
    \draw[double equal sign distance] (A) to node {$\inclSat_N$} (B);
    \draw[right hook->] (A') to node {$\inclSat_M$} (B');
    \draw[->] (B) to node {$f$} (B');
    \draw[->] (A) to [swap] node {$\Sats f$} (A');
  \end{tikzpicture}
 \end{center}
 Thus $f$ factors through $\inclSat_M$ to give $(\Sats f)\inclSat_N^{-1}$. Moreover, this factorisation is unique since $\inclSat_M$ is monic.
 This is precisely the universal property required for $\Sats$ to be the desired coreflector.
\end{proof}

We can now discuss the relationship between $\Sats M$ and the quantale $\MonIdl M$ of monoid ideals.
\begin{proposition}\label{prop:dual_of_SatsR}
 If $M$ is an overt commutative localic monoid, then we have $\MonIdl M \cong \hom(\Sats M, \Omega)$.
\end{proposition}
\begin{proof} %
 If $M$ is overt, then $\Sats M$ can be obtained by splitting the idempotent $(\iota_1)_! \mu_\times$ in $\Sup$.
 Applying the functor $\hom(-, \Omega)$, we find that $\hom(\Sats M, \Omega)$ is given by splitting the idempotent $\hom((\iota_1)_! \mu_\times, \Omega)$. %
 But by \cref{prop:overt_meeting_saturated}, $\hom((\iota_1)_! \mu_\times, \Omega)$ corresponds to the nucleus $a \mapsto a \cdot \top$ on $\SubOW(M) \cong \hom(M,\Omega)$.
 The resulting quotient is then $\MonIdl M$ by definition.
\end{proof}

\begin{corollary}
 If $\Sats M$ is dualisable, then $\MonIdl M$\index{quantale of monoid ideals} is its dual.
\end{corollary}

Note that the isomorphism $\MonIdl M \cong \hom(\Sats M, \Omega)$ preserves the quantale structure, since the comonoid structure on $\Sats M$ inherited as a subobject of $\O M$
and the quantale structure on $\MonIdl M$ inherited as a quotient of $\SubOW(M)$ are both unique.

When $\Sats M$ is dualisable, we can visualise the multiplicative structure on $(\Sats M)^* \cong \MonIdl M$ using string diagrams.
We will assume $M$ is deflationary for simplicity so that $\Sats M \cong \O M$.
As a warm up, we first give diagrams of the comonoid structure of $\O M$.

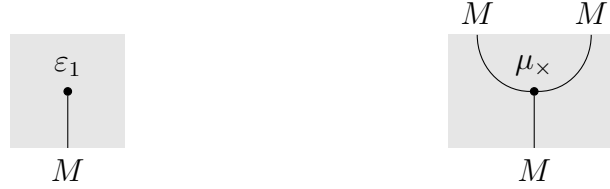
\begin{center}
\vspace{-3pt}
\begin{minipage}{0.4\textwidth}
\centering
\begin{tikzpicture}[scale=0.75]
\path coordinate[label=above:$\phantom{M}$] (top) ++(0,-1) coordinate[dot, label=above:$\epsilon_1$] (e) ++(0,-1) coordinate[label=below:$M$] (b);
\draw (e) -- (b);
\fillbackground{$(b) + (-1,0)$}{$(top) + (1,0)$};
\end{tikzpicture}
\end{minipage}
\begin{minipage}{0.4\textwidth}
\centering
\begin{tikzpicture}[scale=0.75]
\path coordinate[dot, label=above:$\mu_\times$] (mu)
 +(0,-1) coordinate[label=below:$M$] (b)
 +(-1,1) coordinate[label=above:$M$] (tl)
 +(1,1) coordinate[label=above:$M$] (tr);
\draw (tl) to[out=270, in=180] (mu.west) -- (mu.east) to[out=0, in=270] (tr)
      (mu) -- (b);
\fillbackground{$(tl) + (-0.5,0)$}{$(b) + (1.5,0)$};
\end{tikzpicture}
\end{minipage}
\end{center}

We obtain the monoid structure for $M^*$ by applying $\hom(-,\Omega)$.
The action of the first component of the internal hom functor on a morphism $f\colon A \to B$ is given by
the mate of the natural transformation $(-) \otimes f$ under the tensor-hom adjunction.
As a string diagram in the category of endofunctors on $\Sup$ we have:
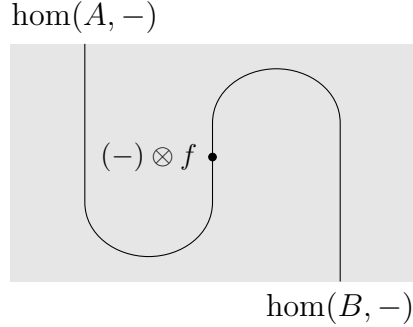
\begin{center}
\begin{tikzpicture}[scale=0.6,baseline={([yshift=-0.5ex]current bounding box.center)}]
\path coordinate (eta) ++(1.4,1.2) coordinate (a1) ++(0,1) coordinate[dot, label=left:\scalebox{0.9}{$(-) \otimes f$}] (f) ++(0,0.75) coordinate (a2) ++(1.4,1.2) coordinate (epsilon)
 ++(1.4,-1.2) coordinate (b) ++(0,-3.5) coordinate[label=below:${\hom(B, -)}$] (br)
 (eta) ++(-1.4,1.2) coordinate (c) ++(0,3.5) coordinate[label=above:${\hom(A, -)}$] (tl);
\draw (tl) -- (c) to[out=-90, in=180] (eta) to[out=0, in=-90] (a1) -- (a2) to[out=90, in=180] (epsilon) to[out=0, in=90] (b) -- (br);
\fillbackground{$(tl) + (-1.65,0)$}{$(br) + (1.65,0)$};
\end{tikzpicture}
\end{center}
Taking the $\Omega$ component gives the map $\hom(f,\Omega)$ and when $A$ and $B$ are dualisable, we may express this within the monoidal category $\Sup$.
Also note that the canonical natural transformation $\hom(B,\Omega) \otimes \hom(B',\Omega) \to \hom(B \otimes B',\Omega)$ making $\hom(-,\Omega)$ a lax monoidal functor
is an isomorphism when $B$ and $B'$ are dualisable. %

Hence we find the monoid structure on $M^*$ can be expressed as follows.
\begin{center}
\vspace{-3pt}
\begin{minipage}[t][][t]{0.4\textwidth}
\centering
\begin{tikzpicture}[scale=0.75]
\path coordinate (eta) ++(1,1) coordinate (a) ++(0,1) coordinate[dot, label=above:$\epsilon_1$] (e)
 (eta) ++(-1,1) coordinate (c) ++(0,2) coordinate[label=above:$M^*$] (tl)
 (eta) ++(1,-0.5) coordinate[label=below:$\phantom{M}$] (b);
\draw (tl) -- (c) to[out=-90, in=180] (eta) to[out=0, in=-90] (a) -- (e);
\fillbackground{$(b) + (0.5,0)$}{$(tl) + (-0.5,0)$};
\end{tikzpicture}
\end{minipage}
\begin{minipage}[t][][t]{0.4\textwidth}
\centering
\begin{tikzpicture}[scale=0.6]
\path coordinate[dot, label=above:$\mu_\times$] (mu)
 +(0,-1) coordinate (a)
 +(-1,1) coordinate (mtl)
 +(1,1) coordinate (mtr);
\path (a) ++(-1.0,-1) coordinate (eta) ++(-1.0,1) coordinate (c) ++(0,4.25) coordinate[label=above:$M^*$] (tl);
\path (mtr) ++(0.75,0.75) coordinate (epsilon) ++(0.75,-0.75) coordinate (d) ++(0,-3.5) coordinate[label=below:$M^*$] (br);
\path (mtl) ++(2.4,1.75) coordinate (epsilon2) ++(2.4,-1.75) coordinate (d2) ++(0,-3.5) coordinate[label=below:$M^*$] (br2);
\path (eta) ++(1.0,-0.5) coordinate (b);
\draw (mtl) to[out=270, in=180] (mu.west) -- (mu.east) to[out=0, in=270] (mtr)
      (mu) -- (a);
\draw (tl) -- (c) to[out=-90, in=180] (eta) to[out=0, in=-90] (a);
\draw (mtr) to[out=90, in=180] (epsilon) to[out=0, in=90] (d) -- (br);
\draw (mtl) to[out=90, in=180] (epsilon2) to[out=0, in=90] (d2) -- (br2);
\fillbackground{$(tl) + (-0.5,0)$}{$(br2) + (0.5,0)$};
\end{tikzpicture}
\end{minipage}
\end{center}

In \cref{section:spectrum_universal_property} we considered the functor $\overline{\OPAI}_R$ of open prime anti-ideals of a localic semiring $R$. %
We can define a similar functor $\overline{\OPMAI}_M\colon \Quant \to \Set$\glsadd[format=(]{OPMAIRoverline} which gives the open prime \emph{monoid} anti-ideals of a commutative localic monoid $M$.
\begin{definition}
The functor $\overline{\OPMAI}_M\colon \Quant \to \Set$\glsadd[format=)]{OPMAIRoverline} is defined on objects by
\begin{align*}
 \overline{\OPMAI}_M\colon Q \mapsto \big\{ u \in Q \oplus \O M \mid {}
 &  (Q \oplus \epsilon_1)(u) =  1, \\
 &  (Q \oplus \mu_\times)(u) = (Q \oplus\iota_1)(u) \cdot (Q \oplus\iota_2)(u) \big\}
\end{align*}
and acts on morphisms in the obvious way.
\end{definition}

Note that if $u \in \overline{\OPMAI}_M(Q)$, then in particular $(Q \oplus \mu_\times)(u) \le (Q \oplus\iota_1)(u)$.
When $M$ is overt, the inserter defining $\Sats M$ is preserved by $Q \oplus (-)$ and so at least in this case,
$\overline{\OPMAI}_M \cong \overline{\OPMAI}_{\Sats M}$.
Thus to understand this functor we can may restrict to the case $M$ is deflationary as above.

We wish to prove a result analogous to \cref{prop:quantic_spectrum_of_approximable_semiring} that $\overline{\OPMAI}_M$ is representable with
$\MonIdl M$ as the representing object whenever $M$ is overt and $\Sats M$ is a dualisable suplattice.
Based on that result, we expect the representing object to be given by $\upsilon = \bigvee_{a \in \O M} \jointcoz(a) \oplus a$.
Using \cref{prop:approximable_supercontinuous,prop:blacktriangle_vs_jointcoz} we can reformulate this as
$\upsilon = \bigvee_{s \in \Sats M} \llbracket s \lll (-) \rrbracket \oplus s$, which using the ideas of \cref{prop:dual_basis,prop:supercontinuous_vs_dualisable}
could then be written as $\upsilon = \bigvee_{x \in X} \sigma_x \otimes r_x$
for any dual basis for $\Sats M$ defined by $(r_x)_{x \in X}$ and $(\sigma_x)_{x \in X}$.
That is, the universal open prime anti-ideal should be given by $\eta(\top)$, where $\eta$ is the unit of the duality between $\Sats M$ and $\MonIdl M$.
This provides a further connection between the spectrum construction and dual suplattices.

\begin{theorem}\label{prop:spectrum_for_localic_monoids}
 Suppose $M$ is an overt commutative localic monoid and $\Sats M$ is a dualisable suplattice\index{dualisable!suplattice}\index{suplattice!dualisable}.
 Then $\overline{\OPMAI}_M$ is representable with representing object $(\Sats M)^*$ and universal element $((\Sats M)^* \otimes \inclSat)\eta(\top)$.
\end{theorem}

\begin{proof}
 We may assume $M$ is deflationary so that $\Sats M \cong \O M$.
 By duality, we have $\abs{Q \otimes M} \cong \Hom(\Omega, Q \otimes M) \cong \Hom(\Omega, M \otimes Q) \cong \Hom(\Omega \otimes M^*, Q) \cong \Hom(M^*, Q)$. %
 Moreover, the identity map $\id_{M^*} \in \Hom(M^*, M^*)$ corresponds to $\eta(\top)$ under this bijection.
 So to prove the desired result it is enough to show that if $(Q, m, e)$ is a two-sided quantale, the elements $u$ of $Q \otimes M$ satisfying
 $(Q \otimes \mu_\times)(u) = (Q \otimes\iota_1)(u) \cdot (Q \otimes\iota_2)(u)$ and $(Q \otimes \epsilon_1)(u) = 1$ correspond to the suplattice homomorphisms from $M^*$ to $Q$
 which respect the multiplicative structure.
 
 Consider a suplattice map $f\colon M^* \to Q$. The corresponding map $f^\sharp \in \Hom(\Omega, Q \otimes M)$ is shown in the following string diagram.
 
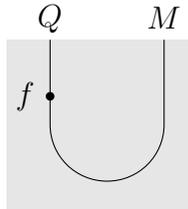
\begin{center}
 \vspace{-3pt}
 \begin{tikzpicture}[scale=0.75]
  \path coordinate (eta) ++(1,1) coordinate (a) ++(0,1.5) coordinate[label=above:$M$] (tr)
  (eta) ++(-1,1) coordinate (c) ++ (0,0.5) coordinate[dot,label=left:$f$] (f) ++(0,1) coordinate[label={[text depth=0.2ex] above:$Q$}] (tl) %
  (eta) ++(1,-0.5) coordinate (b);
  \draw (tl) -- (c) to[out=-90, in=180] (eta) to[out=0, in=-90] (a) -- (tr);
  \fillbackground{$(b) + (0.5,0)$}{$(tl) + (-0.75,0)$};
 \end{tikzpicture}
 \end{center}
 
 Now a suplattice map $f\colon M^* \to Q$ preserves the multiplicative unit if
 \\ 
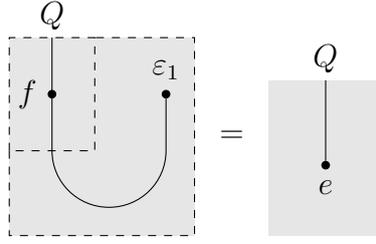
\vspace{-3pt}
 \begin{equation*}
 \begin{tikzpicture}[scale=0.75,baseline={([yshift=-0.5ex]current bounding box.center)}]
  \path coordinate (eta) ++(1,1) coordinate (a) ++(0,1) coordinate[dot, label=above:$\epsilon_1$] (e) %
  (eta) ++(-1,1) coordinate (c) ++ (0,1) coordinate[dot,label=left:$f$] (f) ++(0,1) coordinate[label={[text depth=0.2ex] above:$Q$}] (tl)
  (eta) ++(1,-0.5) coordinate[label=below:$\phantom{M}$] (b);
  \draw (tl) -- (c) to[out=-90, in=180] (eta) to[out=0, in=-90] (a) -- (e);
  \draw[dashed] ($(f)-(0.75,1)$) rectangle ($(f)+(0.75,1)$);
  \draw[dashed] ($(f)-(0.75,1)$) |- ($(b) + (0.5,0)$) |- ($(f)+(0.75,1)$);
  \fillbackground{$(b) + (0.5,0)$}{$(tl) + (-0.75,0)$};
 \end{tikzpicture}
 \enspace=\enspace
 \begin{tikzpicture}[scale=0.75,baseline={([yshift=-0.5ex]current bounding box.center)}]
  \path coordinate[dot, label=below:$e$] (e) ++(0,1.5) coordinate[label={[text depth=0.2ex] above:$Q$}] (t);
  \draw (e) -- (t);
  \fillbackground{$(t) + (-1,0)$}{$(e) + (1,-1.25)$};
 \end{tikzpicture}
 \end{equation*}
 while $(Q \otimes \epsilon_1)f^\sharp (\top) = 1$ if
 \\ 
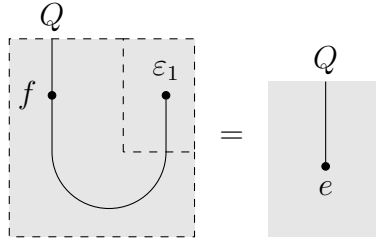
\vspace{-3pt}
 \begin{equation*}
 \begin{tikzpicture}[scale=0.75,baseline={([yshift=-0.5ex]current bounding box.center)}]
  \path coordinate (eta) ++(1,1) coordinate (a) ++(0,1) coordinate[dot, label=above:$\epsilon_1$] (e)
  (eta) ++(-1,1) coordinate (c) ++ (0,1) coordinate[dot,label=left:$f$] (f) ++(0,1) coordinate[label={[text depth=0.2ex] above:$Q$}] (tl)
  (eta) ++(1,-0.5) coordinate[label=below:$\phantom{M}$] (b);
  \draw (tl) -- (c) to[out=-90, in=180] (eta) to[out=0, in=-90] (a) -- (e);
  \draw[dashed] ($(e)-(0.75,1)$) rectangle ($(e)+(0.5,1)$);
  \draw[dashed] ($(e)+(-0.75,1)$) -| ($(f)-(0.75,1)$) |- ($(b) + (0.5,0)$) -- ($(e)+(0.5,-1)$);
  \fillbackground{$(b) + (0.5,0)$}{$(tl) + (-0.75,0)$};
 \end{tikzpicture}
 \enspace=\enspace
 \begin{tikzpicture}[scale=0.75,baseline={([yshift=-0.5ex]current bounding box.center)}]
  \path coordinate[dot, label=below:$e$] (e) ++(0,1.5) coordinate[label={[text depth=0.2ex] above:$Q$}] (t);
  \draw (e) -- (t);
  \fillbackground{$(t) + (-1,0)$}{$(e) + (1,-1.25)$};
 \end{tikzpicture}
 \end{equation*}
 These are the same diagram and hence these two conditions are clearly equivalent.
 
 On the other hand, the map $f$ preserves the multiplication if
 \\ 
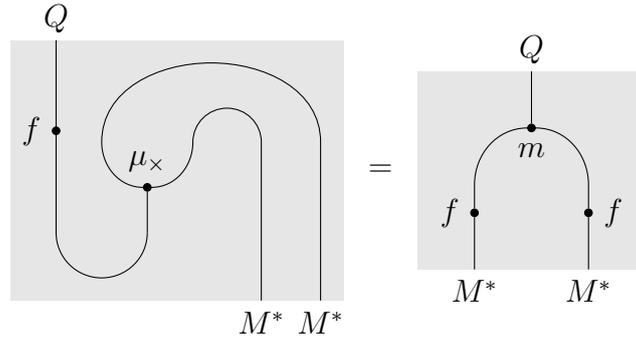
\vspace{-3pt}
 \begin{equation*} %
 \begin{tikzpicture}[scale=0.6,baseline={([yshift=-0.5ex]current bounding box.center)}] %
  \path coordinate[dot, label=above:$\mu_\times$] (mu)
  +(0,-1) coordinate (a)
  +(-1,1) coordinate (mtl)
  +(1,1) coordinate (mtr);
  \path (a) ++(-1.0,-1) coordinate (eta) ++(-1.0,1) coordinate (c) ++(0,2.25) coordinate[dot,label=left:$f$] (f) ++(0,2) coordinate[label={[text depth=0.2ex] above:$Q$}] (tl);
  \path (mtr) ++(0.75,0.75) coordinate (epsilon) ++(0.75,-0.75) coordinate (d) ++(0,-3.5) coordinate[label=below:$M^*$] (br);
  \path (mtl) ++(2.4,1.75) coordinate (epsilon2) ++(2.4,-1.75) coordinate (d2) ++(0,-3.5) coordinate[label=below:$M^*$] (br2);
  \path (eta) ++(1.0,-0.5) coordinate (b);
  \draw (mtl) to[out=270, in=180] (mu.west) -- (mu.east) to[out=0, in=270] (mtr)
        (mu) -- (a);
  \draw (tl) -- (c) to[out=-90, in=180] (eta) to[out=0, in=-90] (a);
  \draw (mtr) to[out=90, in=180] (epsilon) to[out=0, in=90] (d) -- (br);
  \draw (mtl) to[out=90, in=180] (epsilon2) to[out=0, in=90] (d2) -- (br2);
  \fillbackground{$(tl) + (-1.0,0)$}{$(br2) + (0.5,0)$};
 \end{tikzpicture}
 \,=\enspace %
 \begin{tikzpicture}[scale=0.75,baseline={([yshift=-0.5ex]current bounding box.center)}]
  \path coordinate[dot, label=below:$m$] (mu)
   +(0,1) coordinate[label={[text depth=0.2ex] above:$Q$}] (t)
   +(-1,-1) coordinate (mbl)
   +(1,-1) coordinate (mbr);
  \path (mbl) ++ (0,-0.5) coordinate[dot,label=left:$f$] (fl) ++ (0,-1) coordinate[label=below:$M^*$] (bl)
        (mbr) ++ (0,-0.5) coordinate[dot,label=right:$f$] (fr) ++ (0,-1) coordinate[label=below:$M^*$] (br);
  \draw (mbl) to[out=90, in=180] (mu.west) -- (mu.east) to[out=0, in=90] (mbr)
        (mbl) -- (fl) -- (bl)
        (mbr) -- (fr) -- (br)
        (mu) -- (t);
  \fillbackground{$(bl) + (-1.0,0)$}{$(t) + (2.0,0)$};
 \end{tikzpicture}
 \end{equation*}
 
 The condition $(Q \otimes \mu_\times) f^\sharp (\top) = (Q \otimes\iota_1) f^\sharp (\top) \cdot (Q \otimes\iota_2) f^\sharp (\top)$ is represented below.
 \\ 
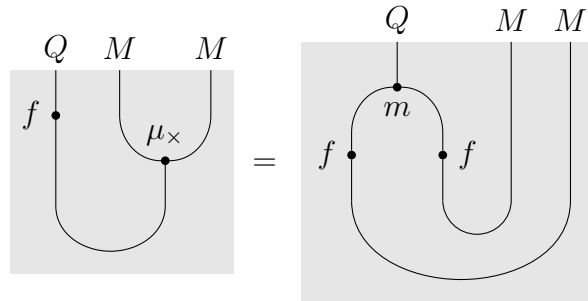
\vspace{-3pt}
 \begin{equation*}
 \begin{tikzpicture}[scale=0.6,baseline={([yshift=-1.5ex]current bounding box.center)}] %
  \path coordinate[dot, label=above:$\mu_\times$] (mu)
  +(0,-1) coordinate (a)
  +(-1,1) coordinate (mtl)
  +(1,1) coordinate (mtr);
  \path (mtl) ++(0,1) coordinate[label=above:$M$] (tl)
        (mtr) ++(0,1) coordinate[label=above:$M$] (tr);
  \path (a) ++(-1.2,-1) coordinate (eta) ++(-1.2,1) coordinate (c) ++(0,2) coordinate[dot,label=left:$f$] (f) ++ (0,1) coordinate[label={[text depth=0.2ex] above:$Q$}] (tll);
  \path (eta) ++(1.2,-0.5) coordinate (b);
  \draw (tl) -- (mtl) to[out=270, in=180] (mu.west) -- (mu.east) to[out=0, in=270] (mtr) -- (tr)
        (mu) -- (a);
  \draw (tll) -- (c) to[out=-90, in=180] (eta) to[out=0, in=-90] (a);
  \fillbackground{$(tll) + (-1.0,0)$}{$(b) + (1.5,0)$};
 \end{tikzpicture}
 \,=\enspace %
 \begin{tikzpicture}[scale=0.6,baseline={([yshift=-1.5ex]current bounding box.center)}]
  \path coordinate[dot, label=below:$m$] (mu)
   +(0,1) coordinate[label={[text depth=0.2ex] above:$Q$}] (t)
   +(-1,-1) coordinate (mbl)
   +(1,-1) coordinate (mbr);
  \path (mbl) ++ (0,-0.5) coordinate[dot,label=left:$f$] (fl) ++ (0,-1) coordinate (a1)
        (mbr) ++ (0,-0.5) coordinate[dot,label=right:$f$] (fr) ++ (0,-1) coordinate (a2);
  \path (a2) ++(0.75,-0.75) coordinate (eta2) ++(0.75,0.75) coordinate (c2) ++ (0,3.5) coordinate[label=above:$M$] (tr2);
  \path (a1) ++(2.4,-1.75) coordinate (eta1) ++(2.4,1.75) coordinate (c1) ++ (0,3.5) coordinate[label=above:$M$] (tr1);
  \draw (mbl) to[out=90, in=180] (mu.west) -- (mu.east) to[out=0, in=90] (mbr)
        (mbl) -- (fl) -- (a1)
        (mbr) -- (fr) -- (a2)
        (mu) -- (t);
  \draw (tr2) -- (c2) to[out=-90, in=0] (eta2) to[out=180, in=-90] (a2);
  \draw (tr1) -- (c1) to[out=-90, in=0] (eta1) to[out=180, in=-90] (a1);
  \fillbackground{$(eta1) + (-3.5,-0.5)$}{$(tr1) + (0.5,0)$};
 \end{tikzpicture}
 \end{equation*}
 
 These two equalities can be turned into each other by `bending the wires' and using the duality identities to `pull them straight'.
 These are inverse operations and so the two conditions are equivalent, as required.
\end{proof} %

\begin{remark}
 In the above discussion, we never actually use that $M$ is a frame. In particular, the same results hold when $M$ is an `overt' commutative comonoid in $\Quant$.
\end{remark}

By using dual suplattices we have been able to give a significantly simpler proof for \cref{prop:spectrum_for_localic_monoids} than for the analogous result for $\OPAI_R$.
It is natural to ask if this approach can be modified to work in that case too. One complication is that the additive structure of $R$ does not restrict to $\Sats R$.
I believe this problem can be surmounted by using a weaker form of addition as in the theory of \emph{hyperrings}, but unfortunately we do not have time to explore this here. %

Let us conclude this section with a simple result, which can be useful for showing when the quantic and localic spectra of a semiring coincide.

\begin{lemma}\label{prop:monoid_ideals_idempotent}
 Let $M$ be an overt commutative localic monoid. Consider the following conditions:
 \begin{enumerate}
  \item $\MonIdl M$ is a frame (i.e.\ the quantale product is given by meet) %
  \item $\Sats M$ is a localic semilattice
  \item $\id \le \nabla_M \mu_\times (\iota_1)_!\mu_\times$, where $\nabla_M\colon \O M \to \O M \oplus \O M$ is the codiagonal map.
 \end{enumerate}
 Conditions (ii) and (iii) are equivalent and imply condition (i). Furthermore, all three conditions are equivalent under the assumption that $\Sats M$ is dualisable.
\end{lemma}
\begin{proof}
 The equivalence of (ii) and (iii) comes from writing the condition for idempotence of the (co)multiplication on $\Sats M$ in terms of $\O M$ %
 and using the fact that $\Sats M$ is deflationary. %
 
 The implication (ii) $\Rightarrow$ (i) is immediate from the fact that $\MonIdl M \cong \hom(\Sats M,\Omega)$ respects the (co)monoid structures.
 When $\Sats M$ is dualisable, we obtain the reverse implication from $\Sats M \cong (\MonIdl M)^*$ in the same way.
\end{proof}
Observe that the third condition can be interpreted in the internal logic as saying $U(x) \vdash_{x\colon M} \exists y\colon M.\ U(x^2y)$ for all predicates $U$. %
In particular, this will hold for any (discrete) regular commutative monoid --- that is, any commutative monoid satisfying $\forall x\, \exists y.\ x = xyx$.

As a quotient of $\MonIdl R$, $\Idl(R)$ is also a frame under these conditions on the multiplicative monoid of a localic semiring $R$.
In particular, this explains the coincidence of the quantic and localic spectra for overt approximable localic distributive lattices.

\subsection{Presenting the coexponential}

In \cite{Niefield2016} Niefield shows that a quantale is coexponentiable if and only if its underlying suplattice is projective (or equivalently dualisable).
The argument is elegant, but extracting the precise form of the coexponential objects from it takes some work. We describe a presentation of the coexponential\index{coexponential!of quantales|(textbf} by generators and relations.

\begin{theorem}\label{prop:presentation_for_coexponential}
Consider quantales $Q$ and $A$. Suppose $A$ admits a dual basis given by $(r_x)_{x \in X}$ and $(\sigma_x)_{x \in X}$ and $Q$ has a presentation with generators $G$ and relations
$\bigvee S_j = t_j$ where $S_j \subseteq G, t_j \in G$ for $j \in J$,\; $a_k b_k = c_k$ where $a_k,b_k,c_k \in G$ for $k \in K$ and $e = 1$ for some $e \in G$.
Then the coexponential ${}^A Q$\glsadd{AQ} has a presentation with generators $(g \oslash \sigma_x)$\glsadd{oslash} for $g \in G, x \in X$ and the following relations:
\begin{enumerate}
 \item $g \oslash \sigma_x = \bigvee_{y \in X} \sigma_x(r_y) \cdot (g \oslash \sigma_y)$ for $g \in G, x \in X$, %
 \item $t_j \oslash \sigma_x = \bigvee_{s \in S_j}  (s \oslash \sigma_x)$ for $j \in J, x \in X$,
 \item $c_k \oslash \sigma_x = \bigvee_{y,y' \in X} \sigma_x(r_y r_{y'}) \cdot (a_k \oslash \sigma_y)(b_k \oslash \sigma_{y'})$ for $k \in K, x \in X$, %
 \item $e \oslash \sigma_x = \sigma_x(1) \cdot 1$ for $x \in X$.
\end{enumerate}
The co-evaluation map $\eta_Q\colon Q \to {}^A Q \oplus A$\glsadd{eta} is given by $\eta_Q(g_i) = \bigvee_{x \in X} (g_i \oslash \sigma_x) \oplus r_x$. \index{coexponential!of quantales|)} %
\end{theorem}
\begin{proof}
 We first show that $\eta_Q$ is a well-defined quantale homomorphism. The definition on generators defines a map $\eta'$ from the free quantale on $G$ to ${}^A Q \oplus Q$.
 For this to factor through $Q$, it needs to send the relations $\bigvee S_j = t_j$, $c_k = a_k b_k$  and $e = 1$ to identities. We quickly see this is the case
 for the $\bigvee$-type relations and for $e = 1$. Now consider $\eta'(c_k) = \bigvee_{x \in X} (c_k \oslash \sigma_x) \oplus r_x$. This equals
 $\bigvee_{x,y,y' \in X} \sigma_x(r_y r_{y'}) \cdot (a_k \oslash \sigma_y)(b_k \oslash \sigma_{y'}) \oplus r_x$ by relation (iii) of ${}^A Q$.
 Applying the first of the dualisable object identities, we get $\bigvee_{y,y' \in X} (a_k \oslash \sigma_y)(b_k \oslash \sigma_{y'}) \oplus r_y r_{y'}$, which then factors to give
 $\eta'(a_k)\eta'(b_k)$, as required.
 
 Suppose $f\colon Q \to R \oplus A$ is a map of quantales. We must show there is a unique map $\overline{f}\colon {}^A Q \to R$ making the following diagram commute.
 \begin{center}
 \begin{tikzpicture}[node distance=3.2cm, auto]
  \node (DQpD) {${}^A Q \oplus A$};
  \node (Q) [below of=DQpD] {$Q$};
  \node (RpD) [right of=Q] {$R \oplus A$};
  \node (DQ) [right=1.2cm of DQpD] {${}^A Q$};
  \node (R) at ($(DQ) + 0.6*(RpD)-0.6*(DQpD)$) {$R$};
  \draw[->] (Q) to node [swap] {$f$} (RpD);
  \draw[->] (Q) to node {$\eta_Q$} (DQpD);
  \draw[->, dashed] (DQ) to node {$\overline{f}$} (R);
  \draw[->] (DQpD) to node {$\overline{f} \oplus A$} (RpD);
 \end{tikzpicture}
 \end{center}
 If such an $\overline{f}$ exists, it must satisfy $f(g) = \bigvee_{x \in X} \overline{f}(g \oslash \sigma_x) \oplus r_x$ by commutativity.
 Applying $R\otimes \sigma_y$ to both sides we have
 $(R\otimes \sigma_y)(f(g)) = \bigvee_{x \in X} \sigma_y(r_x) \cdot \overline{f}(g \oslash \sigma_x) = \overline{f}\left(\bigvee_{x \in X} \sigma_y(r_x) \cdot (g \oslash \sigma_x)\right)
 = \overline{f}(g \oslash \sigma_y)$ by relation (i). This specifies $\overline{f}$ on generators, so we now only need to check that this indeed gives a well-defined map.
 
 Let $\widehat{f}$ be the map from the free quantale on the generators $(g \oslash \sigma_x)_{g \in G, x \in X}$ to $R$ defined in the same way as $\overline{f}$.
 It is easy to see that $\widehat{f}$ respects the relations (i) and (ii) on ${}^A Q$ using preservation of joins and, in the case of (i), the second dualisable object identity. %
 That it sends $e \oslash \sigma_x$ to $\sigma_x(1) \cdot 1$ is immediate. For (iii) note that, as above, we have
 \[\bigvee_{y \in X} \widehat{f}(ab \oslash \sigma_y) \oplus r_y = f(ab) = f(a)f(b) = \bigvee_{y,y' \in X} \widehat{f}(a \oslash \sigma_y)\widehat{f}(b \oslash \sigma_{y'}) \oplus r_y r_{y'}.\]
 Once again we apply $R \otimes \sigma_x$, giving $\widehat{f}(ab \oslash \sigma_x) = \bigvee_{y,y' \in X} \sigma_x(r_y r_{y'}) \cdot (a \oslash \sigma_y)(b \oslash \sigma_{y'})$, as desired. %
\end{proof}

The above construction concerns coexponentials in the category of all (commutative, but not necessarily two-sided) quantales.
Though it is easy to see that ${}^A Q$ is two-sided whenever $Q$ is %
and hence the construction works equally well for computing these coexponentials in $\Quant$.
However, I do not know if there are any additional two-sided quantales that only become coexponentiable in the subcategory. %

\subsection{Tangent bundles of quantales} \label{section:tangent_bundle}

We again consider the two-sided quantale $\Dvar = \langle \epsilon \mid \epsilon^2 = 0\rangle$.\glsadd{Dvar}
We can construct $\Dvar$ in two steps. First we find the corresponding two-sided idempotent semiring for the presentation.
This is given by $\{0, \epsilon, 1\}$ with the obvious operations. Then we take order ideals giving $\Dvar \cong \I \{0, \epsilon, 1\}$.

Incidentally, the same construction can be used for $\O\Srpnsk = \langle m \mid m^2 = m\rangle$. The underlying join-semilattice of the associated idempotent semiring $\{0, m, 1\}$
is isomorphic to $\{0, \epsilon, 1\}$ and so the underlying suplattices of $\Dvar$ and $\O\Srpnsk$ are isomorphic.

This suplattice has a dual basis given by elements $m = {\downarrow} m$ and $1$ and corresponding maps $\sigma_m = \llbracket (-) \ge m \rrbracket$ and $\sigma_1 = \llbracket (-) = 1 \rrbracket$.
Consequently, both $\Dvar$ and $\O\Srpnsk$ are coexponentiable.

The coexponential ${}^\Srpnsk Q$ is isomorphic to the copower of $Q$ by the poset $\{0 \le 1\}$. %
We denote the smaller of the two natural inclusions $Q \to {}^\Srpnsk Q$ by $\iota_0$\glsadd{iota_01} and the larger by $\iota_1$.

As in the category of commutative rings or in synthetic differential geometry, we can interpret the coexponential ${}^\Dvar Q$ as a kind of \emph{tangent bundle}\index{tangent bundle|(textbf} of $Q$. %

\begin{definition}
 Suppose $Q$ and $Q'$ are two-sided quantales and $i\colon Q \to Q'$ is a quantale homomorphism.
 We call a suplattice homomorphism $d\colon Q \to Q'$ an \emph{$i$-derivation}\index{derivation|(textbf} if
 $d \ge i$ and $d(xy) = i(x) d(y) \vee i(y) d(x)$.
\end{definition} %
\begin{remark}
 Notice that in the presence of the second condition, the first condition for being a derivation is equivalent to requiring $d(1) = 1$. %
\end{remark}

\begin{lemma}
 We can express ${}^\Dvar Q$ explicitly as the free quantale on $Q \sqcup \{\d x \mid x \in Q\}$\glsadd{dx} subject to relations such that the obvious map $\iota\colon Q \to {}^\Dvar Q$\glsadd{iota} is a quantale homomorphism,
 the map $\d\colon Q \to {}^\Dvar Q$ is an $\iota$-derivation\index{derivation|)}. The co-evaluation map is given by $\eta_Q (x) = \iota x \oplus 1 \vee \d x \oplus \epsilon$.
\end{lemma}
\begin{proof}
 We just need to apply \cref{prop:presentation_for_coexponential}.
 If we set $\iota(a) = a \oslash \sigma_1$ and $\d a = a \oslash \sigma_\epsilon$, then the relations given by the theorem show that $\iota$ and $\d$
 satisfy the desired properties and ${}^\Dvar Q$ is the free quantale on the generators subject to these conditions.
\end{proof}

We claimed in the introduction that the quantale of ideals of a ring $R$ contains additional `differential' information over the frame of radical ideals.
One example of this is the fact that the quantale $\Idl(R)$ contains a great deal of information about singularities of $R$. In particular, it encodes whether $R$ is nonsingular.
However, the usual algebro-geometric approaches to nonsingularity do not generalise well to other quantales. Here we attempt to formulate a preliminary definition of nonsingular quantale
which might stand a better chance of being appropriate more generally.

We now consider some examples of tangent bundles of quantales. We will look at simple quantales that might be thought of as analogues of algebraic curves.
Since the quantales can be rather complicated, we will describe the localic reflection of the quantales involved.
The resulting frames are all coherent and we will depict them by drawing Hasse diagrams of their lattices of compact elements.
(Under the assumption of excluded middle these finite lattices will coincide with the frames themselves.)

\begin{figure}[H]\label{fig:affine_tangent_bundle}
\paragraph{Affine line}
Take $Q_1 = \langle x \rangle$ --- the free two-sided quantale on one generator. \\
Then $\locrefl({}^\Dvar Q_1) = \langle x, \d x \mid x \le \d x \rangle_\Frm$ and $\locrefl({}^\Srpnsk Q_1) = \langle x_0, x_1 \mid x_0 \le x_1 \rangle_\Frm$. %
\bigskip

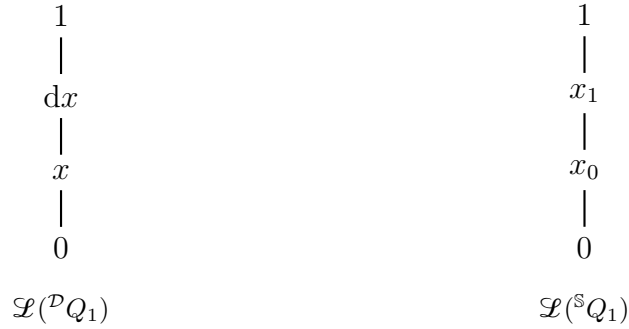
\captionsetup[subfigure]{font=footnotesize}
\centering
\begin{subfigure}[b]{0.45\textwidth}
\centering
\begin{tikzpicture}[auto]
  \matrix (nodes) [matrix of nodes, nodes in empty cells, ampersand replacement=\&, column sep=0.25cm, row sep=0.5cm]{
    $1$ \\
    $\d x$ \\
    $x$ \\
    $0$ \\
  };
  \foreach \i in {1,...,3} {
    \pgfmathtruncatemacro{\iplusone}{\i + 1}
    \draw [thick] (nodes-\i-1) -- (nodes-\iplusone-1);
  }
\end{tikzpicture}
\caption*{$\locrefl({}^\Dvar Q_1)$}
\end{subfigure}
\begin{subfigure}[b]{0.45\textwidth}
\centering
\begin{tikzpicture}[auto]
  \matrix (nodes) [matrix of nodes, nodes in empty cells, ampersand replacement=\&, column sep=0.25cm, row sep=0.5cm]{
    $1$ \\
    $x_1$ \\
    $x_0$ \\
    $0$ \\
  };
  \foreach \i in {1,...,3} {
    \pgfmathtruncatemacro{\iplusone}{\i + 1}
    \draw [thick] (nodes-\i-1) -- (nodes-\iplusone-1);
  }
\end{tikzpicture}
\caption*{$\locrefl({}^\Srpnsk Q_1)$}
\end{subfigure}
\addtocounter{figure}{-1} %
\end{figure}

\begin{figure}[H]
\paragraph{Cuspidal cubic}
Take $Q_2 = \langle x, y \mid y^2 = x^3 \rangle$. \\
Then $\locrefl({}^\Dvar Q_2) = \langle x, \d x, \d y \mid x \le \d x \wedge \d y \rangle$ and $\locrefl({}^\Srpnsk Q_2) = \langle x_0, x_1 \mid x_0 \le x_1 \rangle$. %
\bigskip

\captionsetup[subfigure]{font=footnotesize}
\centering
\begin{subfigure}[b]{0.45\textwidth}
\centering
\begin{tikzpicture}[auto]
  \matrix (nodes) [matrix of nodes, nodes in empty cells, ampersand replacement=\&, column sep=0.25cm, row sep=0.5cm]{
    \& $1$ \& \\
    \& $\d x \vee \d y$ \& \\
    $\d x$ \& \& $\d y$ \\
    \& $\d x \wedge \d y$ \& \\
    \& $x$\rlap{ $(= y)$} \& \\ %
    \& $0$ \& \\
  };

  \newlength\Nodewidth
  \newlength\EmptyLength
  \setlength{\EmptyLength}{9pt} %
  \foreach \i in {1,...,5} {
    \pgfmathtruncatemacro{\iplusone}{\i + 1}
    \foreach \j in {1,...,3} {
      \ifnodedefined{nodes-\i-\j}{
      \getwidthofnode{\Nodewidth}{nodes-\i-\j}
      \ifdimless{\EmptyLength}{\Nodewidth}{
        \foreach \k in {1,...,3} {
          \ifnodedefined{nodes-\iplusone-\k}{
          \getwidthofnode{\Nodewidth}{nodes-\iplusone-\k}
          \ifdimless{\EmptyLength}{\Nodewidth}{
              \draw [thick] (nodes-\i-\j) -- (nodes-\iplusone-\k);
          }{}}{}
        }
       }{}}{}
    }
  }
\end{tikzpicture}
\caption*{$\locrefl({}^\Dvar Q_2)$}
\end{subfigure}
\begin{subfigure}[b]{0.45\textwidth}
\centering
\begin{tikzpicture}[auto]
  \matrix (nodes) [matrix of nodes, nodes in empty cells, ampersand replacement=\&, column sep=0.25cm, row sep=0.5cm]{
    $1$ \\
    $x_1$\rlap{ $(= y_1)$} \\
    $x_0$\rlap{ $(= y_0)$} \\
    $0$ \\
    \\
  };
  \foreach \i in {1,...,3} {
    \pgfmathtruncatemacro{\iplusone}{\i + 1}
    \draw [thick] (nodes-\i-1) -- (nodes-\iplusone-1);
  }
\end{tikzpicture}
\caption*{$\locrefl({}^\Srpnsk Q_2)$}
\end{subfigure}
\addtocounter{figure}{-1} %
\end{figure}

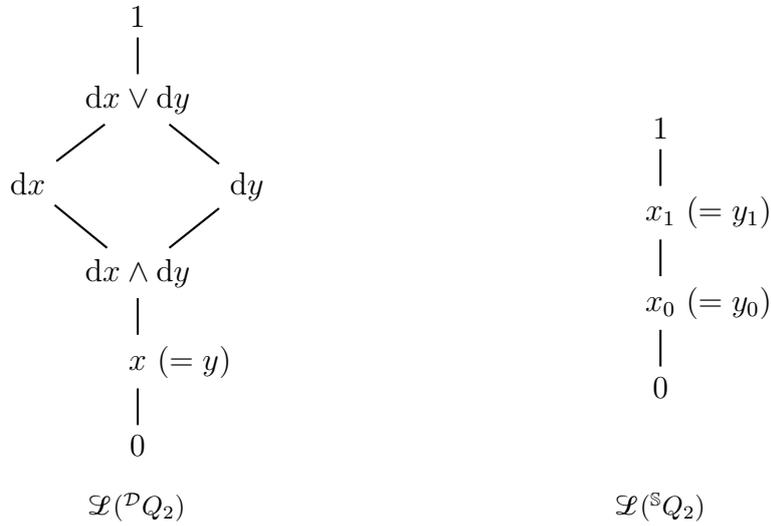
\begin{figure}[H]
\paragraph{Sierpiński space}
Take $Q_3 = \langle x \mid x = x^2 \rangle$. \\
Then $\locrefl({}^\Dvar Q_3) = \langle x \rangle$ and $\locrefl({}^\Srpnsk Q_3) = \langle x_0, x_1 \mid x_0 \le x_1 \rangle$.
\bigskip

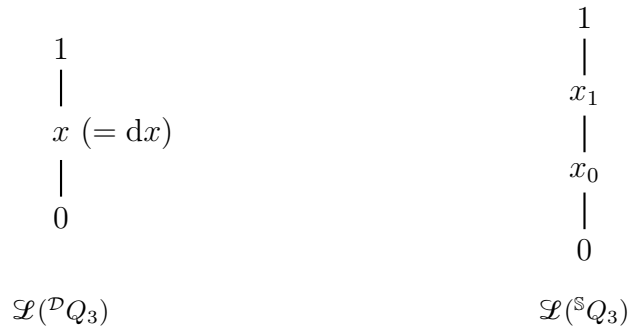
\captionsetup[subfigure]{font=footnotesize}
\centering
\begin{subfigure}[b]{0.45\textwidth}
\centering
\begin{tikzpicture}[auto]
  \matrix (nodes) [matrix of nodes, nodes in empty cells, ampersand replacement=\&, column sep=0.25cm, row sep=0.5cm]{
    $1$ \\
    $x$\rlap{ $(= \d x)$} \\
    $0$ \\
  };
  \foreach \i in {1,...,2} {
    \pgfmathtruncatemacro{\iplusone}{\i + 1}
    \draw [thick] (nodes-\i-1) -- (nodes-\iplusone-1);
  }
  \node[] () [below=8pt  of nodes-3-1] {};
\end{tikzpicture}
\caption*{$\locrefl({}^\Dvar Q_3)$}
\end{subfigure}
\begin{subfigure}[b]{0.45\textwidth}
\centering
\begin{tikzpicture}[auto]
  \matrix (nodes) [matrix of nodes, nodes in empty cells, ampersand replacement=\&, column sep=0.25cm, row sep=0.5cm]{
    $1$ \\
    $x_1$ \\
    $x_0$ \\
    $0$ \\
  };
  \foreach \i in {1,...,3} {
    \pgfmathtruncatemacro{\iplusone}{\i + 1}
    \draw [thick] (nodes-\i-1) -- (nodes-\iplusone-1);
  }
\end{tikzpicture}
\caption*{$\locrefl({}^\Srpnsk Q_3)$}
\end{subfigure}
\addtocounter{figure}{-1} %
\end{figure}
\index{tangent bundle|)}

Notice how in the affine case $\locrefl({}^\Dvar Q_1)$ and $\locrefl({}^\Srpnsk Q_1)$ are isomorphic. In the frame case ${}^\Dvar Q_3$ is `smaller than expected' and in the singular case ${}^\Dvar Q_2$ is
`larger than expected'. This suggests that we might be able to define the notion of nonsingular quantale\index{nonsingular!quantale|(textbf}\index{quantale!nonsingular|(textbf} by looking how these two coexponentials are related.
Though it seems like isomorphism is too much to ask for in general, as the quantale of ideals of $k[x,y]$ shows. %

Algebraic geometry gives some suggestions on how we might proceed. Given an irreducible closed set containing a point, we can induce a subspace of the tangent space at that point.
For instance, curves through a point induce (subspaces generated by) tangent vectors in the usual way. In nonsingular\index{nonsingular!ring|(}\index{ring!nonsingular|(} cases it appears that every subspace can be generated in this way,
but if the point is singular this is no longer the case.
For example, no subvariety generates a 1-dimensional subspace of the nodal cubic curve at the singular point (shown below). This is made precise in the following lemma
(which requires classical logic and a choice principle to ensure enough points exist).

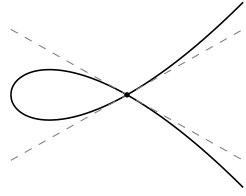
\begin{figure}[H]
\centering
\begin{tikzpicture}
\begin{axis}[hide axis,domain=-1:1,width=0.35\linewidth]
  \addplot[semithick,samples at={-1.0,-0.995,...,-0.65,-0.65,-0.6,...,1.0}] {+x*sqrt(x+1)};
  \addplot[semithick,samples at={-1.0,-0.995,...,-0.65,-0.65,-0.6,...,1.0}] {-x*sqrt(x+1)};
  \addplot[dashed,gray,samples=2] {x};
  \addplot[dashed,gray,samples=2] {-x};
  \fill[black] (0,0) circle(1pt);
\end{axis}
\end{tikzpicture}
\caption*{A nodal cubic curve with two tangent lines at the singularity.}
\end{figure}

\begin{proposition}\label{prop:nonsingular_schemes_via_quantic_tangent_bundles}
 Under the assumption of excluded middle\index{excluded middle} and the axiom of choice, %
 a Noetherian ring $R$ is nonsingular\index{nonsingular!ring|)}\index{ring!nonsingular|)} if and only if for every point $p$ in the spectrum $X$, every subspace $V \le T_p R$\index{tangent space} is induced by an irreducible closed set $Q \subseteq X$.
 Here the subspace induced by $Q$ is found by restricting to the stalk at $p$ to give a proper ideal $I$ in the local ring $(R_p, \m)$ and
 then taking the image of $I$ in $\m/\m^2$, quotienting and dualising. %
\end{proposition}
\begin{proof}
 We may assume $X$ is the spectrum of a local ring $(R_p, \m)$ and that $p = \m$, since any primes can be lifted from and restricted to the localisation as needed.

 Suppose that $R_p$ is a regular local ring of dimension $n$ and $V$ is a $d$-dimensional subspace of the tangent space. Dualising gives a quotient of $\m/\m^2$.
 A basis $\overline{x}_1,\dots,\overline{x}_d$ for the kernel of this quotient can then be extended to a basis $\overline{x}_1,\dots,\overline{x}_n$ of $\m/\m^2$.
 These lift to a regular system of parameters $x_1,\dots,x_n$ for $R_p$. Now $(x_1,\dots,x_d)$ is prime by \cref{prop:regular_sequence_generates_prime_ideal} and the subspace $V$ is
 induced by $(x_1,\dots,x_d)$ by construction.
 
 Conversely, suppose that every quotient, and hence, every subspace of $\m/\m^2$ is induced by a prime. Let $d = \dim \m/\m^2$. We will construct a chain of $d+1$ prime ideals
 $\m = p_0 > p_1 > \dots > p_d$. If $d \ne 0$, then $\m/\m^2$ has a maximal proper subspace, which is induced by a prime $q$ by assumption.
 By the Prime Ideal Theorem for rings there is a prime $p_1 \ne \m$ lying above $q \vee \m^2$. %
 Then $\dim p_1/p_1^2 \ge \dim p_1/\m^2 \ge d-1$. We can now localise at $p_1$ and repeat the process.
\end{proof}

Recall from the discussion after \cref{prop:map_to_D_and_primary_elements} that at least classically, quantale maps from $\Idl(R)$ to $\Dvar$ --- or equivalently, points of $\locrefl({}^\Dvar \Idl(R))$ ---
correspond to points of $\Rad(R)$ equipped with a subspace of the tangent space at that point.
So in \cref{prop:nonsingular_schemes_via_quantic_tangent_bundles} we have a way to map points of $\locrefl({}^\Srpnsk \Idl(R))$ to points of $\locrefl({}^\Dvar \Idl(R))$.
This suggests we should look for a frame homomorphism from $\locrefl({}^\Dvar Q)$ to $\locrefl({}^\Srpnsk Q)$.
The only maps from $\O\Srpnsk$ to $\Dvar$ factor through $\Omega$ and so the maps induced by these ignore
too much structure. The Yoneda lemma implies that any natural transformation from $\Sigma\locrefl({}^\Srpnsk(-))$ to $\Sigma\locrefl({}^\Dvar(-))$ must be induced in this way
and hence the map we construct is unlikely to be natural. %

To give a frame map $\dfrak\colon \locrefl({}^\Dvar Q) \to \locrefl({}^\Srpnsk Q)$\glsadd[format=(]{dfrak} is to give a quantale map $\overline{\dfrak}\colon {}^\Dvar Q \to \locrefl({}^\Srpnsk Q)$.
We can view both these locales as `bundles' over $\locrefl(Q)$ and our comparison map should respect this structure so that $\overline{\dfrak} \iota = \rad \iota_0$. %
Using the universal property of ${}^\Dvar Q$ we can uniquely specify $\overline{\dfrak}$ by a quantale homomorphism $i\colon Q \to \locrefl({}^\Srpnsk Q)$ and
an $i$-derivation $\partial\colon Q \to \locrefl({}^\Srpnsk Q)$. %
The condition $\overline{\dfrak}\iota = \rad\iota_0$ means that $i = \rad\iota_0$.
We can always take $\partial = i$, but this corresponds to the uninteresting map induced by the smallest map from $\O\Srpnsk$ to $\Dvar$.
The interest lies in how differentials $\d x$ map to nontrivial elements lying above $x_0$, while $\partial = i$ simply sends $\d x$ to $x_0$.
On the other hand, we do not want $\d x$ to be sent to anything too large either. We impose that $\partial \le \rad \iota_1$, which corresponds to the idea that the tangent vectors corresponding to
a subvariety should `lie inside it'. %
We take $\partial$ to be the largest map satisfying these conditions. The existence of such a map is guaranteed by the following lemma (and the fact that $\rad\iota_0$ is one map which satisfies the conditions).

\begin{lemma}
 Let $i\colon Q \to Q'$ be a quantale homomorphism between two-sided quantales. The set of $i$-derivations is closed under taking inhabited pointwise suprema.
\end{lemma}
\begin{proof}
 Let $(d_\alpha)_\alpha$ be a family of $i$-derivations and let $d = \bigvee_\alpha d_\alpha$. Certainly, $d$ preserves joins.
 Now observe that
 \begin{align*}
  d(xy) &= \bigvee_\alpha d_\alpha(xy) \\
        &= \bigvee_\alpha \left(i(x) d_\alpha(y) \vee i(y) d_\alpha(x)\right) \\
        &= \bigvee_\alpha i(x) d_\alpha(y) \vee \bigvee_\alpha i(y) d_\alpha(x) \\
        &= i(x) \bigvee_\alpha d_\alpha(y) \vee i(y) \bigvee_\alpha d_\alpha(x) \\
        &= i(x) d(y) \vee i(y) d(x),
 \end{align*}
 as required. Finally, we clearly have $d(1) = 1$ whenever the join is inhabited.
\end{proof}

\begin{definition}
 Let the comparison map $\dfrak_Q\colon \locrefl({}^\Dvar Q) \to \locrefl({}^\Srpnsk Q)$\glsadd[format=)]{dfrak} be the largest frame homomorphism such that $\overline{\dfrak}\iota = \rad\iota_0$ and $\overline{\dfrak}\d \le \rad\iota_1$.
\end{definition}

\begin{example}
It is easy to compute the map $\dfrak$ for the examples on page~\pageref{fig:affine_tangent_bundle}.
For the affine line analogue $Q_1$, we have $\partial(x) = x_1$; for the so-called cuspidal cubic $Q_2$, we find $\partial(x) = \partial(y) = x_1$;
and for Sierpiński space $Q_3$, we obtain $\partial(x) = x_0$.
\end{example}

We can now take inspiration from these examples and \cref{prop:nonsingular_schemes_via_quantic_tangent_bundles} to suggest a definition for a nonsingular quantale.
\begin{definition}
 A quantale $Q$ is \emph{nonsingular}\index{nonsingular!quantale|)}\index{quantale!nonsingular|)} if $\dfrak_Q\colon \locrefl({}^\Dvar Q) \to \locrefl({}^\Srpnsk Q)$ is an injection. %
\end{definition}

Note that, as expected, the quantales $Q_1$ and $Q_3$ are nonsingular according to this definition, but the cuspidal cubic analogue $Q_2$ is not.

In the case that $Q = \Idl(R)$ we expect the map $\dfrak_Q$ to correspond to the assignment of points to subspaces of the tangent space described in \cref{prop:nonsingular_schemes_via_quantic_tangent_bundles}.
It should then be possible to modify the proof to show that $R$ is singular if and only if $\Idl(R)$ is singular in the above sense.
Unfortunately, we have not been able to complete the proof at present and so this remains a conjecture.
\begin{conjecture}
 Let $R$ be a discrete Noetherian ring. Under the assumption of the axiom of choice, $\Idl(R)$ is nonsingular\index{nonsingular!quantale|textbf}\index{quantale!nonsingular|textbf} if and only if $R$ is.
\end{conjecture}
It could also be interesting to explore if the map $\dfrak_{\Idl(R)}$ can be used to say anything about the nature of the singularities in the case that $R$ is singular.

\bibliographystyle{abbrv}
\bibliography{bibliography}

\let\cleardoublepage\clearpage %
\glsaddallunused
\printglossary[type=symbols,style=myglossarystyle,title={Index of notation}]

\index{inhabited|seealso {positive locale}}
\index{semiring!localic|seealso {monoid, localic}}
\index{dense|seealso {strongly dense}}
\index{overt|seealso {sublocale, overt weakly closed}}
\index{prime|see {\textit{under} anti-ideal; ideal (discrete algebraic); filter}}
\def\igobble#1 {}
\index{congruence on a frame!zzzzz@\igobble |seealso {sublocale}}
\index{frame!zzzzz@\igobble |seealso {locale}}
\index{locale!zzzzz@\igobble |seealso {frame}}

\printindex

\end{document}